\documentclass[11pt,namelimits,sumlimits,a4paper]{amsart}
\usepackage{cite}

\usepackage{amssymb,amsmath}
\usepackage[mathscr]{eucal}
\newif\iffigureson
\figuresontrue

\usepackage{graphicx}
\usepackage{amsmath,amsthm}
\usepackage{graphics}
\usepackage{color}
\usepackage{epsfig}
\usepackage{amssymb,amsmath}
\usepackage[mathscr]{eucal}

\newcounter{mtheorem}
\newtheorem{mtheorem}[mtheorem]{Theorem}

\newtheorem{theorem}{Theorem}[section]
\newtheorem{lemma}[theorem]{Lemma}
\newtheorem{prop}[theorem]{Proposition}
\newtheorem{corollary}[theorem]{Corollary}

\newtheorem{definition}[theorem]{Definition}

\theoremstyle{remark}
\newtheorem{remark}[theorem]{Remark}

\numberwithin{equation}{section}

\newcommand{\abs}[1]{\lvert#1\rvert}

\newcommand{\ymax}{y_\text{max}}
\newcommand{\ymin}{y_\text{min}}
\newcommand{\taumax}{\tau_\text{max}}
\newcommand{\fmax}{f_\text{max}}
\newcommand{\pp}{\mathsf{p}}

\newcommand{\tb}{{\underline{\tau}}}
\newcommand{\mb}{{\underline{m}}}
\newcommand{\pt}{\pp_\tau}

\newcommand{\ptb}{\pp_{\tb}}
\newcommand{\pthat}{\widehat{\pp}_\tau}

\newcommand{\ptbhat}{\widehat{\pp}_{\tb}}
\newcommand{\bw}{\mathbf{w}}

\newcommand{\RRR}{\mathsf{R}}
\newcommand{\SSS}{\mathsf{S}}
\newcommand{\TTT}{\mathsf{T}}

\newcommand{\CCC}{\mathsf{C}}
\newcommand{\EEE}{\mathsf{E}}

\newcommand{\MMM}{\mathsf{M}}
\newcommand{\LLL}{\mathsf{L}}
\newcommand{\III}{\mathsf{I}}
\newcommand{\DDD}{\mathsf{D}}
\newcommand{\tbar}{\underline{\TTT}}

\newcommand{\mtilde}{\widetilde{\MMM}}
\newcommand{\sbartilde}{\widetilde{\underline{\SSS}}}
\newcommand{\tbartilde}{\widetilde{\underline{\TTT}}}
\newcommand{\tbartildeplus}{\widetilde{\underline{\TTT}}_+}
\newcommand{\tbartildeminus}{\widetilde{\underline{\TTT}}_-}

\newcommand{\ttilde}{\widetilde{\TTT}}

\newcommand{\ccong}{\CCC}

\newcommand{\tcheck}{\hat{\TTT}}
\newcommand{\tbarcheck}{\underline{\tcheck}}

\newcommand{\sgn}{\operatorname{sgn}}

\newcommand{\beq}{\begin{equation}}
\newcommand{\eeq}{\end{equation}}
\newcommand{\bea}{\begin{eqnarray}}
\newcommand{\eea}{\end{eqnarray}}
\newcommand{\cy}{{C}alabi-{Y}au\ }
\newcommand{\ka}{{K}\"ahler\ }

\newcommand{\slg}{special {L}agrangian\ }
\newcommand{\Slg}{{S}pecial {L}agrangian\ }

\newcommand{\Lag}{\operatorname{Lag}}
\newcommand{\Slag}{\operatorname{SL}}

\newcommand{\C}{\mathbb{C}}
\newcommand{\R}{\mathbb{R}}
\newcommand{\Q}{\mathbb{Q}}
\newcommand{\Z}{\mathbb{Z}}
\newcommand{\N}{\mathbb{N}}
\newcommand{\M}{\mathbb{M}}
\newcommand{\HH}{\mathbb{H}}
\newcommand{\CP}{\mathbb{CP}}
\newcommand{\Sph}{\mathbb{S}}
\newcommand{\ra}{\rightarrow}

\newcommand{\no}{\noindent}

\newcommand{\diag}{\operatorname{diag}}
\newcommand{\vol}{\operatorname{Vol}}
\newcommand{\dvol}{\operatorname{dv}}

\newcommand{\Real}{\operatorname{Re}}
\newcommand{\Imag}{\operatorname{Im}}

\newcommand{\hcf}{\operatorname{hcf}}

\newcommand{\Sym}{\operatorname{Sym}}
\newcommand{\Symtilde}{\widetilde{\Sym}}

\newcommand{\Diff}{\operatorname{Diff}}

\newcommand{\sun}{\text{SU(n)}}
\newcommand{\son}{\text{SO(n)}}
\newcommand{\sech}{\operatorname{sech}}
\newcommand{\unit}[1]{\text{U}(#1)}
\newcommand{\sunit}[1]{\text{SU}(#1)}
\newcommand{\sorth}[1]{\text{SO}(#1)}
\newcommand{\orth}[1]{\text{O}(#1)}
\newcommand{\glc}[1]{\text{GL}(#1,\C)}

\newcommand{\ort}{\text{O}}
\newcommand{\lsorth}[1]{\mathfrak{so}(#1)}
\newcommand{\lorth}[1]{\mathfrak{o}(#1)}
\newcommand{\lsunit}[1]{\mathfrak{su}(#1)}
\newcommand{\Per}{\operatorname{Per}}
\newcommand{\Perh}{\operatorname{Per}_{\frac{1}{2}}}
\newcommand{\Id}{\operatorname{Id}}
\newcommand{\Isom}{\operatorname{Isom}}
\newcommand{\Isoml}{\Isom_{\textrm{Lag}}}
\newcommand{\Isomsl}{\Isom_{\textrm{SL}}}
\newcommand{\Isomslpm}{\Isom_{\pm \textrm{SL}}}
\newcommand{\IsomslpmJ}{\Isomslpm^{J}}
\newcommand{\iso}{\mathfrak{iso}}
\newcommand{\lcm}{\operatorname{lcm}}

\newcommand{\Aut}{\operatorname{Aut}}
\newcommand{\dihedral}[1]{\mathbf{D}_{#1}}
\newcommand{\cyclic}[1]{\mathbf{C}_{#1}}

\newcommand{\symoff}[1]{\mathrm{Sym}_{\textrm{off}}(#1,\R)}
\newcommand{\stabpq}{\mathrm{Stab}_{p,q}}
\newcommand{\surj}{\twoheadrightarrow}

\newcommand{\Stab}{\mathrm{Stab}}

\newcommand{\cylpq}{\mathrm{Cyl}^{p,q}}
\newcommand{\cylpp}{\mathrm{Cyl}^{p,p}}
\newcommand{\cylone}{\mathrm{Cyl}^{1,n-1}}
\newcommand{\merpq}{\mathrm{Mer}^{p,q}}

\newcommand{\merone}{\mathrm{Mer}^{1,n-1}}
\newcommand{\Omegasf}{\mathsf{\Omega}}

\newcommand{\ki}{\mathsf{k}}
\newcommand{\gbar}{\overline{g}}

\newcommand{\hot}{{\text{h.o.t.}}}
\newcommand{\psig}{\pp_\sigma}
\newcommand{\psighat}{\widehat{\pp}_\sigma}
\newcommand{\varphit}{f_{\mathsf{t} }}
\newcommand{\phitilde}{{\widetilde{\phi}}}
\newcommand{\varphitilde}{{\widetilde{\varphi}}}
\newcommand{\bz}{\mathbf{z}}

\begin{document}

\title{Twisted products and $\sorth{p}\times \sorth{q}$-invariant special Lagrangian cones}

\author[M.~Haskins]{Mark~Haskins}
\address{Department of Mathematics, South Kensington Campus, Imperial
College London}
\email{m.haskins@imperial.ac.uk}

\author[N.~Kapouleas]{Nikolaos~Kapouleas}
\address{Department of Mathematics, Brown University, Providence,
RI 02912} \email{nicos@math.brown.edu}


\date{\today}


\keywords{Differential geometry, isolated singularities, calibrated geometry, minimal submanifolds}

\begin{abstract}
We construct $\sorth{p} \times \sorth{q}$-invariant special Lagrangian (SL) cones 
in $\C^{p+q}$. These SL cones are natural higher-dimensional analogues of the $\sorth{2}$-invariant 
SL cones constructed previously by MH and used in our gluing constructions 
of higher genus SL cones in $\C^{3}$. We study in detail the geometry of these $\sorth{p}\times \sorth{q}$-invariant 
SL cones,  in preparation for their application to our higher dimensional special Legendrian gluing constructions. 
In particular the symmetries of these cones and their asymptotics near the spherical limit are analysed.

All $\sorth{p} \times \sorth{q}$-invariant SL cones arise from a more general construction of independent interest 
which we call the special Legendrian twisted product construction. Using this 
twisted product construction and simple variants of it we can construct a constellation of new 
special Lagrangian and Hamiltonian stationary cones in $\C^{n}$. We prove the following theorems: 
A. there are infinitely many topological types of special Lagrangian and Hamiltonian stationary cones in 
$\C^{n}$ for all $n\ge 4$, B. for $n\ge 4$ special Lagrangian and Hamiltonian stationary torus cones in $\C^{n}$ can occur in 
continuous families of arbitrarily high dimension and C. for $n\ge 6$ there are infinitely many 
topological types of special Lagrangian and Hamiltonian stationary cones in $\C^{n}$ that can occur in 
continuous families of arbitrarily high dimension.
\end{abstract}

\maketitle

\newif\ifcommentson
\commentsonfalse

\ifcommentson
\small
\textbf{General:}
\begin{enumerate}
\item Finalise title of this paper and others in series
\item Finalise Organisation of Paper section
\item Acknowledgements?
\item \ref{P:X:tau}: do we need to include a decay statement? Curvature properties of $X_{\tau}$?
\end{enumerate}

\small
\textbf{Mark to do:}
\begin{enumerate}
\item Legn nhd theorem
\item Marked spheres: finalise defn, does a marked sphere determine an SL embedding?
\item Lagn catenoid discussion and the total angle change $\pi/n$ as an integral.
\item \ref{E:w:ic:p:eq:1} Discuss geometric reasons for different choices of initial data for ODEs in $p=1$, $p\neq 1$?
\end{enumerate}

\small
\textbf{Nicos to do:}
\begin{enumerate}
\item Include discussion of marked SL spheres or not? Yes, do this in $\tau \ra 0$ section.
\item Prove carefully results about the necks approximating Lagrangian catenoids.
\item Section \ref{S:asymptotics}:
Finish detailed asymptotics for $\pt$, $\pthat$ and first derivatives
\end{enumerate}

\small
\textbf{Nicos request list:}
\begin{enumerate}
\item Non $\sunit{n}$ torques
\end{enumerate}

\small
\textbf{Done:}
\begin{enumerate}
\item (1/2/10) Wrote Abstract
\item (8/2/10) Tried to use UK spellings consistently throughout
\item (17/3/10) Merged section on heuristic discussion of geometry as $\tau \ra 0$ with Nicos discussion
\item (8/2/10) Corrected Figures 1 and 2 so that don't use $p$ and $q$ with $p>q$
\item (24/11/09) Section \ref{S:slg}. Added definitions of Hamiltonian stationary in $\C^{n}$, contact stationary 
in $\Sph^{2n-1}$; HS cones and CS links etc. Discussed relevance of HS cones to Schoen-Wolfson programme in 2d.
\item (1/2/10) Added explicit form of curves giving links of 2d HS cones in \ref{E:hs:cones:2d}.
\item (23/11/09) Section \ref{S:twist:prod}: Checked Lemma 3.1 for signs in the statement.
\item Added Remark \ref{R:twist:def:p:eq:1} on twisted products in the degenerate case where one factor is a point.
\item Added Remark \ref{R:product:cones} on product cones. Added discussion of products of Euclidean spaces.
\item (24/11/09) Following Prop \ref{P:leg:triple}: added analogous Prop \ref{P:leg:twist:p:eq:1} discussing degenerate case $p=1$ case
\item (1/2/10): Added subsubsection on twisted products of contact stationary immersions
\item (1/2/10): reworked consequences of the closed twisted product constructions, restated as Theorems A, B and C with proofs 
given.
\item Section \ref{S:twist:prod}: Prop \ref{C:so:invariant:leg} $p=2$ okay.
\item (24/11/09) Section \ref{S:twist:prod}: added several concrete theorems (A--C) about SL cones with arbitrary topology in high dimension and in cts families of arbitrarily high dim and similar results for Ham stationary cones.
\item Section \ref{S:sop:soq} intro: added reference to Anciaux in section on work of other authors.
\item Section \ref{S:sop:soq} p15: Discussion of isotropic orbits added.
\item Section \ref{S:sop:soq} p16, added reference for Joyce theorem.
\item (14/10/09) Finished proof of Prop \ref{P:odes:p:n}
\item (21/11/09) Added Defn \ref{D:M:period} and Lemma \ref{L:M:periods} on periods and half-periods of $\MMM$
\item (15/10/09) Section  \ref{S:sop:soq}: Added Figure \ref{fig:fplot} to illustrate $\ymax$, $\ymin$ the function $f(y)$ 
\item (14/10/09) Finished proof of \ref{P:y}: now discuss periodicity of $y_{\tau}$
\item (15/10/09) Added Figure \ref{fig:ctau} to show periodic orbits for $0<\abs{\tau}<\taumax$ and the geometry of the 
singular points.
\item (14/10/09) \ref{E:y:maxmin:asym}: Add error terms in $\tau \ra 0$ asymptotics for $\ymax$ and $\ymin$
\item (21/11/09) Added comment on $G$-invariant versus $G$-equivariant.
\item Lemma \ref{P:w:tau}: added proof of smooth dependence on $\tau$
\item Finish symmetries of $\bw_{\tau}$ in Lemma \ref{L:y:w:symmetry:p:neq:1}
\item (21/11/09) Altered definition of $\Perh(\bw)$ and $\Per(\bw)$ in Defn \ref{D:w:period} and discussion following.
\item (23/11/09) Rewrote proof of Lemma 5.43 on self-intersection points of $X_{\tau}$; it now uses only the ODEs for $\bw_{\tau}$ rather than the Cartan-Kahler theory as previously.
\item (8/12/09) Rewrote proof of Prop \ref{P:isom:pb} more carefully.
\item Wrote correct definition of (internal) central product of groups. (Checked with Liebeck).
\item  Wrote out structure of $\Sym(X_{\tau})$ etc for the case $p\neq1$
\item (24/11/09) Finished proof of Lemma \ref{L:discrete:str} on the structure of $\mathbf{D}$.
\item (9/12/09) Added Remark \ref{R:isom:pb:ex} on enhanced isometry group when $\tau=\taumax$ and situation for $\tau=0$.
\item 28/12/09 Added subsubsection on the rotational period of $X_{\tau}$ and its relation to the rotational period of $\bw_{\tau}$
\item (9/12/09) Added $p>1$ analogues of Remark \ref{R:pm:holo} on the $\pm$ holomorphic isometries in $\Symtilde(X_{\tau})$.
\item Derive results about $\Per(X_{\tau})$ from the structure results for  $\Sym$ and $\Symtilde$
\item (8/01/10) Added section on ``Waists, Bulges and Approximating Spheres''
\item (5/10/09) Added statement and calculation of torques for $X_{\tau}$ 
\item (10/10/09) Added proof of torque vanishing for off-diagonal $\ki$
\item (21/01/10) Added Remark \ref{R:flux:reposn} on behaviour of torques under repositioning
\item Restructured Section \ref{S:embedded} and added Lemmas \ref{L:symtilde:p:eq:1}, \ref{L:symtilde:p:neq:q} and \ref{L:symtilde:p:eq:q} and Corollaries \ref{C:xtau:period:p:eq:1}, 
\ref{C:xtau:period:p:neq:1} and \ref{C:xtau:period:p:eq:q}.
\item (12/3/10) Section  \ref{S:asymptotics}: mention consequences of $\pthat$ asymptotics for closed $(p,q)$-twisted SL curves.
\item  added proofs of Lemmas \ref{L:embed:p:eq:1} and \ref{L:embed:p:eq:q}
\end{enumerate}

\normalsize
\else
\fi

\section{Introduction}
\label{S:intro}
\nopagebreak
Let $M$ be a \cy manifold of complex dimension $n$ with \ka form $\omega$ and non-zero parallel
holomorphic $n$-form $\Omega$. Suitably normalised $\Real{\Omega}$ is a calibrated
form whose calibrated submanifolds are called special Lagrangian (SL) submanifolds \cite{harvey:lawson}.
Beginning in the mid 1990s moduli spaces of SL submanifolds appeared in string theory
\cite{becker,syz,hitchin:moduli,hitchin:slg:lectures,joyce:syz}.
On physical grounds, Strominger, Yau and Zaslow argued that a \cy manifold $M$
with its mirror partner $\widehat{M}$ admits a (singular) fibration by SL tori,
and that $\widehat{M}$ should be obtained by compactifying the dual fibration \cite{syz}.
To make these ideas rigorous one needs control over the singularities and compactness
properties of families of SL submanifolds.
In dimensions three and higher these properties
are not well understood. As a result there has been considerable recent interest
in singular SL subvarieties
\cite{butscher,gross:slg:examples,gross:slgfib1,gross:slgfib2,haskins:thesis,haskins:slgcones,haskins:complexity,joyce:sing:survey,y:lee,schoen:wolfson,schoen:wolfson:1996}.

One natural class of singular SL $n$-folds is the class of SL varieties with isolated
conical singularities \cite{joyce:sing:survey}. These are compact SL $n$-folds
of \cy manifolds which are singular at a finite number of points, near each of which
they resemble asymptotically some SL cone $C$ in $\C^n$ with the origin as the only singular
point of $C$.
This motivates the recent interest in constructing SL cones in $\C^n$ with an
isolated singularity at the origin
\cite{haskins:thesis,haskins:slgcones,mcintosh:slg,carberry:mcintosh,joyce:symmetries}.
Until recently few examples of such SL cones were known.
Between 2000 and 2005, many new families of examples were constructed using techniques from equivariant differential 
geometry, symplectic geometry and integrable systems
\cite{haskins:thesis,haskins:slgcones,joyce:symmetries,mcintosh:slg,carberry:mcintosh}.
In contrast, PDE techniques were not used in the construction of new special Lagrangian cones during this period
(see Joyce's series of papers \cite{joyce:u1:i,joyce:u1:ii,joyce:u1:iii} on $\unit{1}$-invariant special Lagrangians in $\C^3$ for PDE constructions of other special Lagrangian $3$-folds.)

In 2007 we used geometric PDE gluing methods to construct a plethora of
new special Lagrangian cones in $\C^3$ \cite{haskins:kapouleas:invent}. 
The basic building blocks of our gluing constructions were the $\sorth{2}$-invariant 
special Lagrangian cones first constructed in \cite{haskins:slgcones,haskins:thesis} and
studied further in \cite{haskins:kapouleas:invent}. It is very natural to ask if 
one can also use geometric PDE gluing methods to construct a variety of new 
special Lagrangian cones in $\C^{n}$ for $n>3$. Recently we achieved such 
a construction (as announced in our detailed survey article \cite{haskins:kapouleas:survey}). 
The current paper is the first in a series of 
three \cite{haskins:kapouleas:hd2,haskins:kapouleas:hd3}
in which the constructions announced and summarised in our survey 
\cite{haskins:kapouleas:survey} are carried out in detail.

\subsubsection*{Higher dimensional building blocks for gluing constructions} 
The main purpose of the current article is to construct the appropriate higher 
dimensional building blocks, analogous to the $\sorth{2}$-invariant SL cones 
we used in \cite{haskins:kapouleas:invent}, and to describe them at the level of detail needed 
for the gluing applications in \cite{haskins:kapouleas:hd2,haskins:kapouleas:hd3}. 
Appropriate higher dimensional 
building blocks are provided by $\sorth{p} \times \sorth{q}$-invariant SL 
cones in $\C^{p+q}$. When $(p,q)=(1,2)$, these are exactly the $\sorth{2}$-invariant
SL cones we used as building blocks in our three-dimensional constructions. 
$\sorth{p} \times \sorth{q}$-invariant SL cones share many features 
of the $\sorth{2}$-invariant SL cones in $\C^{3}$. 

\subsubsection*{A variety of building blocks}
Because in dimension $n$ 
there are $[\tfrac{n}{2}]$ possible building blocks available, there are many 
different ways to attempt gluing constructions of higher dimensional special Legendrian 
submanifolds. In the forthcoming articles \cite{haskins:kapouleas:hd2,haskins:kapouleas:hd3} and in our survey \cite{haskins:kapouleas:survey} we describe how to use both the $\sorth{n-1}$-invariant
and $\sorth{p}\times \sorth{p}$-invariant special Legendrians as the basic building blocks for two
different gluing constructions of a plethora of higher dimensional special Legendrian submanifolds. 
The constructions using the $\sorth{n-1}$-invariant special Legendrians are the most direct high-dimensional 
analogues of our previous three-dimensional gluing constructions. 
By contrast, the geometry and the analysis 
for constructions based on the $\sorth{p}\times \sorth{p}$-invariant special Legendrians 
turns out to be rather different to the three-dimensional construction.

\subsubsection*{$\sorth{p}\times\sorth{q}$-invariant special Legendrians}
In this article we will describe in detail the geometry of $\sorth{p} \times \sorth{q}$-invariant 
special Legendrians for all $p$ and $q$, even though  in our present gluing constructions 
we will only use the $\sorth{n-1}$-invariant and 
$\sorth{p}\times \sorth{p}$-invariant ones. 
We view the $\sorth{p}\times \sorth{q}$-invariant special Legendrians as very natural 
geometric objects worthy of study themselves and there is little extra effort (and indeed some 
economy) in studying them for all values of $p$ and $q$. Also while a direct analogue of the gluing 
constructions described in \cite{haskins:kapouleas:hd2,haskins:kapouleas:hd3,haskins:kapouleas:survey} will fail using these 
other building blocks (for reasons discussed in our survey article \cite{haskins:kapouleas:survey})
 there are certainly other plausible gluing constructions in which they might 
profitably be deployed in the future. 
While many features of the $\sorth{p}\times \sorth{q}$-invariant special
 Lagrangian cones do not depend on $p$ and $q$, there are some crucial geometric differences between
 the cases (i) $p=1$, (ii) $p=q$ and (iii) $p\neq q$ and $p\neq 1$. These differences are important for our gluing constructions and 
 for this reason we will at a later stage separate our study of $\sorth{p}\times \sorth{q}$-invariant special Lagrangian cones into these three cases. 
 
 As we remarked earlier for $(p,q)=(1,2)$ the $\sorth{p} \times \sorth{q}$-invariant special Legendrians  are precisely the $\sorth{2}$-invariant ones studied previously in \cite{haskins:thesis,haskins:slgcones,haskins:kapouleas:invent}.  To our knowledge, for general $(p,q)$, 
 $\sorth{p}\times \sorth{q}$-invariant special Legendrians were first studied 
 by Castro-Li-Urbano \cite{castro:li:urbano}.
 We study the geometry of  $\sorth{p}\times \sorth{q}$-invariant special Legendrians in 
 considerably more detail than in \cite{castro:li:urbano}; we pay particular attention to the 
 geometric features central to their use as the building blocks for our gluing constructions of higher dimensional special Legendrian submanifolds. In particular:
 (a) we prove the the existence of countably infinitely many closed embedded $\sorth{p}\times \sorth{q}$-invariant special Legendrians
 for any admissible $p$ and $q$, 
 (b) we classify the extra discrete symmetries enjoyed by $\sorth{p} \times \sorth{q}$-invariant special Legendrians, 
 and (c) we study the geometry and detailed asymptotics of the family of $\sorth{p}\times \sorth{q}$-invariant special Legendrians 
 close to the equatorial sphere limit where the family degenerates. 
 All three components play an essential role in the use of $\sorth{p} \times \sorth{q}$-invariant special Legendrians 
 as building blocks in our higher dimensional gluing constructions \cite{haskins:kapouleas:survey,haskins:kapouleas:hd2,haskins:kapouleas:hd3}.
 
 \subsubsection*{Special Legendrian twisted products}
The construction of $\sorth{p} \times \sorth{q}$-invariant special Legendrians 
can be viewed as a special case of what we call the \emph{twisted product} construction. 
This twisted product construction (see Definition \ref{D:twisted:product}) gives a way to combine a 
pair of lower-dimensional Legendrian immersions  $X_{1}: \Sigma_{1} \ra \Sph^{2p-1}$ and 
$X_{2}: \Sigma_{2} \ra \Sph^{2q-1}$ 
 with a Legendrian curve $\bw: I \ra \Sph^{3}\subset \C^{2}$ to produce a new Legendrian immersion  
 $X_{1}*_{\bw}X_{2}: I \times \Sigma_{1}\times \Sigma_{2} \ra \Sph^{2p+2q-1}\subset \C^{p+q}$. 
If the twisting curve  $\bw:I \ra \Sph^{3}$ is appropriately chosen then the cone over  $X_{1}*_{\bw}X_{2}$
is just the product of the cones over $X_{1}$ and $X_{2}$, hence the origin of the term twisted product.

The Lagrangian phase of the twisted product $X_{1}*_{\bw}X_{2}$ is determined by the
Lagrangian phases of $X_{1}$, $X_{2}$ and the twisting curve $\bw$ (see \ref{E:twist:angle}). This formula implies 
that if the twisting curve $\bw:I \ra \Sph^{3}$ satisfies a certain condition 
depending on $p$ and $q$ (see \ref{E:w:slg}) then the $\bw$-twisted product $X_{1}*_{\bw}X_{2}$ 
is special Legendrian whenever both $X_{1}$ and $X_{2}$ are also special Legendrian. 
Twisting curves $\bw:I \ra \Sph^{3}$ satisfying the condition \ref{E:w:slg} we call 
\emph{$(p,q)$-twisted SL curves} in $\Sph^{3}$. The crucial point here about condition \ref{E:w:slg} 
is that it depends on $p$ and $q$ but not on the immersions $X_{1}$ and $X_{2}$.
Thus we can use the twisted product construction 
to produce special Legendrian immersions from lower-dimensional special Legendrian immersions
provided that we can find $(p,q)$-twisted SL curves. To produce special Legendrian immersions 
of closed manifolds via $(p,q)$-twisted SL curves we need to find closed $(p,q)$-twisted SL curves. 

In the special case where the two Legendrian immersions $X_{1}$ and $X_{2}$ are the standard 
real equatorial embeddings $\Sph^{p-1}\ra \Sph^{2p-1}\subset \C^{p}$ and 
$\Sph^{q-1} \ra \Sph^{2q-1}\subset \C^{q}$ respectively the resulting twisted product  
$X_{1}*_{\bw} X_{2}$ is an $\sorth{p} \times \sorth{q}$-invariant Legendrian. 
If additionally, the twisting curve $\bw$ is a $(p,q)$-twisted SL curve then the twisted 
product of these two equatorial embeddings is an $\sorth{p} \times \sorth{q}$-invariant special 
Legendrian. Conversely, every $\sorth{p} \times \sorth{q}$-invariant special Legendrian 
arises in this way as a twisted product with some $(p,q)$-twisted SL curve. Thus the study 
of $\sorth{p} \times \sorth{q}$-invariant special Legendrians in $\Sph^{2p+2q-1}$ 
essentially reduces to the study of   $(p,q)$-twisted SL curves in $\Sph^{3}$.
The twisted product construction first appeared in the work of Castro-Li-Urbano \cite{castro:li:urbano},   
although the terminology twisted product is our invention.

\subsubsection*{A plethora of new special Lagrangian and Hamiltonian stationary cones}
Although our main interest in $(p,q)$-twisted SL curves in $\Sph^{3}$ is their intimate connection 
to $\sorth{p} \times \sorth{q}$-invariant special Legendrians,  surprisingly good mileage can be obtained by applying our results about $(p,q)$-twisted SL curves in $\Sph^{3}$ to more general 
twisted products, i.e. where $X_{1}$ and $X_{2}$ are not just standard equatorial embeddings of 
spheres. 

Our results about $(p,q)$-twisted SL curves in $\Sph^{3}$ together with 
our previous gluing constructions of higher genus SL cones in $\C^{3}$ \cite{haskins:kapouleas:invent} allow 
us to construct a wealth of new topological types of higher-dimensional special Lagrangian and Hamiltonian 
stationary cones. 

\medskip

\noindent
\textbf{Theorem A.}
\begin{itemize}
\item[(i)]
\emph{For any $n\ge 4$ there are infinitely many topological types of special Lagrangian cone in $\C^{n}$, 
each of which is diffeomorphic to  the cone over a product $S^{1}\times \Sigma'$ for some smooth manifold $\Sigma'$, 
and each of which admits infinitely many distinct geometric representatives.}
\item[(ii)]
\emph{For any $n \ge 4$ there are infinitely many topological types of Hamiltonian stationary cone in $\C^{n}$ 
which are not minimal Lagrangian, each of which is diffeomorphic to the cone over a product $S^{1}\times \Sigma'$ for some smooth manifold $\Sigma'$, 
and each of which admits infinitely many distinct geometric representatives.}
\end{itemize}
Similarly combining our results about $(p,q)$-twisted SL curves with the work of Carberry-McIntosh \cite{carberry:mcintosh} 
on special Legendrian $2$-tori via integrable systems methods we obtain the following 

\medskip

\noindent
\textbf{Theorem B.}
\begin{itemize}
\item[(i)] 
\emph{For $n\ge 3$ there exist special Legendrian immersions of $T^{n-1}$ in $\Sph^{2n-1}$ which come in continuous families of arbitrarily high dimension.}
\item[(ii)] 
\emph{For $n \ge 4$ there exist contact stationary (and not minimal Legendrian) immersions of $T^{n-1}$ in $\Sph^{2n-1}$ which come in continuous families 
of arbitrarily high dimension.}
\end{itemize}
Finally, by combining the twisted product construction with both integrable systems methods and our gluing methods 
for special Legendrian surfaces in $\Sph^{5}$ we obtain the following striking results

\medskip

\noindent
\textbf{Theorem C.}
\begin{itemize}
\item[(i)] 
\emph{For any $n\ge 6$ there are infinitely many topological types of special Lagrangian cone in $\C^{n}$ of 
product type which can occur in continuous families of arbitrarily high dimension.}
\item[(ii)]  
\emph{For each $n\ge 6$ there are infinitely many topological types of Hamiltonian stationary cone in $\C^{n}$ of product 
type which are not minimal Lagrangian and which can occur in continuous families of arbitrarily high dimension.}
\end{itemize}
It is difficult to see how either integrable systems methods or gluing methods by themselves could yield 
a result like Theorem C.

\subsubsection*{$(p,q)$-twisted SL curves, ODEs and $\sorth{p}\times \sorth{q}$-invariant special Legendrians}
The key to our study of $(p,q)$-twisted SL curves in $\Sph^{3}$ is the simple Lemma \ref{L:w:slg}
which shows that there is a very natural system of complex $1$st order ODEs (\ref{E:slg:ode})
whose integral curves are $(p,q)$-twisted SL curves in $\Sph^{3}$ and conversely that 
$(p,q)$-twisted SL curves in $\Sph^{3}$ always admit parametrisations satisfying \ref{E:slg:ode}. 
Since  the generic isotropic orbit of $\sorth{p} \times \sorth{q}$ has dimension $p+q-2$
one expects that the $1$st order PDEs for special Legendrians 
should reduce to a system of $1$st order ODEs under this symmetry assumption.
When $(p,q)=(1,2)$ \ref{E:slg:ode} reduces to the fundamental ODEs  used to 
study $\sorth{2}$-invariant special Legendrians in 
\cite[eqn. 3.18]{haskins:kapouleas:invent}. The ODEs \ref{E:slg:ode} for 
$(p,q)$-twisted SL curves in $\Sph^{3}$ can be studied by methods similar to 
those used in our earlier work \cite{haskins:kapouleas:invent}.

\subsubsection*{The $1$-parameter family $X_{\tau}$ and its geometry as $\tau \ra 0$}
Many features of the $\sorth{2}$-invariant special Legendrians in $\Sph^{5}$ 
have  analogues for $\sorth{p} \times \sorth{q}$-invariant special Legendrians in 
$\Sph^{2p+2q-1}$. For instance, for each $(p,q)$ there is a real $1$-parameter family of 
distinct $\sorth{p} \times \sorth{q}$-invariant special Legendrian cylinders 
$X_{\tau}:\R \times \Sph^{p-1} \times \Sph^{q-1} \ra \Sph^{2p+2q-1}$ depending real analytically 
on $\tau$ (see Proposition \ref{P:X:tau}). Geometrically, the parameter $\tau$ controls the maximum of
the absolute value of the curvature that occurs on the cylinder $X_\tau$ and this value tends to infinity
as $\tau \ra 0$. Hence as $\tau \ra 0$, the family $X_\tau$ must degenerate in some way.
In all our cases as $\tau \ra 0$, $X_\tau$ approaches a necklace of equatorial $n-1$
spheres. In this sense our building blocks are reminiscent of building blocks used in other gluing
constructions---Delaunay surfaces in the construction of CMC surfaces in $\R^3$ \cite{kapouleas:bulletin,kapouleas:annals,kapouleas:jdg,kapouleas:cmp} and
Delaunay/Fowler metrics in the construction of constant scalar curvature metrics \cite{schoen}.
The fact that the family $X_\tau$ degenerates to a union of very simple geometric
objects is fundamental to our gluing constructions in \cite{haskins:kapouleas:invent,haskins:kapouleas:hd2,haskins:kapouleas:survey}.

\subsubsection*{Spherical SL necklaces of type $(p,q)$}
For a given value of $p+q=n$, but different values of $p$ and $q$,
the $\sorth{p} \times \sorth{q}$-invariant special Legendrian submanifolds
all approach a necklace of equatorial $n-1$ spheres as $\tau \ra 0$. For different values of $(p,q)$
these give rise to different kinds of spherical necklaces, in which
the geometry of the transition regions that connects two adjacent almost spherical regions,
 and thus the relative positioning of adjacent almost spherical regions,  changes.
For $\sorth{n-1}$-invariant special Legendrians each limiting equatorial sphere has two identical transition regions 
which localise on two antipodal points of the equatorial sphere.  
Suitably enlarged the core of each of these transition regions resembles a Lagrangian catenoid as $\tau \ra 0$. 
If $p>1$, the geometry of the transition regions of the $\sorth{p} \times \sorth{q}$-invariant special Legendrians 
is more complicated. In this case there are two different kinds of transition regions: one which localises on 
a $p-1$ dimensional equatorial subsphere and another which localises on a $q-1$ dimensional equatorial subsphere. 
In the former case the core of the transition region resembles the product of a unit $p-1$ sphere with a small $q$ 
dimensional Lagrangian catenoid, and in the latter case the product of a small $p$ dimensional Lagrangian catenoid 
with a unit $q-1$ sphere. In the special case $p=q$ these two kinds of transition regions are isometric and there exist 
discrete symmetries that exchange the two kinds of transition regions; these symmetries cannot exist in the case $p \neq q$. 
The geometry of the almost spherical regions and the different kinds of transition regions are described in detail in Section \ref{S:xtau:limit}.

\subsubsection*{Symmetries of the building blocks $X_{\tau}$}
The $\sorth{p} \times \sorth{q}$-invariant special Legendrians $X_{\tau}$ possess 
additional discrete symmetries beyond the  $\sorth{p} \times \sorth{q}$ symmetry. 
An important feature of the gluing constructions carried out in \cite{haskins:kapouleas:hd2,haskins:kapouleas:hd3} is that we 
exploit fully these discrete symmetries to simplify the later analysis (as described in our survey article 
\cite{haskins:kapouleas:survey}). Hence we give a detailed description of all symmetries 
the building blocks $X_{\tau}$ enjoy.

\subsubsection*{Closed special Legendrian immersions from $X_{\tau}$ and closed $(p,q)$-twisted SL curves}
The immersions $X_{\tau}$ give us special Legendrian immersions
of the generalised cylinder $\R \times \Sph^{p-1}\times \Sph^{q-1}$.
As in the $\sorth{2}$-invariant case a single angular period $\pthat$ (defined precisely in \ref{E:pthat}) determines
when $X_\tau$ factors through a closed special Legendrian embedding 
and hence gives rise to a SL cone in $\C^n$ with closed link. 
By studying the dependence of $\pthat$ on $\tau$  (see Proposition \ref{P:asymptotics-pthat})
we prove that for a dense set of $\tau$, $X_\tau$ factors as above; 
this is intimately connected with  understanding for what values of $\tau$ we can find 
\emph{closed} $(p,q)$-twisted SL curves in $\Sph^{3}$ (see \ref{T:w:closed} and \ref{T:Xtau:embed}). The latter is important for the applications to
 construct new special Legendrian immersions of the closed manifold 
 $S^{1}\times \Sigma_{1}\times \Sigma_{2}$ from a pair of lower-dimensional special Legendrian immersions of 
$\Sigma_{1}$ and $\Sigma_{2}$. 

\subsubsection*{Asymptotics of the angular period $\pthat$ as $\tau \ra 0$}
The behaviour of the angular period $\pthat$ and its derivative $\frac{d\pthat}{d\tau}$ as $\tau \ra 0$ is 
needed to understand quantitively how the geometry of $X_{\tau}$ changes when we make a small change in $\tau$ (and $\tau$ itself is also small). 
Understanding this behaviour is crucial to the gluing applications in \cite{haskins:kapouleas:invent,haskins:kapouleas:hd2,haskins:kapouleas:hd3,haskins:kapouleas:survey}.
The angular period $\pthat$ can be expressed as an integral, which for $\sorth{2}$-invariant SL cones 
is an elliptic integral (of the third kind) \cite[eqn. 3.34 \& Appendix A]{haskins:kapouleas:invent}. 
In\cite{haskins:kapouleas:invent} 
we exploited results about elliptic integrals to prove our small $\tau$ asymptotics for $\tfrac{d\pthat}{d\tau}$ \cite[3.30]{haskins:kapouleas:invent}. 
In higher dimensions $\pthat$ has an expression in terms of hyperelliptic rather than elliptic integrals.  
In this paper, rather than studying the hyperelliptic integrals directly we adopt a more geometric approach 
that we hope will find application in other similar problems.

Every minimal submanifold of $\Sph^{m-1}$ has an associated homological invariant called the torque 
which arises directly from the First Variation Formula applied to Killing fields $\lorth{m}$ of $\Sph^{m-1}$. 
We show that the torque detects the difference between the $X_{\tau}$ for different values of $\tau$. 
By studying the linearisation of the torque for small rotationally invariant perturbations of $X_{\tau}$ 
and combining it with the Legendrian neighbourhood theorem we derive
an exact formula for the derivative $\tfrac{d\pthat}{d\tau}$ in terms of the values of 
a distinguished solution $Q$ to the rotationally invariant linearised operator. This formula  
is valid for all values of $\tau$ not just small $\tau$ and may itself be useful for other purposes. 
By studying the behaviour of the distinguished solution 
$Q$ for small $\tau$ we are able to prove our result on the small $\tau$ asymptotics of $\tfrac{d\pthat}{d\tau}$ and $\pthat$.

\subsection*{Organisation of the Paper}
$\phantom{ab}$
\nopagebreak
The paper is organised in ten sections and an Appendix.
Section \ref{S:intro} consists of the introduction, this section and some remarks on notation.
In Section \ref{S:slg} we recall basic facts and definitions from symplectic, contact,  special Lagrangian and Hamiltonian 
stationary geometry.
The reader already familiar with these basics could skip this section.

In Section \ref{S:twist:prod} we describe two ways to generate new special Lagrangian or 
special Legendrian immersions from other simpler or lower-dimensional special Lagrangian or 
special Legendrian immersions and a curve either in $\C$ or $\Sph^{3}$ 
satisfying some additional geometric condition. 

The first construction leads to the construction of the so-called Lagrangian catenoids---the unique family of nonflat 
special Lagrangian $n$-folds in $\C^{n}$ foliated by round $n-1$ spheres. 
Although well-known, we describe the Lagrangian catenoid and the associated ODEs in detail as a warmup for the second construction 
and because we will see the Lagrangian catenoids reappear in the ``necks'' of the $\sorth{p} \times \sorth{q}$-invariant 
special Legendrians constructed later in the paper.  

The second construction (see Definition 
\ref{D:twisted:product} and Proposition \ref{P:leg:triple}), which we call the 
\emph{twisted product} construction, is at the heart of the paper. 
Definition \ref{D:twist:slg} introduces the notion of a \emph{$(p,q)$-twisted special Legendrian (SL) curve} in $\Sph^{3}$.  
Corollary \ref{C:sl:triples} explains how to use $(p,q)$-twisted SL curves in $\Sph^{3}$ to construct new special Legendrian 
immersions from a pair of lower-dimensional special Legendrian immersions via the twisted product construction.
Lemma \ref{L:w:slg} reduces the study of $(p,q)$-twisted SL curves in $\Sph^{3}$ to 
a $1$st order system of complex ODEs \ref{E:slg:ode}.

We also sketch briefly the extension of the twisted product construction to the contact stationary realm. 
Definition \ref{D:pq:twist:cs}  introduces \emph{$(p,q)$-twisted contact stationary (CS) curves} in $\Sph^{3}$ 
and Lemma \ref{L:twist:cs} gives the contact stationary analogue of Corollary \ref{C:sl:triples}.
We do not systematically study $(p,q)$-twisted CS curves in this paper and instead content ourselves with 
exhibiting a very simple but still countably infinite family of closed $(p,q)$-twisted CS curves (see \ref{E:twist:cs:curve}).
The existence of this simple family of closed $(p,q)$-twisted CS curves nevertheless suffices to 
construct many new contact stationary (and non minimal Legendrian) immersions of closed manifolds from lower dimensional 
special Legendrian immersions. In the rest of the section, assuming results on the existence of countably infinitely many 
closed $(p,q)$-twisted SL curves proved later in Theorem \ref{T:w:closed},  we prove Theorems A--C quoted above
by combining the SL and CS twisted product constructions with (a) our gluing constructions of 
special Legendrian surfaces of higher genus in $\Sph^{5}$ \cite{haskins:kapouleas:invent} and 
(b) the integrable systems constructions of special Legendrian tori in $\Sph^{5}$ of Carberry-McIntosh \cite{carberry:mcintosh}.

In Section \ref{S:sop:soq} we begin our detailed study of $\sorth{p}\times \sorth{q}$-invariant 
special Legendrians in $\Sph^{2p+2q-1}$. Lemma \ref{L:iso:orbits} describes the isotropic $\sorth{p} \times \sorth{q}$ 
orbits $\mathcal{O} \subset \Sph^{2(p+q)-1}$ and leads to Corollaries \ref{C:so:invariant:leg} and \ref{C:so:invariant:slg} 
on the correspondence between $\sorth{p} \times \sorth{q}$-invariant special Legendrians in $\Sph^{2(p+q)-1}$ 
and $(p,q)$-twisted SL curves in $\Sph^{3}$. Proposition \ref{P:odes:p:n} establishes the basic facts about solutions 
to the $(p,q)$-twisted SL ODEs \ref{E:odes:p:n}: its conserved quantities, its symmetries, 
stationary points, local and global existence and dependence on initial data.
Proposition \ref{P:w:normal:form} gives a normal form for any solution $\bw$ to \ref{E:odes:p:n}. 
Using Propositions \ref{P:odes:p:n} and \ref{P:w:normal:form} we
define the $1$-parameter family $\bw_{\tau}$ of solutions of the fundamental ODE for $(p,q)$-twisted SL curves by 
specifying appropriate initial conditions (see \ref{E:w:ic:p:neq:1}, \ref{E:w:ic:p:eq:1} and Proposition \ref{P:w:tau}). 
In Definition \ref{D:X:tau} we use the $1$-parameter family of solutions $\bw_{\tau}$ to define the $1$-parameter family 
$X_{\tau}$ of $\sorth{p}\times \sorth{q}$-invariant special Legendrian immersions in $\Sph^{2p+2q-1}$. 
Proposition \ref{P:X:tau} establishes some basic properties of $X_{\tau}$. 

Section \ref{S:w:sym} studies the discrete symmetries of the $1$-parameter family $\bw_{\tau}$ of $(p,q)$-twisted 
SL curves defined in Section \ref{S:sop:soq}. 
We also introduce the \emph{periods} and \emph{half-periods} of $\bw_{\tau}$; 
the periods of $\bw_{\tau}$ control when $\bw_{\tau}$ forms a closed curve in $\Sph^{3}$, while the half-periods 
control when the curve of $\sorth{p}\times \sorth{q}$ orbits associated with $\bw_{\tau}$ is closed. 
The half-periods of $\bw_{\tau}$ also control the embedding properties of $X_{\tau}$ (see Proposition \ref{P:xtau:embed}).
The discrete symmetries of $\bw_{\tau}$ give rise to symmetries of $X_{\tau}$ beyond the 
$\sorth{p} \times \sorth{q}$ symmetry implicit in the construction of $X_{\tau}$. Many of these discrete symmetries 
do not belong to $\sunit{n}$, the obvious subgroup of $\orth{2n}$ that sends SL $n$-folds in $\C^{n}$ to other SL 
$n$-folds. For this reason and because the discrete symmetries play an important role in our subsequent gluing constructions 
it is important to give a careful study of the symmetries of $X_{\tau}$.

Thus Section \ref{S:sym:xtau} gives an in-depth analysis of all symmetries enjoyed by 
the $1$-parameter family of $\sorth{p}\times \sorth{q}$-invariant special Legendrian immersions 
$X_{\tau}$. A \emph{symmetry of  $X_{\tau}: \cylpq \ra \Sph^{2n-1}$}
 is a pair $(\MMM,\mtilde) \in \Diff(\cylpq) \times \orth{2n}$ such that 
$$ \mtilde \circ X_{\tau} = X_{\tau} \circ \MMM.$$
If $(\MMM,\mtilde)$ is any symmetry of $X_{\tau}$ then $\MMM \in \Isom(\cylpq,g_{\tau})$ where 
$g_{\tau}:= X_{\tau}^{\tiny *}\,g_{\tiny \,\Sph^{2n-1}}$ is the pullback metric on $\cylpq$ induced by 
the immersion $X_{\tau}$ (see    Remark \ref{R:sym:pullback}). Propositions \ref{P:isom:pb} and \ref{P:isom:pb:str} 
determine the structure of the group $\Isom(\cylpq,g_{\tau})$.
Propositions \ref{P:xtau:sym:p:eq:1}, \ref{P:xtau:sym:p:neq:q} and \ref{P:xtau:sym:p:eq:q} (for the three cases 
$p=1$, $p>1$, $p \neq q$ and $p>1$, $p=q$ respectively) show that in fact every element 
of $\Isom(\cylpq,g_{\tau})$ gives rise to a symmetry of $X_{\tau}$. Using this fact and our results on the structure of 
$\Isom(\cylpq,g_{\tau})$ we determine the structure of the group of domain symmetries $\Sym(X_{\tau})$ 
(defined in \ref{E:sym:xtau:def}) in Corollaries \ref{C:sym:isom:p:eq:1}, \ref{C:sym:isom:p:neq:1} and \ref{C:sym:isom:p:eq:q}.
Together with the results from Section \ref{S:w:sym} on the half-periods of $\bw_{\tau}$ this also allows us to 
determine the structure of the group of target symmetries $\Symtilde(X_{\tau})$ (defined in \ref{E:symtilde:xtau:def}) 
in Lemmas \ref{L:symtilde:p:eq:1}, \ref{L:symtilde:p:neq:q} and \ref{L:symtilde:p:eq:q}. 

Section \ref{S:xtau:limit} studies two related topics. 
The first part of the section introduces subsets of $\cylpq$ called 
the \emph{waists} (Definition \ref{D:waist}) and \emph{bulges} 
(Definition \ref{D:bulges}) of $X_{\tau}$ and associates to each bulge an equatorial $n-1$ sphere in $\Sph^{2n-1}$ 
called its \emph{approximating sphere} (Definition \ref{D:approx:sphere}). 
We study the action of $\Sym(X_{\tau})$ on the waists and bulges of $X_{\tau}$ 
and the action of $\Symtilde(X_{\tau})$ on the approximating spheres of $X_{\tau}$ 
(see Lemmas \ref{L:sym:k} and \ref{L:symtilde:app:sph}). 

The second part of the section studies the geometry of $X_{\tau}$ as $\tau \ra 0$. We define subsets of the bulges, 
called almost spherical regions, and show that as $\tau \ra 0$ the image of an almost spherical region under $X_{\tau}$ 
is close to its associated approximating sphere, thereby justifying the terminology. We also study the geometry of the necks of 
$X_{\tau}$---the core of the transition regions connecting two adjacent almost spherical regions centred around one of the waists---and show that as $\tau \ra 0$ the necks approach a limiting geometry: if $p=1$ the necks all resemble small $n-1$ dimensional 
Lagrangian catenoids, while if $p>1$ there are two kinds of necks both of which resemble the product of a unit sphere 
with a small Lagrangian catenoid of the appropriate dimensions. 

Section \ref{S:torques} introduces a homological invariant of minimal submanifolds of $\Sph^{2n-1}$ 
call its \emph{torque} and a variant called the \emph{restricted torque} for special Legendrian submanifolds of $\Sph^{2n-1}$.
Proposition \ref{P:torques} determines the restricted torque for the $\sorth{p} \times \sorth{q}$-invariant 
SL immersions $X_{\tau}$. This torque calculation is used later in the proof of the 
$\tau \ra 0$ asymptotics of the angular period $\pthat$. 

Section \ref{S:asymptotics} studies the asymptotics of the period $2\pt$ and the angular period $\pthat$ as $\tau \ra 0$. 
To prove the small $\tau$ asymptotics of $\tfrac{d\pthat}{d\tau}$ we need Lemma \ref{L:infinitesimal-force} 
which calculates the linearisation of the torque of $X_{\tau}$ when perturbed by a small rotationally-invariant function 
$\phi$. Lemma \ref{L:infinitesimal-force} is a key ingredient of Lemma \ref{L:phat:dt}, which gives an precise formula 
for $\tfrac{d\pthat}{d\tau}$ valid for any $0<\tau<\taumax$, in terms of the values of a particular solution 
to the rotationally-invariant linearised operator \ref{E:linear:phi}. The small $\tau$ asymptotics of $\pthat$ and 
$\tfrac{d\pthat}{d\tau}$ are easy consequences of this formula (see \ref{P:asymptotics-pthat}).

Section \ref{S:embedded} uses Proposition \ref{P:xtau:embed}, the structure of the half-periods of $\bw_{\tau}$ 
and the asymptotics of $\pthat$ as $\tau \ra 0$ to prove the existence of a countably infinite family of 
closed $(p,q)$-twisted SL curves for every admissible pair of integers $(p,q)$ (see Theorem \ref{T:w:closed}). 
Similar methods prove the existence of a countable dense set of $\tau$ for which $X_{\tau}$ factors 
through an embedding of a closed special Legendrian manifold (see Theorem \ref{T:Xtau:embed}).
In particular we find an infinite sequence of $\tau$ converging to $0$ 
for which $X_{\tau}: \cylpq \ra \Sph^{2n-1}$ factors through an embedding of a 
$S^{1}\times S^{n-2}$ if $p=1$ or $S^1 \times S^{p-1}\times S^{p-1}$ if $p=q\ge 2$
(see Lemma \ref{L:no:half:periods}). These closed special Legendrian ``necklaces'' with small $\tau$ 
are the building blocks for our subsequent gluing constructions \cite{haskins:kapouleas:survey,haskins:kapouleas:hd2,haskins:kapouleas:hd3}.

The Appendix contains material of a more general nature. 
Appendix \ref{A:groups} recalls some elementary group theory used to describe 
the structure of various symmetry groups that arise in the paper. 
Appendix \ref{A:isom:sl} describes all \emph{Lagrangian isometries} and 
\emph{special Lagrangian isometries} of $\C^{n}$: elements of $\orth{2n}$ 
that preserve the Lagrangian Grassmannian or special Lagrangian Grassmannian respectively. 
The structure of the special Lagrangian isometries and more generally the anti-special Lagrangian isometries 
(isometries that take all special Lagrangian planes to special Lagrangian planes with the wrong orientation) 
underpins the structure of the group $\Symtilde(X_{\tau})$ analysed in Section \ref{S:sym:xtau}; 
every element of $\Symtilde(X_{\tau})$ is either a special Lagrangian or anti-special Lagrangian isometry.


\subsection*{Notation and conventions}
$\phantom{ab}$
\nopagebreak
Throughout the paper we use the following notation to express elements of $\Isom(\R)$, the isometries of the real line.
We denote by $\TTT_{x}$,  translation by $x$, $t \mapsto t+x$. We denote by $\tbar$ reflection in the origin 
$t\mapsto -t$ and reflection in $x$, $t \mapsto 2x-t$ by $\tbar_{x}$. 

%
%


\subsection*{Acknowledgments}
N.K. would like to thank the Leverhulme Trust for funding his visit to Imperial College London in Spring 2009 
and to the Department of Mathematics at Imperial for the supportive research environment.
M.H. would like to thank the EPSRC for their continuing support of his research under Leadership Fellowship EP/G007241/1.

\section{Special Lagrangian cones and special Legendrian submanifolds of $\Sph^{2n-1}$}
\label{S:slg}
\nopagebreak

In this section we recall basic facts about \slg geometry in $\C^n$, \slg cones in $\C^n$
and their connection to minimal Legendrian submanifolds of $\Sph^{2n-1}$.
\Slg geometry is an example of a \textit{calibrated geometry} \cite{harvey:lawson}. 

\subsection*{Calibrations and Special Lagrangian geometry in $\C^n$}
\label{SS:slg}
$\phantom{ab}$
\nopagebreak

Let $(M,g)$ be a Riemannian manifold. Let $V$ be an oriented tangent $p$-plane on $M$, \textit{i.e.}
a $p$-dimensional oriented vector subspace of some tangent plane $T_xM$ to $M$. The restriction
of the Riemannian metric to $V$, $g|_V$, is a Euclidean metric on $V$ which together with the
orientation on $V$ determines a natural $p$-form on $V$, the volume form $\vol_{V}$.
A closed $p$-form $\phi$ on $M$ is a \textit{calibration} on $M$ if for every oriented tangent $p$-plane $V$ on $M$
we have $\phi|_V \le \vol_V$. Let $L$ be an oriented submanifold of $M$ with dimension $p$. $L$ is a
\textit{$\phi$-calibrated submanifold} if $\phi|_{T_xL} = \vol_{T_xL}$ for all $x\in L$.
There is a natural extension of this definition to singular calibrated submanifolds  
Geometric Measure Theory and rectifiable currents \cite[\S II.1]{harvey:lawson}.
The key property of calibrated submanifolds (even singular ones) is that they are \textit{homologically volume minimising} \cite[Thm. II.4.2]{harvey:lawson}.
In particular, any calibrated submanifold is automatically \textit{minimal}, \textit{i.e.} has vanishing mean curvature.

Let $z_1 = x_1 + i y_1, \cdots  ,z_n = x_n + i y_n$ be standard complex coordinates on $\C^n$ equipped with the Euclidean metric.
Let
$$\omega = \frac{i}{2} \sum_{j=1}^n{dz_j \wedge d\overline{z}_j} = \sum_{j=1}^n{dx_j \wedge dy_j},$$
be the standard symplectic $2$-form on $\C^n$.
Define a complex $n$-form $\Omega$ on $\C^n$ by
\begin{equation}
\addtocounter{theorem}{1}
\label{E:slg:form}
\Omega = dz_1 \wedge \cdots \wedge dz_n.
\end{equation}
The real $n$-form $\Real{\Omega}$ is a calibration on $\C^n$ whose calibrated submanifolds we call
\textit{\slg submanifolds} of $\C^n$, or SL $n$-folds for short.
There is a natural extension of \slg geometry to any Calabi-Yau manifold $M$ by replacing
$\Omega$ with the natural parallel holomorphic $(n,0)$-form on $M$. \Slg submanifolds 
play an important role in a number of interesting geometric properties
of \cy manifolds, \textit{e.g.} Mirror Symmetry \cite{syz,thomas:yau}.

Let $f:L \ra \C^n$ be a Lagrangian immersion of the oriented $n$-manifold $L$, and $\Omega$ be
the standard holomorphic $(n,0)$-form defined in \ref{E:slg:form}.
Then $f^*\Omega$ is a complex $n$-form on $L$ satisfying $|f^*\Omega| = 1$ \cite[p. 89]{harvey:lawson}. Hence we may write
\addtocounter{theorem}{1}
\begin{equation}
\label{E:lagn:phase}
f^*\Omega = e^{i\theta}\vol_L\quad \text{on \ } L,
\end{equation}
for some \textit{phase function} $e^{i\theta}:L \ra \Sph^1$.
We call $e^{i\theta}$ the \textit{phase of the oriented Lagrangian immersion $f$}.
$L$ is a SL $n$-fold in $\C^n$ if and only if the phase function $e^{i\theta}\equiv 1$.
Reversing the orientation of $L$ changes the sign of the phase function $e^{i\theta}$.
The differential $d\theta$ is a closed $1$-form on $L$ satisfying
\addtocounter{theorem}{1}
\begin{equation}
\label{mc:lagn:angle}
d\theta = \iota_H \omega,
\end{equation}
where $H$ is the mean curvature vector of $L$. In particular, \ref{mc:lagn:angle} implies
that a connected component of $L$ is minimal if and only if the phase function $e^{i\theta}$
is constant.
\ref{mc:lagn:angle} may also be restated as
\addtocounter{theorem}{1}
\begin{equation}
\label{mc:lagn:angle:grad}
H = -J \nabla \theta,
\end{equation}
where $J$ and $\nabla$ denote the
standard complex structure and gradient on $\C^n$ respectively.
For a general Lagrangian submanifold of $\C^{n}$ it is not possible to find a global lift
of the $\Sph^1$ valued phase function $e^{i\theta}$ to a real function $\theta$, although
of course such a lift always exists locally.
When a global lift $\theta$ exists we call $\theta:L \ra \R$ the \textit{Lagrangian angle} of $L$.
In particular, any Lagrangian submanifold which is sufficiently close to
a special Lagrangian submanifold will have a globally well-defined Lagrangian angle $\theta$.

If the Lagrangian phase $e^{i\theta}:L \ra \Sph^{1}$ of a Lagrangian submanifold $L$ is a harmonic map, 
or equivalently if the $1$-form $d\theta$ 
is harmonic, then $L$ is said to be \emph{Hamiltonian stationary}. Hamiltonian stationary submanifolds 
are so-called because the condition that $d\theta$ be harmonic is equivalent to $L$ being a critical point of volume 
with respect to compactly supported Hamiltonian variations \cite{schoen:wolfson}.

\subsection*{Contact geometry}
$\phantom{ab}$
\nopagebreak

We recall some basic definitions from contact geometry \cite{geiges,arnold,marle,mcduff:salamon}.
Let $M$ be a smooth manifold of dimension $2n+1$, and let $\xi$ be a hyperplane field on $M$.
$\xi$ is a (cooriented) \textit{contact structure} on $M$ if there exists a $1$-form $\gamma$ so that
$\ker{\gamma} = \xi$ and
\addtocounter{theorem}{1}
\begin{equation}
\label{contact:nonint}
\gamma \wedge (d\gamma)^n \neq 0.
\end{equation}
The pair $(M,\xi)$ is called a \textit{contact manifold}, and the $1$-form $\gamma$ a \textit{contact form} defining $\xi$.
Condition \ref{contact:nonint} is equivalent to the condition that $(d\gamma)^n|_{\xi} \neq 0$.
In particular, for each $p\in M$ the $2n$-dimensional subspace $\xi_p \subset T_pM$ endowed with the
$2$-form $d\gamma|{\xi_p}$ is a symplectic vector space.
Given a contact form $\gamma$ on $M$, the \textit{Reeb vector field} $R_\gamma$ is the unique vector field on $M$
satisfying
$$\iota(R_\gamma) d\gamma \equiv  0, \qquad \gamma(R_\gamma) \equiv  1.$$
Let $(M,\xi=\ker{\gamma})$ be a contact manifold.
A submanifold $L$ of $(M,\xi)$ is an \textit{integral submanifold of $\xi$} (also called an isotropic
submanifold) if $T_x L \subset \xi_{x}$ for all $x\in L$. Equivalently $L$
is an integral submanifold of $\xi$ if $\gamma|_L =0$.
A submanifold $L$ of $(M^{2n+1},\xi)$ is \textit{Legendrian}
if it is an integral submanifold of maximal dimension $n$.

\subsection*{Special Legendrian submanifolds and special Lagrangian cones}
$\phantom{ab}$

For any compact oriented embedded (but not necessarily connected) submanifold $\Sigma
\subset \Sph^{2n-1}(1)\subset \C^n$ define the \textit{cone on
$\Sigma$},
$$ C(\Sigma) = \{ tx: t\in \R^{\ge 0}, x \in \Sigma \}.$$
A cone $C$ in $\C^n$  (that is a subset invariant under dilations)
is $\textit{regular}$ if there exists
$\Sigma$ as above so that $C=C(\Sigma)$, in which case we call
$\Sigma$ the \textit{link} of the cone $C$. $C'(\Sigma):=C(\Sigma) - \{0\}$ is an
embedded smooth submanifold, but $C(\Sigma)$ has an isolated
singularity at $0$ unless $\Sigma$ is a totally geodesic sphere.
Sometimes it will also be convenient to allow $\Sigma$ to be just immersed not embedded,
in which case $C'(\Sigma)$ is no longer embedded. Then we call $C(\Sigma)$ an \textit{almost regular} cone.

Let $r$ denote the radial coordinate on $\C^n$ and let
$X$ be the Liouville vector field
$$X= \frac{1}{2}r \frac{\partial}{\partial r} = \frac{1}{2}\sum_{j=1}^{n}{x_j \frac{\partial}{\partial x_j} + y_j \frac{\partial}{\partial y_j}}.$$
From its embedding in $\C^n$, the unit sphere $\Sph^{2n-1}$ inherits a natural contact form
$$\gamma = \iota_X\omega | _{\Sph^{2n-1}} = {\sum_{j=1}^n{ x_j dy_j - y_j dx_j}}\Bigr|_{\Sph^{2n-1}}.$$

There is a one-to-one correspondence between regular Lagrangian cones in $\C^n$
and Legendrian submanifolds of $\Sph^{2n-1}$.
The Lagrangian angle or the phase of a Lagrangian cone in $\C^n$ is homogeneous of degree $0$.
We define the Lagrangian angle of a Legendrian submanifold $\Sigma$ of $\Sph^{2n-1}$ to
be the restriction to $\Sph^{2n-1}$ of the Lagrangian angle of the Lagrangian cone $C'(\Sigma)$.
We call a submanifold $\Sigma$ of $\Sph^{2n-1}$ \textit{special Legendrian}
if the cone over $\Sigma$, $C'(\Sigma)$
is special Lagrangian in $\C^n$. In other words, $\Sigma$ is special Legendrian
if and only if its Lagrangian phase is identically $1$ or its Lagrangian angle
is identically $0$ modulo $2\pi$.

A special Legendrian submanifold of $\Sph^{2n-1}$
is  minimal. Conversely, any minimal Legendrian submanifold of $\Sph^{2n-1}$
has constant Lagrangian phase. Hence up to rotation by a constant phase
$e^{i\theta}$ any connected minimal Legendrian submanifold of $\Sph^{2n-1}$ is special Legendrian.
The goal of our paper is thus to construct special Legendrian immersions of (orientable) $n-1$-manifolds
into $\Sph^{2n-1}$.

\subsection*{Contact stationary submanifolds and Hamiltonian stationary cones}
$\phantom{ab}$

A Legendrian submanifold $\Sigma$ of $\Sph^{2n-1}$ is said to be \emph{contact stationary} if 
it is a stationary point of volume with respect to all contact deformations of $\Sigma$. 
One can show that an oriented Legendrian submanifold $\Sigma$ is contact stationary if and only if the Lagrangian phase $e^{i\theta}:\Sigma \ra \Sph^{1}$ 
is harmonic \cite[Prop 2.4]{castro:li:urbano}. Contact stationary submanifolds in $\Sph^{2n-1}$ 
are to Hamiltonian stationary cones in $\C^{n}$ as special Legendrian submanifolds in $\Sph^{2n-1}$ 
are to special Lagrangian cones in $\C^{n}$. Namely, a regular cone $C(\Sigma)$ in $\C^{n}$ is 
Hamiltonian stationary if and only if the link $\Sigma \subset \Sph^{2n-1}$ is a contact stationary 
submanifold \cite[Prop. 5.1]{castro:li:urbano}.

Hamiltonian stationary cones play an important role in analysing singularities of Lagrangian minimisers 
in the variational programme initiated by Schoen-Wolfson \cite{schoen:wolfson}. 
For $2$-dimensional Lagrangian minimisers Schoen-Wolfson developed a rather complete theory.  
Their singularity analysis in two dimensions involves at a key stage 
the classification of $2$-dimensional Hamiltonian stationary cones and the stability analysis of these cones \cite[\S 7]{schoen:wolfson}. The classification of $2$-dimensional Hamiltonian stationary cones is rather straightforward; 
the Hamiltonian stationary cones are parametrised by a pair of relatively prime positive integers $(m,n)$ 
and the corresponding contact stationary curves in $\Sph^{3}$ are (up to unitary equivalence) the following explicit curves
(see Theorem 7.1 in \cite{schoen:wolfson})
\begin{equation}
\label{E:hs:cones:2d}
\gamma_{m,n}(s) = \frac{1}{\sqrt{m+n}}\left( \sqrt{n} e^{is\sqrt{m/n}}, i\sqrt{m} e^{-is\sqrt{n/m}}\right), \quad 
s\in [0,2\pi \sqrt{mn}].
\end{equation}
The simple explicit form \ref{E:hs:cones:2d} permits the second variation operator of these Hamiltonian 
stationary cones to be written down very explicitly. 
 
By contrast, we will show that in high dimensions there is a plethora of Hamiltonian stationary cones 
(besides the ones that are special Lagrangian). This suggests that higher-dimensional versions of the Schoen-Wolfson 
programme  may run into problems if one needs to rely on a classification of Hamiltonian stationary cones.

\section{Twisted products of Legendrian immersions: new immersions from old}
\label{S:twist:prod}

In this section we describe two ways to generate new special Lagrangian or special Legendrian immersions 
from other (simpler or lower dimensional) 
special Lagrangian or special Legendrian immersions and a curve either in $\C$ or $\C^2$ 
satisfying certain ODEs. 

In the first simpler construction, given a curve in $w:I \ra \C$ and 
a Legendrian immersion in $X:\Sigma \ra \Sph^{2n-1}$ we obtain a Lagrangian immersion 
$X_w:I \times \Sigma \ra \C^n$. If the curve $w$ satisfies a certain ODE and $X$ is special Legendrian 
then $X_w$ is a new special Lagrangian immersion. In the second more powerful construction, 
given a Legendrian immersion $\bw: I \ra \Sph^3$ and a pair of Legendrian immersions $X_1: \Sigma_1 \ra \Sph^{2p-1}$ 
and $X_2: \Sigma_2 \ra \Sph^{2q-1}$ we obtain a new Legendrian immersion 
$X_1 *_\bw X_2: I \times \Sigma_1 \times \Sigma_2 \ra \Sph^{2p+2q-1}$, that we call 
the \emph{$\bw$-twisted product of $X_1$ and $X_2$}.  
If the curve $\bw:I \ra \Sph^3 \subset \C^2$ is chosen appropriately then the cone over the $\bw$-twisted product
is precisely the product of the cone over $X_1$ with the cone over $X_2$---hence the name twisted product  
for the general case. 
If $\bw$ satisfies an appropriate ODE and both $X_1$ and $X_2$ are special Legendrian then 
the $\bw$-twisted product $X_1 *_\bw X_2$ is also special Legendrian. 
We call solutions of these ODEs, \emph{$(p,q)$-twisted special Legendrian curves}. 
To construct new special Legendrian immersions of closed manifolds, the key point 
is to find closed $(p,q)$-twisted special Legendrian (SL) curves.  
We achieve a good understanding of closed $(p,q)$-twisted SL curves in Section \ref{S:embedded}. 

Combining this understanding of closed $(p,q)$-twisted SL curves 
with our earlier work on gluing constructions of special Legendrian immersions in $\Sph^5$ 
\cite{haskins:kapouleas:invent} 
and constructions of special Legendrian $2$-tori via integrable systems methods \cite{carberry:mcintosh,mcintosh:slg} 
we are able to prove the existence of a plethora of new special Legendrian immersions with interesting 
geometric properties in dimensions greater than three. 
Very minor modifications also allow us to construct a similar variety of contact stationary 
Legendrian immersions and hence of new Hamiltonian stationary (and not special Lagrangian) cones.
However, all closed special Legendrians constructed via $(p,q)$-twisted SL curves are topologically products 
of the form $S^{1} \times \Sigma$. We construct infinitely many topological types of higher dimensional 
special Legendrians which are not topologically products using gluing methods in \cite{haskins:kapouleas:survey,haskins:kapouleas:hd2,haskins:kapouleas:hd3}.

When the immersions $X_1$ and $X_2$ are chosen to be the simplest possible special Legendrian immersions, 
namely the standard totally real equatorial embeddings of 
$\Sph^{p-1} \subset \R^p \subset \C^p$ and $\Sph^{q-1} \subset \R^q \subset \C^q$, then $\bw$-twisted 
special Legendrian immersions $X_1 *_{\bw} X_2$ turn out to be suitable building blocks for 
higher dimensional gluing constructions of special Legendrian immersions. When $p=1$ and $q=2$ 
these turn out to be precisely the building blocks used in our previous gluing construction 
of special Legendrian surfaces in $\Sph^5$ \cite{haskins:kapouleas:invent,haskins:slgcones,haskins:thesis}.

Throughout this section, given a Legendrian immersion $Y$ into an odd-dimensional sphere 
or a Lagrangian immersion $Y$ into $\C^n$, we shall denote its Lagrangian phase by $e^{i\theta_Y}$.

\subsection*{$n$-twisted SL curves in $\C$ and the Lagrangian catenoid}
In the first construction we combine a curve $w$ in $\C$ and a Legendrian immersion $X$ in $\Sph^{2n-1}$
to obtain a new Lagrangian immersion in $\C^n$ as follows.
\begin{lemma}
\addtocounter{equation}{1}
Let $w: I \subset \R \ra \C$ be a smooth immersion 
and let $X: (\Sigma,g) \ra \Sph^{2n-1}$ be a Legendrian isometric immersion. At points where $w\neq 0$ the map 
$X_w: I \times \Sigma \ra \C^n$ defined by $$X_w(t,\sigma) = w(t) X(\sigma),$$
is a Lagrangian immersion  whose Lagrangian phase $e^{i\theta}$ satisfies
$$e^{i\theta} = e^{i\theta_X} e^{i\theta_w + i(n-1) \arg{w}},$$
and the metric $g_w$ induced by $X_w$ is 
$$ g_w = \abs{\dot{w}}^2 dt^2 + \abs{w}^2 g.$$    
\end{lemma}
\begin{proof}
It is straightforward to check that $X_{w}$ is a Lagrangian immersion away from $w=0$, using the fact 
that $X$ is a Legendrian immersion in $\Sph^{2n-1}$. 
The form for the Lagrangian phase of $X_{w}$ in terms of the Lagrangian phase of $X$ and the curve $w$ 
follows easily from the definition of the Lagrangian phase in \ref{E:lagn:phase}.
\end{proof}

\begin{definition}[cf. Definition \ref{D:twist:slg}] 
\addtocounter{equation}{1}
A smooth immersed curve $w: I \ra \C$ is called an \emph{$n$-twisted SL curve} if the Lagrangian phase of $w$ satisfies
\begin{equation}
\addtocounter{theorem}{1}
\label{E:n:twist:sl}
e^{i\theta_w} = e^{-i(n-1) \arg{w}}.
\end{equation}
\end{definition}
\noindent
The next simple result characterises $n$-twisted SL curves as solutions of a nonlinear ODE for $w$.

\begin{lemma}[cf. Lemma \ref{L:w:slg}]
\addtocounter{equation}{1}
A curve $w: I \ra \C^*$ is an $n$-twisted SL curve if and only if admits a parametrisation such that 
\begin{equation}
\addtocounter{theorem}{1}
\label{E:n:twist:sl:ode}
\dot{w} = \overline{w}^{n-1}.
\end{equation}
\end{lemma}

\begin{proof}
The Lagrangian angle of a smooth immersed curve $w:I \ra \C$ is given by 
$e^{i\theta_w} = \frac{\dot{w}}{\abs{\dot{w}}}$. Hence if $w$ satisfies \ref{E:n:twist:sl} 
then 
$$ e^{i\theta_w} = \frac{\dot{w}}{\abs{\dot{w}}} = \frac{\overline{w}^{n-1}}{\abs{\overline{w}}^{n-1}} = e^{-i(n-1)\arg{w}}.$$
Conversely, if $w: I \ra \C^*$ is an $n$-twisted SL curve then 
$$ \frac{\dot{w}}{\abs{\dot{w}}} = \frac{\overline{w}^{n-1}}{\abs{\overline{w}}^{n-1}}.$$
Since $w$ is a smooth immersed curve in $\C^*$ we can reparametrise it so that $\abs{\dot{w}} = \abs{w}^{n-1}$, 
and then in this parametrisation we have $\dot{w} = \overline{w}^{n-1}$.
\end{proof}

\begin{corollary}[cf. Corollary \ref{C:sl:triples}] 
\addtocounter{equation}{1}
If $X: \Sigma \ra \Sph^{2n-1}$ is a special Legendrian immersion and $w:I \ra \C^*$ is an $n$-twisted SL curve, 
then $X_w: I \times \Sigma \ra \C^n$ is a special Lagrangian immersion.
\end{corollary}
In fact, one can also easily show that if $X$ is a contact stationary Legendrian immersion (which is not special Legendrian) 
and $w$ is an $n$-twisted SL curve, then $X_w$ is a Hamiltonian stationary immersion of $I \times \Sigma$ 
which is not special Lagrangian.

\medskip

If $n=2$ then \ref{E:n:twist:sl:ode} becomes the linear ODE $\dot{w}=\overline{w}$. 
Straightforward calculation shows that any solution to this linear ODE has the form 
$$ w(t) = a e^{t} + i b e^{-t}, \quad \text{for $a$, $b \in \R$},$$
and that $\Imag{w^{2}}(t) \equiv 2ab$. In particular when $n=2$ all solutions of \ref{E:n:twist:sl} are 
defined for all $t\in \R$ (in contrast to the case when $n>2$). 
Solutions with either $a$ or $b$ zero but not both zero 
give horizontal or vertical half-lines through the origin in the complex plane which approach the origin 
either as $t\ra \infty$ or $t \ra -\infty$ and tend to infinity at the other end. 
Any solution with $ab \neq 0$ gives a connected curve (one of the two components of $\Imag{w^{2}}=2ab$) 
contained in a single quadrant asymptotic as $t\ra \pm \infty$ to either a vertical or a horizontal half-line.

From now on we assume $n>2$.
We will analyse \ref{E:n:twist:sl:ode}, its symmetries and solutions in detail as a warmup 
for the more complicated analysis of $(p,q)$-twisted SL curves we will encounter shortly. 
We will see that equation \ref{E:n:twist:sl:ode} reoccurs when analysing limiting 
behaviour of certain $(p,q)$-twisted SL curves. 

%
The ODE \ref{E:n:twist:sl:ode} has the following five obvious types of symmetry:
\begin{enumerate}
\item Time translation invariance, i.e. $t \mapsto t + t_0$ for some constant $t_0\in \R$.
\item Multiplication by an $n$th root of unity, i.e. $w \mapsto \alpha w$ where $\alpha^n=1$.
\item Complex conjugation, i.e. $w \mapsto \overline{w}$.
\item The simultaneous time and spatial transformation given by 
$$ t \mapsto -t, \quad w \mapsto \beta w, \quad \text{where} \quad \beta^n =-1.$$
\item The simultaneous time and spatial rescaling given by 
$$t \mapsto \lambda^{1-2/n} t, \quad w \mapsto \lambda^{1/n} w, \quad \text{for any } \lambda>0.$$ 
More precisely, $w$ is a solution of \ref{E:n:twist:sl:ode} if and only if 
$w_{\lambda}(t):= \lambda^{1/n} w( \lambda^{1-2/n}t)$ is.
\end{enumerate}

We now describe the basic properties  of the ODE \ref{E:n:twist:sl:ode}.
\begin{lemma}[cf. Proposition \ref{P:odes:p:n}] Assume that $n\ge 3$. \hfill
\addtocounter{equation}{1}
\label{L:lagn:cat:ode}
\begin{enumerate}
\item[(i)]
Solutions to the $n$-twisted SL ODEs \ref{E:n:twist:sl:ode}  admit the conserved quantity 
$$\mathcal{I}:= \Imag{w^n}.$$ 
Symmetries (1) and (2) preserve $\mathcal{I}$, (3) and (4) send $\mathcal{I} \mapsto -\mathcal{I}$ and (5) sends $\mathcal{I}\mapsto \lambda \mathcal{I}$. 
Hence by scaling as in (5) we may assume either $\mathcal{I}=0$ or $\mathcal{I}=\pm 1$, and also by using symmetry 
(4) we can assume $\mathcal{I}=1$.
\item[(ii)]
The origin is the only stationary point of \ref{E:n:twist:sl:ode}.
\item[(iii)]
The initial value problem for \ref{E:n:twist:sl:ode} with any initial data $w(0) \in \C^{*}$ has a unique real analytic solution $w:J \ra \C$ defined on a bounded interval $J \subset \R$.
\item[(iv)]
For any solution of \ref{E:n:twist:sl:ode} with $\mathcal{I}(w) = \lambda \neq 0$ 
write $w(t) = \sqrt{y(t)}\,e^{i\psi(t)}$ where $y:= \abs{w}^{2}$. Then 
$y$ and $\psi$ satisfy the following
$$
\frac{1}{2}\dot{y} + i \lambda = {w}^{n}, \quad \dot{y} = {2y}^{n/2}\cos{n\psi}, \quad \lambda = 
{y}^{n/2} \sin{n\psi}, \quad \dot{y}^{2} = 4 (y^{n} - \lambda^{2}),
\quad 
 y \,\dot{\psi} = -\lambda.
$$
In particular $y(t)$ is increasing when $\cos{n\psi}(t)=\Real(w^{n})(t)>0$, decreasing when $\cos{n\psi}(t)=\Real(w^{n})(t)<0$ and has a stationary point if and only if $\cos{n\psi}(t)=\Real(w^{n})(t)=0$, and for $\lambda>0$ $\psi(t)$ is a decreasing function of $t$.
\item[(v)]
Any solution of \ref{E:n:twist:sl:ode} with $\mathcal{I}(w) = \lambda>0$ can be written in the form 
$$
w_{\lambda,t_{0},k}(t) = \lambda^{1/n}\, e^{2\pi ki/n} w_{1}(\lambda^{1-2/n}(t+t_{0})),
$$
for some $k\in \Z/n\Z$ and $t_{0} \in \R$, where $w_{1}$ is the unique solution of \ref{E:n:twist:sl:ode} with 
$$
w(0)=e^{i\pi/2n}$$ 
(and hence $\mathcal{I}(w_{1})=1)$. 
\item[(vi)]
The solution $w_{1}$ defined in part (v) has the symmetry
$$ w_{1} \circ \tbar = e^{-i\pi/n} \overline{w}_{1},$$
where $\tbar$ denotes reflection in the origin $t\mapsto -t$.
\item[(vii)]
The solution $w_{\lambda,t_{0},k}$ from (v) is defined on the time interval 
$$ I_{\lambda,t_{0}} = \lambda^{-1+2/n}\left(-T_{1}-t_{0},T_{1}-t_{0}\right),$$
where
$$ T_{1}:=\int_{1}^{\infty}{\frac{dy}{2\sqrt{y^{n}-1}}} = \frac{\Gamma(\tfrac{1}{2}) \Gamma(\tfrac{1}{2}-\tfrac{1}{n})}{2\Gamma(-\tfrac{1}{n})} = \frac{\sqrt{\pi}\, \Gamma(\tfrac{1}{2}-\tfrac{1}{n})}{2\Gamma(-\tfrac{1}{n})},$$
and $\Gamma$ denotes the gamma function. 
\end{enumerate} 
\end{lemma}

%
%

\begin{proof}
(i) For any solution of \ref{E:n:twist:sl:ode} we have 
$$ \frac{d}{dt}w^n = nw^{n-1}\dot{w} = nw^{n-1} \overline{w}^{n-1} = n \abs{w}^{2(n-1)} \in \R,$$
and hence $\mathcal{I}:= \Imag{w^n}$ is conserved. It is straightforward to check the action of the symmetries on $\mathcal{I}$
is as claimed.\\
(ii) Stationary points are zeros of the vector field $V(w):=w^{n-1}$. \\
(iii) The vector field $V$ from (ii) defining \ref{E:n:twist:sl:ode} is clearly real algebraic. Hence from standard local existence and uniqueness results for the initial 
value problem,  locally \ref{E:n:twist:sl:ode} admits a unique real analytic solution for any initial data and this local solution 
depends real analytically on the initial data. In (vii) we show that for $n>2$ any nonstationary 
solution exists only on a finite interval $J$. \\
(iv) Once we prove the first equation the others follow from looking at real and imaginary parts or comparing the modulus 
squared of both sides. 
\ref{E:n:twist:sl:ode} implies $2\Real(w^{n})=2\Real(\overline{w}^{n})=2\Real(\dot{w}\overline{w}) = \dot{y}$ 
and $\Imag(w^{n})=\lambda$. \\
(v) By time translation invariance we can assume $\dot{y}(0)=0$ and therefore by (iv) $w^{n}(0)=i\lambda$. 
By scaling we can assume that $\lambda=1$ and hence $w^{n}(0)=i$. Multiplying by an $n$th root of unity we 
can assume $w(0)=e^{i\pi/2n}$ and the result follows. \\
(vi) Define $\hat{w}= e^{i\pi/n} \overline{w}_{1} \circ \tbar$. $\hat{w}$ is also a solution of \ref{E:n:twist:sl:ode} 
and $\hat{w}(0) = w(0)$. Hence by uniqueness of the initial value problem $\hat{w}\equiv w_{1}$.\\
(vii) It suffices to prove the result for $w_{1}$: the result for $w_{\lambda,t_{0},k}$ then follows immediately by applying the scaling symmetry (5) to $w_{1}$ and a time translation. We can apply local existence repeatedly until $y=\abs{w}^{2}$ 
goes to infinity. Hence it suffices to study the maximal interval of existence for $y$. 
By the symmetry (vi) of $w_{1}$, $y:=\abs{w_{1}}^{2}$ is an even function 
of $t$, and hence is defined on a symmetric interval $(-T_{1},T_{1})$ and $T_{1}$ is given by
$$T_{1} = \int_{1}^{\infty}\frac{1}{2\sqrt{y^{n}-1}}dy.$$
\end{proof}

By choosing the $1$-parameter family of solutions $w_{\lambda,0,0}$ from \ref{L:lagn:cat:ode}.v we obtain
\addtocounter{equation}{1}
\begin{prop}[\mbox{\cite[Thm A]{haskins:slgcones}},\cite{joyce:symmetries,castro:urbano:foliated}]
\label{P:acslg}
Let $\Sigma$ be any special Legendrian submanifold in $\Sph^{2n-1}$.
For any $d\in \R$, let $\Sigma_d$ denote the set
$$ \{w\,p\in \C^n:\ p\in \Sigma,\, w\in \C,\  \text{with}\  \Imag{w^n}=d,\  \arg{w}\in (0,\tfrac{\pi}{n})\}.$$
Then for $d\neq 0$, $\Sigma_d$ is an immersed asymptotically conical special Lagrangian submanifold of $\C^n$
with two ends asymptotic to the two SL cones $C(\Sigma)$ and $e^{i\pi/n}C(\Sigma)$.
\end{prop}

When we apply Proposition \ref{P:acslg} to $\Sigma = \Sph^{n-1}\subset \R^n \subset \C^n$  the 
asymptotically conical SL $n$-folds we obtain are known as the \emph{Lagrangian catenoids} $L$. 
The Lagrangian catenoids $L$ can be characterised in a number of ways. For example, they are the only non-flat SL $n$-folds in $\C^n$ foliated by round $n-1$-spheres \cite{castro:urbano:foliated}. 
Also any nonsingular SL $n$-fold in $\C^n$ invariant under the standard complex linear action of $\sorth{n}$ on $\C^n$ is (conjugate to a piece of) a Lagrangian catenoid.

\begin{definition}[Standard embeddings of the Lagrangian catenoid]\hfill
\addtocounter{equation}{1}
\label{D:lag:cat:std}
\begin{enumerate}
\item [(i)]
Suppose $n>2$. 
Let $w_1:\R \ra \C^{*}$ be the unique solution of \ref{E:n:twist:sl:ode} with $w_{1}(0)= e^{i\pi/2n}$ as in \ref{L:lagn:cat:ode}.v and 
let $T_{1}$ denote the lifetime defined in \ref{L:lagn:cat:ode}.vii. We call the special Lagrangian embedding 
$X_{1}: (-T_{1},T_{1}) \times \Sph^{n-1} \ra \C^{n}$ defined by 
$$X_{1}(t,\sigma) = w_{1}(t) \sigma, \quad \text{for\ } t\in (-T_{1},T_{1}), \quad \sigma \in \Sph^{n-1}\subset \R^{n},$$
the \emph{standard embedding of the Lagrangian catenoid of size $1$} or just the \emph{standard unit Lagrangian catenoid} for short. 
\item[(ii)]
 Suppose $n=2$. Let $w_{1}:\R \ra \C^{*}$ 
be the unique solution of \ref{E:n:twist:sl:ode} with $w_{1}(0)=e^{i\pi/4}$, i.e. 
$$w_{1}(t) = \frac{1}{\sqrt{2}} (e^{t}+ie^{-t}).$$
We call the special Lagrangian embedding 
$X_{1}: \R \times \Sph^{1} \ra \C^{2}$ defined by 
$$X_{1}(t,\sigma) := w_{1}(t) \sigma, \quad \text{for\ } t\in \R, \ \sigma \in \Sph^{1}\subset \R^{2},$$
the \emph{standard embedding of the $2$-dimensional Lagrangian catenoid of size $1$} or just the \emph{standard unit $2$-dimensional Lagrangian catenoid} for short. 
\end{enumerate}
\end{definition} 
More generally, if we replace $w_{1}$ above with the rescaled solution $w_{\lambda}(t):= \lambda^{1/n}w_{1}(\lambda^{1-2/n}t)$ 
and define $X_{\lambda}(t,\sigma) = w_{\lambda}(t) \sigma$ then we obtain the standard embedding of the 
Lagrangian catenoid of size $\lambda^{1/n}$. We call it size $\lambda^{1/n}$ because the waist of $X_{\lambda}$, 
i.e. the sphere $w_{\lambda}(t) \cdot \Sph^{n-2}$ of minimal radius has radius $\lambda^{1/n}$.
Symmetry \ref{L:lagn:cat:ode}.vi implies the following discrete symmetry of the standard embedding of the Lagrangian catenoid
(of any size $\lambda$) 
$$ \tbartilde \circ X_{\lambda} = X_{\lambda} \circ \tbar,$$
where $\tbar \in \Diff(\R \times \Sph^{n-1})$ is given by $\tbar(t,\sigma) = (-t,\sigma)$ and 
$\tbartilde \in \orth{2n}$ is defined by 
$$\tbartilde (z) = e^{-i\pi/n} \overline{z} \quad \text{for \ } z\in \C^{n}.$$

\subsection*{Twisted products of spherical Legendrian immersions}
$\phantom{ab}$
We turn now to a similar mechanism for producing a new 
Legendrian immersion $X_1 *_{\bw} X_2$ into a higher-dimensional 
odd sphere from a Legendrian curve $\bw$ in $\Sph^3$ and a pair 
of Legendrian immersions $X_1$ and $X_2$ into lower-dimensional odd spheres. We call $X_1 *_{\bw} X_2$ 
the \emph{$\bw$-twisted product} of $X_1$ and $X_2$ for reasons explained in \ref{R:product:cones}. 

\addtocounter{equation}{1}
\begin{definition}
\label{D:twisted:product}
Let $I \subseteq \R$ be a connected interval, 
$\Sigma_1$ and $\Sigma_2$ be two smooth manifolds of dimensions $n_1$ and $n_2$ respectively, and
$X_i: \Sigma_i \ra \Sph^{2m_i -1}$ for $i=1,2$ be smooth maps into odd-dimensional spheres.
Let $\bw=(w_1(t),w_2(t)\,): I \ra \Sph^3$ be a smooth immersed curve in $\Sph^3$. 
Then the $\bw$-twisted product of $X_1$ and $X_2$, denoted $X_1 *_{\bw} X_2$, is the smooth map
$$X_1 *_{\bw} X_2: I \times \Sigma_1 \times \Sigma_2 \ra \Sph^{2m_1+2m_2-1} \subset \C^{m_1+m_2} = \C^{m_1} \times \C^{m_2},$$
defined by
\addtocounter{theorem}{1}
\begin{equation}
\label{E:X1wX2}
X_1 *_{\bw} X_2\,(t,\sigma_1,\sigma_2) = (w_1(t)X_1(\sigma_1),\, w_2(t)X_2(\sigma_2)\,).
\end{equation}
\end{definition}
\begin{remark}
\addtocounter{equation}{1}
\label{R:twist:def:p:eq:1}
In the definition of a twisted product above it is also convenient to allow the degenerate case where $\Sigma_{1}$ 
is $0$-dimensional. We will need the case where $\Sigma_{1}$ is a single point $p$ and the map 
$X_{1}$ maps $p \mapsto (1,0)\in \Sph^{3}$. In this case we will drop the reference to $X_{1}$ and 
$\Sigma_{1}$ and the subscript 
for $X_{2}$ and $\Sigma_{2}$ and write $X_{\bw}: I \times \Sigma \ra \Sph^{2m-1}$
for the map defined by 
\begin{equation}
\addtocounter{theorem}{1}
\label{E:twist:def:p:eq:1}
X_{\bw}(t,\sigma) = (w_{1}(t), w_{2}(t)\,X(\sigma)).
\end{equation}
We will still refer to this degenerate case as a twisted product.
\end{remark}

The following extended remark explains the origin of the term \emph{twisted product} in Definition \ref{D:twisted:product}. 
\begin{remark}
\addtocounter{equation}{1}
\label{R:product:cones}
Let $C_{1}$ and $C_{2}$ be cones in $\C^{m_{1}}$ and $\C^{m_{2}}$ respectively.
The product $C_{1}\times C_{2} \subset \C^{m_{1}} \times \C^{m_{2}} \cong \C^{m_{1}+m_{2}}$
is also a cone. Suppose now that $C_{1}$ and $C_{2}$ are both regular cones, i.e. $C_{i}=C(\Sigma_{i})$ is the cone over a smooth closed submanifold 
$\Sigma_{i} \subset \Sph^{2m_{i}-1}$ and hence has an isolated singularity at $\mathbf{0} \in \C^{m_{i}}$.
Let $\Sigma_{12}\subset \Sph^{2m_{1}+2m_{2}-1}$ denote the link of the product cone $C_{1}\times C_{2} \subset \C^{m_{1}+m_{2}}$. 
Clearly
\begin{equation}
\addtocounter{theorem}{1}
\Sigma_{12}= \{( \cos{t}\, \sigma_{1}, \sin{t}\, \sigma_{2})\,| \, t\in [0,\tfrac{1}{2}\pi],\, \sigma_{1}\in \Sigma_{1}, \, \sigma_{2}\in \Sigma_{2}\} \subset \Sph^{2m_{1}+2m_{2}-1}.
\end{equation}
There is an obvious surjective map 
$$\Pi: [0,\pi/2] \times \Sigma_{1} \times \Sigma_{2} \ra \Sigma_{12}$$
from the manifold with boundary $[0,\pi/2] \times \Sigma_{1} \times \Sigma_{2}$ to the link of our product cone
 $\Sigma_{12}$
defined by 
\begin{equation}
\addtocounter{theorem}{1}
\label{E:Pi}
\Pi(t,\sigma_{1},\sigma_{2}) = (\cos{t} \,\sigma_{1},\sin{t} \,\sigma_{2}).
\end{equation}
Clearly, the map $\Pi$ can be written as a $\bw$-twisted product by taking $X_{1}$ and $X_{2}$ to be the inclusion maps 
$i_{1}: \Sigma_{1} \ra \Sph^{2m_{1}-1}$ and  $i_{2}: \Sigma_{2} \ra \Sph^{2m_{2}-1}$ respectively and 
$\bw:I \ra \Sph^{3}$ to be the equatorial curve $ \bw:[0,\pi/2] \ra \Sph^{1} \subset \Sph^{3}$ defined by 
\begin{equation}
\addtocounter{theorem}{1}
\label{E:w:s1}
\bw(t) = (\cos{t}, \sin{t}).
\end{equation}

We therefore view the $\bw$-twisted product defined in \ref{D:twisted:product} as a ``twisted'' version 
of taking the product of two regular cones. It ``twists'' the product construction by allowing a general curve $\bw \in \Sph^{3}$ 
instead of the standard equatorial curve $\Sph^{1}\subset \Sph^{3}$ defined in \ref{E:w:s1}. 
It is natural therefore to call the curve $\bw:I \ra \Sph^{3}$ the \emph{twisting curve}.

The degenerate case discussed in Remark \ref{R:twist:def:p:eq:1} also specialises to a product of cones $C_{1}\times C_{2}$
when the twisting curve is the equatorial curve \ref{E:w:s1} and $C_{1} = \R^{+}\subset \C$ and $C_{2}=C(\Sigma)$.
Thus we can still view $X_{\bw}$ (defined in \ref{E:twist:def:p:eq:1}) 
as a twisted version of the product of two cones $\R^{+}\times C$ and hence the name 
twisted product is appropriate even in this degenerate case.
\end{remark}

The product cone $C_{1} \times C_{2}$ is not a regular cone even when both $C_{1}$ and $C_{2}$ are regular cones. 
Equivalently, the link $\Sigma_{12} \subset \Sph^{2m_{1}+2m_{2}-1}$ is not a smooth submanifold. 
As a topological space we can think of $\Sigma_{12}$ as being obtained from the 
generalised cylinder $[0,\pi/2]\times \Sigma_{1}\times \Sigma_{2}$ by 
a modified ``coning-off the boundary''  construction. Namely, at the endpoint $t=0$ we cone-off 
$\Sigma_{2}$ inside $\{\mathbf{0}\}\times \Sigma_{1}\times \Sigma_{2}$ but leave $\Sigma_{1}$ untouched, 
whereas at the endpoint $t=\pi/2$ we instead cone-off $\Sigma_{1}$ but leave $\Sigma_{2}$ alone.
Thus  $\Sigma_{12}$ has two different types of singularities: conical singularities modelled on $\Sigma_{2}$ 
along a copy of $\Sigma_{1}$ and conical singularities modelled on $\Sigma_{1}$ along a copy of $\Sigma_{2}$.

$\Pi$ is a smooth embedding away from the endpoints of the interval $[0,\pi/2]$ 
and induces a Riemannian metric $g$ on  $(0,\pi/2)\times \Sigma_{1}\times \Sigma_{2}$ 
defined by 
$$g = dt^{2} + \cos^{2}{t}\, g_{1} + \sin^{2}{t}\,g_{2},$$
where $g_{1}$ and $g_{2}$ are the Riemannian metrics induced on $\Sigma_{1}$ and $\Sigma_{2}$ by the spherical inclusions $i_{1}$ and $i_{2}$.
In particular, we see that the metric $g$ degenerates at $t=0$ and $t=\pi/2$ in a manner consistent with the description of the  
singularities of $\Sigma_{12}$ we gave in the previous paragraph.

In the exceptional case where $C_{1}=\R^{m_{1}}\subset \C^{m_{1}}$ and $C_{2}= \R^{m_{2}} \subset \C^{m_{2}}$ then 
obviously $C_{1}\times C_{2} \cong \R^{m_{1}+m_{2}}$ and therefore 
$\Sigma_{12}=\Sph^{m_{1}+m_{2}-1}\subset \Sph^{2m_{1}+2m_{2}-1}$ is not singular. 
In this case the images of the hypersurfaces with $t$ constant under the map 
$$\Pi: [0,\tfrac{1}{2}\pi]\times \Sph^{m_{1}-1} \times \Sph^{m_{2}-1} \ra \Sph^{m_{1}+m_{2}-1}$$
give a (singular) codimension one foliation of $\Sph^{m_{1}+m_{2}-1}$ by hypersurfaces isometric to the product 
of spheres $\Sph^{m_{1}-1}(\cos{t}) \times \Sph^{m_{2}-1}(\sin{t})$. 
As $t\ra 0$ the second spherical factor shrinks to radius $0$, 
while the first spherical factor shrinks to radius $0$ as $t \ra \pi/2$. Restricting $\Pi$ to the open interval $(0,\pi/2)$ gives a 
foliation of $\Sph^{m_{1}+m_{2}-1}\setminus (\Sph^{m_{1}-1},0) \cup (0,\Sph^{m_{2}-1})$ that omits the two singular 
leaves corresponding to the endpoints $t=0$ and $t=\pi/2$. The leaves of this singular foliation of $\Sph^{m_{1}+m_{2}-1}$ 
are exactly the orbits of the group $\sorth{m_{1}}\times \sorth{m_{2}} \subset \sorth{m_{1}+m_{2}}$. 
When $m_{1}=m_{2}=2$ the singular foliation above yields the  standard singular foliation of 
$\Sph^{3}$ by an open interval of $2$-tori which degenerates at the ends of the interval to the linked Hopf circles $(\Sph^{1},0) \subset \Sph^{3}$ and $(0,\Sph^{1})\subset \Sph^{3}$.

\medskip
Moving from the smooth to the Legendrian category we can refine the notion of twisted product to generate new 
Legendrian immersions from a pair of lower-dimensional Legendrian immersions, provided the twisting curve itself 
is Legendrian in $\Sph^{3}$.


\addtocounter{equation}{1}
\begin{prop}[Legendrian twisted products,  \cite{castro:li:urbano} Thm 3.1]
\label{P:leg:triple}
Suppose that the twisting curve $\bw$ is a Legendrian curve in $\Sph^3$, that
$(\Sigma_1,g_1)$ and $(\Sigma_2,g_2)$ are oriented Riemannian manifolds of dimension $p-1>0$ and $q-1>0$ respectively, and
that $X_1:\Sigma_{1} \ra \Sph^{2p-1}$ and $X_2:\Sigma \ra \Sph^{2q-1}$ are Legendrian isometric immersions.
Away from points where $w_{1}$ or $w_{2}$ vanish the $\bw$-twisted product
$$X_1 *_{\bw} X_2: I \times \Sigma_1 \times \Sigma_2 \ra \Sph^{2p+2q-1} \subset \C^{p+q} = \C^p \times \C^{q},$$
defined in \ref{D:twisted:product} is a Legendrian immersion
whose Lagrangian phase $e^{i\theta_X}$  satisfies the following twisted product relation
\addtocounter{theorem}{1}
\begin{equation}
\label{E:twist:angle}
e^{i\theta_{X}} = (-1)^{p-1}e^{i\theta_{X_1}} e^{i\theta_{X_2}} e^{i\theta_{\bw} + i(p-1)\arg{w_1} +i(q-1) \arg{w_2}},
\end{equation}
and the metric $g$ induced by $X_1 *_{\bw} X_2$ is
\addtocounter{theorem}{1}
\begin{equation}
\label{E:twist:metric}
g = \abs{\dot{\bw}}^2 dt^2 + \abs{w_1}^2 g_1 + \abs{w_2}^2 g_2.
\end{equation}
\end{prop}
\noindent
The analogue of Proposition \ref{P:leg:triple} in the degenerate case $p=1$ considered in \ref{E:twist:def:p:eq:1} is 
\begin{prop}
\label{P:leg:twist:p:eq:1}
\addtocounter{equation}{1}
Suppose that the twisting curve $\bw$ is a Legendrian curve in $\Sph^{3}$,  that $(\Sigma,g)$ is an oriented Riemannian 
manifold of dimension $n-2$ and that $X: \Sigma \ra \Sph^{2n-3}$ is a Legendrian isometric immersion. 
Away from points where $w_{2}$ vanishes the $\bw$-twisted product 
$$X_{w}: I \ra \Sigma \ra \Sph^{2n-1}\subset \C^{n}=\C \times\C^{n-1},$$
defined in \ref{E:twist:def:p:eq:1} is a Legendrian immersion whose Lagrangian phase $e^{i\theta}$ satisfies 
the twisted product relation
$$ e^{i\theta} = e^{i\theta_{X}} e^{i\theta_{\bw}+i(n-2)\arg{w_{2}}},$$
and the metric  induced by $X_{\bw}$ is $\abs{\dot{w}}^{2}dt^{2} + \abs{w_{2}}^{2}\, g$.
\end{prop}

\addtocounter{equation}{1}
\begin{remark}
$X_1 *_{\bw} X_2$ fails to be an immersion at points where either $w_{1}$ or $w_{2}$ vanish.
Away from such points we have $\vol_g = \abs{\dot{\bw}}\, \abs{w_1}^{p-1} \, \abs{w_2}^{q-1} dt \vol_{g_1}\, \vol_{g_2}$, and hence when both $\Sigma_{1}$ and $\Sigma_{2}$ are closed 
the $\bw$-twisted product has volume
\addtocounter{theorem}{1}
\begin{equation}
\label{E:twist:vol}
\vol{(X_1 *_{\bw} X_2)} = \vol(X_1) \vol(X_2) \int_{I}{ \abs{\dot{\bw}}\, \abs{w_1}^{p-1} \abs{w_2}^{q-1}\, dt}.
\end{equation}
The obvious analogue of \ref{E:twist:vol} holds for the degenerate case $p=1$.
\end{remark}

\subsubsection*{Twisted products of special Legendrians and $(p,q)$-twisted special Legendrian curves}
From now on we will always consider the case where the integers $p$ and $q$ satisfy $p \le q$, $p\ge 1$ and $q \ge 2$. 
There is no loss of generality in making this assumption. 
We call such a pair $(p,q)$ of positive integers \emph{admissible}. 
For each admissible pair of integers $(p,q)$ we define a distinguished class of Legendrian curves in $\Sph^3$.
\addtocounter{equation}{1}
\begin{definition}
\label{D:twist:slg}
We call a Legendrian curve $\bw$ in $\Sph^3$ a \emph{$(p,q)$-twisted special Legendrian (SL) curve} if
the Lagrangian phase of $\bw$ satisfies
\addtocounter{theorem}{1}
\begin{equation}
\label{E:w:slg}
e^{i\theta_{\bw}} = (-1)^{p-1}e^{-i(p-1)\arg{w_1}-i(q-1) \arg{w_2}}.
\end{equation}
\end{definition}
Proposition \ref{P:leg:triple} (and \ref{P:leg:twist:p:eq:1} for the degenerate case $p=1$) 
has the following easy corollary which allows us to generate a new
special Legendrian immersion in $\Sph^{2(p+q)-1}$ from  a $(p,q)$-twisted SL
curve in $\Sph^3$ and a pair of special Legendrian immersions into $\Sph^{2p-1}$ and
$\Sph^{2q-1}$ respectively.
\addtocounter{equation}{1}
\begin{corollary}[Special Legendrian twisted products]
\label{C:sl:triples}
Let $X_1$, $X_2$ and $\bw$ be as in Proposition \ref{P:leg:triple}.
If additionally, $X_1$ and $X_2$ are both special Legendrian then the $\bw$-twisted product $X_1 *_{\bw} X_2$ is
special Legendrian if and only if $\bw$ is a $(p,q)$-twisted SL curve in $\Sph^3$.
Similarly, let $X$ and $\bw$ be as in Proposition \ref{P:leg:twist:p:eq:1}.
If additionally, $X$ is special Legendrian then the $\bw$-twisted product $X_{\bw}$ is special Legendrian 
if and only if $\bw$ is a $(1,n-1)$-twisted SL curve in $\Sph^{3}$.
\end{corollary}

The following characterisation of $(p,q)$-twisted SL curves in $\Sph^3$ is central to the rest of this paper
\addtocounter{equation}{1}
\begin{lemma}[\mbox{\cite[Cor 1]{castro:li:urbano}}]
\label{L:w:slg}
Any curve $\bw: I \ra \C^2$ satisfying
\addtocounter{theorem}{1}
\begin{equation}
\label{E:slg:ode}
\overline{w}_1 \dot{w}_1 = - \overline{w}_2 \dot{w}_2 = (-1)^p \overline{w}_1^p\, \overline{w}_2^{q}, \qquad \abs{\bw(0)}=1,
\end{equation}
is a $(p,q)$-twisted SL curve in $\Sph^3$.
Conversely, any $(p,q)$-twisted SL curve in $\Sph^3$ containing no points with $w_1(t)=0$ or $w_2(t)=0$
admits a parametrisation satisfying \ref{E:slg:ode}.
\end{lemma}
\begin{proof}
First notice that the Lagrangian phase $e^{i\theta_{\bw}}$ of any Legendrian curve $\bw$ in $\Sph^3$ can be expressed as
\addtocounter{theorem}{1}
\begin{equation}
\label{E:legn:curve:phase}
e^{i\theta_{\bw}} = \frac{w_1 \dot{w}_2 - \dot{w}_1 w_2}{\abs{\dot{\bw}}},
\end{equation}
since $\bw$ has norm $1$ and is hermitian orthogonal to $\dot{\bw}$.

Now suppose $\bw$ is a curve in $\C^2$ satisfying \ref{E:slg:ode}.
The real part of the equality $\overline{w}_1 \dot{w}_1 + \overline{w}_2 \dot{w}_2=0$ implies that $\tfrac{d}{dt}\abs{\bw}^2=0$, and
hence $\bw$ lies in $\Sph^3$. The imaginary part of the same equality implies that $\bw$ is a Legendrian curve.
Straightforward calculation using \ref{E:slg:ode} shows that $\bw$ satisfies
\addtocounter{theorem}{1}
\begin{equation}
\label{E:w:reparam}
\abs{\dot{\bw}} = \abs{w_1}^{p-1} \abs{w_2}^{q-1},
\end{equation}
and
\addtocounter{theorem}{1}
\begin{equation}
\label{E:w1:w2:dot}
w_1 \dot{w}_2 - \dot{w}_1 w_2 = (-1)^{p-1}\overline{w}_1^{p-1}\overline{w}_2^{q-1}.
\end{equation}
Combining \ref{E:legn:curve:phase}, \ref{E:w:reparam} and \ref{E:w1:w2:dot} it follows that the Lagrangian phase of $\bw$ satisfies
\ref{E:w:slg} as required.

For the converse, notice that any Legendrian curve $\bw$ in $\Sph^3$ satisfies the first and third equalities in \ref{E:slg:ode}, \textit{i.e.}
$\overline{w}_1 \dot{w}_1 = - \overline{w}_2 \dot{w}_2$ and $\abs{\bw(0)}=1$.
Also we can rewrite \ref{E:w:slg} as
$$e^{i\theta_{\bw}} = (-1)^{p-1} \frac{\overline{w}_1^{p-1} \overline{w}_2^{q-1}}{\abs{w_1}^{p-1} \abs{w_2}^{q-1}},$$
and hence using \ref{E:legn:curve:phase} also as
$$ \frac{w_1 \dot{w}_2 - \dot{w}_1 w_2}{\abs{\dot{\bw}}} = (-1)^{p-1} \frac{\overline{w}_1^{p-1} \overline{w}_2^{q-1}}{\abs{w_1}^{p-1} \abs{w_2}^{q-1}}.$$
Now if we reparametrise $\bw$ so that it satisfies \ref{E:w:reparam} then from the previous equality
we see that \ref{E:w:slg} is equivalent to equation \ref{E:w1:w2:dot}.
Multiplying \ref{E:w1:w2:dot} by $\overline{w}_1 \overline{w}_2$ and using the fact that $\bw$ satisfies
$\abs{\bw}^2=1$ and $\overline{w}_1 \dot{w}_1 = - \overline{w}_2 \dot{w}_2$, we get the second equality of \ref{E:slg:ode} as required.
\end{proof}

\addtocounter{equation}{1}
\begin{remark}
\label{R:sl:twist:sign}
By changing the parameter $t$ of the curve $\bw$ to $-t$ if necessary one can always absorb the dimension-dependent sign $(-1)^p$ from \ref{E:slg:ode}
and therefore it suffices to study curves $\bw$ in $\Sph^3$ satisfying
$$\overline{w}_1 \dot{w}_1 = - \overline{w}_2 \dot{w}_2 = \overline{w}_1^p \overline{w}_2^{q}, $$
with initial condition $\abs{\bw(0)}=1$.
Moreover, away from points where $w_{1}w_{2}=0$ these ODEs are equivalent to 
\begin{equation}
\addtocounter{theorem}{1}
\label{E:slg:ode2}
\dot{w}_{1} = \overline{w}_{1}^{p-1}\overline{w}_{2}^{q}, \quad 
\dot{w}_{2} = - \overline{w}_{1}^{p} \overline{w}_{2}^{q-1}.
\end{equation}
\ref{E:slg:ode2} will be the most convenient form of the equations to use since it allows the cleanest treatment of 
the degenerate solutions where $w_{1}$ or $w_{2}$ can become zero.
\end{remark}

\addtocounter{equation}{1}
\begin{remark}
\label{R:sl:twist:vol}
If $\bw$ is a $(p,q)$-twisted SL curve in $\Sph^3$ with $p>1$, parametrized as in \ref{E:slg:ode}, then by combining \ref{E:twist:vol} and \ref{E:w:reparam}
we see that when $\Sigma_{1}$ and $\Sigma_{2}$ are both closed 
\addtocounter{theorem}{1}
\begin{equation}
\label{E:sl:twist:vol}
\vol{(X_1 *_{\bw} X_2)}  = \vol(X_1) \vol(X_2) \int_{I}{ \abs{\dot{\bw}}^2 \, dt}.
\end{equation}
Again the obvious analogue of \ref{E:sl:twist:vol} holds in the degenerate case $p=1$.
Therefore there is a close relation between volume of special Legendrian twisted products and the \textit{energy} of
$(p,q)$-twisted SL curves in $\Sph^3$ when using the parametrisation forced by \ref{E:slg:ode}.
\end{remark}

\subsubsection*{Twisted products of contact stationary immersions}
Although the main focus of this paper is the construction of $\sorth{p} \times\sorth{q}$-invariant 
special Lagrangian cones in $\C^{n}$ 
or equivalently  $\sorth{p} \times\sorth{q}$-invariant special Legendrian submanifolds of $\Sph^{2n-1}$ with very little extra effort one can also construct many
Hamiltonian stationary cones in $\C^{n}$ or equivalently contact stationary submanifolds in $\Sph^{2n-1}$ via the twisted product construction. 

To this end we define the following class of Legendrian curves in $\Sph^{3}$ generalising \ref{E:slg:ode}
\begin{definition}
\addtocounter{equation}{1}
\label{D:pq:twist:cs}
We call a curve $\bw: I \subset \R \ra \Sph^{3}$ a \emph{$(p,q)$-twisted contact stationary (CS) curve} if it satisfies the ODEs
\begin{equation}
\addtocounter{theorem}{1}
\label{E:pq:twist:cs}
\overline{w}_{1}\dot{w}_{1} = -\overline{w}_{2}\dot{w}_{2} = e^{i(a+bt)} \overline{w}_{1}^{p} \overline{w}_{2}^{q},
\quad t\in I,
\end{equation}
for some $a$, $b \in \R$.
\end{definition}
\begin{remark}
\addtocounter{equation}{1}
\label{R:sw:compare}
Note in the degenerate case $p=q=1$ these ODEs occur as equation (7.1) in Schoen-Wolfson's work on the classification 
of $2$-dimensional Hamiltonian stationary cones in $\C^{2}$ \cite{schoen:wolfson}. The system \ref{E:pq:twist:cs} is very simple 
to understand in this case because $\bw$ satisfies a system of linear equations.
Moreover, by direct differentiation of the equations for $\dot{w}_{1}$ and $\dot{w}_{2}$, 
$w_{1}$ and $w_{2}$ each satisfy autonomous second order linear equations.
\end{remark}
The reason for making this definition is the following
\begin{lemma}[Contact stationary twisted products\cite{castro:li:urbano}, Cor 3.2]
\addtocounter{equation}{1}
\label{L:twist:cs}
Let $X_{1}$, $X_{2}$ and $\bw$ be as in Proposition \ref{P:leg:triple}. If additionally $X_{1}$ and $X_{2}$ 
are both oriented contact stationary immersions and $\bw$ is a $(p,q)$-twisted contact stationary curve then the 
$\bw$-twisted product $X_{1} *_{\bw} X_{2}: I \times \Sigma_{1} \times \Sigma_{2} \ra \Sph^{2(p+q)-1}$ is also
a contact stationary immersion away from points where $w_{1}$ or $w_{2}$ vanish.
Moreover, if either $X_{1}$ or $X_{2}$ is contact stationary but not minimal Legendrian or if $\bw$ is a 
$(p,q)$-twisted CS curve with $b\neq 0$ then  $X_{1} *_{\bw} X_{2}$ is contact stationary but not minimal Legendrian.

Similarly, let $X$ and $\bw$ be as in Proposition \ref{P:leg:twist:p:eq:1}.
If additionally, $X$ is an oriented contact stationary immersion then the $\bw$-twisted product $X_{\bw}$ is an oriented contact 
stationary immersion if $\bw$ is a $(1,n-1)$-twisted CS curve in $\Sph^{3}$.
\end{lemma}
\begin{proof}
The proof follows from Proposition \ref{P:leg:triple} together with the characterisation of 
contact stationary and minimal Legendrian submanifolds of $\Sph^{2n-1}$ in terms of harmonicity and constancy
 of the Lagrangian phase $e^{i\theta}$ respectively. We sketch the proof.
Using the form of the metric $g$ induced by $X_{1} *_{\bw}X_{2}$ given in \ref{E:twist:metric} and the 
Lagrangian phase $e^{i\theta_{X}}$ of $X_{1} *_{\bw}X_{2}$ given by \ref{E:twist:angle} we calculate $\Delta_{g}
 e^{i\theta_{X}}$. 
 Using the fact that $X_{1}$ and $X_{2}$ are contact stationary 
we have $\Delta_{g_{1}}e^{i\theta_{X_{1}}}=  \Delta_{g_{2}}e^{i\theta_{X_{2}}}=0$, which
together with the fact that $\bw$ satisfies \ref{E:pq:twist:cs} allows us to conclude that 
$\Delta_{g}  e^{i\theta_{X}}=0$ and hence that  $X_{1}*_{\bw}X_{2}$ is contact stationary.

The proof in the case $p=1$ follows in the same way using Proposition \ref{P:leg:twist:p:eq:1} in place of \ref{P:leg:triple}.
\end{proof}

\begin{remark}
\label{R:pq:twist:curve}
\addtocounter{equation}{1}
Clearly, \ref{E:slg:ode} is a special case of \ref{E:pq:twist:cs} where $a=p\pi$ and $b=0$. 
If $\bw$ is a solution of \ref{E:pq:twist:cs} with parameters $(a,b)$ then for any constant $d \in \R$, 
$\bw' = e^{id} \bw$ is another solution of \ref{E:pq:twist:cs} with parameters $(a',b')=(a+(p+q)d,b)$.
Hence if $b=0$ then by choosing $d$ appropriately we can reduce \ref{E:pq:twist:cs} to \ref{E:slg:ode}. 
The analysis of \ref{E:pq:twist:cs} when $b \neq 0$ 
is more complicated than that of \ref{E:slg:ode} because the system \ref{E:pq:twist:cs} is no longer autonomous. 
In this paper we will analyse in great detail solutions of \ref{E:slg:ode} and say almost nothing further about 
solutions of \ref{E:pq:twist:cs} with $b \neq 0$. However, following \cite[eqn. 13] {castro:li:urbano} we note 
that for any  $c\in (0,\pi/2)$ the Legendrian curve $\bw: \R \ra \Sph^{3}$
\begin{equation}
\addtocounter{theorem}{1}
\label{E:twist:cs:curve}
\bw (t) = ( \cos{c} \exp( it \sin^{p}{c}\cos^{q-2}{c}), \sin{c} \exp(-it \sin^{p-2}{c} \cos^{q}{c})), \quad t\in \R
\end{equation}
satisfies \ref{E:pq:twist:cs} with $a=\pi/2$ and $ b = \sin^{p-2}{c}\,\cos^{q-2}{c}\, ( p \sin^{2}{s} - q \cos^{2}{c})$.
Clearly $b=0$ if and only if $\tan^{2}{c}=q/p$.

The $(p,q)$-twisted CS curve \ref{E:twist:cs:curve} is closed if and only if  $\tan^{2}{c} \in \Q$. 
In particular given relatively prime positive integers $m$ and $n$ 
choose the unique value of $c_{m,n} \in (0,\pi/2)$ so that $\tan^{2}{c_{m,n}}=m/n$, and therefore
$\cos{c_{m,n}}=\sqrt{n/(m+n)}$,  $\sin{c_{m,n}} = \sqrt{m/(m+n)}$. Hence for each fixed $(p,q)$ there is a countably infinite family 
of closed $(p,q)$-twisted CS curves $\bw_{m,n}$ of the form \ref{E:twist:cs:curve} parametrised by the pair of relatively prime positive integers $m$ and $n$. 
In the degenerate case when $p=q=1$ these closed curves $\bw_{m,n}$ 
are (up to a unitary transformation) nothing but the closed contact stationary curves $\gamma_{m,n}$ described in \ref{E:hs:cones:2d}.
\end{remark}
\begin{remark}
\label{R:cs:non:min:leg}
\addtocounter{equation}{1}
Combining Lemma \ref{L:twist:cs} and Remark \ref{R:pq:twist:curve} gives us 
two ways to construct contact stationary submanifolds that are not minimal Legendrian using 
the twisted product construction: (i) we take at least one of our initial immersions $X_{i}$ to be contact stationary but not minimal 
Legendrian and $\bw$ to be a $(p,q)$-twisted SL curve or 
(ii) we take the twisting Legendrian curve $\bw$ to be a $(p,q)$-twisted CS curve of the form 
\ref{E:twist:cs:curve} with $\tan^{2}{c} \neq q/p$. 
In the latter case we can allow both $X_{1}$ and $X_{2}$ to be special Legendrian, yielding a very simple method to 
generate higher-dimensional contact stationary immersions from a pair of lower-dimensional special Legendrians.
\end{remark}

To construct special Legendrian or contact stationary immersions of the closed manifold 
$S^{1} \times \Sigma_{1 } \times \Sigma_{2}$ from a pair of immersions of closed manifolds $\Sigma_{1}$ and $\Sigma_{2}$ 
we need $(p,q)$-twisted SL or CS curves that are closed. We call Legendrian immersions which arise this way, 
\emph{closed twisted products}. For each fixed $p$ and $q$ 
Remark \ref{R:pq:twist:curve} exhibited a countably infinite family of closed $(p,q)$-twisted CS curves $\bw_{m,n}$ 
parametrised by relatively prime positive integers $m$ and $n$. Moreover, $\bw_{m,n}$ is congruent to a $(p,q)$-twisted SL curve 
if and only if $m/n= p/q$.
 
We study closed $(p,q)$-twisted SL curves in Section \ref{S:embedded} by analysing
the periodicity conditions for solutions $\bw$ of \ref{E:slg:ode}. We will prove the following result 
(Theorem \ref{T:w:closed})

\vspace{0.3cm}
\noindent
\emph{For each admissible pair $(p,q)$ of positive integers there exists a countably infinite number of distinct closed $(p,q)$-twisted SL curves in $\Sph^{3}$.}

\vspace{0.3cm}

By the SL twisted product construction of Corollary \ref{C:sl:triples}, Theorem \ref{T:w:closed} implies that  
every pair of closed special Legendrian submanifolds
$\Sigma_1$ and $\Sigma_2$ in $\Sph^{2p-1}$ and $\Sph^{2q-1}$ respectively, 
gives rise to a countably infinite family of closed SL twisted products, \textit{i.e.}
special Legendrian immersions of $S^1 \times \Sigma_1 \times \Sigma_2$ in $\Sph^{2p+2q-1}$.
Similarly, by using closed $(1,n-1)$-twisted SL curves every closed special Legendrian submanifold $\Sigma$ in $\Sph^{2n-3}$ 
gives rise to a countably infinite family of closed special Legendrian submanifolds in $\Sph^{2n-1}$ 
with topology $S^{1}\times \Sigma$.

By combining the closed twisted product construction with existing constructions of closed special Legendrian immersions
we generate a plethora of new closed special Legendrian and contact stationary immersions in essentially all dimensions.
For example, we have the following result on topological types of special Lagrangian and Hamiltonian stationary cones 
\begin{mtheorem}[Infinitely many topological types of SL and HS cones in $\C^{n}$ for $n\ge 4$]\hfill
\addtocounter{equation}{1}
\label{T:slg:infinite:top}
\begin{itemize}
\item[(i)]
For any $n\ge 4$ there are infinitely many topological types of special Lagrangian cone in $\C^{n}$, 
each of which is diffeomorphic to  the cone over a product $S^{1}\times \Sigma'$ for some smooth manifold $\Sigma'$, 
and each of which admits infinitely many distinct geometric representatives.
\item[(ii)]
For any $n \ge 4$ there are infinitely many topological types of Hamiltonian stationary cone in $\C^{n}$ 
which are not minimal Lagrangian, each of which is diffeomorphic to the cone over a product $S^{1}\times \Sigma'$ for some smooth manifold $\Sigma'$, 
and each of which admits infinitely many distinct geometric representatives.
\end{itemize}
\end{mtheorem}
\begin{proof}
In \cite{haskins:kapouleas:invent} we proved the existence of infinitely many special Legendrian surfaces in $\Sph^{5}$ of every odd genus (and also of genus $4$). By Theorem \ref{T:w:closed}
there is a countably infinite family of closed $(1,3)$-twisted SL curves. 
Appealing to  \ref{C:sl:triples} using this infinite family of closed $(1,3)$-twisted SL curves and the infinite number of topological types of 
SL surfaces in $\Sph^{5}$ described above we conclude that there are infinitely many topological types of special Legendrian 3-folds 
in $\Sph^{7}$ of the form $S^{1} \times \Sigma$, where $\Sigma$ is a oriented surface and that each topological type 
is realised by infinitely many distinct geometric representatives. To prove part (i) for any $n> 4$
we can keep iterating the process using the fact that by Theorem \ref{T:w:closed} 
for each $n \ge 3$ there is a countably  infinite family of closed $(1,n-1)$-twisted SL curves.
To prove (ii) we simply substitute Lemma \ref{L:twist:cs} on CS twisted products 
for Corollary \ref{C:sl:triples} and Remark \ref{R:pq:twist:curve} 
for Theorem \ref{T:w:closed} and argue as before using our gluing results for SL surfaces in $\Sph^{5}$ as the starting point
once again.
\end{proof}


We can also combine the twisted product construction with the SL $2$-tori produced by integrable 
systems methods. McIntosh \cite{mcintosh:slg} proved that all SL $2$-tori in $\Sph^5$
can be constructed by integrable systems methods and more specifically by so-called spectral curve methods.
Using these methods Carberry-McIntosh \cite{carberry:mcintosh} produced a very rich variety of
special Legendrian $2$-tori; in particular they proved the existence of appropriate SL spectral data
in which the genus of the spectral curve genus can be any positive even integer.
A simple consequence of their result is the remarkable fact that SL
$2$-tori can come in continuous families of arbitrarily high dimension,
by choosing SL spectral data of higher and higher spectral curve genus.
We can extend Carberry-McIntosh's result to every dimension and also to contact stationary tori of dimension at least $3$
using the closed twisted product construction.

\begin{mtheorem}[SL/CS tori in $\Sph^{2n-1}$ occur in families of arbitrarily high dimension]\hfill
\begin{itemize}
\item[(i)] 
For $n\ge 3$ there exist special Legendrian immersions of $T^{n-1}$ in $\Sph^{2n-1}$ which come in continuous families of arbitrarily high dimension.
\item[(ii)] 
For $n \ge 4$ there exist contact stationary (and not minimal Legendrian) immersions of $T^{n-1}$ in $\Sph^{2n-1}$ which come in continuous families 
of arbitrarily high dimension.
\end{itemize}
\end{mtheorem}
\begin{proof}
(i) For $n=3$ we simply appeal to the results of Carberry-McIntosh\cite{carberry:mcintosh}.
For $n=4$ we use the $(1,3)$-twisted SL product of a $2$-torus coming from the Carberry-McIntosh construction 
and any closed $(1,3)$-twisted SL curve. Clearly, the resulting twisted product depends continuously on the input $2$-torus. 
Hence by Carberry-McIntosh's work for any $d \in \N$ we can find a special Legendrian immersion of $S^{1}\times T^{2}$ 
which moves in a continuous family of dimension at least $d$. For $n=5$ we use the $(2,3)$-twisted product 
where $X_{1}: \Sph^{1} \ra \Sph^{3}\subset \C^{2}$ 
is the standard totally real equatorial circle, $X_{2}: T^{2} \ra \Sph^{5}\subset \C^{3}$ is a $2$-torus coming from the 
Carberry-McIntosh construction and $\bw$ is any closed $(2,3)$-twisted SL curve.
For $n\ge 6$ we use the twisted $(n-3,3)$-twisted SL product where 
$X_{1}: T^{n-3} \ra \Sph^{2n-7}$ is the unique SL $n-3$ torus invariant under the diagonal subgroup 
$T^{n-3} \subset \sunit{n-3}$,  $X_{2}: T^{2} \ra \Sph^{5}$ is a $2$-torus coming from the 
Carberry-McIntosh construction and $\bw$ is any closed $(n-3,3)$-twisted SL curve.
Part (ii) is proved in the same way using the twisted CS product construction \ref{L:twist:cs} and  
the closed $(p,q)$-twisted CS curves exhibited in Remark \ref{R:pq:twist:curve}.
\end{proof}

Finally, by combining the twisted product construction with both integrable systems constructions and our 
gluing methods we obtain the following striking result
\begin{mtheorem}\hfill
\addtocounter{equation}{1}
\label{T:cts:fam:top:types}
\begin{itemize}
\item[(i)] 
For any $n\ge 6$ there are infinitely many topological types of special Lagrangian cone in $\C^{n}$ of 
product type which can come in continuous families of arbitrarily high dimension.
\item[(ii)]  
For each $n\ge 6$ there are infinitely many topological types of Hamiltonian stationary cone in $\C^{n}$ of product 
type which are not minimal Lagrangian and which can come in continuous families of arbitrarily high dimension. 
\end{itemize}
\end{mtheorem}

\begin{proof}
(i) Since $n-3 \ge 3$ by the gluing results of \cite{haskins:kapouleas:invent} and Theorem \ref{T:slg:infinite:top}(i) 
there are infinitely many topological types of SL $n-3$ fold in $\Sph^{2(n-3)-1}$. 
The result follows by applying the $(n-3,3)$-twisted SL product construction where $X_{1}$ is any of these SL $n-3$ folds, 
$X_{2}$ is a SL $2$-torus coming from the Carberry-McIntosh construction and $\bw$ is any closed 
$(n-3,3)$-twisted SL curve.

Part (ii) follows in the same way using the twisted CS product construction and the 
closed $(p,q)$-twisted  CS curves exhibited in Remark \ref{R:pq:twist:curve}.
\end{proof}
\noindent
It is difficult to see how integrable systems methods or gluing methods alone could yield a result 
like Theorem \ref{T:cts:fam:top:types}.



\section{$\sorth{p}\times \sorth{q}$-invariant special Legendrian submanifolds}
\label{S:sop:soq}
\nopagebreak

\subsection*{Introduction}
$\phantom{ab}$
\nopagebreak
Given an admissible pair of integers $p$ and $q$ (i.e. satisfying $1\le p \le q$ and $q\ge 2$) we set $n=p+q$ and 
define round cylinders of type $(p,q)$, $\cylpq_I$, by
\addtocounter{theorem}{1}
\begin{equation}
\label{E:cylpq}
{\text{Cyl}^{p,q}_I}
:=
\begin{cases}
I \times \Sph^{p-1} \times \Sph^{q-1}, \quad & \text{if $p>1$;}\\
I \times \Sph^{n-2} , \quad & \text{if $p=1$,}
\end{cases}
\end{equation}
where $I\subset\R$ is an interval which we omit in the notation when $I=\R$.

This section studies $\sorth{p} \times \sorth{q}$-invariant special Legendrian immersions from $\cylpq$  to $\Sph^{2(p+q)-1}$.
When we specialise to $p=1$ and $q=2$ we obtain the $\sorth{2}$-invariant
cylindrical special Legendrian immersions $X_{\tau}: \Sph^1 \times \R \ra \Sph^5$ used as building blocks
in the gluing constructions of \cite{haskins:kapouleas:invent}.
To generalise the gluing methods of \cite{haskins:kapouleas:invent} 
to construct higher dimensional special Legendrian submanifolds,
building blocks analogous to the $\sorth{2}$-invariant special Legendrian cylinders $X_\tau$ are needed
in higher dimensions. 
The most natural such generalisations of the $\sorth{2}$-invariant special Legendrian cylinders
are the $\sorth{p} \times \sorth{q}$-invariant special Legendrian cylinders in $\Sph^{2(p+q)-1}$ first studied
by Castro-Li-Urbano in \cite{castro:li:urbano}.
Many features of the $\sorth{2}$-invariant special Legendrian cylinders
generalise to these $\sorth{p} \times \sorth{q}$-invariant special Legendrian submanifolds.

First, $\sorth{p}\times\sorth{q}$-invariant special Legendrians are governed by 
a first order system of complex ODEs generalising 3.18 in \cite{haskins:kapouleas:invent}; 
we will see that all such special Legendrians arise from the twisted product construction  
and hence are governed by the ODEs \ref{E:slg:ode2} as in Remark \ref{R:sl:twist:sign}.

Second, for fixed $p$ and $q$, the set of $\sorth{p} \times \sorth{q}$-invariant special Legendrian submanifolds
of $\Sph^{2(p+q)-1}$ depends essentially on one real parameter $\tau$. For any admissible value of 
$\tau$ we get a special Legendrian immersion $X_\tau$ of a generalised cylinder
$\cylpq$  in $\Sph^{2n-1}$ (see Proposition \ref{P:X:tau}). 
Moreover, there is one angular period $\pthat$---defined precisely in \ref{E:pthat}---which determines
when $X_\tau$ factors through a special Legendrian embedding of 
$S^1 \times S^{p-1} \times S^{q-1}$ or $S^1 \times S^{n-2}$ for the case $p=1$ .
When $X_{\tau}$ factors through an embedding of a closed manifold it 
gives rise to a SL cone in $\C^n$ with link $S^1\times S^{p-1}\times S^{q-1}$ or $S^1 \times S^{n-2}$ 
(or a $\Z_{2}$ quotient of these).
By studying the behaviour of $\pthat$ as $\tau$ varies we prove that for a dense set of $\tau$, $X_\tau$ factors as above 
(Theorem \ref{T:Xtau:embed}).

Third, the $\tau \ra 0$ limit is singular and geometrically $X_\tau$ degenerates in interesting ways in this limit.
Fully understanding these degenerations is a major part of this paper
and crucial to the applications to gluing constructions.

\smallskip
\noindent 
\textit{Relation with work of other authors.}
$\sorth{p}\times \sorth{q}$-invariant SL submanifolds of $\C^n$ are studied in \cite[\S 3]{castro:urbano:construct}
and $\sorth{p}\times \sorth{q}$-invariant Legendrian submanifolds of $\Sph^{2n-1}$ are studied in 
\cite[\S 3]{castro:li:urbano}. 
%
%
The ODEs for $\sorth{p}\times \sorth{q}$-invariant special Legendrian
submanifolds of $\Sph^{2(p+q)-1}$ appear in \cite[Lemma 2]{castro:urbano:construct} and \cite[Cor 1]{castro:li:urbano}.
However,  Castro-Li-Urbano did not prove results about the behaviour of the angular period $\pthat$. Therefore
almost all of the closed $\sorth{p}\times \sorth{q}$-invariant SL cones over $S^1 \times S^{p-1}\times S^{q-1}$
(or $S^1 \times S^{n-2}$ for $p=1$) described in this paper appear to be new SL cones.

The special case of  $\sorth{n-1}$-invariant special Legendrians (for $n>3$) has  also recently been studied by Anciaux \cite{anciaux} 
from a slightly different point-of-view. Anciaux \cite[Thm 2]{anciaux} 
gives the following nice geometric characterisation of $\sorth{n-1}$-invariant special Legendrians: 
any minimal Legendrian submanifold of $\Sph^{2n-1}$ which  is foliated by round $n-2$ spheres is 
either a totally geodesic $\Sph^{n-1}$ or congruent to an 
 $\sorth{n-1}$-invariant special Legendrian. Anciaux goes on to study $\sorth{n-1}$-invariant special Legendrians 
 in $\Sph^{2n-1}$ noting that they arise from a Legendrian curve $\bw$ in $\Sph^{3}$ satisfying \ref{E:w:slg} with $(p,q)=(1,n-1)$.
Rather than working directly with this first order condition and deriving an equation like \ref{E:slg:ode} from it, Anciaux 
differentiates \ref{E:w:slg} and interprets the resulting second order equation (see \cite[eqn. 3]{anciaux}) 
as an equation on the projected curve $\pi(\bw) \subset \CP^{1}$ where $\pi: \Sph^{3} \ra \CP^{1}$ denotes the Hopf projection. 
Using this approach he can prove the existence of a countable family of  closed integral curves in $\CP^{1}$ and this suffices 
to prove the existence of closed minimal Lagrangian submanifolds of $\CP^{n-1}$ (see \cite[Thm 3]{anciaux}). 
However, the horizontal lift to $\Sph^{3}$ of a closed integral curve in $\CP^{1}$ is not necessarily closed. 
In Anciaux's approach an additional period condition must be satisfied for the spherical lift to be closed 
and because of this his method does not prove the existence of suitable closed curves in $\Sph^{3}$ (see his discussion 
following Theorem 3).

The key to overcoming this period problem 
is to work directly with the first order system \ref{E:odes:p:n} rather than the second order system that Anciaux exploits. 
This approach allows us to prove the existence of countably infinitely many closed $(p,q)$-twisted special Legendrian curves in 
$\Sph^{3}$ for general $p$ and $q$. For our gluing constructions 
\cite{haskins:kapouleas:hd2,haskins:kapouleas:hd3,haskins:kapouleas:survey}
it is crucial that we have closed $\sorth{p} \times \sorth{q}$-invariant special Legendrians at our disposal.


$\sorth{2} \times \sorth{2}$-invariant SL cones in $\C^4$ can be
constructed in a different manner, namely
as a special case of Joyce's work on $T^{n-2}$-invariant SL cones in $\C^n$.
To obtain this $\sorth{2} \times \sorth{2}$ action we should set $n=4$ and take $a_1=a_2=-1,\  a_3=a_4=1$ in
\cite[Prop. 7.6]{joyce:symmetries}. 
Among all $T^2$-actions  allowed in Joyce's constructions, the 
$\sorth{2} \times \sorth{2}$ action is distinguished by having the largest fixed point set.
This is directly related to the fact that the $\tau \ra 0$ limit of $X_\tau$ is singular and geometrically interesting in
this case.

\subsection*{Isotropic orbits of the $\sorth{p}\times\sorth{q}$ action on $\C^{p+q}$}
As previously we assume that $(p,q)$ is an admissible pair of positive integers, i.e. $p\le q$, $q\ge 2$ and $p\ge 1$, 
and we set $n=p+q$. 

$\sorth{p}\times\sorth{q}$ acts via isometries on $\C^{p+q}\cong \C^{p}\times \C^{q}$ via the product 
of the standard complex linear actions of $\sorth{p}$ and $\sorth{q}$ on the $\C^{p}$ and $\C^{q}$ factors respectively.
Since $\sorth{p} \times \sorth{q} \subset \sorth{p+q} \subset \sunit{n}$ it is natural to look for 
$\sorth{p} \times \sorth{q}$-invariant special Lagrangians in $\C^{p+q}$ and in particular for special Lagrangian cones 
or equivalently special Legendrian submanifolds of $\Sph^{2n-1}$ invariant under $\sorth{p} \times \sorth{q}$. 
If a Legendrian submanifold of $\Sph^{2n-1}$ is a union of orbits then each orbit $\mathcal{O}$ must be $\gamma$-isotropic, 
i.e. $\gamma|_{\mathcal{O}}=0$. 
The following simple lemma describes the $\gamma$-isotropic orbits $\mathcal{O}$ of $\sorth{p}\times\sorth{q}$ 
in $\Sph^{2n-1}$.

\begin{lemma}[Isotropic orbits of $\sorth{p}\times\sorth{q}$]\hfill
\addtocounter{equation}{1}
\label{L:iso:orbits}
\begin{itemize}
\item[(i)] If $p\ge 2$,  $q\ge 2$ then any $\gamma$-isotropic $\sorth{p}\times\sorth{q}$ orbit 
$\mathcal{O} \subset \Sph^{2(p+q)-1}$ has the form 
\begin{equation}
\addtocounter{theorem}{1}
\label{E:iso:orbit:p:neq:1}
\mathcal{O}_{\bw} = (w_{1}\cdot \Sph^{p-1},\, w_{2}\cdot \Sph^{q-1})
\end{equation}
for some $\bw=(w_{1},w_{2}) \in \Sph^{3}$. Moreover, if $\bw$ and $\bw' \in \Sph^{3}$ 
then $\mathcal{O}_{\bw}=\mathcal{O}_{\bw'}$ if and only if $\bw'= \rho_{jk} \bw$ 
for some $(j,k)\in \Z_{2}\times \Z_{2}$ where $\rho: \Z_{2} \times \Z_{2} \ra \orth{2} \subset \unit{2}$ 
is the homomorphism defined by 
$$ (j,k) \mapsto
\rho_{jk}:= \left(
\begin{matrix}
(-1)^j & 0 \\
0 & (-1)^k
\end{matrix}
\right).
$$
In particular, spherical isotropic $\sorth{p}\times\sorth{q}$ orbits are in one-to-one correspondence 
with points in $\Sph^{3}/\Z_{2}\times \Z_{2}$.

\item[(ii)]
Similarly, for $n\ge 3$ any $\gamma$-isotropic $\sorth{n-1}$ orbit $\mathcal{O} \subset \Sph^{2n-1}$ has the form 
\begin{equation}
\addtocounter{theorem}{1}
\label{E:iso:orbit:p:eq:1}
\mathcal{O}_{\bw} = (w_{1},\, w_{2}\cdot \Sph^{n-2})
\end{equation}
for some $\bw=(w_{1},w_{2}) \in \Sph^{3}$. Moreover, if $\bw$ and $\bw' \in \Sph^{3}$ 
then $\mathcal{O}_{\bw}=\mathcal{O}_{\bw'}$ if and only if $\bw'= \rho_{jk} \bw$ 
for $(j,k) \in \langle (+-) \rangle \cong \Z_{2} \leqslant \Z_{2}\times \Z_{2}$. 
In particular, isotropic $\sorth{n-1}$ orbits in $\Sph^{2n-1}$ are in one-to-one correspondence 
with points in $\Sph^{3}/\Z_{2}$, where $\Z_{2}=\langle \rho_{+-} \rangle$.
\end{itemize}
\end{lemma}

\begin{proof}
We begin with a more general result that applies to isotropic orbits of any connected Lie subgroup of $\sunit{n}$.
Let $G$ be any connected Lie subgroup of $\sunit{n}$,  $\mathfrak{g}$ denote the Lie algebra of $G$ 
and $x$ be any point in $\Sph^{2n-1}$.
Then the orbit $\mathcal{O}_{x}:= G \cdot x$ is contained in $\Sph^{2n-1}$ 
and is $\gamma$-isotropic if and only if $\gamma(v)=0$ for all $v\in T_{y}\mathcal{O}_{x}$ and $y\in \mathcal{O}_{x}$.
By homogeneity it suffices to check this at $x$. But since $\mathcal{O}_{x}$ is a $G$-orbit we have
$T_{x} \mathcal{O}_{x}=\mathfrak{g} \cdot x$. Therefore $\mathcal{O}_{x}$ is isotropic if and only if 
$\gamma_{x}(\mathfrak{g}\cdot x)=0$. Hence using the definition of the standard contact form $\gamma$ on $\Sph^{2n-1}$
we see that $\mathcal{O}_{x}$ is isotropic if and only if 
\begin{equation}
\addtocounter{theorem}{1}
\label{E:moment:map:general}
\Imag{\,\langle x,Ax \rangle} = 0, \quad \text{for any\ }A \in \mathfrak{g}
\end{equation}
where $\langle \cdot, \cdot \rangle$ denotes the standard Hermitian inner product on $\C^{n}$.
In the language of moment maps \ref{E:moment:map:general} 
is equivalent to the condition $x\in \mu^{-1}(\mathbf{0})$ where $\mu: \C^{n} \ra \mathfrak{g}^{*}$
is the moment map associated to the action of $G\subset \sunit{n}$. 
(For the definition and basic properties of 
the moment map we refer the reader to Section 4 of \cite{joyce:symmetries}.)

Specialising to 
$G=\sorth{p}\times\sorth{q}$ and $\mathfrak{g}=\lsorth{p} \times \lsorth{q}$ (with $p\ge 2$ and $q\ge 2$)
we have $\mathcal{O}_{x}$ is isotropic if and only if
\begin{equation}
\addtocounter{theorem}{1}
\label{E:moment:map:sopq}
\Imag{\,\langle x,Ax \rangle} = 0, \quad \text{for any\ }A \in \lsorth{p}\times \lsorth{q}.
\end{equation}
To analyse \ref{E:moment:map:sopq},  decompose $x = (x',x'') \in \C^{p}\times \C^{q}$ and $A=(A',A'') \in \lsorth{p} \times \lsorth{q}$. 
By considering $x=(x',0)$ and $A=(A',0)$ or $x=(0,x'')$ and $A=(0,A'')$ we find it is equivalent to
\begin{equation}
\addtocounter{theorem}{1}
\label{E:moment:map:0}
\Imag{\,\langle x',A'x' \rangle} =  \Imag{\,\langle x'',A''x'' \rangle} = 0, \text{ for all\ } A'\in \lsorth{p}, \,A'' \in \lsorth{q}.
\end{equation}
One can check that $\Imag{\,\langle z,Az \rangle}=0$ for all $A\in \lsorth{m}$ if and only if 
$z\in \C^{m}$ has the form $z\in w\cdot \Sph^{m-1}$ for some $w\in \C$. 
Applying this to \ref{E:moment:map:0} twice (for different values of $m$) 
we obtain $x'\in w_{1} \cdot \Sph^{p-1}$ and $x'' \in w_{2} \cdot \Sph^{q-1}$ for some $\bw=(w_{1},w_{2}) \in \C^{2}$. 
But since $\mathcal{O}_{x}\subset \Sph^{2(p+q)-1}$ we have $\bw \in \Sph^{3}$ and hence 
\ref{E:iso:orbit:p:neq:1} follows. It is straightforward to verify the conditions on $\bw$ and $\bw'$ under which 
the orbits $\mathcal{O}_{\bw}$ and  $\mathcal{O}_{\bw'}$ coincide are as stated.


\smallskip
The proof of Lemma \ref{L:iso:orbits} for $\orth{n-1}$ is a minor modification of the proof above and therefore we omit it. 
The main difference is the condition under which two orbits $\mathcal{O}_{\bw}$ and  $\mathcal{O}_{\bw'}$ coincide. 
\end{proof}
By Lemma \ref{L:iso:orbits} the generic $\gamma$-isotropic orbit of $\sorth{p} \times \sorth{q}$ 
has dimension $n-2$ and therefore we can look for $\sorth{p} \times \sorth{q}$-invariant special Legendrians 
that are curves of $\sorth{p} \times \sorth{q}$ orbits, and these curves will satisfy some first order system of ODEs. 

\subsection*{$\sorth{p}\times\sorth{q}$-invariant special Legendrians and $(p,q)$-twisted SL curves}\phantom{ab}
An immediate consequence of Lemma \ref{L:iso:orbits} is that all $\sorth{p}\times\sorth{q}$-invariant Legendrian 
submanifolds of $\Sph^{2(p+q)-1}$ arise from the twisted product construction of \ref{D:twisted:product}.
\addtocounter{equation}{1}
\begin{corollary}[$\sorth{p}\times\sorth{q}$-invariant Legendrians are twisted products]\hfill
\label{C:so:invariant:leg}
\begin{itemize}
\item[(i)]
For $p\ge 2$, $q\ge 2$ a Legendrian immersion $Y:\Sigma \ra \Sph^{2(p+q)-1}$ is $\sorth{p}\times\sorth{q}$-invariant
if and only if $Y$ is locally congruent to a twisted product $X_1 *_\bw X_2$ where $X_1: \Sph^{p-1} \ra \Sph^{2p-1}$ and
$X_2: \Sph^{q-1} \ra \Sph^{2q-1}$ are the standard totally geodesic special Legendrian embeddings.
\item[(ii)]
If $p=1$ a Legendrian immersion $Y: \Sigma \ra \Sph^{2n-1}$ is $\sorth{n-1}$-invariant if and only if $Y$ 
is locally congruent to  a (degenerate) twisted product $X_{\bw}$ (as defined in \ref{E:twist:def:p:eq:1})
where the immersion $X:\Sph^{n-2} \ra \Sph^{2n-3}$ is the standard totally geodesic special Legendrian embedding.
\end{itemize}
\end{corollary}
In particular, by combining Corollary \ref{C:so:invariant:leg} with Corollary \ref{C:sl:triples} we obtain
\begin{corollary}[$\sorth{p}\times\sorth{q}$-invariant special Legendrians and $(p,q)$-twisted SL curves]\hfill
\addtocounter{equation}{1}
\label{C:so:invariant:slg}
\begin{itemize}
\item[(i)] For $p$, $q \ge 2$ any $\sorth{p}\times\sorth{q}$-invariant special Legendrian immersion is locally congruent to 
a twisted product with $X_{1}$ and $X_{2}$ as in \ref{C:so:invariant:leg}, where the twisting curve $\bw$ is a $(p,q)$-twisted SL curve in $\Sph^{3}$.
\item[(ii)]
For $p=1$ any $\sorth{n-1}$-invariant special Legendrian immersion is locally congruent to a (degenerate) twisted product with 
$X: \Sph^{n-2} \ra \Sph^{2n-3}$ the standard totally geodesic Legendrian embedding and $\bw$ a $(1,n-1)$-twisted SL curve 
in $\Sph^{3}$.
\end{itemize}
\end{corollary}
\noindent
Corollary \ref{C:so:invariant:leg} appears in Castro-Li-Urbano in the statement of Thm 3.1\cite{castro:li:urbano}. Note, however, the assumption $p, q\ge 3$ made in their statement can be relaxed as in our 
statement.
We could also derive these results about 
$\sorth{p}\times\sorth{q}$-invariant special Legendrians using the methods Joyce developed 
to study cohomogeneity one special Legendrians. 
We describe this approach briefly. The following result is a minor rephrasing of Theorem 6.3 in \cite{joyce:symmetries}.
\begin{prop}
\addtocounter{equation}{1}
\label{P:joyce:orbits}
Let $G$ be a connected Lie subgroup of $\sunit{n}$ with Lie algebra $\mathfrak{g}$ and moment map 
$\mu: \C^{n} \ra \mathfrak{g}^{*}$ (with $\mu(\mathbf{0})=\mathbf{0}$). Let $\mathcal{O}$ be an oriented orbit of $G$ in $\Sph^{2n-1}$ 
of dimension $n-2$, and suppose $\mathcal{O} \subset \mu^{-1}(\mathbf{0})$. 
Then there exists a locally unique, $G$-invariant 
special Legendrian $\Sigma$ containing $\mathcal{O}$. Furthermore, $\Sigma \subset \mu^{-1}(\mathbf{0}) \cap \Sph^{2n-1}$
and near $\mathcal{O}$, $\Sigma$ is fibred by $G$-orbits isomorphic to $\mathcal{O}$. Thus $\Sigma$ is 
locally diffeomorphic to $(-\epsilon,\epsilon) \times \mathcal{O}$ for some $\epsilon >0$ and 
we can view $\Sigma$ as a smooth curve of $G$-orbits. 
\end{prop}
\noindent
Moreover, one can also use Joyce's methods (see \cite[\S 7]{joyce:symmetries} where he treats the case of $T^{n-2}$ actions on $\Sph^{2n-1}$) 
to derive explicit ODEs for  cohomogeneity one $G$-invariant special Legendrians. In this way one always gets a system of $1$st order ODEs 
on the space of (generic) isotropic orbits of $G$. Locally one has existence and uniqueness for these ODEs but problems
may develop if we run into singular $G$-orbits. To study the global geometry of the $G$-invariant special Legendrians 
one needs to understand the possible singular $G$-orbits and how solutions to the ODEs behave on approach to these orbits.

We now explain the link between Joyce's approach applied to $G=\sorth{p}\times\sorth{q}$ 
and the approach via $(p,q)$-twisted SL curves in $\Sph^{3}$ we have discussed.
Recall Lemma \ref{L:iso:orbits} shows that the space of isotropic orbits in $\Sph^{2(p+q)-1}$ for  
$G=\sorth{p}\times\sorth{q}$ is $\Sph^{3}/\Z_{2}$ if $p=1$ or $\Sph^{3}/\Z_{2}\times \Z_{2}$ if $p>1$.
Thus the ODEs for $\sorth{p}\times\sorth{q}$-invariant special Legendrians should be a $1$st order ODE on  
this space of orbits. These ODEs are essentially the fundamental ODEs \ref{E:slg:ode} for $(p,q)$-twisted SL curves in $\Sph^{3}$.
The only difference is that to describe the evolution of the isotropic orbits we should consider the equivalence class of $\bw$ 
in $\Sph^{3}/\Z_{2}$ or $\Sph^{3}/\Z_{2}\times \Z_{2}$ rather than the point $\bw \in \Sph^{3}$. 

Both methods therefore yield local classification results for $\sorth{p}\times\sorth{q}$-invariant special Legendrians in 
$\Sph^{2(p+q)-1}$. The global structure of $\sorth{p}\times\sorth{q}$-invariant special Legendrians is then studied 
as a later step.

\subsection*{The fundamental ODE system for $(p,q)$-twisted SL curves}
$\phantom{ab}$
\nopagebreak

As before we fix $1\le p \le q,\, 2\le q$ and define $n:=p+q$.
We now study the basic properties of the first order system of complex ODEs \ref{E:slg:ode2} which describe 
all (appropriately parametrised)  $(p,q)$-twisted SL curves.  
We begin by discussing symmetries of these ODEs.

For any $p$ and $q$ the $(p,q)$-twisted SL ODEs \ref{E:slg:ode2} have six obvious types  of symmetry:
\begin{enumerate}
\item Time translation invariance $\bw \mapsto \bw \circ \TTT_{t_{0}}$ for any $t_0\in \R$.
\item Multiplication by an $n$th root of unity $\bw \mapsto z \bw$, where $z^n=1.$
\item $\bw \mapsto \tcheck_{x} \circ \bw$ where $\tcheck_x \in \unit{1} \times \unit{1} \subset \unit{2}$ is the $1$-parameter subgroup (depending on $p$ and $q$)
\addtocounter{theorem}{1}
\begin{equation}
\label{E:M:defn}
\tcheck_x =  \left(
\begin{matrix}
e^{ix/p} & 0  \\
0 & e^{-ix/q}\\
\end{matrix}
\right).
\end{equation}
\item Complex conjugation  $\bw \mapsto \overline{\bw}$.
\item The simultaneous time reflection and spatial rotation given by 
$$t \mapsto -t, \quad \bw \mapsto z\bw,$$
where $z$ is any $n$th root of $-1$.
\item The simultaneous time and spatial rescaling given by 
$$t \mapsto \lambda^{1-2/n} t, \quad \bw \mapsto \lambda^{1/n} \bw, \quad \text{for any } \lambda>0.$$ 
More precisely, $\bw$ is a solution of \ref{E:slg:ode2} if and only if $\bw_{\lambda}(t):= \lambda^{1/n} \bw( \lambda^{1-2/n}t)$ is.
\end{enumerate}

Before establishing the basic facts about solutions to the $(p,q)$-twisted SL ODES
we discuss the geometry of the $1$-parameter group of symmetries $\{\tcheck_{x}\}_{x\in\R}$ (which depends on $p$ and $q$) 
appearing in symmetry (3) above. 
As in \ref{L:iso:orbits}, for any $\bw \in \Sph^{3}$ let $\mathcal{O}_{\bw}\subset \Sph^{2(p+q)-1}$ denote the associated isotropic $\sorth{p} \times \sorth{q}$ orbit.

\begin{definition}
\label{D:M:period}
\addtocounter{equation}{1}
For fixed integers $p$ and $q$ define a \emph{period of the $1$-parameter group $\{\tcheck_{x}\}$} by
$$\Per(\{\tcheck_{x}\}):= \{x\in \R\, | \, \tcheck_{x}=\Id\}.$$
Clearly, if $x\in \Per(\{\tcheck_{x}\})$ then $\mathcal{O}_{\tcheck_{x}\bw}=\mathcal{O}_{\bw}$ for any 
$\bw \in \Sph^{3}$. In other words, for any $x \in \Per(\{\tcheck_{x}\})$, $\tcheck_{x}$ leaves invariant all isotropic 
$\sorth{p}\times \sorth{q}$ orbits in $\Sph^{2(p+q)-1}$.

More generally, we call $x\in \R$ a \emph{half-period of $\{\tcheck_{x}\}$} if $\tcheck_{x}$ leaves invariant
all  isotropic $\sorth{p}\times \sorth{q}$ orbits in $\Sph^{2(p+q)-1}$. In other words,  
$$\Perh(\{\tcheck_{x}\}):=\{ x\in \R \, |\, \mathcal{O}_{\tcheck_{x}\bw}=\mathcal{O}_{\bw} \ \forall\  \bw \in \Sph^{3}\}.$$
A half-period of $\{\tcheck_{x}\}$ which is not a period of $\{\tcheck_{x}\}$ we call a \emph{strict half-period of $\{\tcheck_{x}\}$}.
\end{definition}

Define the finite subgroup $\stabpq \subset \unit{2}$ by 
\begin{equation}
\addtocounter{theorem}{1}
\label{E:stabpq}
 \stabpq = 
\begin{cases}
\left(\begin{array}{cc}\pm 1 & 0 \\0 & \pm 1 \end{array}\right) \cong \Z_{2}\times \Z_{2} \quad & \text{if $p>1$};\\
\left(\begin{array}{cc} 1 & 0 \\0 & \pm1\end{array}\right) \cong \Z_{2} \quad & \text{if $p=1$}.
\end{cases}
\end{equation}
It follows from \ref{L:iso:orbits} that 
\begin{equation}
\addtocounter{theorem}{1}
\label{E:M:half:period}
x\in \Perh(\{\tcheck_{x}\}) \iff \tcheck_{x} \in \stabpq.
\end{equation}
An immediate consequence of \ref{E:M:half:period} is that $2\Perh(\{\tcheck_{x}\}) \subset \Per(\{\tcheck_{x}\})$; 
this explains the choice of the terminology half-period.
If $\tcheck_{x}=\rho_{jk}$ for some $(j,k)\in \Z_{2}\times \Z_{2}$ 
with $\rho_{jk}\in \stabpq$ as defined in \ref{L:iso:orbits} we call 
$x$ \emph{a half-period of type $(jk)$}.
If $x$ is a half-period of type $(jk)$ then $e^{ix}=(-1)^{jp}$ and $e^{-ix}=(-1)^{kq}$ and hence 
$jp+kq \equiv 0 \mod{2}$.

The following lemma describes the periods and half-periods of $\{\tcheck_{x}\}$.
\begin{lemma}
\label{L:M:periods}
\addtocounter{equation}{1}
Fix a pair of admissible integers $p$ and $q$ and let $\{\tcheck_{x}\}$ denote the $1$-parameter subgroup defined in \ref{E:M:defn}.
\begin{itemize}
\item[(i)]
The periods of $\{\tcheck_{x}\}$ are given by 
$$\Per(\{\tcheck_{x}\})= 2\pi \lcm(p,q)\Z.$$
\item[(ii)]
If $p>1$ then the half-periods of $\{\tcheck_{x}\}$ are given by 
$$\Perh(\{\tcheck_{x}\})=\tfrac{1}{2}\Per(\{\tcheck_{x}\}) = \pi \lcm(p,q)\Z.$$
Moreover, any strict half-period of $\{\tcheck_{x}\}$ is of type $(jk)$ where $j = q/\hcf(p,q) \mod{2}$ and $k=p/\hcf(p,q) \mod{2}$.
In particular, for any fixed $p$ and $q$ exactly one type of strict half-period occurs.
\item[(iii)]
If $p=1$ then the half-periods of $\{\tcheck_{x}\}$ are given by 
$$\Perh(\{\tcheck_{x}\})=
\begin{cases}
\tfrac{1}{2}\Per(\{\tcheck_{x}\}) = \pi \lcm(p,q)\Z \quad & \text{if\  $n$ is odd;}\\
\Per(\{\tcheck_{x}\})  = 2\pi \lcm(p,q)\Z \quad & \text{if\  $n$ is even.}
\end{cases}
$$
\end{itemize}
\end{lemma}
\begin{proof}
The proof is a straightforward use of the various definitions;  the case $p=1$ being different because 
$\stabpq$ (defined in \ref{E:stabpq}) is defined differently in this case. In either case $p=1$ or $p>1$ 
$$x\in \Per(\{\tcheck_{x}\}) \iff e^{ix/p}=e^{-ix/q}=1 \iff x \in 2\pi p\Z \cap 2\pi q \Z = 2\pi (p\Z \cap q\Z)=2\pi \lcm(p,q)\Z.$$
Suppose now that $p>1$ then 
$$x\in \Perh(\{\tcheck_{x}\}) \iff e^{ix/p}=\pm 1,\, e^{-ix/q}=\pm 1 \iff e^{2ix/p}=e^{-2ix/q}=1\iff 2x \in \Per(\{\tcheck_{x}\}).$$
Since $\{\tcheck_{x}\}$ is a $1$-parameter group, by the definition of $\Per(\{\tcheck_{x}\})=2\pi \lcm(p,q)\Z$ we have 
$$ \tcheck_{k\pi\lcm(p,q)} = 
\begin{cases}
\Id, \quad & \text{if\ } k\in 2\Z;\\
\tcheck_{\pi \lcm(p,q)} \quad & k\in 2\Z +1;
\end{cases}
$$
for any $k\in \Z$.
Let $x_{0}=\lcm(p,q)\pi = pq\pi/\hcf(p,q)$. Then $x_{0}/p=q\pi/\hcf(p,q)  = j\pi$ and $x_{0}/q = p\pi/\hcf(p,q) = k\pi$, 
with $j$ and $k$ as defined in the statement of \ref{L:M:periods}. Hence 
$$\tcheck_{x_{0}} = \left(\begin{array}{cc}e^{ix_{0}/p} & 0 \\0 & e^{-ix_{0}/q}\end{array}\right) = 
\left(\begin{array}{cc}(-1)^{j} & 0 \\0 & (-1)^{k}\end{array}\right)= \rho_{jk}.$$

Similarly the result in the case $p=1$ follows taking into account that the only strict
half-periods  that belong to $\stabpq$ in this case are of type $(+-)$.
\end{proof}
\begin{remark}
\addtocounter{equation}{1}
Notice that $j$ and $k$ defined in \ref{L:M:periods}
$jp+kq = 2pq/\hcf(p,q) \equiv 0 \mod{2}$ as required.
\end{remark}

The following result establishes basic facts about solutions to the $(p,q)$-twisted SL ODEs.
\addtocounter{equation}{1}
\begin{prop}(\textit{cf.} equation \ref{E:slg:ode2} and Remark \ref{R:sl:twist:sign})
\label{P:odes:p:n}
\begin{enumerate}
\item[(i)]
Solutions to the $(p,q)$-twisted SL ODEs
\addtocounter{theorem}{1}
\begin{equation}
\label{E:odes:p:n}
\begin{aligned}
 \dot{w}_1 & = \phantom{-} \overline{w}_1^{p-1}  \overline{w}_2^{q},\\
\dot{w}_2 & = - \overline{w}_1^p \overline{w}_2^{q-1},\\
\end{aligned}
\end{equation}
admit two conserved quantities
$$ \mathcal{I}_1(\bw):=\abs{\boldsymbol{w}}^{2} \quad \text{and} \quad \mathcal{I}_2(\bw):=\Imag(w_1^p w_2^q).$$
The symmetries (1), (2) and (3)  preserve both conserved quantities $\mathcal{I}_1$ and $\mathcal{I}_2$. 
Symmetries (4) and (5) preserve $\mathcal{I}_1$ but send $\mathcal{I}_2 \mapsto -\mathcal{I}_2$.
Symmetry (6) sends $(\mathcal{I}_1,\mathcal{I}_2) \mapsto (\lambda^{2/n} \mathcal{I}_1, \lambda \mathcal{I}_2)$.

Hence if $\bw$ is a solution of \ref{E:odes:p:n} with $\mathcal{I}_1(\bw)\neq 0$ then we may rescale using symmetry $(6)$ to obtain another solution of \ref{E:odes:p:n} with $\mathcal{I}_1(\bw) = 1$. 
For any solution with $\mathcal{I}_1(\boldsymbol{w})=1$, the possible range of  values of 
$\mathcal{I}_2=\Imag(w_1^p w_2^q)$ is $[-2 \taumax\,, 2\taumax]$,
where
\addtocounter{theorem}{1}
\begin{equation}
\label{E:tau:max}
2\tau_{\text{max}} = \sqrt{\frac{p^p q^{q}}{n^n}}.
\end{equation}

\item[(ii)] The stationary points of \ref{E:odes:p:n} are 
$$ \C\times \{0\} \cup \,\{0\} \times\C \quad \text{if\ } p>1 \quad \text{or} \quad \C \times \{0\} \quad \text{if\ }p=1.$$

\item[(iii)] The initial value problem for \ref{E:odes:p:n} with any initial data $\boldsymbol{w}(0) \in \C^2$ 
has a unique real analytic solution $\boldsymbol{w}:\R \ra \C^{2}$ defined for all $t \in \R$, which depends 
real analytically on the initial data.

\item[(iv)] For any solution of \ref{E:odes:p:n} with $\mathcal{I}_{1}(\bw) =1$ and 
$\mathcal{I}_{2}(\bw) = \Imag(w_1^p w_2^q)=-2\tau$ 
(and hence by part (i) $\tau \in [-\taumax,\taumax]$) 
the function $y:=\abs{w_2}^2:\R \ra [0,1]$ satisfies the equation 
\addtocounter{theorem}{1}
\begin{equation}
\label{E:ydot:cx}
\frac{1}{2}\dot{y} + 2i\tau = -w_1^p  w_2^{q}.
\end{equation}
Therefore $y$ satisfies the energy conservation equation
\addtocounter{theorem}{1}
\begin{equation}
\label{E:y:dot}
\dot{y}^2  = 4(f(y)-4\tau^2) = 4y^q (1-y)^{p} - 16 \tau^2,
\end{equation}
and hence also the second-order ODE
\addtocounter{theorem}{1}
\begin{equation}
\label{E:y:ddot}
\ddot{y}  = 2f'(y) = 2y^{q-1}(1-y)^{p-1}(q-ny),
\end{equation}
where we define the function $f:\R \ra \R$ (depending on $p$ and $q$) by 
\addtocounter{theorem}{1}
\begin{equation}
\label{E:f}
f(y) = y^q(1-y)^{p}.
\end{equation}
\end{enumerate}
\end{prop}

\begin{remark}
\addtocounter{equation}{1}
\label{R:sing:pts}
The difference between the stationary points of \ref{E:odes:p:n} in the case $p>1$ and the case $p=1$ reflects the 
difference in the geometry of the nongeneric isotropic orbits of $\sorth{p} \times \sorth{q}$ and $\sorth{n-1}$ respectively. 
For $p>1$ the  nongeneric isotropic orbits of $\sorth{p} \times \sorth{q}$ have the form $(w_{1}\cdot \Sph^{p-1},0)$ and 
$(0,w_{2}\cdot \Sph^{q-1})$. For $p=1$ the only nongeneric isotropic orbits are of the form $(w_{1},0)$. In particular, the orbits 
of the form $(0,w_{2}\cdot \Sph^{n-2})$ are generic provided $w_{2}\neq 0$.
\end{remark}

\iffigureson
\begin{figure}[h!t]
\vspace{-0.5in}
\hspace{-1in}
\includegraphics[scale=1.0]{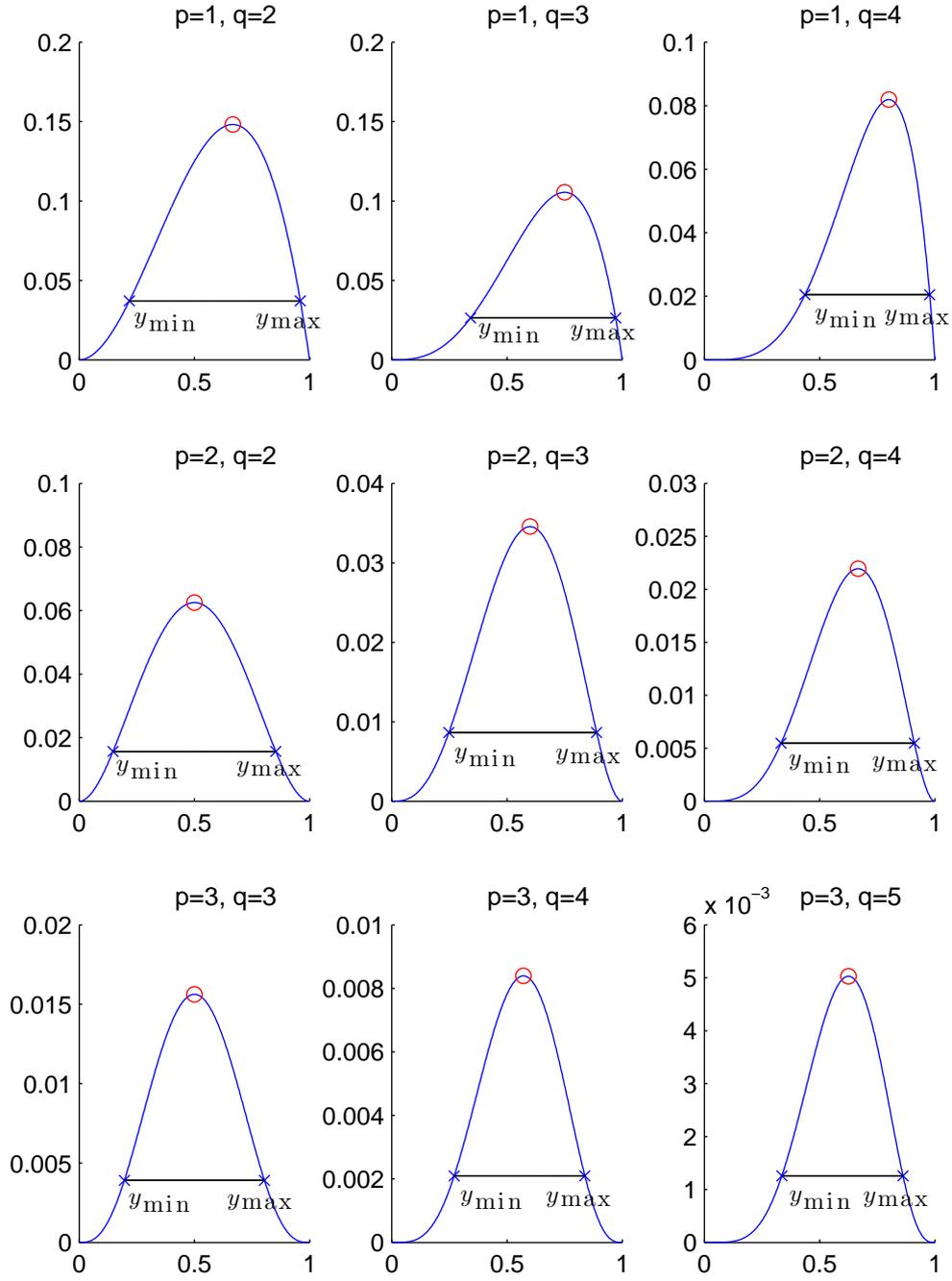}
\hspace{-1in}
\vspace{-0.5in}
\caption{The graph of $f(y)=y^{q}(1-y)^{p}$ on the interval $[0,1]$ for various choices of $(p,q)$. $\ymin$ and $\ymax$ ---
the two solutions of $f(y)=4\tau^{2}$ in the interval $[0,1]$ --- are shown for $\tau=\tfrac{1}{2}\taumax$. The maximum value $f_{\text{max}}=4\taumax^{2}$ which occurs at $y=\tfrac{q}{n}$ is marked by $\circ$.}
\label{fig:fplot}
\end{figure}
\else
\fi

\begin{proof}
\emph{(i) Conserved quantities.} 
We verify $\mathcal{I}_{1}$ and $\mathcal{I}_{2}$ are conserved by direct calculation. Firstly,  
$$ \dot{\mathcal{I}_1} = \tfrac{d}{dt}\abs{\bw}^2= \tfrac{d}{dt}(w_1 \overline{w}_1) + \tfrac{d}{dt}(w_2 \overline{w}_2) =
2\Real{(\dot{w}_1\overline{w}_1 + \dot{w}_2 \overline{w}_2)} = 0,$$ where we have used \ref{E:odes:p:n} in the final equality.
Secondly, since 
$$\tfrac{d}{dt}(w_1^p w_2^{q}) = p w_1^{p-1}w_2^{q}\dot{w}_1 + q w_1^p w_2^{q-1} \dot{w}_2,$$
using
\ref{E:odes:p:n} we obtain
\addtocounter{theorem}{1}
\begin{equation}
\label{E:dt:w1p:w2q}
\frac{d}{dt}(w_1^p w_2^{q}) = \abs{w_1}^{2p-2} \abs{w_2}^{2q-2} (p\abs{w_2}^2 - q \abs{w_1}^2) \in \R^.
\end{equation}
Hence $\tfrac{d}{dt} \mathcal{I}_2 = \tfrac{d}{dt} \Imag{(w_1^pw_2^{q})}=0$.
It is straightforward to check the action of the symmetries on $\mathcal{I}_{1}$ and $\mathcal{I}_{2}$ is as claimed.
Define $y=\abs{w_{2}}^{2}$. When $\mathcal{I}_{1}(\bw)=1$ 
$$\abs{\mathcal{I}_{2}(\bw)} = \abs{\Imag w_{1}^{p}w_{2}^{q}} \le \abs{w_{1}}^{p} \abs{w_{2}}^{q} = \sqrt{y^{q}(1-y)^{p}}=\sqrt{f(y)},$$
for the function $f$ defined in \ref{E:f}.
A short calculation shows that
\addtocounter{theorem}{1}
\begin{equation}
\label{Ef:dot}
f'(y) = y^{q-1}(1-y)^{p-1}(q-ny),
\end{equation}
and therefore the critical points of $f$ are 
\begin{equation}
\addtocounter{theorem}{1}
\label{E:crit:f}
\text{Crit}(f) = 
\begin{cases}
\{0, \, \tfrac{q}{n},\, 1\} & \text{if\ }p>1;\\
\{0, \, \tfrac{q}{n}\} & \text{if\ } p=1.
\end{cases}
\end{equation}
Since $f$ is non-negative on $[0,1]$ and vanishes only at the two endpoints, the maximum value of $f$ for $y\in [0,1]$ occurs when $y=\tfrac{q}{n}$ and
hence
$$f_{\text{max}}=f(\tfrac{q}{n}) = \frac{p^p q^{q}}{n^n} = 4\taumax^2,$$
where $\taumax$ is defined in \ref{E:tau:max}.
Hence $\abs{\mathcal{I}_{2}(\bw)} \le \sqrt{f_{\text{max}}}\le 2\taumax$ as claimed.
See Figure \ref{fig:fplot} for the graph of the function $f$ on the interval $[0,1]$ for various choices of $(p,q)$.

\no \emph{(ii) Stationary points.} Stationary points of \ref{E:odes:p:n} are given by common zeros of the two polynomials
\begin{equation}
\addtocounter{theorem}{1}
\label{E:crit:pts}
w_1^{p-1} w_2^q = 0 = w_1^p w_2^{q-1}.
\end{equation}

\no \emph{(iii) Global existence, uniqueness and analyticity.}
The vector field 
$$V(\bw) = (\overline{w}_{1}^{p-1}\overline{w}_{2}^{q}\, ,\overline{w}_{1}^{p}\overline{w}_{2}^{q-1})$$
on $\C^{2}$ defining \ref{E:odes:p:n} is clearly real algebraic. It follows then from the standard local existence and uniqueness 
results for the initial value problem that locally  \ref{E:odes:p:n} admits a unique real analytic solution for any initial data and this 
solution  depends real analytically on the initial condition. Since $\mathcal{I}_{1}(\bw)=\abs{\bw}^{2}$ is constant, this local 
solution remains in a compact subset of $\C^{2}$ and hence global existence follows immediately. 

\no \emph{(iv) ODEs for $y:=\abs{w_{2}}^{2}$.}
Using \ref{E:odes:p:n}, we have 
\begin{align*}
\dot{y} &= 2\Real{(\dot{w}_2 \overline{w}_2)} = -2\Real{(\overline{w}_1^p \overline{w}_2^{q})} = -2\Real{({w}_1^p {w}_2^{q})},\\
2\tau &= \Imag{(\overline{w}_1 \dot{w}_1)} = \Imag{(\overline{w}_1^p \overline{w}_2^{q})} = -\Imag{(w_1^p w_2^{q})}.
\end{align*}
Hence we obtain \ref{E:ydot:cx}. 
Taking the modulus squared of both sides of \ref{E:ydot:cx} proves that $\dot{y}$ satisfies \ref{E:y:dot}.
Differentiating \ref{E:ydot:cx} with respect to $t$ and using \ref{E:dt:w1p:w2q} we see that $y$ satisfies the second-order equation \ref{E:y:ddot}.
Note that the stationary points of \ref{E:y:ddot} are exactly the critical points of $f$ and hence by \ref{E:crit:f} are 
$0$ and $\tfrac{q}{n}$ when $p=1$ and $0$, $\tfrac{q}{n}$ and $1$ when $p>1$. 
\end{proof}
To understand the space of solutions to \ref{E:odes:p:n} modulo the action of the symmetries (1)--(6) 
we need the following auxiliary result about solutions of \ref{E:y:dot}
\begin{lemma}
\label{P:y}
Let $\bw$ be any solution of \ref{E:odes:p:n} with $\mathcal{I}_{1}(\bw) =1$ and 
$\mathcal{I}_{2}(\bw) = \Imag(w_1^p w_2^q)=-2\tau$ and let $y:=\abs{w_2}^2:\R \ra [0,1]$
be the associated solution of \ref{E:y:dot}.
\begin{itemize} 
\item[(i)] If  $0<\abs{\tau}< \taumax$,  the following holds:
\begin{itemize}
\item[a.] $y$ is periodic of period $2\pt>0$ and hence any
two solutions of \ref{E:y:dot} with the same value of $\tau$ differ only by a time translation.
 Moreover, the period $\pt$ satisfies
\addtocounter{theorem}{1}
\begin{equation}
\label{E:pt:taumax}
\lim_{\tau \ra \taumax}{2\pt} = \frac{\pi}{\taumax} \sqrt{\frac{pq}{2n^3}}.
\end{equation}
\item[b.] The range of $y$ is $[\ymin\, , \ymax]$, 
where $0 < \ymin <\tfrac{q}{n} < \ymax < 1$ are the only two solutions of the degree $n$ polynomial equation
\addtocounter{theorem}{1}
\begin{equation}
\label{E:y:polynomial}
f(y) = y^q (1-y)^{p} = 4\tau^2,
\end{equation}
that lie in the interval $[0,1]$.
\item[c.]
As $\tau \ra 0$ we have
\addtocounter{theorem}{1}
\begin{equation}
\label{E:y:maxmin:asym}
y_{\text{min}} = (2\tau)^{2/q} (1 + O(\tau^{2/q})), \quad y_{\text{max}} = 1- (2\tau)^{2/p}(1+O(\tau^{2/p})).
\end{equation}\
\end{itemize}
\item[(ii)]
If $\abs{\tau}=\taumax$, then $y \equiv \tfrac{q}{n}$.
\item[(iii)]
If $\tau=0$ and $p>1$ then one of the following holds:
\begin{itemize}
\item[a.] $y\equiv 0$
\item[b.] $y \equiv 1$
\item[c.] $y$ is strictly monotone and satisfies
$$ y = 
\begin{cases}
y_{0} \circ \TTT_{t_{0}} \quad & \text{some \ } t_{0} \in \R; \quad \text{if $y$ is decreasing}\\
y_{0} \circ \TTT_{t_{0}} \circ \tbar \quad & \text{some \ } t_{0} \in \R; \quad \text{if $y$ is increasing}
\end{cases}
$$
where $y_{0}: \R \ra (0,1)$ denotes the unique (decreasing) solution to the initial value problem
\begin{equation*}
\dot{y} = - 2\sqrt{f(y)}, \quad y(0)  =\frac{q}{n}.
\end{equation*}
Alternatively, $y_{0}$ can be characterised as the unique solution to \ref{E:y:ddot} with initial conditions
$$ y(0) = \frac{q}{n}, \quad \dot{y}(0) = -4\taumax.$$
Moreover, $y_{0}$ satisfies 
$$\lim_{t\ra -\infty}y_{0}(t)=1\quad \text{and} \quad \lim_{t \ra \infty}y_{0}(t)=0.$$
\end{itemize}

\item[(iv)] If $\tau=0$ and $p=1$ then one of the following holds:
\begin{itemize}
\item[a.] $y\equiv 0$, 
\item[b.] $ y = y_{0} \circ \TTT_{t_{0}}$  for some  $t_{0}\in \R,$ 
where $y_{0}:\R \ra (0,1]$ is the unique solution to \ref{E:y:ddot} with initial conditions 
$$y(0)=1, \quad \dot{y}(0)=0.$$
Moreover, $y_{0}$ is even, increasing on $(-\infty,0)$ and satisfies $\lim_{t \ra \pm \infty}y_{0}(t)=0$.
\end{itemize}
\end{itemize}
\end{lemma}

\begin{remark}\hfill
\addtocounter{equation}{1}
\begin{itemize}
\item[(i)]
Detailed asymptotics for the $\tau \ra 0$ limit of the period $2\pt$ are established in Section \ref{S:asymptotics}.
\item[(ii)]
Since $y$ satisfies an equation of the form $\dot{y}^2=P(y)$ where $P$ is a polynomial of degree $n$, 
any solution of \ref{E:y:dot} can be expressed in terms of hyperelliptic functions.
When $n=3$ or $4$ $y$ can be expressed in terms of Jacobi elliptic functions---see \cite{haskins:slgcones,haskins:kapouleas:invent} for such expressions in the $(p,q)=(1,2)$ case.
Moreover, in the $\tau \ra 0$ limit the modulus $k^{2}$ of the elliptic functions tends to $1$. 
In this limit these elliptic functions become hyperbolic trigonometric functions. 
e.g. $y_{0}=\sech^{2}{t}$  when $p=1$, $q=2$ and $y_{0}=\tfrac{1}{2}(1-\tanh{t})$ when $p=q=2$.
\item[(iii)]
Figure \ref{fig:fplot} shows $\ymin$ and $\ymax$ on the graph of $f(y)$ for various $(p,q)$ 
for  $\tau = \tfrac{1}{2} \taumax$.
\end{itemize}
\end{remark}


\begin{proof}
Motivated by \ref{E:y:dot} we define the $2$-variable polynomial $P_\tau: \R^2 \ra \R$
$$P_\tau(y,z) =z^{2}-4f(y) + 16\tau^{2}= z^2 - 4y^q(1-y)^p + 16 \tau^2.$$
Let $C_{\tau}$ denote the real affine curve in $\R^{2}$ defined by $P_{\tau}=0$.
We can also view $P_\tau$ as a $2$-variable complex polynomial and consider the complex affine curve $C_\tau^{\C}$ in $\C^2$ defined by $P_\tau =0$.
We find
$$ (y,z) \in \text{Sing}(C_{\tau}^{\C}) \Longleftrightarrow f(y)=4\tau^{2},\  f'(y)=0,\ z=0.$$
Hence from \ref{Ef:dot} we have 
\addtocounter{theorem}{1}
\begin{equation}
\label{E:sing:pts}
\text{Sing}(C_\tau^{\C}) = \text{Sing}(C_\tau) = 
\begin{cases}
\phantom{(}\emptyset , &\text{for $0<\abs{\tau} < \taumax$;}\\
(\tfrac{q}{n}\, , 0), &\text{for $\abs{\tau}=\taumax$;}\\
(0\, , 0) , &\text{for $\tau = 0$ and $p=1$;}\\
(0\, , 0) \cup (1\, ,0) &\text{for $\tau=0$ and $p>1$.}
\end{cases}
\end{equation}
Since $P_{zz}=2$, all singular points of $C_\tau^{\C}$ are double point singularities.
Further calculation yields:
\begin{itemize}
\item[]  $(\tfrac{q}{n},\,0)$ is always an ordinary double point,
\item[] $(0\, , 0)$ is an ordinary double point if $q=2$ but a node if $q\ge 3$,
\item[]  $(1\, ,0)$ is an ordinary double point if $p=2$ but a node if $p \ge 3$.
\end{itemize}
See also Figure \ref{fig:ctau}.

\iffigureson
\begin{figure}
\vspace{-0.5in}
\hspace{-1in}
\includegraphics[scale=1.0]{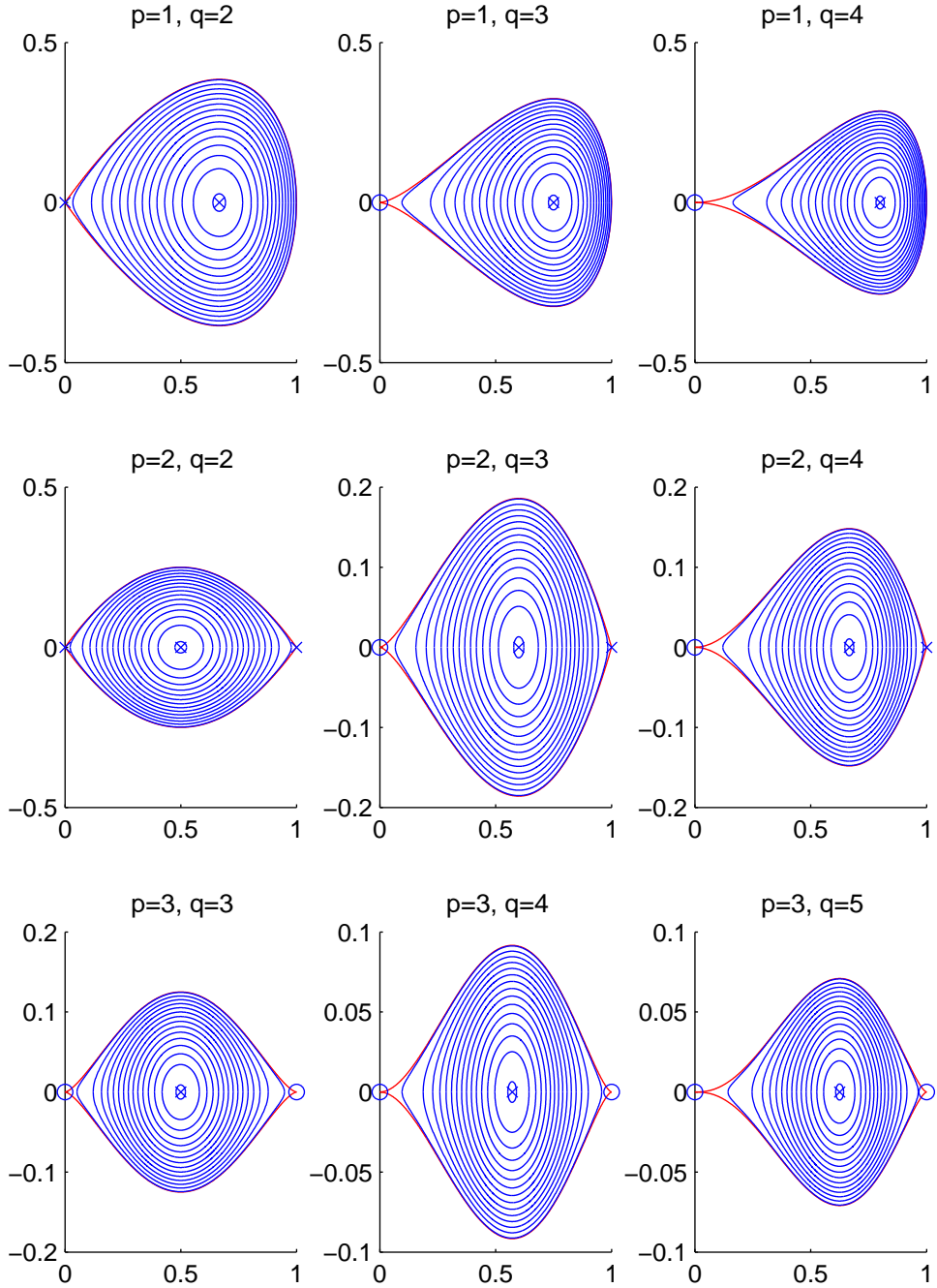}
\hspace{-1in}
\vspace{-0.5in}
\caption{The curves $C_{\tau}^{0}$ for $\tau \in [0,\taumax]$ and various choices of $(p,q)$. 
\newline
Singular points are marked: ordinary double points as $\times$ and nodes as $\circ$. }
\label{fig:ctau}
\end{figure}
\else
\fi

(i) : \emph{$0<\abs{\tau}<\taumax$.} \phantom{a} 
If $0<\abs{\tau}<\taumax$, \ref{E:sing:pts} implies that the real affine curve $C_{\tau}$ is nonsingular. 
$C_{\tau}$ is not necessarily connected, so let $C_{\tau}^{0}$ denote the component containing the point 
$(\tfrac{q}{n}, 4\sqrt{\taumax^{2}-\tau^{2}})$.  $(y,z) \in C_{\tau}$ implies $f(y) \ge 4\tau^{2}$.
The set $f^{-1}([4\tau^{2},\infty)\,) \subset \R$ is not necessarily connected but the component containing 
$\tfrac{q}{n}$ is the closed interval $[\ymin,\ymax] \subset (0,1)$.
Since $f(y) \le 4\taumax^{2}$ for $y\in (0,1)$ any point $(y,z)\in C_{\tau}^{0}$ satisfies 
$(y,z) \in [0,1] \times [-4\taumax,4\taumax]$. In particular, the component $C_{\tau}^{0}$ is a
compact nonsingular curve and hence is diffeomorphic to $S^{1}$. 
Hence all solutions of \ref{E:y:dot} with $0<\abs{\tau}<\taumax$ are non-constant and periodic 
with period $2\pt>0$ depending only on $\tau$. In particular two solutions of \ref{E:y:dot}
with the same values of $\tau$ differ only by time translation. 

The geometry of the curves $C_{\tau}^{0}$ is illustrated for various choices of $(p,q)$ in Figure \ref{fig:ctau}.
The different types of singular points which can occur in the $\tau=0$ energy level are clearly visible in this figure.

a. Asymptotics of $\pt$ as $\tau \ra \taumax$:
we consider the first order corrections to the stationary point $y \equiv q/n$ when $\tau = \taumax$.
If we write $\tilde{y} = y - q/n$, then \ref{E:y:ddot} becomes
$$ \ddot{\tilde{y}} = - \omega^2 \tilde{y} + O(\tilde{y}^2),$$
where
$$ \omega^2 = 2n \left(\frac{q}{n}\right)^{q-1} \left(\frac{p}{n}\right)^{p-1} = \frac{8n^3}{pq}\taumax^2.$$
Hence $\lim_{\tau \ra \taumax}{2\pt} = 2\pi/ \omega$, as claimed.

b. Since $\abs{\bw}=1$ and $y=\abs{w_2}^2$, we have $0\le y\le 1$ for all $t\in \R$.
At any critical point of $y$, \ref{E:y:dot} implies that $y$ satisfies equation
\ref{E:y:polynomial}. It follows from the definitions of $\ymax$ and $\ymin$ in terms of 
roots of the polynomial \ref{E:y:polynomial} that the maximum and minimum values
of $y$ are therefore $\ymax$ and $\ymin$ respectively.

c. The stated asymptotics of $\ymin$ and $\ymax$ as $\tau \ra 0$ follow immediately from the characterisation of 
$\ymin$ and $\ymax$ as the only solutions of \ref{E:y:polynomial} in the range $[0,1]$.

(ii): \emph{ $\abs{\tau}=\taumax$.} \phantom{a} When $\tau^2 = \taumax^2$, from \ref{E:y:dot} 
we have $\dot{y}^2 = 4(f(y) - 4\taumax^2) \le 4(\fmax - 4\taumax^2) \le 0$,
with equality if and only if $f(y)=\fmax$, \textit{i.e.} if and only if $y=q/n$.
Hence we have $\dot{y}=0$ for all $t \in \R$ and $y\equiv q/n$.

(iii): \emph{ $\tau=0$ and $p>1$.} \phantom{a} Recall from \ref{E:crit:f} that for $p>1$ both $y=0$ and $y=1$ 
are critical points of $f$ and hence give rise to constant solutions $y\equiv 0$ and $y \equiv 1$ 
of  \ref{E:y:ddot}.

\ref{E:y:dot} implies $\dot{y}=0$ if and only if $y=0$ or $y=1$. Since $y\in [0,1]$  and $\{0,1\} \subset \text{Crit}(f)$
a non-constant solution $y$ contains no points with $\dot{y}=0$ and is therefore
monotone with $0<y<1$ for all $t$. If $y$ is increasing then $y \circ \tbar$ is decreasing and hence by 
composing with $\tbar$ if necessary we can assume $y$ satisfies the 1st order ODE
\begin{equation}
\addtocounter{theorem}{1}
\label{E:y:dec}
\dot{y} = -2 \sqrt{f(y)}.
\end{equation}
Since $y$ is monotone and bounded it must approach constant values $c_{-}$ and $c_{+}$ as $t \ra \pm \infty$.
Recall the elementary fact that if $\gamma$ is an integral curve of a vector field $V$ and $\lim_{t \ra \infty} \gamma(t)=\gamma_{\infty}$, 
then $\gamma_{\infty}$ must be a zero (or stationary point) of the vector field $V$.
Hence
we see that $c_{\pm}$ must be stationary points of \ref{E:y:ddot} which also belong to the zero energy level.
Therefore  $c_{\pm} \in \text{Crit}(f) \cap f^{-1}(0)= \{0,1\}$.
Since $y$ is strictly decreasing we must have $\lim_{t\ra -\infty}y(t)=1$ and $\lim_{t \ra \infty}y(t)=0$.
In particular, for any such solution of \ref{E:y:dec}  there exists $t_{0}\in \R$ so that $y(t_{0})=q/n$. 
Hence $\hat{y}:=y \circ \TTT_{t_{0}}$ is a solution of \ref{E:y:dec} with $\hat{y}(0)=q/n$, and so 
by uniqueness of the initial value problem $\hat{y} \equiv y_{0}$.

(iv):  \emph{$\tau=0$ and $p=1$.} \phantom{a} Recall from \ref{E:crit:f} that for $p=1$, $y=0$ (but not $y=1$) 
is a critical point of $f$ and so gives rise to the stationary point $y\equiv 0$ of \ref{E:y:ddot}.

Again from \ref{E:y:dot}, $\dot{y}=0$ if and only if $y=0$ or $y=1$. 
For $p=1$, $y=0$ is a stationary point of \ref{E:y:ddot} but $y=1$ is not. If $y$ is non-constant, 
then $y$ cannot attain an interior minimum since $\dot{y}(t)=0$ implies $y(t)=1$. 
Therefore, as $\text{Crit}(f) \cap f^{-1}(0)= \{0\}$ for $p=1$, $y$ must approach $0$ as $t \ra \pm \infty$.
Since $y\in [0,1]$ is non-constant and tends to $0$ as $t\ra \pm \infty$,  $y$ attains an interior maximum at some point $t_{0}\in \R$. 
Hence $\dot{y}(t_{0})=0$ and therefore $y(t_{0})=1$. Then by uniqueness of the initial value problem $y \circ \TTT_{t_{0}}=y_{0}$. 
Evenness of $y_{0}$ follows from the invariance of \ref{E:y:ddot} and the initial conditions $y(0)=1,\, \dot{y}(0)=0$
under $t \mapsto -t$.
\end{proof}

We can use Lemma \ref{P:y} to establish normal forms for solutions of \ref{E:odes:p:n}.
\begin{prop}
Fix a pair of admissible integers $p$ and $q$ and let $\bw$ be any solution of \ref{E:odes:p:n} with 
$\mathcal{I}_{1}(\bw)=1$ and $\mathcal{I}_{2}(\bw)=-2\tau$ with $0\le\abs{\tau}\le \taumax$.
\addtocounter{equation}{1}
\label{P:w:normal:form}
\begin{itemize}
\item[(i)]
If $p>1$  and $0<\abs{\tau}\le\taumax$ then $\bw$ is equivalent under symmetries (1)--(3) to 
$\bw_{\tau}: \R \ra \Sph^{3}$
defined as the unique solution to \ref{E:odes:p:n} with initial value 
\begin{equation*}
\bw_\tau(0)= \left(\sqrt{\tfrac{p}{n}} e^{i\alpha_\tau/2p}, \sqrt{\tfrac{q}{n}} e^{i \alpha_\tau / 2q}\right),
\end{equation*}
where $\alpha_\tau \in [-\tfrac{\pi}{2}, \tfrac{\pi}{2}]$ is defined by
\begin{equation*}
\alpha_\tau := \arcsin{\left(-\frac{\tau}{\taumax}\right)}.
\end{equation*}
\item[(ii)] If $p>1$ and $\tau=0$ then $\bw$ is equivalent under symmetries (1)--(3) to the unique solution 
of \ref{P:odes:p:n} with one of the following four initial conditions
\begin{itemize}
\item[a.] $\bw(0)=(1,0)$,
\item[b.] $\bw(0)=(0,1)$,
\item[c.] $\bw(0)=\left(\sqrt{\tfrac{p}{n}},\sqrt{\tfrac{q}{n}}\right)$,
\item[d.] $\bw(0)=\left(e^{i\pi/2p}\sqrt{\tfrac{p}{n}}, e^{i\pi/2q}\sqrt{\tfrac{q}{n}}\right)$.
\end{itemize}
\item[(iii)]
If $p=1$ and $0<\abs{\tau}\le\taumax$ then $\bw$ is equivalent under symmetries (1)--(3) to $\bw_{\tau}: \R \ra \Sph^{3}$
defined as the unique solution to \ref{E:odes:p:n} with initial value 
\begin{equation*}
\bw_{\tau}(0) = (-i\sgn{\tau}  \sqrt{1-\ymax},\sqrt{y_{max}}).
\end{equation*}
\item[(iv)] If $p=1$ and $\tau=0$ then $\bw$ is equivalent under symmetries (1)--(3) to the unique solution 
of \ref{P:odes:p:n} with one of the following two initial conditions
\begin{itemize}
\item[a.] $\bw(0)=(1,0)$,
\item[b.] $\bw(0)=(0,1)$.
\end{itemize}
\end{itemize}
\end{prop}

\begin{proof}
 Let $\bw$ be any solution of \ref{E:odes:p:n} with $\mathcal{I}_{1}(\bw)=\abs{\bw}^{2}=1$ , 
$\mathcal{I}_{2}(\bw)=-2\tau$ and $0<\abs{\tau}\le\taumax$.
Set $y=\abs{w_{2}}^{2}$ and write $\bw(0)=(\sqrt{1-y(0)}\, e^{i\theta_{1}},\sqrt{y(0)}\, e^{i\theta_{2}})$.

(i) \emph{Case $p>1$ and $\tau\neq 0$.}
If $\tau=\pm \taumax$ then by \ref{P:y}(ii) $y\equiv q/n$ and hence $y(0)=q/n$ and $\dot{y}(0)=0$.
If $0<\abs{\tau}<\taumax$ then by \ref{P:y}(i)a and using the time translation invariance of \ref{E:odes:p:n} 
(symmetry 1) we can arrange that 
$y(0)=\abs{w_{2}}^{2}(0)=\tfrac{q}{n}$ and that $\dot{y}(0)<0$. 
In both cases we have 
$$y(0)=\tfrac{q}{n} \quad \text{and}  \quad \dot{y}(0)\le 0.$$
The former together with \ref{E:ydot:cx} implies that 
$$ \sin(p\theta_{1}+q\theta_{2}) = - \tfrac{\tau}{\taumax},$$
while the latter together with \ref{E:ydot:cx} implies 
$$\cos(p\theta_{1}+q\theta_{2})\ge 0.$$
Acting with the $n$th root of unity $z_{k}= e^{2\pi ki/n}$ (symmetry 2) leaves $y(0)$ invariant and sends 
$p\theta_{1}+q\theta_{2} \mapsto p\theta_{1}+q\theta_{2} + 2k\pi$. 
Hence by using symmetry (2) we can arrange that $p\theta_{1}+q\theta_{2} \in [-\pi,\pi).$
Finally by using symmetry (3) we can arrange that $p\theta_{1}=q\theta_{2}$.
Therefore we have 
$$\sin 2p\theta_{1}= \sin 2q\theta_{2}=-\tfrac{\tau}{\taumax} = \sin{\alpha_{\tau}}, \quad  
\cos{2p\theta_{1}}\ge 0\quad \text{and} \quad 2p\theta_{1}\in [-\pi,\pi).$$ 
Hence $2p\theta_{1}= 2q\theta_{2}=\alpha_{\tau}$ as claimed.
Notice that in this case
$$w_1^p w_2^q(0) = \left(\sqrt{\tfrac{p}{n}}\right)^p \left(\sqrt{\tfrac{q}{n}}\right)^q e^{i\alpha_\tau} = 
2\taumax e^{i\alpha_\tau}.$$

(ii) \emph{Case $p>1$ and $\tau=0$.} By \ref{P:y}(iii) $y=\abs{w_{2}}^{2}$ must be one of the following: 
(a) $y \equiv 0$, (b) $y\equiv 1$, (c) $y=y_{0} \circ \TTT_{t_{0}}$, (d) $y=y_{0} \circ \TTT_{t_{0}} \circ \tbar$ 
for some $t_{0} \in \R$, where $y_{0}:\R \ra (0,1)$ is the function defined in \ref{P:y}(iii)c.
 It is easily seen that (a) implies $\bw$ is a stationary point of the form $\bw=(e^{i\theta_{1}},0)$, 
while (b) implies $\bw$ is a stationary point of the form $\bw=(0,e^{i\theta_{2}})$. Hence $\bw$ is equivalent using symmetry 
(3) to the stationary points $(1,0)$ or $(0,1)$ in cases (a) and (b) respectively. 
Suppose we are now in case (c) or (d) and hence $y=y_{0} \circ \TTT_{t_{0}}\circ \tbar^{j}$ for some $t_{0}\in \R$ and 
$j \in \{0,1\}$. By time translation invariance of \ref{E:odes:p:n} we can arrange that $y(0)=\tfrac{q}{n}$ (i.e. that $t_{0}=0$.) Thus we have 
$$y(0) = \tfrac{q}{n}, \quad \text{and} \quad \dot{y}(0)=(-1)^{j+1}4\taumax.$$
Substituting these initial conditions into \ref{E:ydot:cx} and simplifying yields 
$$ e^{i(p\theta_{1}+q\theta_{2})} = (-1)^{j}.$$
As in case (i) by using symmetry (2) we may arrange that $p\theta_{1}+q\theta_{2} \in [-\pi,\pi)$ and then use symmetry 
(3) to arrange that also $p\theta_{1}=q\theta_{2}$. Hence the previous equality reduces to 
$ e^{2ip\theta_{1}}=(-1)^{j}$ with $2p\theta_{1}\in [-\pi,\pi).$
In case (c) $j=0$ and so $2p\theta_{1}=2q\theta_{2}=0$ is the unique solution in the required range, whereas in 
case (d) $j=1$ and so $2p\theta_{1}=2q\theta_{2}=\pi$ as claimed.

(iii) \emph{Case $p=1$ and $\tau \neq 0$.} 
If $0<\abs{\tau}<\taumax$ by using time translation invariance of \ref{E:odes:p:n} (symmetry 1) we may assume that 
$y(0)=\ymax$ and therefore also $\dot{y}(0)=0.$ If $\tau= \pm \taumax$ then by \ref{P:y}(ii) $y \equiv q/n = \ymax$ and $\dot{y}(0)=0$.
Hence in either case from \ref{E:ydot:cx} we have 
$$ 2i\tau=-w_{1}w_{2}^{n-1}(0) = -i\sqrt{f(\ymax)} \sin (\theta_{1}+(n-1)\theta_{2}) = -2i\abs{\tau} \sin (\theta_{1}+(n-1)\theta_{2}),$$
and therefore
$$\sin(\theta_{1}+(n-1)\theta_{2})= -\frac{\tau}{\abs{\tau}} = -\sgn{\tau}.$$
As in the previous cases by acting with an $n$th root of unity we can arrange that $\theta_{1}+(n-1)\theta_{2}\in [-\pi,\pi)$
and by acting with symmetry (3) that $\theta_{2}=0$. 
Therefore we have $\sin{\theta_{1}}= -\sgn{\tau}$ with $\theta_{1}\in [-\pi,\pi)$.
Hence $\theta_{1} = -\sgn{\tau} \cdot \tfrac{1}{2}\pi$ as claimed.


(iv) \emph{Case $p=1$ and $\tau=0$.} By \ref{P:y}(iv) $y=\abs{w_{2}}^{2}$ must be one of the following: (a) $y\equiv 0$ 
or (b) $y=y_{0} \circ \TTT_{t_{0}}$ where $t_{0}\in \R$ and $y_{0}:\R \ra (0,1]$ is the function defined in \ref{P:y}(iv)b.
As in (ii),  case (a) implies that $\bw$ is a stationary point of the form $(e^{i\theta_{1}},0)$ and hence is equivalent using 
symmetry (3) to $(1,0)$ as claimed. Suppose now that we are in case (b). By time translation invariance we can 
arrange that $t_{0}=0$ and hence $y(0)=1$. This implies $w_{1}(0)=0$ and $w_{2}(0)=e^{i\theta_{2}}$
for some $\theta_{2}\in \R$. Using symmetry (3) we can arrange that $\theta_{2}=0$ and hence that $\bw(0)=(0,1)$
as claimed.

\end{proof}

\begin{remark}
\addtocounter{equation}{1}
Note that in cases (ii)a and (ii)b of Proposition \ref{P:w:normal:form} 
the initial conditions are stationary points of \ref{E:odes:p:n} and hence the 
corresponding solutions with this initial data are $\bw \equiv (1,0)$ and $\bw \equiv (0,1)$ respectively. 
Let $\bw_{0}$ denote the unique solution to \ref{E:odes:p:n} with initial condition 
$\bw_{0}(0)=\left(\sqrt{\tfrac{p}{n}},\sqrt{\tfrac{q}{n}}\right)$ as in (ii)c. Then by uniqueness of the initial value problem 
for \ref{E:odes:p:n} we see that 
\begin{equation}
\addtocounter{theorem}{1}
\label{E:w0:pgt1}
\bw_{0}(t) = (\sqrt{1-y_{0}(t)},\sqrt{y_{0}(t)}\,),
\end{equation}
where $y_{0}(t):\R \ra (0,1)$ is the even function defined in \ref{P:y}(iii)c.

Note that $(0,1)$ is not a stationary point of \ref{E:odes:p:n} for $p=1$. Let $\bw_{0}$ denote the unique solution of \ref{E:odes:p:n} with initial condition $\bw_{0}(0)=(0,1)$ as in (iv)b. 
Then by the uniqueness of the initial value problem for \ref{E:odes:p:n} we see that
\begin{equation}
\addtocounter{theorem}{1}
\label{E:w0:peq1}
\bw_{0}(t) = (\sgn{t}\sqrt{1-y_{0}(t)},\sqrt{y_{0}(t)}\,),
\end{equation}
where  $y_{0}:\R \ra (0,1]$ is the even function defined in \ref{P:y}(iv)b. (Since $w_{2}(t)=\sqrt{y_{0}(t)\,}$ is real and positive for all $t$, 
the equation for $\dot{w}_{1}$ in \ref{P:odes:p:n} implies that $\dot{w}_{1}>0$ for all $t$. By \ref{P:y}iv 
$\sqrt{1-y_0(t)}$ is decreasing for $t<0$ and increasing $t>0$, so $\sgn{t} \sqrt{1-y_0(t)}$ is increasing for all $t$ as required.)
\end{remark}

\begin{remark}
\addtocounter{equation}{1}
\label{R:w:normal:form}
The argument from \ref{P:w:normal:form}(iii) applied in the case $p>1$ implies that any solution of \ref{E:odes:p:n} with 
$\mathcal{I}_{1}(\bw)=1$ and $\mathcal{I}_{2}(\bw)=-2\tau$ and $0<\abs{\tau}\le\taumax$ is equivalent under symmetries 
(1) to (3) to 
$$\hat{\bw}_{\tau} = \left(-e^{i\pi/2p}\, \sgn(\tau) \sqrt{1-\ymax},\sqrt{\ymax}\right).$$
Similarly, the argument from \ref{P:w:normal:form}(i) works for $p=1$ as well as $p>1$.  
However,  we will only make use of the normal forms stated in \ref{P:w:normal:form}. The difference in our choice of normal 
form for $p=1$ and $p>1$ reflects differences in the geometry of the resulting special Legendrian immersions in these cases as we will explain later.
\end{remark}

\subsection*{$\bw_{\tau}$ and the $\sorth{p}\times \sorth{q}$-invariant special Legendrian immersions $X_\tau$.}
$\phantom{ab}$
\nopagebreak

We now define the particular $1$-parameter family of $(p,q)$-twisted SL curves we will use throughout the 
rest of the paper by specifying initial data $\bw_{\tau}(0)$ as in the normal form Proposition 
\ref{P:w:normal:form}. Associated to the $1$-parameter family $\bw_{\tau}$ is the $1$-parameter family 
$X_{\tau}$ of $\sorth{p} \times \sorth{q}$-invariant special Legendrians. 
Proposition \ref{P:w:normal:form} implies that any $\sorth{p} \times \sorth{q}$-invariant special Legendrian in $\Sph^{2p+2q-1}$ is congruent to $X_{\tau}$ for some $\tau$.




\begin{prop}
\addtocounter{equation}{1}
\label{P:w:tau}
Fix a pair of admissible integers $p$ and $q$ and choose any $\tau \in [-\taumax,\taumax]$.  Define $\bw_\tau: \R \ra \Sph^3$ 
as the unique solution of \ref{E:odes:p:n} with initial data 
\addtocounter{theorem}{1}
\begin{equation}
\label{E:w:ic:p:neq:1}
\bw_\tau(0)= \left(\sqrt{\tfrac{p}{n}} e^{i\alpha_\tau/2p}, \sqrt{\tfrac{q}{n}} e^{i \alpha_\tau / 2q}\right) \qquad \text{if $p>1$;}
\end{equation}
where $\alpha_\tau \in [-\tfrac{\pi}{2}, \tfrac{\pi}{2}]$ is defined by
\addtocounter{theorem}{1}
\begin{equation}
\label{E:alpha:tau}
\alpha_\tau := \arcsin{\left(-\frac{\tau}{\taumax}\right)}, 
\end{equation}
or
\addtocounter{theorem}{1}
\begin{equation}
\label{E:w:ic:p:eq:1}
\bw_{\tau}(0) = (-i\sgn{\tau}  \sqrt{1-\ymax},\sqrt{y_{max}}) \quad \text{if $p=1$.}
\end{equation}
Then $\bw_\tau$ depends real analytically on $\tau \in (-\taumax,\taumax)$
and satisfies $\bw_{-\tau} = \overline{\bw}_{\tau}.$
In particular, $\bw_0: \R \ra \Sph^3 \subset \C^2$ is contained in $\R^2 \subset \C^2$.
\end{prop}

\begin{proof}
To prove that $\bw_\tau$ depends analytically on $\tau$ it suffices by \ref{E:odes:p:n}.iii to prove that the initial condition 
$\bw_\tau(0)$ given by  \ref{E:w:ic:p:neq:1} or \ref{E:w:ic:p:eq:1} depends analytically on $\tau$ 
for $\tau   \in (-\taumax,\taumax)$.
For $p> 1$ it is clear from \ref{E:alpha:tau} that $\alpha_{\tau}$ depends real analytically 
on $\tau$ for $\abs{\tau}< \taumax$. Hence by \ref{E:w:ic:p:neq:1} $\bw_{\tau}(0)$ depends analytically on $\tau$ for $\abs{\tau}<\taumax$.
For $p=1$ we have $f'(\ymax) = \ymax^{n-2}(q-n \ymax) \neq 0$ for $\abs{\tau}<\taumax$. Hence by
the real analytic Implicit Function Theorem (see e.g. \cite[Thm 2.3.5]{krantz}) 
$\ymax$ is an analytic function of $\tau \in (-\taumax,\taumax)$. 
Therefore $\sqrt{\ymax}$ is also an analytic function of $\tau \in (-\taumax,\taumax)$ (recall $\ymax \ge (n-1)/n$).
Write $\bw_{\tau}(0) = (ir_{\tau}, \sqrt{\ymax})$. 
Because $\mathcal{I}_{2}(\bw_{\tau}(0)) = w_{1}w_{2}^{n-1}(0) = -2i\tau$ 
$$r_{\tau}= -\frac{2\tau}{\sqrt{\ymax}^{\,n-1}}$$ 
and hence is an analytic function of $\tau \in (-\taumax,\taumax)$.
From \ref{E:w:ic:p:neq:1} or \ref{E:w:ic:p:eq:1} we have $\bw_{-\tau}(0)=\overline{\bw}_\tau(0)$ 
and hence $\bw_{-\tau} = \overline{\bw}_\tau$ by uniqueness of the initial value problem for \ref{E:odes:p:n}.
\end{proof}

\subsubsection*{The associated function $y_{\tau}:=\abs{w_{2}}^{2}$ and its initial value characterisation}

For the solution $\bw_{\tau}$ defined in \ref{P:w:tau}, define  $y_\tau:= \abs{w_2}^2$.  
By \ref{P:odes:p:n} $y_{\tau}$ satisfies equations \ref{E:y:dot} and \ref{E:y:ddot}. 
Analytic dependence of $y_{\tau}$ on $\tau\in (-\taumax,\taumax) $ follows immediately from analytic dependence of $\bw_{\tau}$. 

For $p=1$, $y_\tau$ is the unique solution of \ref{E:y:ddot} satisfying the
initial conditions
\begin{equation}
\addtocounter{theorem}{1}
\label{E:y:ic:p:eq:1}
y(0) = \ymax, \quad \dot{y}(0) = 0.
\end{equation}
In particular, $y_0$ is the unique solution of \ref{E:y:ddot} satisfying $ y(0)=1, \ \dot{y}(0)=0$ 
introduced in \ref{P:y}.iv.b.

Similarly, for $p> 1$, $y_\tau$ is the unique solution of \ref{E:y:ddot} satisfying the initial conditions
\begin{equation}
\addtocounter{theorem}{1}
\label{E:y:ic:p:neq:1}
y(0) = \frac{q}{n}, \quad \dot{y}(0) = -4 \taumax \cos{\alpha_\tau} = -4 \sqrt{\taumax^{2}-\tau^{2}}.
\end{equation}
$y_0$ coincides with the solution of \ref{E:y:ddot} satisfying $y(0)= q/n, \ \dot{y}(0) = -4\taumax$ 
introduced in \ref{P:y}.iii.c. 

For both $p=1$ and $p>1$ it follows from these initial value characterisations of $y_\tau$ that
$y_{-\tau}=y_{\tau}$ which is consistent with the fact that $\bw_{-\tau} = \overline{\bw}_\tau$.



\medskip

Earlier we thought about $\sorth{p}\times \sorth{q}$-invariant special Legendrians 
by treating our special Legendrians as (unparametrised) subsets of $\Sph^{2(p+q)-1}$.
From now on it will be more convenient to deal with special Legendrian immersions $X_{\tau}: \cylpq \ra \Sph^{2(p+q)-1}$
and to talk about $\sorth{p} \times \sorth{q}$-equivariant immersions with respect to 
the obvious actions of $\sorth{p} \times \sorth{q}$ on both domain and target.
This will facilitate future discussion of additional discrete symmetries possessed by $X_{\tau}$.

We now define the family of special Legendrian immersions 
$X_\tau: \cylpq  \ra \Sph^{2(p+q)-1}$
using the $(p,q)$-twisted SL curves $\boldsymbol{w}_\tau$ defined in Proposition \ref{P:w:tau},
where $\cylpq$ denotes the round cylinder of type $(p,q)$ defined in \ref{E:cylpq}.

\addtocounter{equation}{1}
\begin{definition}
\label{D:X:tau}
For $\tau \in [-\taumax,\taumax]$ define an immersion 
$X_\tau: \cylpq \ra \Sph^{2(p+q)-1}$
by
\begin{eqnarray*}
X_\tau(t,\sigma_1, \sigma_2)  = & (w_1(t)\cdot \sigma_1,\  w_2(t) \cdot \sigma_2), \quad & \text{for \ } p>1;\\
X_\tau(t,\sigma) = & (w_1(t), w_2(t) \cdot \sigma), \quad  & \text{for \ } p=1,
\end{eqnarray*}
where $t \in \R$, $\sigma_1 \in \Sph^{p-1}$, $\sigma_2\in \Sph^{q-1}$,  $\sigma \in \Sph^{n-2}$
and $\bw_\tau=(w_1,w_2)$ is the unique solution to \ref{E:odes:p:n} specified in Proposition \ref{P:w:tau}.
\end{definition}
We now establish some basic properties of $X_\tau$.
\addtocounter{equation}{1}
\begin{prop}
\label{P:X:tau}
For $\tau \in [-\taumax,\taumax]$ the immersion 
$X_\tau: \cylpq \ra \Sph^{2(p+q)-1}$ defined in \ref{D:X:tau} has the following properties:
\begin{itemize}
\item[(i)] $X_\tau$ is a smooth special Legendrian immersion
depending analytically on $\tau$ for $\tau \in (-\taumax,\taumax),$
and satisfies $X_{-\tau} = \overline{X}_\tau$. 
In particular, $X_0$ is contained in $\Sph^{p+q-1} \subset \R^{p+q} \subset \C^{p+q}$.
\item[(ii)] For $p>1$, the metric $g_{\tau}$ on $\cylpq$ induced by $X_{\tau}$ is 
$$\abs{\dot{\bw}}^2 dt^2 + \abs{w_1}^2 g_{\Sph^{p-1}} + \abs{w_2}^2 g_{\Sph^{q-1}}
= y^{q-1}(1-y)^{p-1}dt^2 + (1-y) g_{\Sph^{p-1}} + y\, g_{\Sph^{q-1}}.$$
For $p=1$,  the induced metric $g_{\tau}$ on $\cylone$ is
$$\abs{\dot{\bw}}^2 dt^2 +  \abs{w_2}^2 g_{\Sph^{n-2}}
= y^{n-2}dt^2 + y\, g_{\Sph^{n-2}}.$$
\item[(iii)]  $X_\tau$ is $\sorth{p} \times \sorth{q}$-equivariant, i.e. 
for any $\MMM=(\MMM_1, \MMM_2) \in \sorth{p} \times \sorth{q}$ 
we have $$ \mtilde \circ X_{\tau} = X_{\tau}\circ \MMM,$$
where $\MMM=(\MMM_1, \MMM_2)$ acts on $\cylpq$ by
$ \MMM \cdot (t,\sigma_1, \sigma_2) = (t, \MMM_1 \sigma_1, \MMM_2 \sigma_2),$
and 
$$\mtilde = \left(\begin{array}{cc}\MMM_{1} & 0  \\0 & \MMM_{2}\end{array}\right) \in \sorth{p}\times \sorth{q}
 \subset \sorth{p+q}.$$
\item[(iv)]
When $\tau=0$ we have
$$X_0(\cylpq) = 
\begin{cases}
\Sph^{p+q-1} \setminus (\Sph^{p-1},0) \cup (0,\Sph^{q-1}), \quad &\text{for $p>1$;}\\
\Sph^{n-1} \setminus (\pm 1, 0) \in \R \oplus \R^{n-1}, \quad &\text{for $p=1$.}
\end{cases}
$$
\item[(v)] When $\tau = \taumax$, we have
\addtocounter{theorem}{1}
\begin{equation}
\label{E:X:taumax}
X_{\taumax} =
\begin{cases}
\left(-i \sqrt{\tfrac{1}{n}} e^{2in\tau t}, \sqrt{\frac{n-1}{n}}e^{-2in\tau t/(n-1)}\,\right), \quad & \text{for\ } p=1;\\
\left( \sqrt{\tfrac{p}{n}} e^{-i\pi/(4p)} e^{2ni\tau t/p}, \sqrt{\tfrac{q}{n}} e^{-i\pi/(4q)} e^{-2ni\tau t/q} \right), \quad & \text{for\ } p>1.
\end{cases}
\end{equation}
\item[(vi)] If $X: \cylpq \ra \Sph^{2(p+q)-1}$ is any non totally geodesic 
$\sorth{p} \times \sorth{q}$-invariant special Legendrian immersion 
then $ X = e^{i\omega}\, \ttilde_{x} \circ X_{\tau} \circ \TTT_{y}$ for some $x$, $y \in \R$,  $0<\abs{\tau}<\taumax$ and $n$th root of unity $\omega \in \Sph^{1}$
where $\ttilde_{x} \in \sunit{n}$ is defined by 
\begin{equation}
\addtocounter{theorem}{1}
\label{E:ttilde}
\ttilde_x = \left(
\begin{matrix}
e^{ix/p} \Id_{p}& 0  \\
0 & e^{-ix/q}\Id_{q} \\
\end{matrix}
\right).
\end{equation}
\end{itemize}
\end{prop}
\begin{proof}
(i) For $\tau\neq 0$ we have $\abs{w_{1}}^{2}\ge \ymin>0$ and $\abs{w_{2}}^{2} \ge 1-\ymax >0$. 
Because there are no points where $w_{1}$ or $w_{2}$ vanish, \ref{P:leg:triple} implies that $X_{\tau}$ is 
a Legendrian immersion. 
Since $\bw_\tau$ is a solution of \ref{E:odes:p:n} \ref{C:sl:triples} and\ref{L:w:slg} imply that $X_\tau$ is special Legendrian.
We deal with the exceptional case $\tau=0$ separately in part (iv).
Analytic dependence of $X_\tau$ on $\tau$ follows from the analytic dependence 
of $\bw_\tau$ on $\tau$ proved in Proposition \ref{P:w:tau}.
The final part follows from the fact that $\bw_{-\tau}  = \overline{\bw}_{\tau}$ (see \ref {P:w:tau}).

\noindent
(ii) follows immediately from equations \ref{E:twist:metric} and \ref{E:w:reparam}.

\noindent
(iii) The $\sorth{p} \times \sorth{q}$-equivariance of $X_\tau$ is clear from the definition of $X_\tau$ in \ref{D:X:tau}.

\noindent
(iv) $\tau=0$ limit. From part (i), $X_0(\cylpq) \subset \Sph^{p+q-1}$. 

Consider first the case where $p>1$. 
From \ref{E:w0:pgt1}  
$$X_{0}(t,\sigma_{1},\sigma_{2}) = (\sqrt{1-y_{0}(t)}\, \sigma_{1}, \sqrt{y_{0}(t)}\, \sigma_{2})$$ 
where 
$y_{0}: \R \surj (0,1)$ is the decreasing function defined in \ref{P:y}.iii.c.
Recall from Remark \ref{R:product:cones} that the 
map $\Pi: [0,\pi/2] \times \Sph^{p-1}\times \Sph^{q-1} \ra \Sph^{p+q-1}$ given by 
$$ \Pi\,(t,\sigma_{1},\sigma_{2}) = (\cos{t}\, \sigma_{1}, \sin{t}\, \sigma_{2}),$$
is surjective and on restriction 
to the interval $(0,\pi/2)$ gives a diffeomorphism between $(0,\pi/2) \times \Sph^{p-1}\times \Sph^{q-1}$ and  
$\Sph^{p+q-1}\setminus (\Sph^{p-1},0) \cup (0,\Sph^{q-1})$.
Since by \ref{P:y}.iii.c $y_{0}$ is decreasing with $\lim_{t\ra -\infty}y(t)=1$ and $\lim_{t \ra \infty}y(t) =0$ 
we see that $X_{0}$ is a reparametrisation of this diffeormorphism. 

Similarly, from \ref{E:w0:peq1} for $p=1$ we have 
$$X_{0}(t,\sigma) = (-\sgn{t}\sqrt{1-y_{0}(t)}, \sqrt{y_{0}(t)}\, \sigma),$$
where $y_{0}: \R \surj (0,1]$ is the even function defined in \ref{P:y}.iv.b.
The map $\Pi: [0,\pi]\times \Sph^{n-2} \surj \Sph^{n-1}$ defined by 
$ \Pi\,(t,\sigma) = (\cos{t}, \,\sin{t}\, \sigma)$
on restriction to the open interval $(0,\pi)$ gives a diffeomorphism between $(0,\pi) \times \Sph^{n-2}$ and 
$\Sph^{n-1}\setminus (\pm 1,0)$. Since  by \ref{P:y}.iv.b $y_{0}$ is even, increasing on $(-\infty,0)$, satisfies $y_{0}(0)=1$ and 
$\lim_{t\ra \pm \infty}y_{0}(t)=0$ we see that 
$X_{0}$ is a reparametrisation of this diffeomorphism.

\noindent
(v) $\tau = \taumax$ limit.  We leave this as an elementary exercise for the reader.

\noindent
(vi) follows from \ref{C:so:invariant:slg} and the normal form for solutions of \ref{E:odes:p:n} established in \ref{P:w:normal:form}.
\end{proof}

\section{Discrete symmetries of $\bw_{\tau}$}
\label{S:w:sym}
In this section we study the discrete symmetries of $\bw_{\tau}$ and the 
conditions under which $\bw_{\tau}$ corresponds to a closed curve of $\sorth{p}\times\sorth{q}$ orbits.
We will use these results in the following section to study the full group of symmetries of $X_{\tau}$.

\subsection*{Symmetries of $y_{\tau}$}
We begin by establishing the 
symmetries of $y_{\tau}:=\abs{w_{2}}^{2}$ in the three cases 
(i) $p=1$, (ii) $p>1$ and $p\neq q$ and (iii) $p>1$ and $p=q$.

To state these results we need to introduce some notation to describe the basic properties of $y_{\tau}$.
For $p>1$, recall from \ref{E:y:ic:p:neq:1} that $y_{\tau}$ satisfies the initial conditions 
$$y(0) = \frac{q}{n}, \quad \dot{y}(0) = -4 \taumax \cos{\alpha_\tau} = -4 \sqrt{\taumax^{2}-\tau^{2}},$$
whereas for $p=1$ from \ref{E:y:ic:p:eq:1} it satisfies 
$$y(0) = \ymax, \quad \dot{y}(0) = 0.$$
The different initial conditions for $y_{\tau}$ affect where the $2\pt$-periodic function $y_{\tau}$ attains its 
maxima and minima in the cases $p=1$ and $p> 1$. 
In the case $p>1$ the choice of initial data for $y_\tau$ implies that there exist unique real numbers $\pt^+, \pt^- \in (0,\pt)$ satisfying
\begin{equation}
\addtocounter{theorem}{1}
\label{E:y:max:min}
y_\tau(-\pt^-) = \ymax, \quad y_\tau(\pt^+) = \ymin,
\end{equation}
and so that $y_\tau$ is strictly decreasing on $(-\pt^-, \pt^+)$. 
We call these two numbers the \emph{partial-periods} of $y_\tau$, since 
\begin{equation}
\addtocounter{theorem}{1}
\label{E:part:periods}
2\pt = 2\pt^+ + 2\pt^-.
\end{equation}
In general, $\pt^+$ and $\pt^-$ are not related except when $p=q$ when we will prove shortly that they are equal.
Illustrative plots of $y_{\tau}$ are shown in Figures \ref{fig:w2:pequals1} and \ref{fig:w2:pneq1} for $p=1$ and $p>1$, $p \neq q$ 
respectively.

\iffigureson
\begin{figure}
\centering
\input{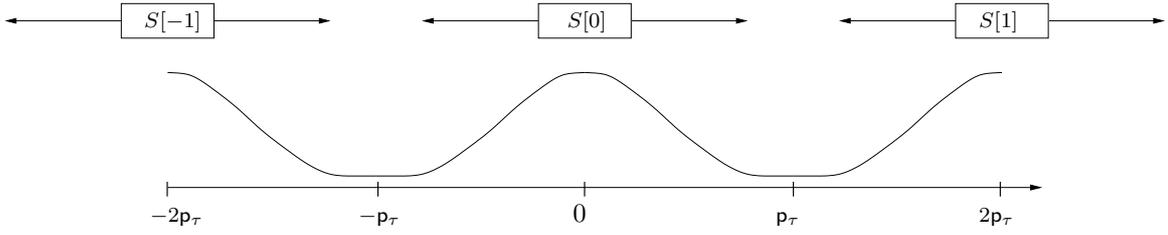}
\caption{Profile of $y_{\tau}:=\abs{w_2}^{2}$ for $p=1$}      
\label{fig:w2:pequals1}
\end{figure}

\begin{figure}
\centering
\input{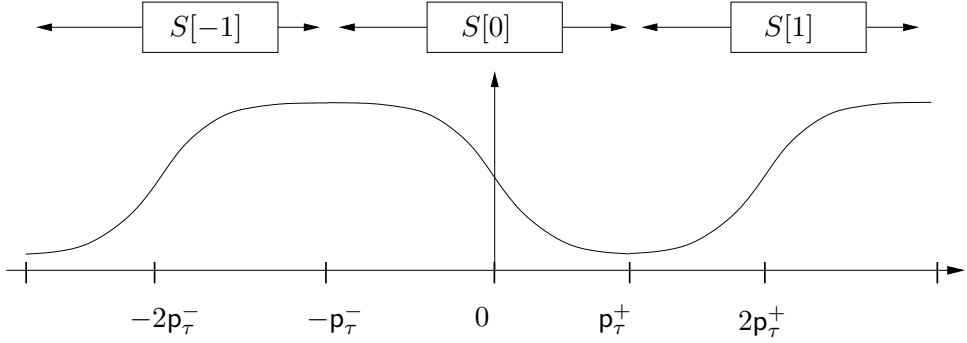}
\caption{Profile of $y_{\tau}=\abs{w_2}^{2}$ for $p>1$}
\label{fig:w2:pneq1}
\end{figure}
\else
\fi

Throughout the following lemma we assume $0<\abs{\tau}<\taumax$ and discuss the exceptional cases $\tau=0$ 
and $\abs{\tau}=\taumax$ in Remark \ref{R:y:sym:ex} below.
Recall, also the notation for elements in $\Isom(\R)$ introduced in Section 1 in Notation and Conventions.
\begin{lemma}[Symmetries of $y_{\tau}$]\hfill
\addtocounter{equation}{1}
\label{L:y:symmetry}
\begin{itemize}
\item[(i)] 
For $p=1$, $q=n-1$ the symmetries of $y_{\tau} =\abs{w_{2}}^{2}$ are generated by 
\begin{equation}
\addtocounter{theorem}{1}
\label{E:y:sym:p:eq:1}
y_{\tau} \circ \TTT_{2\pt} = y_{\tau} \quad \text{and} \quad  y_{\tau} \circ \tbar = y_{\tau}.
\end{equation}
That is, $y_{\tau}$ is an even $2\pt$-periodic function.
Moreover, we have
\begin{equation}
\addtocounter{theorem}{1}
\label{E:y0:p:eq:1}
 y_\tau(0) = \ymax \quad \text{and} \quad y_\tau(\pt) = \ymin.
\end{equation}
\item[(ii)] For $p>1$ and $p\neq q$ the symmetries of $y_{\tau}$ are generated by 
\begin{equation}
\addtocounter{theorem}{1}
\label{E:y:sym:p:neq:1}
y_\tau \circ \TTT_{2\pt}= y_\tau , \quad y_\tau  \circ \tbar_{\pt^+} = y_\tau  \quad 
\text{and} \quad y_\tau  \circ \tbar_{\,-\pt^-}=y_\tau.
\end{equation}
\item[(iii)] 
For $p>1$ and $p=q$ the symmetries of $y_{\tau}$ are generated by 
\begin{equation}
\addtocounter{theorem}{1}
\label{E:y:sym:p:eq:q}
y_\tau \circ \TTT_{2\pt}= y_\tau , \quad y_\tau  \circ \tbar_{\pt/2} = y_\tau,  \quad y_\tau  \circ \tbar_{-\pt/2}=y_\tau \quad \text{and} \quad y_{\tau}\circ \tbar = 1 - y_{\tau},
\end{equation}
and the partial-periods defined in \ref{E:part:periods}  satisfy 
\begin{equation}
\addtocounter{theorem}{1}
\label{E:part:periods:equal}
\pt^+ = \pt^- = \tfrac{1}{2}\pt \quad \text{and} \quad y_\tau(\tfrac{1}{2}\pt) = \ymin,  \quad y_\tau(-\tfrac{1}{2}\pt) = \ymax.
\end{equation}
\end{itemize}
\end{lemma}

\begin{remark}
\addtocounter{equation}{1}
\label{R:dihedral:cyl}
It follows from the partial-period relation \ref{E:part:periods}
that the reflections $\tbar_{\pt^+}$ and $\tbar_{\,-\pt^-}$
satisfy
\addtocounter{theorem}{1}
\begin{equation}
\label{E:tbar:pm:commute}
\tbar_{\,-\pt^-} \circ \tbar_{\pt^+} = \TTT_{-2\pt}, \quad \tbar_{\pt^+} \circ \tbar_{\,-\pt^-} = \TTT_{2\pt}.
\end{equation}
Hence the first symmetry of $y_\tau$ in \ref{E:y:sym:p:neq:1} is a consequence of the second and third symmetries.

Similarly, it is straightforward to check that $\tbar \circ \tbar_{\pt/2} \circ \tbar = \tbar_{-\pt/2}$. 
It follows that the two symmetries $\tbar$ and $\tbar_{\pt/2}$ are sufficient to generate all four symmetries in 
\ref{E:y:sym:p:eq:q}.
\end{remark}

\begin{remark}
\label{R:y:sym:ex}
\addtocounter{equation}{1}
For $\tau=0$, $y_{\tau}$ is no longer periodic (the period $2\pt \ra \infty$ as $\tau \ra 0$).
For $p=1$ we have already seen in \ref{P:y}.iv.b that $y_{0}$ is still even. 
For $p=q$, $y_{0}(0)$ is invariant under $y \mapsto 1-y$,  and hence $y_{0}$ retains the reflectional symmetry 
$$ y_{0} \circ \tbar = 1 - y_{0}.$$
When $\abs{\tau}=\taumax$, $y_{\tau}$ is the constant function $q/n$, as noted in Proposition \ref{P:y}.
\end{remark}

\begin{proof}[Proof of Lemma \ref{L:y:symmetry}]
Since the ODE \ref{E:y:dot} is autonomous we have time translation symmetry, i.e.  
for any solution $y$ of \ref{E:y:dot} and any $t_{0}\in \R$, $y \circ \TTT_{t_{0}}$ is also a solution of \ref{E:y:dot}.
Moreover, if $y$ is a solution of \ref{E:y:dot} then so is $y \circ \TTT$. 
Hence \ref{E:y:dot} is invariant under the whole of $\Isom(\R)$.
\ref{E:y:dot} is also invariant under $y \mapsto 1-y$ when $p=q$.

\noindent
(i) Proof of \ref{E:y:sym:p:eq:1}: The first equality is immediate since $y_\tau$ has period $2\pt$ 
by Proposition \ref{P:y}(i) and \ref{E:y:ic:p:eq:1}. 
The second symmetry follows from the fact that $y_\tau(0) = \ymax$ as in \ref{E:y:ic:p:eq:1}. 

\noindent
(ii) Proof of \ref{E:y:sym:p:neq:1}: $y_\tau$ is periodic of period $2\pt$ by Proposition \ref{P:y}(i).
Since $y_{\tau}$ has a maximum and a minimum at $-\pt^-$ and $\pt^+ $ respectively
it has the two additional reflection symmetries listed in \ref{E:y:sym:p:neq:1}. \\
(iii) We need to prove that $y_{\tau}$ admits the new symmetry $y_{\tau}\circ \tbar = 1 - y_{\tau}$.
The rest of the claims made will then follow by combining this symmetry with the ones already established in part (ii).
Define $\tilde{y}:= (1-y_{\tau})\circ \tbar$.  $\tilde{y}$ is also a solution of \ref{E:y:dot} and see from 
\ref{E:y:ic:p:neq:1} that $\tilde{y}$ satisfies the same initial conditions 
as $y_{\tau}$. 
Hence by the uniqueness of solutions of the initial value problem $y_{\tau}\equiv (1-y_{\tau}) \circ \tbar$ as required. 
It follows that 
\begin{equation}
\addtocounter{theorem}{1}
\label{E:ymin:ymax}
\ymax+\ymin =1,
\end{equation}
and that 
$ y_\tau(\pt^-) = 1 - y_\tau(-\pt^-) = 1 - \ymax = \ymin = y_\tau(\pt^+).$
Hence $\pt^- = \pt^+ = \tfrac{1}{2}\pt$.
Since $\pt^+ = \tfrac{1}{2}\pt$, the existing reflectional symmetries $y_\tau \circ \tbar_{\pt^+}=y_\tau$ 
and $y_\tau \circ \tbar_{-\pt^-}=y_\tau$ become $y_\tau \circ \tbar_{\pt/2}=y_\tau$ and 
$y_\tau \circ \tbar_{-\pt/2}=y_\tau$ respectively.
\end{proof}

\begin{corollary}
\addtocounter{equation}{1}
\label{C:y:dihedral}
The discrete subgroup $\dihedral{}$ of $\Isom(\R)$ generated by the symmetries of $y_{\tau}$ is
$$\dihedral{}=
\begin{cases}
\langle \tbar \,, \TTT_{2\pt} \rangle \quad & \text{if \ } p=1;\\
\langle \tbar_{\pt^{+}}\,, \tbar_{-\pt^{-}} \rangle \quad & \text{if\ } p>1, \, p \neq q;\\
\langle \tbar \,, \tbar_{\pt/2} \rangle \quad & \text{if\ } p>1, \, p =q.
\end{cases}
$$
In all three cases $\dihedral{} \cong \dihedral{\infty}$ the infinite dihedral group.
\end{corollary}
\begin{proof}
Recall the two standard presentations for the infinite dihedral group $\dihedral{\infty}$
$$ \langle r, f\, | \, f^{2}=1, frf=r^{-1} \rangle,$$
and
$$\langle s, t\, |\, s^2=1,\, t^2=1 \rangle.$$
The commutation relation 
\begin{equation}
\label{E:tx:tbar}
\addtocounter{theorem}{1}
\tbar \circ \TTT_{x} \circ \tbar = \TTT_{-x}
\end{equation}
together with the first presentation of 
$\dihedral{\infty}$ shows that $\dihedral{}\cong \dihedral{\infty}$ in the case $p=1$.
The commutation relations \ref{E:tbar:pm:commute} for the reflection symmetries $\tbar_{\pt^+}$
and $\tbar_{\,-\pt^-}$ together with the second presentation of $\dihedral{\infty}$ yield the result for $p>1$ and $p \neq q$.
Similarly, for $p=q$, $\dihedral{}$ is a group generated by two independent reflections $s$ and $t$ with no relation of the form 
$(st)^{k}=1$,  hence isomorphic to $\dihedral{\infty}$.
\end{proof}

\subsection*{Symmetries of $\bw_{\tau}$}\phantom{ab}
In this section we study the symmetries of $\bw_{\tau}$. Since $X_{\tau}$ is determined by $\bw_{\tau}$ these 
symmetries are intimately connected to the extrinsic geometry of $X_{\tau}$. However, from the point of view of 
more general twisted products the symmetries of the $(p,q)$-twisted SL curves $\bw_{\tau}$ are of their own interest. 
The symmetries of $\bw_{\tau}$ are themselves closely related to the symmetries of $y_{\tau}$ studied in the 
previous section.
Since by Propositions \ref{P:w:tau} and \ref{P:X:tau}.i,  $\bw_{-\tau} = \overline{\bw}_\tau$ and $X_{-\tau} = \overline{X}_\tau$, it suffices to consider the
case where $\tau \ge 0$.

If $\bw_{\tau}=(w_{1},w_{2})$, $y_{\tau}=\abs{w_{2}}^{2}$ and 
$\psi_{1}$ and $\psi_{2}$ denote the arguments of $w_{1}$ and $w_{2}$ respectively then the equations 
$$\Imag(\overline{w}_{1} \dot{w}_{1}) = - \Imag(\overline{w}_{2} \dot{w}_{2}) = 2\tau,$$
are equivalent to 
\begin{equation}
\addtocounter{theorem}{1}
\label{E:y:psi}
(1-y_{\tau}) \dot{\psi}_1 = 2\tau, \qquad y_{\tau} \dot{\psi}_2 = -2\tau.
\end{equation}

To study the symmetries of $\bw_\tau$ it is convenient to write $\bw_\tau$ in the form
\addtocounter{theorem}{1}
\begin{equation}
\label{E:w1}
w_1(t) =
\begin{cases}
\sgn{t}\,\sqrt{1-y_0(t)}, &\text{for $\tau=0$;}\\
-i \sqrt{1-y_\tau(t)} e^{i\psi_1}, &\text{for $\tau >0$;}
\end{cases}
\qquad
w_2(t) =
\begin{cases}
\sqrt{y_0(t)}, &\text{for $\tau=0$;}\\
\sqrt{y_\tau(t)} e^{i\psi_2}, &\text{for $\tau >0$;}
\end{cases}
\end{equation}
if $p=1$ and 
\addtocounter{theorem}{1}
\begin{equation}
\label{E:w1:p:neq:1}
w_1(t) =
\begin{cases}
\sqrt{1-y_0(t)}, &\text{for $\tau=0$;}\\
\sqrt{1-y_\tau(t)} e^{i\alpha_\tau/2p} e^{i\psi_1}, &\text{for $\tau >0$;}
\end{cases}
\quad
w_2(t) =
\begin{cases}
\sqrt{y_0(t)}, &\text{for $\tau=0$;}\\
\sqrt{y_\tau(t)} e^{i \alpha_\tau/2q} e^{i\psi_2}, &\text{for $\tau >0$;}
\end{cases}
\end{equation}
if $p>1$, where $\alpha_{\tau}\in [-\pi/2,\pi/2]$ was defined in \ref{E:alpha:tau} and where in both cases for $0<\tau\le \taumax$, 
$\psi_1,\  \psi_2:\R \ra \R$ are the unique solutions of \ref{E:y:psi} with initial conditions 
\begin{equation}
\addtocounter{theorem}{1}
\label{E:psi:ic}
\psi_1(0)=\psi_2(0)=0.
\end{equation}
The slightly different forms the above $w_{i}$ take in the cases $p=1$ and $p>1$ stem from the fact that we have chosen 
the initial data $\bw(0)$ for $\bw_{\tau}$ differently in these two cases (recall \ref{E:w:ic:p:neq:1} and \ref{E:w:ic:p:eq:1}).

Define the function $\Psi$ by 
\begin{equation}
\addtocounter{theorem}{1}
\label{E:psi:def}
\Psi:= p \psi_{1} + q \psi_{2}.
\end{equation}
$\Psi$ plays at important role at several points later in the paper.
Written in terms of $y$ and $\Psi$ the real and imaginary parts of equation \ref{E:ydot:cx} are equivalent to 
\addtocounter{theorem}{1}
\begin{eqnarray}
\label{E:psi:real}
\dot{y}_{\tau} &=& -2\sqrt{f(y)} \sin \Psi,\\
\addtocounter{theorem}{1}
\label{E:psi:imag}
2\tau & =& \quad \sqrt{f(y)} \cos \Psi,
\end{eqnarray}
for $p=1$ and to 
\addtocounter{theorem}{1}
\begin{eqnarray}
\label{E:psi:real:pneq1}
\dot{y}_{\tau} &=& -2\sqrt{f(y)} \cos(\Psi + \alpha_\tau),\\
\addtocounter{theorem}{1}
\label{E:psi:imag:pneq1}
-2\tau &=&\quad \sqrt{f(y)} \sin(\Psi + \alpha_\tau)
\end{eqnarray}
for $p>1$ with $\alpha_{\tau}$ as defined in \ref{E:alpha:tau}.

\begin{definition}
\addtocounter{equation}{1}
For any $\tau$ with $0<\abs{\tau}<\taumax$ we define the \emph{angular period} $\pthat$ in terms of $\psi_1$ by
\begin{equation}
\addtocounter{theorem}{1}
\label{E:pthat}
2\pthat:= p\psi_1(2\pt).
\end{equation}
\end{definition}
\smallskip

In Section \ref{S:asymptotics} we prove that the angular period $2\pthat$ is an analytic function of $\tau$ for $0<\abs{\tau}<\taumax$ 
that satisfies 
\addtocounter{theorem}{1}
\begin{equation}
\label{E:pt:hat:asymp}
\lim_{\tau \ra 0} \pthat = \frac{\pi}{2}.
\end{equation}
More precise asymptotics for $\pthat$ as $\tau \ra 0$ will be important in our subsequent gluing 
constructions and are also established in Section \ref{S:asymptotics}.

\begin{lemma}[Discrete symmetries of $\bw_{\tau}$ for $p=1$]
\addtocounter{equation}{1}
\label{L:y:w:symmetry}
For $p=1$, $q=n-1$ and $0<\tau<\taumax$ the angular period $\pthat$ defined in \ref{E:pthat} satisfies
\addtocounter{theorem}{1}
\begin{equation}
\label{E:psi:period:p:eq:1}
2\pthat:= \psi_1(2\pt) = 2\,\psi_1(\pt) = -2(n-1)\,\psi_2(\pt) = -(n-1)\,\psi_2(2\pt).
\end{equation}
$\bw_\tau$ has the following symmetries:
\begin{equation}
\addtocounter{theorem}{1}
\label{E:w:sym:p:eq:1}
\bw_\tau \circ \TTT_{2\pt} = \tcheck_{2\pthat} \circ \bw_\tau, \quad 
\bw_\tau \circ \tbar = \tbarcheck \circ \bw_{\tau}, \quad 
\bw_\tau \circ \tbar_{\pt}  = \tcheck_{2\pthat} \circ \tbarcheck  \circ \bw_\tau,
\end{equation}
where $\tcheck_x \in \unit{2}$ was defined in \ref{E:M:defn} and 
$\tbarcheck \in \orth{4}$ is defined by 
$$\tbarcheck(w_{1},w_{2}) = (-\overline{w}_1,\overline{w}_2).$$
\end{lemma}
\noindent
Using the fact that $\overline{\bw}_{\tau} = -\bw_{\tau}$ the symmetries of $\bw_{\tau}$ for $\tau<0$ can be inferred 
immediately from the symmetries in the case $\tau>0$.

We have the following analogue of Lemma \ref{L:y:w:symmetry} for $p>1$.
\begin{lemma}[Discrete symmetries of $\bw_{\tau}$ for $p> 1$]
\addtocounter{equation}{1}
\label{L:y:w:symmetry:p:neq:1}
Fix a pair of admissible integers $p$ and $q$ with $p>1$, then for $0<\tau<\taumax$, the angular period $\pthat$ satisfies
\addtocounter{theorem}{1}
\begin{equation}
\label{E:psi:period:p:neq:1}
2 \pthat:= p\psi_1(2\pt)  = 2p (\psi_1(\pt^+) - \psi_1(-\pt^-)) = -2q (\psi_2(\pt^+) - \psi_2(-\pt^-)) =q \psi_2(2\pt).
\end{equation}
$\bw_\tau$ has the following symmetries:
\begin{equation}
\addtocounter{theorem}{1}
\label{E:w:sym:p:neq:1}
\bw_\tau \circ \TTT_{2\pt} = \tcheck_{2\pthat} \circ \bw_\tau, \quad 
\bw_\tau \circ \tbar_{\pt^{+}} = \tbarcheck_{+} \circ \bw_{\tau}, \quad 
\bw_\tau \circ \tbar_{-\pt^{-}}  = \tbarcheck_{-} \circ \bw_{\tau}
\end{equation}
where $\tcheck_x \in \unit{2}$ was defined in \ref{E:M:defn} and $\tbarcheck_{+}, \tbarcheck_{-} \in \orth{4}$
are defined by 
\begin{eqnarray*}
\tbarcheck_{+}(w_{1},w_{2}) &= (e^{i\alpha_{\tau}/p} e^{i\psi_{1}(2\pt^{+})} \overline{w}_{1},
e^{i\alpha_{\tau}/q} e^{i\psi_{2}(2\pt^{+})} \overline{w}_{2}),\\
\tbarcheck_{-}(w_{1},w_{2}) &= (e^{i\alpha_{\tau}/p} e^{i\psi_{1}(-2\pt^{-})} \overline{w}_{1},
e^{i\alpha_{\tau}/q} e^{i\psi_{2}(-2\pt^{-})} \overline{w}_{2}).
\end{eqnarray*}

When $p=q$, $\bw_\tau$ has the following extra symmetry:
\begin{equation}
\addtocounter{theorem}{1}
\label{E:w:ex:sym}
w_1 \circ \tbar = w_2 \quad \text{and} \quad w_2 \circ \tbar = w_1.
\end{equation}
Hence $\psi_1$ and $\psi_2$ have the following additional symmetries:
\begin{equation}
\addtocounter{theorem}{1}
\label{E:psi:ex:sym}
\psi_1 \circ \tbar = \psi_2, \quad \psi_2 \circ \tbar = \psi_1, \quad
\psi_1 \circ \TTT_{\pt} = - \psi_2 + \psi_1(\pt), \quad \psi_2 \circ \TTT_{\pt} = -\psi_1 + \psi_2(\pt).
\end{equation}
The angular period $\pthat$  satisfies
\addtocounter{theorem}{1}
\begin{equation}
\label{E:psi:period:p:eq:q}
2 \pthat:= p\psi_1(2\pt)  = p (\psi_1(\pt) - \psi_1(-\pt)) = -p \psi_2(2\pt).
\end{equation}
\end{lemma}


The proofs of Lemmas \ref{L:y:w:symmetry} and \ref{L:y:w:symmetry:p:neq:1} are very similar. 
First, we establish symmetries of $\psi_{i}$  using the symmetries of $y_{\tau}$ from Lemma \ref{L:y:symmetry}
together with the definitions of $\psi_{i}$ in terms of $y_{\tau}$ (recall \ref{E:y:psi}).
The symmetries of $\bw_{\tau}$ then follow by using the definition of $\bw_{\tau}$ in terms of $y_{\tau}$ 
and $\psi_{i}$ and their symmetries. For completeness, we give details in each case.

\begin{proof}[Proof of Lemma \ref{L:y:w:symmetry}]
Proof of \ref{E:psi:period:p:eq:1} and \ref{E:w:sym:p:eq:1}: 
The discrete symmetries of $y_{\tau}$ given in \ref{E:y:sym:p:eq:1} and the definition of $\psi_i$ in terms
of $y_{\tau}$ given in \ref{E:y:psi} imply the following symmetries for
$\psi_i\  (i=1,2)$
\addtocounter{theorem}{1}
\begin{equation}
\label{E:psi:sym:p:eq:1}
\psi_i \circ \TTT_{2\pt} = \psi_i + \psi_i(2\pt), \quad \psi_i \circ \tbar
= -\psi_i, \quad \psi_i \circ \tbar_{\pt} = -\psi_i + \psi_i (2\pt).
\end{equation}
\no
Proof of \ref{E:psi:period:p:eq:1}:  $\psi_i(2\pt) = 2\psi_i(\pt)$ for $i=1,2$ follows from
the third symmetry of \ref{E:psi:sym:p:eq:1} when $t=\pt$.
It remains to prove that $ \Psi(\pt)=\psi_1(\pt) + (n-1) \psi_2(\pt)=0$.
Since $\Psi(0)=0$ and $\sqrt{f(y)(t)}$ is continuous in $t$ and positive,
\ref{E:psi:imag} implies that $\Psi(t) \in (-\tfrac{\pi}{2},\tfrac{\pi}{2})$ for all $t\in \R$.
Then since $\dot{y}(\pt)=0$, from $\ref{E:y:dot}$ it follows that
$\sqrt{f(y)(\pt)}=2\tau$ and hence from \ref{E:psi:imag} that $\cos(\Psi)(\pt)=1$ as required.

The symmetries of $\psi_i$ given in \ref{E:psi:sym:p:eq:1}, 
together with the fact that $\Psi(2\pt)=2\Psi(\pt)=0$, imply the following
simpler symmetries for $\Psi$
\addtocounter{theorem}{1}
\begin{equation}
\label{E:Psi:sym:p:eq:1}
\Psi \circ \TTT_{2\pt} = \Psi, \quad \Psi \circ \tbar
= -\Psi, \quad \Psi \circ \tbar_{\pt} = -\Psi.
\end{equation}
In other words (unlike $\psi_1$ or $\psi_2$ individually), $\Psi$ is an odd periodic function of
$t$ of period $2\pt$.

\noindent
Proof of \ref{E:w:sym:p:eq:1}: 
The symmetries of $\bw_\tau$ claimed in \ref{E:w:sym:p:eq:1} follow from \ref{E:y:sym:p:eq:1}, \ref{E:psi:period:p:eq:1}
and \ref{E:psi:sym:p:eq:1} and the expression \ref{E:w1} for $\bw_\tau$ in terms of $y_\tau$, $\psi_1$ and $\psi_2$.
\end{proof}

\begin{proof}[Proof of Lemma \ref{L:y:w:symmetry:p:neq:1}]
Symmetries of $\psi_i$: The symmetries of $y_\tau$ given in \ref{E:y:sym:p:neq:1} and the definition of $\psi_i$ in terms
of $y_\tau$ given in \ref{E:y:psi} imply the following symmetries for
$\psi_i\  (i=1,2)$
\addtocounter{theorem}{1}
\begin{equation}
\label{E:psi:sym:p:neq:1}
\psi_i \circ \tbar_{\pt^+} = -\psi_i + \psi_i (2\pt^+), \quad \psi_i \circ \tbar_{\,-\pt^-} = -\psi_i + \psi_i (-2\pt^-), \quad
\psi_i \circ \TTT_{2\pt} = \psi_i + \psi_i(2\pt).
\end{equation}
\noindent
Proof of \ref{E:psi:period:p:neq:1}: the first two symmetries of $\psi_i$ in \ref{E:psi:sym:p:neq:1} imply that $\psi_i(2\pt^+) = 2\psi_i(\pt^+)$ and
$\psi_i(-2\pt^-) = 2\psi_i(-\pt^-)$.
The third symmetry of \ref{E:psi:sym:p:neq:1} with $t=-2\pt^-$ implies that
$$2\pthat = p\psi_1(2\pt) = p(\psi_1(2\pt^+) - \psi_1(-2\pt^-)) = 2p (\psi_1(\pt^+) - \psi_1(-\pt^-)).$$
It remains to prove the last equality of \ref{E:psi:period:p:neq:1}. By the equalities on the previous line 
it suffices to prove that $\Psi(2\pt)=p\psi_{1}(2\pt) + q\psi_{2}(2\pt)=0$.
%
Since $\tau>0$, $\alpha_\tau \in [-\tfrac{\pi}{2},0)$. Now since
$\Psi(0)=0$ and $\sqrt{f(y)(t)}$ is continuous in $t$ and positive,
\ref{E:psi:imag:pneq1} implies that $\Psi(t) + \alpha_\tau \in (-\pi,0)$ holds for all $t\in \R$.
At $t=2\pt$, we have $f(y)=\fmax = 4\taumax^2$ and $\dot{y} = -4\taumax \cos{\alpha_\tau}$.
Hence \ref{E:psi:real:pneq1} and \ref{E:psi:imag:pneq1} imply that $e^{i(\Psi+\alpha_\tau)} = e^{i\alpha_\tau}$ holds
at $t=2\pt$. Hence $\Psi(2\pt)=0$, since $\Psi+\alpha_\tau \in (-\pi,0)$.

\noindent
Proof of \ref{E:w:sym:p:neq:1}: The symmetries of $\bw_\tau$ claimed in \ref{E:w:sym:p:neq:1} follow from \ref{E:y:sym:p:neq:1}, \ref{E:psi:period:p:neq:1}
and \ref{E:psi:sym:p:neq:1} and the expression \ref{E:w1:p:neq:1} for $\bw_\tau$ in terms of $y_\tau$, $\psi_1$ and $\psi_2$.

\smallskip

\noindent
\emph{Extra symmetries for case $p=q$:} \phantom{ab}
Define $\boldsymbol{z}:\R \ra \Sph^3$ by $\boldsymbol{z}=(w_2 \circ \tbar,w_1 \circ \tbar)$.
Since $p=q$, we see that $\boldsymbol{z}$ also satisfies \ref{E:odes:p:n}.
Moreover, since $p=q$ the initial data $\bw_\tau(0)$ (recall \ref{E:w:ic:p:neq:1}) 
is invariant under exchange of $w_1$ and $w_2$, and therefore
$\boldsymbol{z}(0)=\bw_\tau(0)$. Hence by uniqueness 
of the initial value problem
$\boldsymbol{z}$ coincides with $\bw_\tau$ as claimed . 

The first two symmetries of \ref{E:psi:ex:sym} follow from \ref{E:w:ex:sym} and the relation between $w_i$ and $\psi_i$, given in \ref{E:w1:p:neq:1}.
The final two symmetries of \ref{E:psi:ex:sym} follow from the first two and the existing symmetry
$ \psi_i \circ \tbar_{\pt/2} = - \psi_i + \psi_i(\pt)$ for $i=1, 2$ (obtained from \ref{E:psi:sym:p:neq:1} using the fact that $\pt^+ = \tfrac{1}{2}\pt$).
\end{proof}

\subsection*{Periods and half-periods of $\bw_{\tau}$}\phantom{ab}
To understand the extrinsic geometry of $X_{\tau}$ and in particular when $X_{\tau}$ factors through a 
closed embedding we need to understand when the $(p,q)$-twisted SL curves $\bw_{\tau}$ form closed curves in $\Sph^{3}$. 
In fact, as we remarked earlier to understand when $X_{\tau}$ closes up we need to understand when $\bw_{\tau}$
gives rise to a closed curve in the space of isotropic $\sorth{p}\times\sorth{q}$ orbits.  As described in Lemma \ref{L:iso:orbits} this orbit space is $\Sph^{3}/\stabpq$ where $\stabpq \subset \unit{2}$ is the finite subgroup defined in \ref{E:stabpq}.

To this end we define the periods and half-periods of $\bw_{\tau}$.
The periods and half-periods of $\bw_{\tau}$ control when the curve of isotropic orbits $\mathcal{O}_{\bw_{\tau}}$ 
determined by $\bw_{\tau}$ is a closed curve in the space of $\sorth{p}\times \sorth{q}$ orbits.
Recall from \ref{D:M:period} the definitions of the periods  and half-periods  
of the $1$-parameter group $\{\tcheck_x\}$ defined in \ref{E:M:defn}.
The periods and half-periods of $\bw_{\tau}$ and the periods and half-periods of $\{\tcheck_x\}$ are 
intimately connected because of the first discrete symmetry of $\bw_{\tau}$ listed in \ref{E:w:sym:p:eq:1} and 
\ref{E:w:sym:p:neq:1} (for $p=1$ and $p>1$ respectively).
\addtocounter{equation}{1}
\begin{definition}
\label{D:w:period}
Fix a pair of admissible integers $p$ and $q$ and let $\bw_{\tau}$ be any of the $(p,q)$-twisted SL curves defined in \ref{P:w:tau}.
We define \emph{the period lattice of $\bw_\tau$} by
\addtocounter{theorem}{1}
\begin{equation}
\label{E:w:period}
\Per(\bw_\tau):= \{ x\in \R\, | \,\bw_{\tau} \circ \TTT_{x}=\bw_{\tau}\},
\end{equation}
and the \emph{half-period lattice of $\bw_{\tau}$} by 
\addtocounter{theorem}{1}
\begin{equation}
\label{E:w:half:period}
\Perh(\bw_\tau):= \{ x\in \R\, | \,\mathcal{O}_{\bw_{\tau}\circ \TTT_{x}(t)}=\mathcal{O}_{\bw_{\tau}(t)}\ \forall \,t\in \R\},
\end{equation}
where as previously $\mathcal{O}_{\bw} \subset \Sph^{2(p+q)-1}$ denotes the isotropic $\sorth{p} \times \sorth{q}$ orbit 
associated with any point $\bw \in \Sph^{3}$.
In other words, $x$ is a half-period of $\bw_{\tau}$ if $\bw_{\tau} \circ \TTT_{x}$ and $\bw_{\tau}$ 
give rise to the same parametrised curve of isotropic $\sorth{p} \times \sorth{q}$-orbits in $\Sph^{2(p+q)-1}$.
We call elements of $\Perh(\bw_{\tau})$ the \emph{half-periods of $\bw_{\tau}$},  
and elements of $\Per(\bw_{\tau})$ the \emph{periods of $\bw_{\tau}$}. 
A \emph{strict half-period} is  any half-period which is not a period of $\bw_{\tau}$.
\end{definition}
Using \ref{L:iso:orbits} we see that $x$ is a half-period of $\bw_{\tau}$ if and only if 
\begin{equation}
\addtocounter{theorem}{1}
\label{E:w:period:alt}
\bw_{\tau} \circ \TTT_{x} = \rho_{jk} \circ \bw_{\tau} \quad \text{for some\ } \rho_{jk}\in \stabpq,
\end{equation}
where as above $\stabpq$ is the finite subgroup of $\unit{2}$ defined in \ref{E:stabpq}.
More explicitly, we have 
\addtocounter{theorem}{1}
\begin{equation}
\label{E:w:half:period:eq:1}
\Perh(\bw_\tau):= \{ x\in \R\, |\  \exists \ (j,k) \in \langle (+,\pm)\rangle \le\Z_2 \times \Z_2 \ \text{such that}\
\rho_{jk} \circ \bw_\tau = \bw_\tau \circ \TTT_x\,\}, \quad \text{if $p=1$;}
\end{equation}
or 
\addtocounter{theorem}{1}
\begin{equation}
\label{E:w:half:period:gt:1}
\Perh(\bw_\tau):= \{ x\in \R\, |\  \exists \ (j,k) \in  \Z_2 \times \Z_2 \ \text{such that}\
\rho_{jk} \circ \bw_\tau = \bw_\tau \circ \TTT_x\,\}, \quad \text{if $p>1$.}
\end{equation}
If $x$ satisfies \ref{E:w:period:alt} for $(j,k) \in \Z_{2}\times \Z_{2}$ then we call $x$ a \emph{half-period of $\bw_{\tau}$ of type $(jk)$}.
We see immediately from  \ref{E:w:period:alt} that $2 \Perh(\bw_{\tau}) \subset \Per(\bw_{\tau})$; 
this explains the terminology half-period.

The importance of the half-periods of $\bw_{\tau}$ is explained by the following
\addtocounter{equation}{1}
\begin{prop}
\label{P:xtau:embed}
Suppose $0<\abs{\tau}<\taumax$ and let  $X_{\tau}$ be one of the $\sorth{p}\times\sorth{q}$-invariant special 
Legendrian cylinders defined in \ref{D:X:tau}. Suppose there exist triples $(t_1,\sigma_1,\sigma_2), (t_2,\sigma_1',\sigma_2')
\in \cylpq$ such that
\addtocounter{theorem}{1}
\begin{equation}
\label{E:self:intersect}
X_{\tau}(t_1, \sigma_1, \sigma_2) = X_{\tau}(t_2, \sigma_1', \sigma_2').
\end{equation}
Then $t_2 - t_1 \in \Perh(\bw_\tau)$.
Moreover, if $t_2-t_1 \in
\Per(\bw_\tau)$ then $\sigma_1=\sigma_1'$ and $\sigma_2=\sigma_2'$. 
\end{prop}
\begin{proof}
From the definition of $X_{\tau}$ in terms of $\bw_{\tau}$ and the isotropic $\sorth{p} \times \sorth{q}$ 
orbits $\mathcal{O}_{\bw}$ 
we see that \ref{E:self:intersect} implies that $\mathcal{O}_{\bw_{\tau}(t_{1})} \cap \mathcal{O}_{\bw_{\tau}(t_{2})} \neq \emptyset$ 
and therefore  $\mathcal{O}_{\bw_{\tau}(t_{1})}=\mathcal{O}_{\bw_{\tau}(t_{2})}$.
Hence by \ref{L:iso:orbits} we have 
\begin{equation}
\addtocounter{theorem}{1}
\label{E:intersect:w}
\bw_{\tau}(t_{1}) = \rho_{jk} \bw_{\tau}(t_{2}) \quad \text{for some\ } \rho_{jk}\in \stabpq,
\end{equation}
and 
\begin{equation}
\addtocounter{theorem}{1}
\label{E:alpha:beta}
\sigma_1=(-1)^{j}\sigma_2, \quad \sigma_2=(-1)^{k}\sigma_2'.
\end{equation}
Using conservation of $\mathcal{I}_{2}=\Imag(w_{1}^{p}w_{2}^{q})$ and \ref{E:intersect:w} we have
$$\Imag w_{1}^{p}w_{2}^{q}(t_{2}) = \Imag w_{1}^{p}w_{2}^{q}(t_{1})= (-1)^{jp+kq}\Imag w_{1}^{p}w_{2}^{q}(t_{2}).$$
Hence we have 
\begin{equation}
\addtocounter{theorem}{1}
\label{E:jk:admissible}
jp+kq \equiv 0 \mod{2}.
\end{equation}
Now define $\tilde{\bw}$ by 
$$\tilde{\bw}:= \rho_{jk} \circ \bw_{\tau} \circ \TTT_{t_{2}-t_{1}}.$$
Using the definition of $\tilde{\bw}$ and \ref{E:intersect:w} we have
$$\tilde{\bw}(t_{1}) = \rho_{jk} \circ \bw_{\tau}(t_{2})=\bw_{\tau}(t_{1}).$$
Because $j$ and $k$ satisfy \ref{E:jk:admissible} $\tilde{\bw}$ is another solution of \ref{E:odes:p:n} 
and therefore by uniqueness of the initial value problem $\tilde{\bw}\equiv \bw_{\tau}$.
It follows that $t_{2}-t_{1} \in \Perh(\bw_{\tau})$.
The final statement in \ref{P:xtau:embed} follows from \ref{E:alpha:beta}.
\end{proof}

For completeness here is the analogue of \ref{P:xtau:embed} for the case $\tau=\taumax$.
\addtocounter{equation}{1}
\begin{lemma}
\label{P:embedded:taumax}
Let $X_\tau: \cylpq  \ra \Sph^{2(p+q)-1}$
be the $\sorth{p}\times \sorth{q}$-equivariant special Legendrian immersion
defined in \ref{D:X:tau}, with $\tau=\taumax$.
Then there exist a pair of triples $(t_1,\sigma_1,\sigma_2), (t_2,\sigma_1',\sigma_2')
\in \R \times \Sph^{p-1} \times \Sph^{q-1}$ such that
\addtocounter{theorem}{1}
\begin{equation}
\label{E:self:intersect:taumax}
X_{\tau}(t_1, \sigma_1, \sigma_2) = X_{\tau}(t_2, \sigma_1', \sigma_2').
\end{equation}
if and only if
$$ t_2-t_1 = \frac{\lcm(p,q)\pi}{n\taumax}l, \quad \sigma_1 = (-1)^{jl}\sigma_1',
\quad \sigma_2 = (-1)^{kl}\sigma_2', \quad \text{for any}\ l\in \Z,$$
where $j=q/\hcf(p,q)$ and $k=p/\hcf(p,q)$.
$$\Per(X_\tau) = \langle \TTT_x \circ  ((-1)^j\Id_{\Sph^{p-1}}, (-1)^k\Id_{\Sph^{q-1}}) \rangle
\quad \text{where} \ x={\frac{\lcm(p,q) \pi}{n\taumax}}.$$
\end{lemma}
\begin{proof}
This is a straightforward computation using the
explicit expression for $\bw_{\tau}$ (see \ref{E:X:taumax}.v).
\end{proof}


As a simple corollary of \ref{P:xtau:embed} we have
\addtocounter{equation}{1}
\begin{corollary}
\label{C:wtau:closed}
Suppose there exist $t_{0} \in \R$ and $x_{0}\in \R^{+}$ such that $\bw_{\tau}(t_{0}+x_{0})=\bw_{\tau}(t_{0})$, \textit{i.e.} the curve $\bw_{\tau}$ has a point of self-intersection, then $x_0 \in \Per(\bw_{\tau})$.
Hence either 
\begin{itemize}
\item[(i)] $\Per(\bw_{\tau})=(0)$ in which case $\bw_{\tau}:\R \ra \Sph^3$ is an injective immersion, or
\item[(ii)] there exists $T>0$, such that $T\in \Per(\bw_{\tau})$ is the smallest nontrivial
period of $\bw_{\tau}$ and the restriction $\bw_{\tau}:[0,T] \ra \Sph^3$ is a closed embedded curve.
\end{itemize}
In particular, $\bw_{\tau}$ forms a closed curve in $\Sph^{3}$ if and only if $\Per(\bw_{\tau})=0$.
\end{corollary}

We now completely determine the periods and half-periods of $\bw_{\tau}$.
We see from \ref{E:w:sym:p:eq:1} and \ref{E:w:sym:p:neq:1} that 
$\tcheck_{2\pthat} \in \unit{2}$ ($\tcheck_{x}\in \unit{2}$ is defined in \ref{E:M:defn} and 
$\pthat$ is defined by \ref{E:pthat})
plays a fundamental role in the geometry of $\bw_{\tau}$.
We call $\tcheck_{2\pthat}$ the \emph{rotational period of $\bw_{\tau}$}, since by \ref{E:w:sym:p:eq:1} and \ref{E:w:sym:p:neq:1} $\tcheck_{2\pthat}$
 controls how $\bw_{\tau}$ gets rotated as we move from one domain of periodicity of $y_{\tau}$ to the next.
This motivates the following definition.

\begin{definition}
\label{D:rp:order}
\addtocounter{equation}{1}
Define $k_{0}\in \N \cup \{+\infty\}$ to be the order of the rotational period $\tcheck_{2\pthat} \in \unit{2}$. 
We set $k_{0}= +\infty$ if the rotational period has infinite order.
\end{definition}
\noindent
We have the following simple result relating the periods and half-periods of $\bw_{\tau}$ 
to the periods and half-periods of $\{\tcheck_x\}$.
\begin{lemma}
\label{L:w:half:period:a}
\addtocounter{equation}{1}
For $0<\abs{\tau}<\taumax$, we have 
\begin{itemize}
\item[(i)]
$x\in \Perh(\bw_{\tau}) \iff x=2k\pt$ for some $k \in \Z$ and $2k\pthat \in \Perh(\{\tcheck_x\})$.
\item[(ii)]
$x\in \Per(\bw_{\tau}) \iff x=2k\pt$ for some $k \in \Z$ and $2k\pthat \in \Per(\{\tcheck_x\}) = 2\pi\lcm(p,q)$.
\end{itemize}
If the rotational period $k_{0}$ of $\tcheck_{2\pthat}$ defined in \ref{D:rp:order} has infinite order then
 $\Per(\bw_{\tau})=\Perh(\bw_{\tau})=(0)$, otherwise
\begin{equation}
\addtocounter{theorem}{1}
\label{E:w:period:lattice}
\Per(\bw_{\tau}) = 2k_{0}\pt \Z.
\end{equation}
$k_{0}$ the order of the rotational period $\tcheck_{2\pthat}$ defined in \ref{D:rp:order} can also be characterised as 
\begin{equation}
\addtocounter{theorem}{1}
\label{E:k0:alt}
k_0= \min\{k \in \Z^+ | \ k\pthat \in \pi \lcm(p,q) \Z\}.
\end{equation}

\end{lemma}

\begin{proof}
\emph{(i)}. Suppose $x\in \Perh(\bw_{\tau})$.
From the definition of $\Perh(\bw_{\tau})$, $w_{2} \circ \TTT_{x}= \pm w_{2}$.
Since $y_{\tau}=\abs{w_{2}}^{2}$ this implies $y_{\tau} \circ \TTT_{x}=y_{\tau}$ 
and hence $x \in \Per(y_{\tau})=2\pt\Z$.
Then from \ref{E:w:sym:p:eq:1} or \ref{E:w:sym:p:neq:1} (according to whether $p=1$ or $p\neq 1$) 
we have 
$$ \bw_{\tau} \circ \TTT_{2k\pt} = \tcheck_{2k\pthat} \circ \bw_{\tau}.$$
Hence $2k\pt$ is a half-period of $\bw_{\tau}$ of type $(jk)$ if and only if $\tcheck_{2k\pthat} = \rho_{jk}$.
This is equivalent to $2k\pthat$ being a half-period of $\{\tcheck_x\}$ of type $(jk)$ and (i) now follows using \ref{L:M:periods}.

(ii) follows from (i) by looking only at half-periods of type $(++)$ and using \ref{L:M:periods}. 
The structure of $\Per(\bw_{\tau})$ claimed follows from 
(ii).
\end{proof}
\begin{remark}
\addtocounter{equation}{1}
For $\tau=0$, $X_{0}$ is an embedding whose image is contained in the standard totally real equatorial sphere 
$\Sph^{n-1}\subset \R^{n} \subset \C^{n}$. In this case $\Per(\bw_{\tau}) = \Perh(\bw_{\tau}) = (0).$
For $\tau = \taumax$, we leave it as an elementary exercise for the reader to use the explicit expression 
given in \ref{E:X:taumax} to write down the period and half-period lattices of $\bw_{\tau}$ in this case (see also Proposition \ref{P:embedded:taumax}).
\end{remark}
For the rest of this section we always assume $0<\abs{\tau}<\taumax$ unless stated otherwise.
We can completely describe the half-period lattice $\Perh(\bw_\tau)$ as follows:
\addtocounter{equation}{1}
\begin{lemma}
Fix a pair of admissible integers $p$ and $q$ and let $n=p+q$. Then 
\label{L:w:half:period}
$$\pthat \notin \pi \Q \iff k_{0}= \infty \iff \Perh(\bw_{\tau})=\Per(\bw_{\tau})=(0).$$
If $\pthat \in \pi \Q$, then
\begin{itemize}
\item[(i)] If $k_0$ is odd then $\Perh(\bw_{\tau}) = \Per(\bw_{\tau}) = 2k_0 \pt \Z$, \textit{i.e.} $\bw_{\tau}$ has no strict half-periods.
\item[(ii)] If $k_0$ is even and $p>1$ then  $\Perh(\bw_{\tau}) = \frac{1}{2} \Per(\bw_{\tau}) = k_0 \pt\Z$. Moreover, 
for fixed $p$ and $q$ every strict half-period of $\bw_{\tau}$ is of type $(jk)$ where  $j = q/\hcf(p,q) \mod{2}$ and $k=p/\hcf(p,q) \mod{2}$.
\item[(iii)] a. If $k_{0}$ is even,  $p=1$ and $n$ is even then $\Perh(\bw_{\tau})=\Per(\bw_{\tau}) = 2k_0 \pt \Z$, \textit{i.e.} $\bw_{\tau}$ has no strict half-periods.\\
b. If $k_{0}$ is even,  $p=1$ and $n$ is odd then $\Perh(\bw_{\tau}) = \frac{1}{2} \Per(\bw_{\tau}) = k_0 \pt\Z$ (and 
 every strict-half period is necessarily of type $(+-)$.)
\end{itemize}
\end{lemma}
\begin{proof}
The equivalences in the first line of the statement follow from the characterisation of $k_{0}$ given in \ref{E:k0:alt} 
together with \ref{E:w:period}.
Now suppose $x\in \Perh(\bw_{\tau})$ and $\pthat \in \pi \Q$, so that the rotational period $k_{0}$ is finite.
Then from \ref{L:w:half:period:a} we have
\begin{equation}
\addtocounter{theorem}{1}
\label{E:perh}
x\in 2\pt \Z \cap \tfrac{1}{2}\Per(\bw_{\tau}) = 2\pt \Z \cap k_{0}\pt\Z = \lcm(2,k_{0})\pt \Z = 
\begin{cases}
2k_{0}\pt \Z & k_{0} \text{\ odd};\\
\ k_{0}\pt \Z & k_{0} \text{\ even}.
\end{cases}
\end{equation}
(i) If $k_0$ is odd then from \ref{E:perh} $x\in 2k_0 \pt \Z = \Per(\bw_{\tau})$ and hence $\Perh(\bw_{\tau}) = \Per(\bw_{\tau})$ as required. 

If $k_0$ is even, then from \ref{E:perh} $x \in k_0 \pt \Z$. Furthermore, if $x$ is a strict half-period of $\bw_{\tau}$ then 
$x\in k_0 \pt (2\Z+1)$.

\noindent
(ii) Suppose now that $p>1$ and hence by \ref{E:w:half:period:gt:1} we should consider all types of half-period.
Given any $x\in  k_0 \pt (2\Z+1)$ notice that $\bw_{\tau} \circ \TTT_{x} = \bw_{\tau} \circ \TTT_{k_0\pt}$ 
since $2k_0 \pt \Z = \Per(\bw_{\tau})$. Since $k_{0}$ is assumed even, $k_{0}\pt \in \Per(y_{\tau})$
and hence  $\bw_{\tau} \circ \TTT_{k_0\pt} = \tcheck_{k_{0}\pthat} \circ \bw_{\tau}  
$.
By \ref{L:w:half:period} and the definition of $k_{0}$,  
$\tcheck_{k_{0}\pthat}\neq \Id$ but $\tcheck_{2k_{0}\pthat}=\Id$. 
Hence from the diagonal form of $\tcheck_{x}$
 we must have $\tcheck_{k_{0}\pthat} = \rho_{jk} \neq \Id$ for some 
$(jk)\neq (++)$. Hence $x$ is a strict half-period as claimed. Moreover, since 
$k_{0} \pthat$ is a strict half-period of $\{\tcheck_x\}$ then by \ref{L:M:periods} it must 
be a half-period of type $(jk)$ with $j$ and $k$ as in  \ref{L:w:half:period}.ii.

\noindent
(iii) If $p=1$ the result follows using the structure of $\Perh(\{\tcheck_x\})$ established 
in \ref{L:M:periods}.iii.
%
\end{proof}
By combining Proposition \ref{P:xtau:embed} with our results on $\Per(\bw_{\tau})$ and $\Perh(\bw_{\tau})$ in 
\ref{L:w:half:period:a} and \ref{L:w:half:period} we have a complete understanding of the self-intersection points 
of $X_{\tau}$.

\section{Discrete symmetries of $X_{\tau}$}
\label{S:sym:xtau}
In addition to its intrinsic interest, 
the full group of symmetries of our $\sorth{p} \times \sorth{q}$-equivariant building blocks $X_\tau$
plays a fundamental role in our subsequent gluing constructions \cite{haskins:kapouleas:hd2,
haskins:kapouleas:hd3}. These additional discrete symmetries that the $X_\tau$ possess
allow us to impose certain symmetries throughout our entire gluing construction (see the 
discussion in our survey paper \cite{haskins:kapouleas:survey}.)
The imposition of these symmetries simplifies some aspects of the gluing construction. 

\subsection*{General features of the symmetries of $X_\tau$} \phantom{ab}
Fix a pair of admissible integers $p$ and $q$ and set $n=p+q$.
We define a \emph{symmetry of $X_{\tau}$} to be a pair $(\mtilde,\MMM)\in \orth{2n} \times \Diff(\cylpq)$ 
such that 
\begin{equation}
\addtocounter{theorem}{1}
\label{E:xtau:sym:def}
\mtilde \circ X_{\tau} = X_{\tau} \circ \MMM,
\end{equation}
where $\Diff(\cylpq)$ denotes the group of diffeomorphisms of the domain of $X_\tau$.
The set of all symmetries of $X_{\tau}$ forms a group with the obvious multiplication. 

Recall from Appendix \ref{A:isom:sl} the definition and structure 
of the groups $\Isomsl \subset \orth{2n}$ and $\Isomslpm \subset \orth{2n}$, the groups of all special 
Lagrangian and  $\pm$-special Lagrangian isometries of $\C^{n}$ respectively. 
$\pm$-special Lagrangian isometries are the natural class of symmetries of special Legendrian immersions in $\Sph^{2n-1}$ 
in the following sense: if $(\mtilde,\MMM)$ is a symmetry of a special Legendrian immersion $X$ then we expect 
$\mtilde$ to be a special Lagrangian isometry if $\MMM$ is orientation-preserving and $\mtilde$ to be an 
anti-special Lagrangian isometry if $\MMM$ is orientation-reversing. More precisely, we will see later in this section that for
any symmetry $(\mtilde,\MMM)$ of $X_{\tau}$ then $\mtilde \in \Isomslpm$ and moreover $\mtilde \in \Isomsl$ if and only 
if $\MMM$ is orientation preserving.

Rather than thinking of the symmetries of $X_{\tau}$ as a subgroup of $\orth{2n} \times \Diff(\cylpq)$ we prefer to work with
 subgroups of the domain or target separately. 
To this end we define a subgroup of $\Diff(\cylpq)$
\addtocounter{theorem}{1}
\begin{equation}
\label{E:sym:xtau:def}
\Sym(X_\tau):= \{\MMM \in \Diff(\cylpq)\, |\
\exists\, \mtilde\in \orth{2n} \text{\ s.t.\ } \mtilde \circ X_\tau = X_\tau \circ \MMM\}.
\end{equation}

We define the subgroup $\Per(X_\tau) \subset \Sym(X_\tau)$ by
\begin{equation}
\addtocounter{theorem}{1}
\label{E:per:xtau:def}
\Per(X_\tau):= \{ \MMM \in \Diff(\cylpq) \, |\ X_\tau \circ \MMM=X_\tau\}.
\end{equation}
It follows that if $\MMM \in \Per(X_{\tau})$ then $\MMM$ must be orientation-preserving. 
We also define a subgroup $\widetilde{\Sym}(X_\tau) \subset \Isom(\Sph^{2n-1})=\orth{2n}$ by
\begin{equation}
\addtocounter{theorem}{1}
\label{E:symtilde:xtau:def}
\Symtilde(X_\tau):= \{ \mtilde \in \orth{2n}\ | \ \mtilde \circ X_\tau =
X_\tau \circ \MMM \text{\ for some\ } \MMM \in \Sym(X_\tau)\}.
\end{equation}
As already mentioned above we will see later in the section that $\Symtilde(X_{\tau}) \subset \Isomslpm$. 
In particular, since $\Isomslpm$ has four connected components then we can consider various subgroups 
of $\Symtilde(X_{\tau})$ namely: (i) $\Symtilde(X_{\tau}) \cap \sunit{n}$, (ii) 
$\Symtilde(X_{\tau}) \cap \IsomslpmJ = \Symtilde(X_{\tau}) \cap \unit{n}$ or 
(iii) $\Symtilde(X_{\tau}) \cap \Isomsl$. 

The three groups $\Sym(X_\tau)$, $\Per(X_\tau)$ and $\Symtilde(X_\tau)$ are related by the following
\begin{lemma}
\addtocounter{equation}{1}
\label{L:sym:subgroups}
For $\tau \neq 0$ there exists a canonical surjective homomorphism 
$\rho: \Sym(X_\tau) \ra \Symtilde(X_\tau)$ with 
$\ker \rho = \Per(X_\tau)$. Hence $\Per(X_\tau)$ is a normal subgroup of $\Sym(X_\tau)$ and 
$$\Symtilde(X_\tau) \cong \Sym(X_\tau)/\Per(X_\tau).$$
\end{lemma}

\begin{proof}
Using the fact that any Legendrian submanifold of $\Sph^{2n-1}$ that is not totally geodesic is linearly
full \cite[Lemma 3.13]{haskins:complexity}, one can see that if $\mtilde_1, \mtilde_2 \in \orth{2n}$
and $\mtilde_1 \circ X_\tau = \mtilde_2 \circ X_\tau$ for some $\tau \neq 0$, then $\mtilde_1 = \mtilde_2$.
Hence, by the definitions of $\Sym(X_\tau)$ and $\Symtilde(X_\tau)$, 
given any $\MMM \in \Sym(X_\tau)$ there exists a unique element $\mtilde \in \orth{2n}$ such that
$\mtilde \circ X_\tau =  X_\tau \circ \MMM$. 
We define the map $\rho: \Sym(X_\tau) \ra \Symtilde(X_\tau)$
by $\MMM \mapsto \mtilde$. $\rho$ is readily seen to be a homomorphism which by the definitions of 
$\Sym(X_\tau)$ and $\Symtilde(X_\tau)$ is surjective. It follows immediately from the definition of 
$\Per(X_\tau)$ that $\ker \rho = \Per(X_\tau)$.
\end{proof}

\begin{remark}
\addtocounter{equation}{1}
\label{R:sym:pullback}
For any $\MMM\in \Sym(X_{\tau})$ we observe that 
$$ \MMM^{*}(X_{\tau}^{\tiny *}\,g_{\tiny \,\Sph^{2n-1}}) = (X_{\tau} \circ \MMM)^{*} g_{\tiny \,\Sph^{2n-1}}= (\mtilde \circ X_{\tau})^{*}g_{\tiny \,\Sph^{2n-1}} = X_{\tau}^{*} \circ \mtilde^{*} g_{\tiny \,\Sph^{2n-1}} = X_{\tau}^{*}g_{\tiny \,\Sph^{2n-1}}.$$
Therefore  any $\MMM\in \Sym(X_{\tau})$ 
must be an isometry of the pullback metric $g_{\tau}:=X_{\tau}^{\tiny *}\,g_{\tiny \,\Sph^{2n-1}}$ on $\cylpq$.
For this reason we want to determine the 
isometry group of $\cylpq$ endowed with the pullback metric $g_{\tau}$.
In fact,  we will below see that $\Sym(X_{\tau})=\Isom(\cylpq,g_{\tau})$.
\end{remark}

\subsection*{Symmetries of the pullback metric $g_{\tau}$} \phantom{ab}
In this section we study $\Isom(\cylpq,g_{\tau})$ the group of 
 isometries of the pullback metric $g_{\tau}:=X_{\tau}^{\tiny *}\,g_{\tiny \,\Sph^{2n-1}}$ on the cylinder $\cylpq$.
In other words we study the symmetries of the intrinsic geometry of $X_{\tau}$. 
We will study the extrinsic 
geometry of $X_{\tau}$ and the related isometries of $\Sph^{2n-1}$ beginning in \ref{P:xtau:sym:p:eq:1}.
Recall from \ref{P:X:tau}(ii) that the pullback metric $g_{\tau}$ on $\cylpq$ depends only on the function $y_{\tau}$;
isometries of $g_{\tau}$ are thus  intimately connected with the symmetries of $y_{\tau}$ studied in \ref{L:y:symmetry}.

We begin by establishing notation.
If $p>1$ any $\MMM=(\MMM_{1},\MMM_{2})\in \orth{p}\times\orth{q}$ acts as an element of 
$\Diff(\cylpq)$ by 
\begin{equation}
\addtocounter{theorem}{1}
\label{E:orth:diff:p:neq:1}
(t,\sigma_{1},\sigma_{2}) \mapsto (t, \MMM_{1}\sigma_{1},\MMM_{2}\sigma_{2}).
\end{equation}
Similarly, any element of $\MMM\in \orth{n-1}$ acts as an element of $\Diff(\cylone)$ by 
\begin{equation}
\addtocounter{theorem}{1}
\label{E:orth:diff:p:eq:1}
(t,\sigma) \mapsto (t,\MMM \sigma).
\end{equation}
We also define the exchange map $\EEE \in \Diff(\cylpp)$ by
\begin{equation}
\addtocounter{theorem}{1}
\label{E:exchange}
\EEE(t,\sigma_1, \sigma_2) = (t,\sigma_2, \sigma_1).
\end{equation}
Finally, any element $\TTT \in \Isom(\R)$ acts as an element of $\Diff(\cylpq)$ by 
\begin{equation}
\addtocounter{theorem}{1}
\label{E:isomr:ext}
(t,\sigma_{1},\sigma_{2}) \mapsto (\TTT t,\sigma_{1},\sigma_{2}).
\end{equation}
Finally the reader should also consult Appendix \ref{A:groups} for a review of some elementary group theory 
assumed throughout the rest of this section and in Section \ref{S:xtau:limit}.

\medskip

The main result of this section is the following

\begin{prop}[Isometries of the pullback metric on $\cylpq$]
\addtocounter{equation}{1}
\label{P:isom:pb}
Let $g_{\tau}:=X_{\tau}^{\tiny *}\,g_{\tiny \,\Sph^{2n-1}}$ denote the metric induced on $\cylpq$ by the immersion $X_{\tau}: \cylpq \ra \Sph^{2(p+q)-1}$.
For $0<\abs{\tau}<\taumax$, $$\Isom(\cylpq,g_{\tau}) = \dihedral{} \cdot \ort$$ where 
\begin{itemize}
\item[(i)] 
for $p=1$, $\dihedral{}=\langle \tbar, \,\TTT_{2\pt}\rangle$ and $\ort = \orth{n-1}$.
\item[(ii)] for $p>1$ and $p \neq q$, $\dihedral{} = \langle \tbar_{\pt^{+}},\, \tbar_{-\pt^{-}} \rangle$ and 
$\ort = \orth{p}\times\orth{q}$. 
\item[(iii)] 
for $p>1$ and $p=q$, $\dihedral{}=\langle \tbar \circ \EEE,\, \tbar_{\pt/2} \rangle$ and 
$\ort = \orth{p}\times\orth{p}$.
\end{itemize}
\end{prop}
\begin{proof}
Recall from \ref{P:X:tau} that $g_{\tau}$ can be written in terms of $y_{\tau}$ as
\begin{equation}
\addtocounter{theorem}{1}
\label{E:gtau}
g_{\tau} = 
\begin{cases}
y_{\tau}^{q-1}(1-y_{\tau})^{p-1}dt^2 + (1-y_{\tau}) g_{\Sph^{p-1}} + y_{\tau}\, g_{\Sph^{q-1}},\quad & \text{for\ }p>1;\\
y_{\tau}^{n-2} dt^{2} + y_{\tau} g_{\Sph^{n-2}}; \quad & \text{for \ } p=1.
\end{cases}
\end{equation}
It follows immediately from \ref{E:gtau} that for $p>1$ any element of $\orth{p} \times \orth{q}$ 
acting as in \ref{E:orth:diff:p:neq:1} is an isometry of $g_{\tau}$. Similarly, for $p=1$  we have 
$\orth{n-1} \subset \Isom(\cylpq, g_{\tau})$. 
For any $\SSS\in \Isom(\R)$ satisfying $y_{\tau} \circ \SSS= y_{\tau}$, 
extend $\SSS$ to a diffeomorphism of $\cylpq$ as described in \ref{E:isomr:ext}.
Since $\SSS$ preserves $y_{\tau}$ it follows from \ref{E:gtau} that $\SSS \in \Isom(\cylpq, g_{\tau})$. 
Recall from \ref{E:y:sym:p:eq:q} that in the special case $p=q$, $y_{\tau}$ possesses an additional symmetry $\tbar$ 
sending $y_{\tau} \mapsto 1-y_{\tau}$.
Because of this symmetry and the form of \ref{E:gtau} the map $\tbar \circ \EEE \in \Diff(\cylpp)$ defined by  
$(t,\sigma_{1},\sigma_{2}) \mapsto (-t,\sigma_{2},\sigma_{1})$ also belongs to $\Isom(\cylpp, g_{\tau})$.
By using the symmetries of $y_{\tau}$ established in \ref{L:y:symmetry} it follows that
$\dihedral{}$ forms a subgroup of $\Isom(\cylpq, g_{\tau})$ where $\dihedral{}$ is the discrete group defined 
for the three cases (i) $p=1$, (ii) $p>1$, $p\neq q$, (iii) $p>1$, $p=q$ in the statements \ref{P:isom:pb}(i)-(iii) respectively.
Hence we have established that $\dihedral{} \cdot \ort \subseteq \Isom(\cylpq)$ where $\ort = \orth{p} \times \orth{q}$ 
if $p>1$ and $\ort = \orth{n-1}$ if $p=1$. 

It remains to prove that any element in $\Isom(\cylpq, g_{\tau})$ belongs to $\dihedral{} \cdot \ort$.
We begin by introducing some useful terminology. 
Define $\merpq$ by
\addtocounter{theorem}{1}
\begin{equation}
\label{E:merpq}
\merpq
:=
\begin{cases}
\Sph^{p-1} \times \Sph^{q-1}, \quad & \text{if $p>1$;}\\
\Sph^{n-2} , \quad & \text{if $p=1$,}
\end{cases}
\end{equation}
so that $\cylpq = \R \times \merpq$, i.e. $\merpq$ is the cross section of $\cylpq$.
A \emph{meridian} of $\cylpq$ is any hypersurface of the form $\{t\} \times \merpq$ for any fixed $t\in \R$.
Let $\Pi: \cylpq \ra \merpq$ denote projection $(t,\sigma) \mapsto \sigma$.
The \emph{generator} of $\cylpq$ through the point $\sigma \in \merpq$ is the curve $\gamma_{\sigma}: \R \ra \cylpq$ 
given by $$ t\mapsto (t,\sigma),$$
i.e. a generator is a curve $\gamma_{\sigma}$ whose projection $\Pi \circ \,\gamma_{\sigma}$ to the cross section $\merpq$
 is the constant map $\sigma: \R \ra \merpq$.
Suitably parametrised any generator is a minimising geodesic, i.e. 
it minimises the $g_{\tau}$-distance between any two points on its image.
Note that the meridians can be characterised as the integral manifolds of the distribution 
$\mathcal{D}=\langle \partial_{t} \rangle^{\perp}$  of hyperplanes normal to the tangent lines to the generators of $\cylpq$.

The key to proving $\Isom(\cylpq, g_{\tau}) = \dihedral{} \cdot \ort$ is to establish that any  
$\mathsf{I} \in \Isom(\cylpq, g_{\tau})$ maps meridians to meridians.
It suffices to prove that any minimising geodesic of $g_{\tau}$ must be a generator, 
since then any isometry must map generators to generators, preserve the hyperplane distribution $\mathcal{D}$ and therefore map meridians to meridians. To prove that any minimising geodesic is a generator we will make use of some special isometries of $\Isom(\cylpq,g_{\tau})$ 
which we now describe. 

If $t_{c}$ is any critical point of $y_{\tau}$ then reflection $\tbar_{t_{c}}\in \Diff(\cylpq)$ 
across the meridian $\{t_{c}\}\times \merpq$  is contained in the group $\dihedral{}$ 
and hence by the first part of the proposition is an isometry of  $\Isom(\cylpq,g_{\tau})$. 
For $p=1$ and any $\sigma \in \Sph^{n-2}$ we denote by $\RRR_{\sigma} \in \orth{n-1}$ reflection 
with respect to the line through $\sigma$ in $\R^{n-1}$.
For $p>1$ and $\sigma = (\sigma', \sigma'') \in \Sph^{p-1}\times \Sph^{q-1}$ 
we define $\RRR_{\sigma} := (\RRR_{\sigma'},\RRR_{\sigma''}) \in \orth{p} \times \orth{q}$ 
where $\RRR_{\sigma'} \in \orth{p}$ and $\RRR_{\sigma''} \in \orth{q}$ denote 
reflections in the line through $\sigma'$ in $\R^{p}$ and the line through $\sigma''$ in $\R^{q}$ respectively. 
By the first part of the proposition $\tbar_{t_{c}} \circ \RRR_{\sigma} \in \dihedral{} \cdot \ort \subset \Isom(\cylpq,g_{\tau})$.

The key properties of the isometry $\tbar_{t_{c}} \circ \RRR_{\sigma}$ are that it fixes the 
point $(t_c,\sigma) \in \{t_{c}\} \times \merpq$ and acts by $-\Id$ on the tangent space $T_{(t_{c},\sigma)}\cylpq$.
Therefore $\tbar_{t_{c}} \circ \RRR_{\sigma}$ sends any geodesic $\gamma$ passing through $(t_{c},\sigma)$ 
to another geodesic passing through $(t_{c},\sigma)$ whose tangent vector at this point is the negative of the tangent vector 
of the initial geodesic. Hence  uniqueness of the initial value problem for geodesics implies the following symmetry of $\gamma$
\begin{equation}
\addtocounter{theorem}{1}
\label{E:geod:reflect}
\gamma \circ \tbar_{s} = \tbar_{t_{c}} \circ \RRR_{\sigma} \circ \gamma, \quad\quad \text{where } \gamma(s) = (t_{c},\sigma).
\end{equation}

Let $2d$ denote the distance between the boundary meridians of any domain of periodicity of $g_{\tau}$. 
Equivalently, $d$ is the distance between two consecutive critical meridians (a meridian of the form $\{t_{k}\}\times \merpq$ for some critical point $t_{k}$ of $y_{\tau}$.)
$d$ is realised along any generator and any other curve connecting two such meridians has strictly greater length.
(Also $2d$ depends smoothly on $\tau$ and tends to $\pi$, the diameter of the unit sphere $\Sph^{p+q-1}$ 
as $\tau \ra  0$).

Suppose $\gamma: \R \ra \cylpq$ is a geodesic parametrised by arc-length which is minimising. By using the 
 obvious piecewise smooth comparison curve, we see that the diameter of $k$ consecutive domains of periodicity of $g_{\tau}$ 
 is bounded above by $2kd + D$ where 
$D$ is the largest diameter of any meridian in a domain of periodicity of $g_{\tau}$. 
Since $\gamma$ is a minimising geodesic the diameter of its image is infinite, and therefore 
$\gamma$ intersects every meridian $\{t\}\times \merpq$.
$y_{\tau}$ is non-constant and $2\pt$-periodic and therefore has countably infinitely many critical points $t_{c}$
that we label by the strictly increasing sequence $(t_{k})_{k\in \Z}$. By \ref{L:y:symmetry} and \ref{E:tbar:pm:commute}
the sequence $(t_{k})$  satisfies 
$$t_{k}-t_{l}=(k-l)\pt, \quad \text{for any\ } k>l.$$
(If $p=1$ we could normalise so that $t_{0}=0$ and hence $t_{k}=k\pt$. 
If $p>1$ we could normalise so that $t_{0}=\pt^{+}$ and therefore $t_{-1}=-\pt^{-}$.)
Since $\gamma$ intersects every meridian, 
 there exists an increasing sequence $(s_{k})_{k\in \Z}$ and a unique sequence of points $\sigma_{k} \in \merpq$ 
so that
\begin{equation}
\addtocounter{theorem}{1}
\label{E:geodesic:sk}
\gamma(s_{k}) = (t_{k},\sigma_{k}).
\end{equation}
In other words, $s_{k}$ is the arc-length parameter at which the minimising geodesic $\gamma$ intersects the $k$th critical meridian $\{t_{k}\} \times \merpq$.
By time-translation invariance of geodesics without loss of generality we may assume that $s_{0}=0$.
Applying the isometry $\tbar_{t_{k}} \circ \RRR_{\sigma_{k}}$ as in 
\ref{E:geod:reflect} we deduce that the minimising geodesic $\gamma$ has the symmetries
\begin{equation}
\addtocounter{theorem}{1}
\label{E:geod:reflect:k}
\gamma \circ \tbar_{s_{k}} = \tbar_{t_{k}} \circ \RRR_{\sigma_{k}} \circ \gamma, \quad 
\text{for any}  \ k \in \Z,
\end{equation}
for the sequence $(s_{k})$ defined above in \ref{E:geodesic:sk}.
Composing the two reflectional symmetries arising from \ref{E:geod:reflect:k} by setting 
$k=0$ and $k=1$  and using \ref{E:tx:tbar} together with the fact that $\tbar, \TTT_{x} \in \Diff(\cylpq)$ commute with $\orth{p} \times \orth{q}$ for $p>1$ (respectively with $\orth{n-1}$ for $p=1$) we obtain
\begin{equation}
\addtocounter{theorem}{1}
\label{E:geod:trans}
\gamma \circ \TTT_{2s_{1}} = \TTT_{2(t_{1}-t_{0})}\circ \RRR_{\sigma_{1}} \circ \RRR_{\sigma_{0}} \circ 
\gamma = \TTT_{2\pt}\circ \RRR_{\sigma_{1}} \circ \RRR_{\sigma_{0}} \circ 
\gamma.
\end{equation}
By iteration of \ref{E:geod:trans} we have  
$$
\gamma \circ \TTT_{2k s_{1}} = (\TTT_{2\pt}\circ \RRR_{\sigma_{1}} \circ \RRR_{\sigma_{0}})^{k} \circ 
\gamma = \TTT_{2k\pt} \circ (\RRR_{\sigma_{1}} \circ \RRR_{\sigma_{0}})^{k} \circ \gamma,
\quad \text{for any\ }  k \in \N$$ 
and hence that 
$$\gamma(2ks_{1}) \in \{2k\pt +t_{0}\} \times \merpq = \{t_{2k}\} \times \merpq.$$
It follows from the definition of $s_{k}$ given in \ref{E:geodesic:sk} 
that $s_{2k}=2ks_{1}$. Therefore since $\gamma$ is a minimising geodesic parametrised by arc-length we have 
$$ dist(\gamma(s_{2k}),\gamma(0)) = s_{2k} = 2ks_{1}, \quad \text{for all\ }k \in \N.$$
On the other hand, by our previous (crude) diameter bound for the union of any $k$ consecutive domains of periodicity of $g_{\tau}$
we have 
$$2ks_{1} \le 2kd +D, \quad \text{for any\ }k \in \N,$$
where as previously $D$ denotes the largest diameter of any meridian and $d$ is the distance between two consecutive 
critical meridians (which as we have already stated is attained only by generators). 
Dividing by $k$ and taking $k\ra \infty$ we conclude that $s_{1}\le d$ and hence 
that $\gamma$ is a generator.

It remains to use the fact that any $\III \in \Isom(\cylpq,g_{\tau})$ maps meridians to meridians to prove that 
$\III \in \dihedral{} \cdot  \ort$. 
For the case $p>1$ we will need the following standard facts about the geometry of the product of two spheres of radii 
$r_{1}$ and $r_{2}$
\begin{equation}
\addtocounter{theorem}{1}
\label{E:isom:prod:sph}
\Isom (\,\Sph^{p-1}_{r_{1}} \times \Sph^{q-1}_{r_{2}}\,) = 
\begin{cases}
\orth{p} \times \orth{q} \quad & \text{if $p\neq q$ or $r_{1} \neq r_{2}$};\\
\orth{p} \times \orth{p} \rtimes_{\rho}\langle E \rangle \quad & \text{if $p=q$ and $r_{1}=r_{2}$},
\end{cases}
\end{equation}
(the semidirect product structure in the latter case is discussed in more detail in \ref{E:exchange'})
and that for  $p \neq q$,  $\Sph^{p-1}_{r_{1}} \times \Sph^{q-1}_{r_{2}}$ and  $\Sph^{p-1}_{r'_{1}} \times \Sph^{q-1}_{r'_{2}}$ 
are isometric if and only if $r_{1}=r'_{1}$ and $r_{2}=r'_{2}$ and for $p=q$ are isometric if and only if 
the sets $\{r_{1},r_{2}\}$ and $\{r'_{1},r'_{2}\}$ are the same.

Let $\III$ be any element in $\Isom(\cylpq,g_{\tau})$. 
Choose any meridian $M=\{t_{k}\}\times \merpq$ so that $k\in \Z$ satisfies $y_{\tau}(t_{k})=\ymin$, i.e. 
so that $y_{\tau}$ is minimal on $M$.
We established above that $\III$ maps any meridian of $\cylpq$ to another (isometric) meridian. 
In particular, when $p=1$ or when $p>1$ and $p \neq q$ this implies 
(using the standard facts about when products of two spheres are isometric) that $\III$ maps $M$
to another meridian where $y_{\tau}$ is minimal. If $p>1$ and $p=q$ then  $\III$ maps $M$
to another meridian where $y_{\tau}$ is either minimal or maximal (recall \ref{E:ymin:ymax}).
In any case of the three cases (i)--(iii), it follows that by composing with a suitable isometry $\DDD \in \dihedral{}$ we can arrange that 
$\DDD \circ \III$ fixes $M$ as a set. Hence $\DDD \circ \III$ restricted to $M$ yields an isometry of $M$. 
Since $M$ is a meridian with 
$y_{\tau}$ minimal  by \ref{E:isom:prod:sph} we have $\Isom(M) = \ort$ with $\ort$ as in \ref{P:isom:pb}(i)--(iii). 
Hence there exists $\MMM \in \ort$ such that $\MMM \circ \DDD\circ \III$ fixes $M$ 
pointwise. Therefore $\MMM \circ \DDD\circ \III$ sends any generator $\gamma_{\sigma}$ to itself
and hence we have 
$$ \MMM \circ \DDD \circ \III = 
\begin{cases}
\Id;\\
\tbar_{t_{k}},
\end{cases}
$$
according to whether $\MMM \circ \DDD \circ \III$ fixes or reflects all generators.
In either case it follows that $\III \in \dihedral{} \cdot \ort$ for any $\III \in \Isom(\cylpq,g_{\tau})$ as claimed.
\end{proof}

\begin{remark}
\label{R:isom:pb:ex}
\addtocounter{equation}{1}
When $\abs{\tau}=\taumax$ by \ref{P:y}.ii $y_{\tau} \equiv \tfrac{q}{n}$ and therefore $y_{\tau}$ is invariant under 
the whole of $\Isom(\R)$. In particular all meridians $\{t\}\times \merpq$ of $\cylpq$ are isometric. 
When $p\neq q$ the isometry group of each meridian is the group $\ort$ (defined in \ref{P:isom:pb}). 
If $p=q$ then each meridian is a product of two $p-1$ spheres of the same radius and hence the isometry group of each 
meridian is the extension of $\orth{p}\times \orth{p}$ given in case two of \ref{E:isom:prod:sph}. 
For $p\neq q$ the isometry group of $g_{\tau}$ for $\abs{\tau}=\taumax$ is $\Isom(\R) \cdot \ort$, 
whereas for $p=q$ it is $\Isom(\R) \cdot \Isom(\Sph^{p-1}_{r}\times \Sph^{p-1}_{r})$.
In all cases the action of the isometry group is transitive on $\cylpq$  thus making it into a Riemannian 
homogeneous space.

When $\tau=0$ we have from \ref{P:X:tau}.iv that $g_{0}$ is isometric to the restriction of the round metric to 
$\Sph^{p+q-1}\setminus (\Sph^{p-1},0) \cup (0,\Sph^{q-1})$ if $p>1$ or to $\Sph^{n-1}\setminus (\pm 1,0)$ 
if $p=1$. Hence for $p=1$ we have $\Isom(\cylone,g_{0}) \cong \orth{1} \times \orth{n-1}$ the subgroup 
of $\orth{n}$ leaving invariant the line through $e_{1}$ (the $\orth{1}$ factor being generated by $\tbar \in \Diff(\cylone)$).
Similarly, for $p>1$ and $p\neq q$  we have $\Isom(\cylpq,g_{0})= \orth{p} \times \orth{q}$, 
the subgroup of $\orth{n}$ leaving invariant the subset $(\Sph^{p-1},0) \cup (0,\Sph^{q-1}) \subset \Sph^{p+q-1}$.
Finally, for $p>1$ and $p=q$ we have  $\Isom(\cylpp,g_{0})= \langle \tbar \circ \EEE \rangle \cdot \orth{p} \times \orth{p}$,  which is isomorphic to the subgroup  of $\orth{n}$ leaving invariant the subset $(\Sph^{p-1},0) \cup (0,\Sph^{p-1}) \subset \Sph^{2p-1}$.
\end{remark}

\begin{lemma}[Structure of the discrete part $\dihedral{}$ of $\Isom(\cylpq,g_{\tau})$\,]
\addtocounter{equation}{1}
\label{L:discrete:str}
All three discrete groups $\dihedral{}$ defined in Proposition \ref{P:isom:pb}(i)-(iii) are isomorphic to the infinite dihedral group
$\dihedral{\infty}$.
\end{lemma}
\begin{proof}
This is essentially already proved in \ref{C:y:dihedral}, although when $p=q$ we are considering the subgroup 
of $\Isom(\cylpp)$ generated by $\tbar \circ \EEE$ and $\tbar_{\pt/2}$, rather than the subgroup of $\Isom(\R)$ 
generated by $\tbar$ and $\tbar_{\pt/2}$. Nevertheless, the same argument applies. 
\end{proof}
\begin{remark}
\addtocounter{equation}{1}
\label{R:dprime}
When $p=q$ the subgroup $\dihedral{}' \subset \dihedral{}$ 
generated by the two elements
$\tbar_{\pt/2}$ and
$\tbar_{\,-\pt/2} = (\tbar \circ \EEE) \circ \tbar_{\pt/2} \circ (\tbar \circ \EEE) $
is also isomorphic to the infinite dihedral group $\dihedral{\infty}$, 
and  corresponds to the symmetries that are shared with the case $p\neq q$.
Alternatively, $\dihedral{}' \subset \dihedral{}$ is the subgroup consisting of all words containing an even 
number of copies of $\tbar \circ \EEE$.
\end{remark}

Conjugation by the exchange map $\EEE$ defined in \ref{E:exchange} defines an involution 
$\EEE' \in \Aut{\Diff(\cylpp)}$. $\EEE'$ leaves the subgroup $\orth{p} \times \orth{p} \subset \Diff(\cylpq)$ 
invariant and on it acts by 
\begin{equation}
\addtocounter{theorem}{1}
\label{E:exchange'}
\EEE'(A,B)=(B,A).
\end{equation}
There is an obvious isomorphism
$\rho_E : \langle E \rangle \ra \langle \EEE' \rangle \subset \Aut \orth{p}\times \orth{p} \subset \Aut \Diff(\cylpp)$ 
given by
$$\EEE \mapsto \EEE'.$$
\ref{E:exchange'} implies that  $\langle \EEE \rangle \cdot (\orth{p}\times\orth{p})  = (\orth{p}\times\orth{p}) \cdot \langle \EEE \rangle$
and hence by the discussion in Appendix \ref{A:groups} the set
$\langle \EEE \rangle \cdot (\orth{p}\times\orth{p}) \subset \Isom(\Sph^{p-1}\times\Sph^{p-1}) \subset \Diff(\cylpq)$
forms a group $G$. Moreover $\orth{p}\times\orth{p}$ is a normal subgroup of $G$ and clearly
$\orth{p} \times \orth{p} \cap \langle \EEE \rangle = (\Id)$.
Hence $G$ is the semidirect product of $\orth{p} \times \orth{p}$ by $\langle \EEE \rangle \cong \Z_2$ where
the twisting homomorphism $\rho: \langle \EEE \rangle \ra \Aut\orth{p} \times \orth{p}$ is $\rho_E$ defined above, i.e.
$G = (\orth{p} \times \orth{p}) \rtimes_{\rho_E} \langle \EEE \rangle$.
In fact, $G \cong \Isom(\Sph^{p-1}_{r}\hspace{-0.05in}\times\Sph^{p-1}_{r})$ as in \ref{E:isom:prod:sph}.

An easy consequence of Proposition \ref{P:isom:pb} is the following structure result for $\Isom(\cylpq, g_{\tau})$
\begin{prop}[Structure of $\Isom(\cylpq,g_{\tau})$ for $0<\abs{\tau}<\taumax$]\hfill
\addtocounter{equation}{1}
\label{P:isom:pb:str}
\begin{itemize}
\item[(i)]
For $p=1$,  the isometry group $\Isom(\cylone, g_{\tau}) \cong \dihedral{\infty}\times \orth{n-1}$.
\item[(ii)]
For $p>1$ and $p\neq q$, the isometry group $\Isom(\cylpq, g_{\tau}) \cong \dihedral{\infty}\times 
\orth{p}\times \orth{q}$.
\item[(iii)] For $p>1$ and $p=q$, the isometry group is a semidirect product
$$\Isom(\cylpp, g_{\tau}) \cong (\orth{p} \times \orth{p}) \rtimes_{\rho} \dihedral{}\, $$
where the twisting homomorphism 
$$\rho: \langle \tbar \circ \EEE, \,\tbar_{\pt/2} \rangle \cong \dihedral{} \ra \Aut \orth{p} \times \orth{p}$$
is defined by
$$ \rho(\gamma) =
\begin{cases}
\Id & \text{if $\gamma$ is a word containing an even number of copies of $\tbar \circ \EEE$},\\
\EEE' & \text{if $\gamma$ is a word containing an odd number of copies of $\tbar \circ \EEE$},
\end{cases}
$$
where $\EEE'$ is the involution defined in \ref{E:exchange'}.
\end{itemize}

\end{prop}

\begin{proof}
(i) By Proposition \ref{P:isom:pb}.i $\Isom(\cylone, g_{\tau}) = \dihedral{} \cdot \orth{n-1}$ where
$\dihedral{}= \langle \tbar,\, \TTT_{2\pt}\rangle$. Since $\dihedral{}$ acts only on the $\R$ factor of $\cylone$ 
and $\orth{n-1}$ acts only on the $\Sph^{n-2}$ factor it is clear that $\dihedral{}$ centralises $\orth{n-1}$
and also that $\dihedral{} \cap \orth{n-1} = (\Id)$. Hence 
$\dihedral{} \cdot \orth{n-1} \cong \dihedral{} \times \orth{n-1}$.\\
(ii) By Proposition \ref{P:isom:pb}.ii $\Isom(\cylpq, g_{\tau}) = \dihedral{} \cdot \orth{p}\times \orth{q}$ where
$\dihedral{}= \langle \tbar_{\pt^{+}},\, \tbar_{-\pt^{-}} \rangle$. By the same argument as in part (i) 
$\dihedral{} \cdot \orth{p}\times \orth{q} \cong \dihedral{} \times \orth{p}\times \orth{q}$.\\
(iii) By Proposition \ref{P:isom:pb}.ii $\Isom(\cylpp, g_{\tau}) = \dihedral{} \cdot \orth{p}\times \orth{p}$ where
$\dihedral{}= \langle \tbar \circ \EEE,\, \tbar_{\pt/2} \rangle$. 
$\dihedral{}$ does not centralise $\orth{p}\times \orth{p}$, 
since $\dihedral{}$ no longer acts only on the $\R$ factor of $\cylpp$.
However, conjugation by any element of
$\dihedral{}$ does preserve the subgroup $\orth{p} \times \orth{p} \subset \Diff(\cylpp)$.
More precisely, if $\gamma \in \langle \tbar \circ \EEE, \,\tbar_{\pt/2} \rangle$ then
$$\gamma (\MMM_1, \MMM_2) \gamma^{-1} =
\begin{cases}
(\MMM_1,\MMM_2) & \text{if $\gamma$ is a word containing an even number of copies of $\tbar \circ \EEE$},\\
(\MMM_2,\MMM_1) & \text{if $\gamma$ is a word containing an odd number of copies of $\tbar \circ \EEE$}.
\end{cases}
$$
It follows that the set $\dihedral{} \cdot \orth{p} \times \orth{p}$
coincides with the set $\orth{p} \times \orth{p}\cdot \dihedral{}$ and hence
that $\dihedral{} \cdot \orth{p} \times \orth{p}$ is a group, containing 
the subgroup $\orth{p} \times \orth{p}$ as a normal subgroup of this group. 
Since clearly $\orth{p} \times \orth{p} \cap \dihedral{} = (\Id)$, we have the semidirect product structure claimed.
\end{proof}

\begin{remark}
\addtocounter{equation}{1}
\label{R:isom:pb:p:eq:q}
The kernel $\ker{\rho}$ of the twisting homomorphism 
$\rho: \dihedral{} \ra \Aut \orth{p}\times \orth{p}$
is precisely the (normal) subgroup $\dihedral{}' \subset \dihedral{}$ introduced in \ref{R:dprime}.
Hence $\Isom(\cylpp)$ contains a subgroup isomorphic to $(\orth{p} \times \orth{p}) \times \dihedral{}'$.
This subgroup is exactly the subgroup of isometries of $g_{\tau}$ we obtained in the case $p\neq q$.
\end{remark}

\subsection*{Discrete symmetries of $X_{\tau}$}\phantom{ab}
In this section we exhibit all the discrete symmetries enjoyed by $X_{\tau}$ building on the work of the previous section 
and the symmetries of $\bw_{\tau}$ established in Section \ref{S:w:sym}. In particular we establish that
$\Sym(X_{\tau})=\Isom(\cylpq,g_{\tau})$.

\subsubsection*{Discrete symmetries of $X_{\tau}$ for $p=1$}
\begin{prop}[Discrete symmetries of $X_{\tau}$ for $p=1$]
\addtocounter{equation}{1}
\label{P:xtau:sym:p:eq:1}
For $p=1$ and $0 < \abs{\tau} < \taumax$,  $X_\tau$ admits the following symmetries
\addtocounter{theorem}{1}
\begin{subequations}
\label{E:sym:p:eq:1}
\begin{align}
\label{E:orth:p:eq:1}
\MMM \circ X_{\tau} &= X_{\tau} \circ \MMM, \quad  \quad \text{for all\ } \ \MMM \in \orth{n-1},\\
\label{E:ttilde:sym:p:eq:1}
\ttilde_{2\pthat} \circ X_\tau &= X_\tau \circ \TTT_{2\pt},\\
\label{E:tbar:sym:p:eq:1}
\tbartilde \circ X_\tau &= X_\tau \circ \tbar,\\
\label{E:tbarkp:sym:p:eq:1}
\tbartilde_{\pthat} \circ X_\tau & = X_\tau \circ \tbar_{\pt},
\end{align}
\end{subequations}
where  $\ttilde_x \in \sunit{n}$ is defined in \ref{E:ttilde}, 
$\tbartilde \in \orth{2n}$ is defined by
\begin{equation}
\addtocounter{theorem}{1}
\label{E:tbartilde:p:eq:1}
\tbartilde(z_1, \ldots, z_n) = (-\overline{z}_1, \overline{z}_2, \ldots , \overline{z}_n),
\end{equation}
and $\tbartilde_x \in \orth{2n}$ is defined by
\begin{equation}
\addtocounter{theorem}{1}
\label{E:tbartilde_x}
\tbartilde_x = \ttilde_{2x} \circ \tbartilde.
\end{equation}
\end{prop}

\begin{proof}
The $\orth{n-1}$-equivariance expressed by \ref{E:orth:p:eq:1} follows immediately from the definition of $X_{\tau}$ 
(and extends the $\sorth{n-1}$-invariance used to construct $X_{\tau}$ in the first place).
The symmetries \ref{E:ttilde:sym:p:eq:1}, \ref{E:tbar:sym:p:eq:1} and \ref{E:tbarkp:sym:p:eq:1} of $X_{\tau}$ 
are equivalent to the three symmetries of $\bw_{\tau}$ established in \ref{E:w:sym:p:eq:1}.

\end{proof}

\begin{remark}
\addtocounter{equation}{1}
\label{R:extreme:tau}
Symmetries when $\tau=0$: from \ref{P:X:tau}.iv $X_0:\cylone \ra \Sph^{2n-1}$ is an embedding whose image is the totally real equatorial 
sphere $\Sph^{n-1} \subset \R^n \subset \C^n$ minus the two antipodal points $\pm e_1$. 
Clearly the subgroup $\orth{n} \subset \orth{2n}$ leaves this equatorial $n-1$ sphere invariant. 
$\orth{n-1} \subset \orth{n}$ is the subgroup of $\orth{n}$ fixing the line spanned by $e_1$.
There is no analogue of the symmetries in \ref{E:ttilde:sym:p:eq:1} and \ref{E:tbarkp:sym:p:eq:1}
in this case since the period $2\pt \ra \infty$ as $\tau \ra 0$ (see Section \ref{S:asymptotics}).
However, the isometry $\tbartilde \in \orth{2n}$ leaves $\Sph^{n-1}$ invariant and sends $e_1$ to $-e_1$ 
(cf. Remark \ref{R:isom:pb:ex}).
Hence the symmetry \ref{E:tbar:sym:p:eq:1} still holds in the case $\tau=0$. This symmetry is equivalent 
to the fact that $y_{0}$ is even in the case $p=1$ (recall \ref{P:y}.iv).

Symmetries when $\abs{\tau} = \taumax$: 
In this case $y_\tau$ is constant and hence $X_\tau$ has the additional continuous symmetries
$$ \ttilde_x \circ X_\tau = X_\tau \circ \TTT_x, \quad \text{for all $x\in \R$}.$$
The discrete symmetry \ref{E:tbar:sym:p:eq:1} still holds in this case and so  
the analogue of \ref{E:tbarkp:sym:p:eq:1} holds for all $x\in \R$. 
\end{remark}

\begin{corollary}[Structure of $\Sym(X_{\tau})$ for $p=1$]
\addtocounter{equation}{1}\label{C:sym:isom:p:eq:1}
For $p=1$ and $0<\abs{\tau}<\taumax$, 
$$\Sym(X_{\tau})= \Isom(\cylone,g_{\tau}) \cong \dihedral{} \times \orth{n-1},$$ 
where $\dihedral{}=\langle \TTT_{2\pt},\,\tbar \rangle \cong \dihedral{\infty}$.
\end{corollary}
\begin{proof}
If $(\mtilde,\MMM)$ is a symmetry of $X_{\tau}$ then by \ref{R:sym:pullback} and \ref{P:isom:pb}
$\MMM \in \Isom(\cylone,g_{\tau}) = \dihedral{} \cdot \orth{n-1}$. 
Hence $\Sym(X_{\tau}) \le \Isom(\cylone,g_{\tau})$. But since $\tbar$ and $\TTT_{2\pt}$ generate $\dihedral{}$, 
\ref{E:sym:p:eq:1} shows that also $\dihedral{} \cdot \orth{n-1} \le \Sym(X_{\tau})$.
Hence by \ref{P:isom:pb:str}.i, $\Sym(X_{\tau}) = \Isom(\cylone,g_{\tau}) = \dihedral{} \cdot \orth{n-1} \cong \dihedral{} \times \orth{n-1}$ 
as claimed.
\end{proof}

Using the definitions \ref{E:ttilde}, \ref{E:tbartilde:p:eq:1} and \ref{E:tbartilde_x}
 it is easy to check the following
\begin{prop}[Properties of  discrete symmetries of target for $p=1$]\hfill
\addtocounter{equation}{1}
\label{P:commute:p:eq:1}
\begin{itemize}
\item[(i)] $\tbartilde \circ \ttilde_x \circ \tbartilde = \ttilde_{-x}$. 
\item[(ii)] The $\orth{2}$ subgroup 
generated by $\tbartilde$ and $\{\ttilde_x\}$ 
centralises $\orth{n-1} \subset \orth{n} \subset \orth{2n}$.
\item[(iii)] $\ttilde_x$ commutes with $J$, while $\tbartilde$ and $\tbartilde_x$ anticommute with $J$.
\item[(iv)] $\ttilde_x$ preserves both $\Omega$ and $\omega$.
\item[(v)] $\tbartilde^* \Omega = -\overline{\Omega}, \quad \tbartilde^* \omega = -\omega.$
\end{itemize}
\end{prop}

\begin{remark}
\addtocounter{equation}{1}
\label{R:cf:isom:r}
We see from \ref{P:commute:p:eq:1}.v that $\tbartilde$ is both an anti-special Lagrangian and anti-holomorphic isometry. 
While for any $x\in \R$, $\ttilde_{x} \in \sunit{n} \subset \Isomsl \subset \Isomslpm$. 

Since $\tbar$ is reflection in the origin in $\R$, we have the commutation relation
$\tbar \circ \TTT_x \circ \tbar = \TTT_{-x}$. Part (i) above says that the 
same relations also hold for $\tbartilde$ and $\ttilde_x$ 
and that $\ttilde_x$ and $\tbartilde$ generate a subgroup of $\orth{2n}$ isomorphic to 
$\orth{2} \cong \Sph^1 \rtimes \Z_2$, where $\tbartilde$ generates
the $\Z_2$ factor and acts by inversion (thinking of the group generated by $\TTT_{x}$ as an abelian group) 
on the $\Sph^1$ factor.

Also, since they act on different factors of $\R \times \Sph^{n-2}$, 
every element in $\Isom(\R) \subset \Isom(\cylone)$ 
commutes with every element in $\orth{n-1} \subset \Isom(\cylone)$.
Part (ii) above is the analogue of this result for 
the group $\orth{2}$ generated by $\tbartilde$ and $\ttilde_x$.
\end{remark}

\subsubsection*{Discrete symmetries of $X_{\tau}$ for $p>1$ and $p \neq q$} \phantom{ab}
\begin{prop}[Discrete symmetries of $X_{\tau}$ for $p> 1$; cf. Prop. \ref{P:xtau:sym:p:eq:1}]
\addtocounter{equation}{1}
\label{P:xtau:sym:p:neq:q}
For $p>1$ and $0 < \abs{\tau} < \taumax$,  $X_\tau$ admits the following symmetries
\addtocounter{theorem}{1}
\begin{subequations}
\label{E:sym:p:neq:q}
\begin{align}
\label{E:orth:p:neq:q}
\MMM \circ X_{\tau} &= X_{\tau} \circ \MMM, \quad  \quad \text{for all\ \,} \MMM \in \orth{p} \times \orth{q},\\
\label{E:ttilde:sym:p:neq:q}
\ttilde_{2\pthat} \circ X_\tau &= X_\tau \circ \TTT_{2\pt},\\
\label{E:tbar:ptplus}
\tbartildeplus \circ X_\tau & = X_\tau \circ \tbar_{\pt^+},\\
\label{E:tbar:ptminus}
\tbartildeminus \circ X_\tau & = X_\tau \circ \tbar_{\,-\pt^-}
\end{align}
\end{subequations}
where $\ttilde_x \in \sunit{n}$ is defined by \ref{E:ttilde}
and $\tbartildeplus$, $\tbartildeminus \in \orth{2n}$ (which depend on $\tau$) are defined by
\begin{equation}
\addtocounter{theorem}{1}
\label{E:tbartildeplus}
\tbartildeplus (z,w) = (e^{i\alpha_\tau/p} e^{i\psi_1(2\pt^+)} \overline{z}, e^{i\alpha_\tau/q} e^{i\psi_2(2\pt^+)}
\overline{w}),
\end{equation}
and
\begin{equation}
\addtocounter{theorem}{1}
\label{E:tbartildeminus}
\tbartildeminus (z,w) = (e^{i\alpha_\tau/p} e^{i\psi_1(-2\pt^-)} \overline{z}, e^{i\alpha_\tau/q} e^{i\psi_2(-2\pt^-)}
\overline{w}),
\end{equation}
where $z\in \C^p$ and $w\in \C^q$.

\end{prop}

\begin{proof}
\ref{E:orth:p:neq:q} follows immediately from the definition of $X_{\tau}$ 
(and extends the $\sorth{p}\times\sorth{q}$-invariance used to construct $X_{\tau}$ in the first place).
The  symmetries \ref{E:ttilde:sym:p:neq:q}, \ref{E:tbar:ptplus} and \ref{E:tbar:ptminus} of $X_{\tau}$ 
are equivalent to the symmetries of $\bw_{\tau}$ established in \ref{E:w:sym:p:neq:1}.
\end{proof}
\begin{remark}
\addtocounter{equation}{1}
\label{R:extreme:tau:p:neq:q}
Symmetries when $\tau=0$: by \ref{P:X:tau}.iv $X_0:\cylpq \ra \Sph^{2n-1}$ is an embedding whose 
image is the totally real equatorial sphere $\Sph^{n-1} \subset \R^n \subset \C^n$ 
minus the two equatorial subspheres $\Sph^{n-1}\cap (\R^{p},0)$ and $\Sph^{n-1}\cap (0,\R^{q})$.
Clearly the subgroup $\orth{n} \subset \orth{2n}$ leaves this equatorial $n-1$ sphere invariant. 
$\orth{p}\times \orth{q} \subset \orth{n}$ is the subgroup of $\orth{n}$ fixing 
this distinguished pair of orthogonal equatorial subspheres. 
Therefore $X_{0}$ is still $\orth{p} \times \orth{q}$-equivariant as in \ref{E:orth:p:neq:q}.
However, there is no analogue of any of the other symmetries in \ref{E:sym:p:neq:q} in this case.
This is consistent with the fact that from \ref{R:isom:pb:ex} we have 
$\Isom(\cylpq,g_{0})= \orth{p} \times \orth{q}$ when $p>1$ and $p\neq q$.

Symmetries when $\abs{\tau} = \taumax$: 
As in the $p=1$ case discussed in Remark \ref{R:extreme:tau} 
$y_\tau$ is constant and hence $X_\tau$ has the additional continuous symmetries
$$ \ttilde_x \circ X_\tau = X_\tau \circ \TTT_x, \quad \text{for all $x\in \R$},$$
and $X_{\tau}$ is therefore homogeneous rather than cohomogeneity one as for other values of $\tau$.
\end{remark}

\begin{corollary}[Structure of $\Sym(X_{\tau})$ for $p>1$ and $p \neq q$]
\addtocounter{equation}{1}
\label{C:sym:isom:p:neq:1}
For $p>1$, $p\neq q$ and $0<\abs{\tau}<\taumax$, 
$$\Sym(X_{\tau})= \Isom(\cylpq,g_{\tau}) \cong \dihedral{} \times \orth{p}\times\orth{q},$$ 
where $\dihedral{}=\langle \tbar_{\pt^{+}}\,,\tbar_{-\pt^{-}} \rangle \cong \dihedral{\infty}$.
\end{corollary}
\begin{proof}
We use the same argument as in the proof of \ref{C:sym:isom:p:eq:1}.
If $(\mtilde,\MMM)$ is a symmetry of $X_{\tau}$ then $\MMM \in \Isom(\cylpq,g_{\tau}) = \dihedral{} \cdot\orth{p}\times \orth{q}$. 
Hence $\Sym(X_{\tau}) \le \Isom(\cylpq,g_{\tau})$. But since $\tbar_{\pt^{+}}$ and $\tbar_{-\pt^{-}}$ generate $\dihedral{}$, 
\ref{E:sym:p:neq:q} shows that also $\dihedral{} \cdot \orth{p} \times \orth{q} \le \Sym(X_{\tau})$.
Hence by \ref{P:isom:pb:str}(ii), $\Sym(X_{\tau}) = \Isom(\cylpq,g_{\tau}) = \dihedral{} \cdot \orth{p}\times\orth{q} \cong \dihedral{} \times \orth{p}\times\orth{q}$ as claimed.
\end{proof}

Using definitions \ref{E:ttilde}, \ref{E:tbartildeplus} and \ref{E:tbartildeminus} one can check the following:
\begin{prop} [Properties of  target discrete symmetries for $p>1,\,p\neq q$, cf. \ref{P:commute:p:eq:1}]\hfill
\addtocounter{equation}{1}
\label{P:commute:p:neq:q}
\begin{itemize}
\item[(i)] $\tbartildeplus \circ \ttilde_x \circ \tbartildeplus = \tbartildeminus \circ \ttilde_x \circ \tbartildeminus = \ttilde_{-x}$.
\item[(ii)] $\tbartildeminus \circ \tbartildeplus = \ttilde_{-2\pthat}, \quad 
\tbartildeplus \circ \tbartildeminus = \ttilde_{2\pthat}$.
\item[(iii)] the dihedral subgroup $\widetilde{\dihedral{}}:= \langle\tbartildeplus, \tbartildeminus \rangle$ 
 centralises $\orth{p} \times \orth{q} \subset \orth{n} \subset \orth{2n}$.
\item[(iv)] $\ttilde_x$ commutes with $J$, while  
$\tbartildeplus$ and $\tbartildeminus$ anticommute with $J$. 
\item[(v)] $\ttilde_x$ preserves both $\Omega$ and $\omega$.
\item[(vi)] ${\tbartildeplus}^* \Omega = {\tbartildeminus}^* \Omega = -\overline{\Omega}, 
\quad {\tbartildeplus}^* \omega = {\tbartildeminus}^* \omega = -\omega$.
\end{itemize}
\end{prop}
\noindent
To prove the first two equalities of \ref{P:commute:p:neq:q}.vi one also needs to use \ref{E:psi:real:pneq1}, 
\ref{E:psi:imag:pneq1} and \ref{E:psi:sym:p:neq:1}. \ref{P:commute:p:neq:q}.vi implies that 
both $\tbartildeplus$ and $\tbartildeminus$ are anti-special Lagrangian,  anti-holomorphic isometries in $\Isomslpm$.

\begin{remark}
\addtocounter{equation}{1}
\label{R:dihedral:sphere}
Remark \ref{R:dihedral:cyl} showed that the group generated by 
$\tbar_{\pt^+}$, $\tbar_{\pt^-} \in \Isom(\R) \subset \Isom(\cylpq)$
is isomorphic to the infinite dihedral group $\dihedral{\infty}$. 
Part (ii) above gives the analogous result for the subgroup $\widetilde{\dihedral{}}$ of $\orth{2n}$ 
generated by $\tbartildeplus$ and $\tbartildeminus$. 
See Lemma \ref{L:symtilde:p:neq:q} for the precise structure of $\widetilde{\dihedral{}}$.

Part (iii) is the analogue 
of the fact that $\dihedral{}$ and $\orth{p} \times \orth{q} \subset \Isom(\cylpq)$ 
centralise each other.
\end{remark}

\subsubsection*{Discrete symmetries of $X_{\tau}$ for $p>1$ and $p=q$} \phantom{ab}
\begin{prop}[Discrete symmetries of $X_{\tau}$ for $p=q$; cf. Props. \ref{P:xtau:sym:p:eq:1} 
and \ref{P:xtau:sym:p:neq:q}]
\addtocounter{equation}{1}
\label{P:xtau:sym:p:eq:q}
For $p>1$, $p=q$ and $0 < \abs{\tau} < \taumax$,  $X_\tau$ admits the following symmetries
\addtocounter{theorem}{1}
\begin{subequations}
\label{E:sym:p:eq:q}
\begin{align}
\label{E:orth:p:eq:q}
\MMM \circ X_{\tau} &= X_{\tau} \circ \MMM, \quad  \quad \text{for all\ } \MMM \in \orth{p} \times \orth{p},\\
\label{E:ttilde:sym:p:eq:q}
\ttilde_{2\pthat} \circ X_\tau &= X_\tau \circ \TTT_{2\pt},\\
\label{E:tbar:ptplus:p:eq:q}
\tbartildeplus \circ X_\tau & = X_\tau \circ \tbar_{\pt/2},\\
\label{E:tbar:ptminus:p:eq:q}
\tbartildeminus \circ X_\tau & = X_\tau \circ \tbar_{\,-\pt/2},\\
\label{E:tbar:sym:p:eq:q}
\tbartilde \circ X_\tau &= X_\tau \circ \tbar \circ \EEE,\\
\label{E:sbar}
\ttilde_{\pthat} \circ \sbartilde \circ X_\tau &= X_\tau \circ \TTT_{\pt} \circ \EEE,
\end{align}
\end{subequations}
where $\ttilde_{x}$, $\tbartildeplus$ and $\tbartildeminus$ are defined as in \ref{P:xtau:sym:p:neq:q}, 
$\tbartilde \in \orth{2p} \subset \unit{2p}$ is defined by
\begin{equation}
\addtocounter{theorem}{1}
\label{E:tbartilde:p:eq:q}
\tbartilde\,(z,w)= (w,z) \quad \text{where} \ w,z\in \C^p,
\end{equation}
and $\sbartilde \in \orth{4p}$ is defined by
\begin{equation}
\addtocounter{theorem}{1}
\label{E:sbartilde}
\sbartilde\, (z,w) = e^{-i\pi/2p} (\overline{w},\overline{z}), \quad \text{where} \ w,z\in \C^p.
\end{equation}
Furthermore, the reflections $\tbartildeplus$ and $\tbartildeminus$ can also be expressed as 
\addtocounter{theorem}{1}
\begin{equation}
\label{E:tbartildeplus:p:eq:q}
\tbartildeplus = \ttilde_{\pthat} \circ \sbartilde \circ \tbartilde,
\end{equation}
and
\addtocounter{theorem}{1}
\begin{equation}
\label{E:tbartildeminus:p:eq:q}
\tbartildeminus = \ttilde_{-\pthat} \circ \sbartilde \circ \tbartilde.
\end{equation}
\end{prop}

\begin{proof}
The $\orth{p}\times\orth{p}$-equivariance expressed by \ref{E:orth:p:eq:q} follows as a special case of \ref{E:orth:p:neq:q}.
Similarly, since $\pt^+ = \pt^- = \tfrac{1}{2}\pt$,
\ref{E:ttilde:sym:p:eq:q}, \ref{E:tbar:ptplus:p:eq:q} and \ref{E:tbar:ptminus:p:eq:q} are each special cases of
\ref{E:ttilde:sym:p:neq:q}, \ref{E:tbar:ptplus} and \ref{E:tbar:ptminus} respectively.
\ref{E:tbar:sym:p:eq:q} is equivalent to the symmetry of $\bw_\tau$ with respect to $\tbar$ given in \ref{E:w:ex:sym}.
\ref{E:sbar} follows from the symmetries \ref{E:tbar:ptplus:p:eq:q} and \ref{E:tbar:sym:p:eq:q}, using
\ref{E:tbartildeplus:p:eq:q}.
\end{proof}

\begin{remark}
\addtocounter{equation}{1}
\label{R:extra:sym}
\ref{E:tbar:sym:p:eq:q} and \ref{E:sbar} express the two additional symmetries that $X_\tau$ possesses when $p=q$
and both utilise the additional exchange isometry $\EEE \in \Isom(\cylpp)$.
\ref{E:tbar:sym:p:eq:q} expresses an additional reflectional symmetry of
$X_\tau$ about the $\sorth{p}\times \sorth{p}$-orbit for which the radii of both $p-1$ spheres are equal.
\end{remark}

\begin{corollary}[Structure of $\Sym(X_{\tau})$ for $p>1$ and $p=q$]
\addtocounter{equation}{1}
\label{C:sym:isom:p:eq:q}
For $p>1$, $p=q$ and $0<\abs{\tau}<\taumax$, 
$$\Sym(X_{\tau})= \Isom(\cylpp,g_{\tau}) \cong  (\orth{p} \times \orth{p}) \rtimes_{\rho} \dihedral{},$$ 
where $\dihedral{}=\langle \tbar \circ \EEE\,,\tbar_{\pt/2} \rangle \cong \dihedral{\infty}$ and $\rho$ is the 
homomorphism defined in \ref{P:isom:pb:str}(iii).
\end{corollary}
\begin{proof}
We use the argument of \ref{C:sym:isom:p:eq:1} again. 
If $(\mtilde,\MMM)$ is a symmetry of $X_{\tau}$ then $\MMM \in \Isom(\cylpp,g_{\tau}) = \dihedral{} \cdot \orth{p}\times \orth{p}$. 
Hence $\Sym(X_{\tau}) \le \Isom(\cylpp,g_{\tau})$. But since $\tbar \circ \EEE$ and $\tbar_{\pt/2}$ generate $\dihedral{}$, 
\ref{E:sym:p:eq:q} shows that also $\dihedral{} \cdot \orth{p}\times\orth{p}  \le \Sym(X_{\tau})$.
Hence by \ref{P:isom:pb:str}(iii), $\Sym(X_{\tau}) = \Isom(\cylpp,g_{\tau}) = \dihedral{} \cdot \orth{p}\times\orth{p} \cong \ (\orth{p} \times \orth{p}) \rtimes_{\rho} \dihedral{}$ as claimed.
\end{proof}

Using \ref{E:tbartilde:p:eq:q}--\ref{E:tbartildeminus:p:eq:q} one can check the following:
\begin{prop}[Properties of  target discrete symmetries for $p=q$, cf. \ref{P:commute:p:eq:1},
\ref{P:commute:p:neq:q}]\hfill\\
\addtocounter{equation}{1}
\label{P:commute:p:eq:q}
\noindent
$\tbartildeplus$, $\tbartildeminus$, $\ttilde_x$ have all the properties detailed in Proposition \ref{P:commute:p:neq:q}. 
Additionally the new isometries $\sbartilde$ and $\tbartilde$ satisfy
\begin{itemize}
\item[(i)] $\tbartilde \circ \ttilde_x \circ \tbartilde = \ttilde_{-x}$ and 
$\tbartilde \circ \tbartildeplus \circ \tbartilde = \tbartildeminus$.
\item[(ii)] $\sbartilde$ commutes with $\tbartildeplus$, $\tbartildeminus$, $\tbartilde$ and with $\ttilde_x$. 
\item[(iii)]$\tbartilde$ commutes with $J$ while $\sbartilde$ anticommutes with $J$.
\item[(iv)] $\sbartilde^* \Omega = (-1)^{p-1}\, \overline{\Omega}, \quad \sbartilde^* \omega = -\omega, \quad
\tbartilde^* \Omega = (-1)^p\, \Omega, \quad \tbartilde^* \omega = \omega.$
\end{itemize}
\end{prop}
\noindent
\ref{P:commute:p:eq:q}.iv implies that $\tbartilde \in \sunit{2p}^{\pm} = \IsomslpmJ$ and that $\tbartilde \in \sunit{2p}$ 
if and only if $p$ is even.

\subsubsection*{The rotational period of $X_{\tau}$} \phantom{ab}
\ref{P:isom:pb} implies that for any admissible $p$ and $q$ and $0<\abs{\tau}<\taumax$ the translation $\TTT_{2\pt} \in \Diff(\cylpq)$
 belongs to $\Sym(X_{\tau})=\Isom(\cylpq,g_{\tau})$.  Therefore for any such $\tau$ we call the immersion 
 $X_{\tau}:\cylpq \ra \Sph^{2(p+q)-1}$  \emph{$2\pt$-periodic}; the corresponding element $\ttilde_{2\pthat} = \rho(T_{2\pt}) \in 
 \Symtilde(X_{\tau})$ we call the \emph{rotational period of $X_{\tau}$}. ($\ttilde_{x} \in \sunit{p+q}$ is defined in  \ref{E:ttilde} 
and $p>1$ respectively and $2\pthat$ is the angular period defined in \ref{E:pthat}).
In \ref{D:M:period} we defined the rotational period $\tcheck_{2\pthat}$ of $\bw_{\tau}$ and also its order $k_{0} \in \N \cup \{\infty\}$. Using the definition of $X_{\tau}$ in terms of the $(p,q)$-twisted SL curve $\bw_{\tau}$ 
we see that symmetries of $X_{\tau}$ corresponding to $\TTT_{2\pt}$, 
(\ref{E:ttilde:sym:p:eq:1}, \ref{E:ttilde:sym:p:neq:q}, and \ref{E:ttilde:sym:p:eq:q}---in the three cases (i) $p=1$, (ii) $p>1$ and $p \neq q$, and (iii) $p>1$ and $p=q$) are equivalent to the symmetry of $\bw_{\tau}$
$$ \tcheck_{2\pthat} \circ \bw_{\tau} = \bw_{\tau} \circ \TTT_{2\pt},$$
described earlier.
Similarly, it follows directly from the definitions that $k_{0}$ the order of the rotational period $\tcheck_{2\pthat} \in \unit{2}$ of $\bw_{\tau}$ is equal to the order of the rotational period $\ttilde_{2\pthat} \in \sunit{p+q}$ of $X_{\tau}$ just defined.
Hence we will simply refer to $k_{0}$ as the order of the rotational period.
\medskip
\subsection*{The structure of $\Symtilde(X_{\tau})$}
In this section we determine the structure of $\Symtilde(X_{\tau})\subset \orth{2n}$ as an abstract group in the three cases 
(i) $p=1$, (ii) $p>1$ and $p \neq q$ and (iii) $p>1$ and $p=q$.

\subsubsection*{The structure of $\Symtilde(X_{\tau})$ for $p=1$}\phantom{ab}
\begin{lemma}[Structure of $\Symtilde(X_{\tau})$ for $p=1,q=n-1$]
\addtocounter{equation}{1}
\label{L:symtilde:p:eq:1}Let $\widetilde{\mathbf{D}}$ be the subgroup of $\Symtilde(X_{\tau}) \subset \orth{2n}$ 
generated by $\tbartilde$ and $\ttilde_{2\pthat}$. 
For $0<\abs{\tau}<\taumax$ , we have 
$$ \widetilde{\mathbf{D}} \cong 
\begin{cases}
\dihedral{\infty} \quad \text{if\ } k_{0}=\infty;\\
\dihedral{k_{0}} \quad \textit{if\ } k_{0} \text{\ is finite;}
\end{cases}
$$
and 
$$\widetilde{\dihedral{}} \cap \IsomslpmJ = \widetilde{\dihedral{}} \cap \unit{n} = \langle \ttilde_{2\pthat}\rangle \subset \sunit{n}.$$
The structure of $\Symtilde(X_{\tau})$ is given as follows:
\begin{itemize}
\item[(i)] If $k_{0}$, the order of the rotational period, is infinite or odd or the dimension $n$ is even
then 
$$ \Symtilde(X_{\tau}) \cong \widetilde{\mathbf{D}} \times \orth{n-1}.$$
\item[(ii)] If $k_{0}$ is even and $n$ is odd then 
$$\widetilde{\dihedral{}} \cap \orth{n-1} = 
\langle \ttilde_{k_{0}\pthat} \rangle = \left\langle \left(\begin{array}{cc}1 & 0 \\0 & -\Id_{n-1}\end{array}\right)\right\rangle \cong \Z_{2},$$
and this $\Z_{2}$ subgroup belongs to the centre of $\Symtilde(X_{\tau})$. 
$ \Symtilde(X_{\tau}) $ is the internal central product of $\widetilde{\dihedral{}}$ and 
$\orth{n-1}$ identifying the central subgroup $\widetilde{\dihedral{}} \cap \orth{n-1} \cong \Z_{2}$ \cite[p. 29]{gorenstein}.
\end{itemize}
\end{lemma}
\begin{proof}
Recall the presentation $\langle r,\, f\, |\, r^{k}=1,\, f^{2}=1,\, frf=f^{-1} \rangle$ 
for the finite dihedral group $\dihedral{k}$ 
and the presentation $\langle r,\, f\, |\, f^{2}=1,\, frf=f^{-1} \rangle$ 
for the infinite dihedral group $\dihedral{\infty}$. 
The structure of $\widetilde{\dihedral{}}$ claimed follows from \ref{P:commute:p:eq:1}.i and 
the definition of $k_{0}$.
Any element in $\widetilde{\mathbf{D}} = \langle \ttilde_{2\pthat}, \tbartilde \rangle$ 
is of the form $\tbartilde_{k\pthat}$ (recall \ref{E:tbartilde_x}) or $\ttilde_{2k\pthat}$ for some $k\in \Z$. 
By \ref{P:commute:p:eq:1}.iii $\tbartilde_{k\pthat}$ acts antiholomorphically while $\ttilde_{2k\pthat}$ 
acts holomorphically. Hence $\widetilde{\dihedral{}} \cap \IsomslpmJ = \langle \ttilde_{2\pthat}\rangle \subset \sunit{n}$ 
as claimed.

By \ref{P:commute:p:eq:1}.ii $\widetilde{\dihedral{}}$ centralises $\orth{n-1}$, 
and hence the set $\widetilde{\dihedral{}}\cdot \orth{n-1}$ forms 
a group  in which both $\widetilde{\dihedral{}}$ and $\orth{n-1}$ are normal subgroups. 
In general  $\widetilde{\mathbf{D}}  \cap \orth{n-1} \neq  (\Id)$, 
and hence in general the isomorphism $\widetilde{\mathbf{D}} \cdot \orth{n-1}
 \cong \widetilde{\mathbf{D}} \times \orth{n-1}$ fails. 
We analyse the intersection $\widetilde{\mathbf{D}}  \cap \orth{n-1}$ as follows.
Clearly, $\widetilde{\mathbf{D}}  \cap \orth{n-1} \subset \widetilde{\mathbf{D}}  \cap \unit{n} = \langle \ttilde_{2\pthat}\rangle.$
Using the definitions of $\ttilde_{x}$, $\tcheck_{x}$ and $\orth{n-1}\subset \orth{n} \subset \unit{n}$ we have 
$$\ttilde_{2k\pthat} \in \orth{n-1} \iff \tcheck_{2k\pthat} = 
\left(\begin{array}{cc}  1& 0 \\0 & \pm 1\end{array}\right).$$
Hence by Lemma \ref{L:w:half:period:a} $\widetilde{\mathbf{D}} \cap \orth{n-1}\neq (\Id)$ 
if and only if $\bw_{\tau}$ admits half-periods of type $(+-)$. 
By Lemma \ref{L:w:half:period}.iii $\bw_{\tau}$ admits half-periods of type $(+-)$ if and only if 
$k_{0}$, the order of the rotational period, is even and the dimension $n$ is odd. 
In this case $\widetilde{\mathbf{D}}  \cap \orth{n-1} \cong \Z_{2}$ 
where $\Z_{2}$ is the group generated by the involution $\ttilde_{k_{0}\pthat}$. 
Since $k_{0}\pthat$ is a half-period of type $(+-)$ then we have
\begin{equation}
\addtocounter{theorem}{1}
\label{E:centre:p:eq:1}
\ttilde_{k_{0}\pthat}=\left(\begin{array}{cc}1 & 0 \\0 & -\Id_{n-1}\end{array}\right).
\end{equation}
One can verify that the $\Z_{2}$ subgroup generated by $\ttilde_{k_{0}\pthat}$ belongs to the 
centre of $\Symtilde(X_{\tau})$.
\end{proof}
\noindent
\begin{remark}
\addtocounter{equation}{1}
\label{R:dihedral}
With reference to the central product structure discussed above for $\Symtilde(X_{\tau})$ we remark that
the centre of the finite dihedral group $\dihedral{k}$ is trivial if $k$ is odd and isomorphic to $\Z_{2}$ 
if $k$ is even. In the presentation $\dihedral{k} \cong \langle r,\, f\,|\, f^{2}=1,\,r^{k}=1, \, frf=r^{-1} \rangle$ 
the centre of $\dihedral{k}$ is generated by $r^{k/2}$ when $k$ is even.
Applied to the group $\widetilde{\dihedral{}}=\langle \ttilde_{2\pthat}, \tbartilde \rangle \cong \dihedral{k_{0}}$,  this is consistent with \ref{E:centre:p:eq:1}. 
\end{remark}

\begin{remark}
\addtocounter{equation}{1}
\label{R:pm:holo}
Lemma \ref{L:symtilde:p:eq:1} implies that
every element $\mtilde \in \Symtilde(X_\tau)$ can be written (not uniquely) in the form
$$ \mtilde = {\tbartilde}\phantom{}^i \circ \ttilde_{2k\pthat} \circ \MMM, \quad
\text{where} \quad i \in \Z_2, \ k\in \Z \ \text{and} \ \MMM \in \orth{n-1}.$$
In particular, every element of $\Symtilde(X_{\tau})$ belongs to $\Isomslpm$ and hence acts 
on $\C^n$ either by
holomorphic (if $i=0$) or antiholomorphic (if $i=1$) isometries.
Also, $\mtilde \in \Symtilde(X_\tau)$ is special unitary if and only if $i=0$ and $\MMM \in \sorth{n-1}\subset \orth{n-1}$.
\end{remark}

\begin{corollary}[Structure of $\Per(X_{\tau})$ for $p=1$]
\addtocounter{equation}{1}
\label{C:xtau:period:p:eq:1}
For $0<\abs{\tau}<\taumax$, we have
\begin{equation*}
\Per(X_{\tau}) = 
\begin{cases}
(\Id) & \text{if\ }k_{0}=\infty;\\
\langle \TTT_{2k_{0}\pt} \rangle & \text{if\ } k_{0} \text{\ is odd or $k_{0}$ is even and $n$ is even;}\\
\langle \TTT_{k_{0}\pt} \circ -\Id_{\Sph^{n-1}}\rangle & \text{if\ } k_{0} \text{\ is even and $n$ is odd.}
\end{cases}
\end{equation*}
\end{corollary}
\begin{proof}
This follows from the results for the structures of $\Sym(X_{\tau})$ and $\Symtilde(X_{\tau})$ proved 
in \ref{C:sym:isom:p:eq:1} and \ref{L:symtilde:p:eq:1} 
together with the fact that $\Per(X_{\tau}) = \ker{\rho}$ where $\rho: \Sym(X_{\tau}) \ra \Symtilde(X_{\tau})$ is the homomorphism described in \ref{L:sym:subgroups}.

Alternatively, we can see more directly that determining $\Per(X_{\tau})$ is essentially 
equivalent to determining the intersection $\widetilde{\dihedral{}}\cap \orth{n-1}$ 
and therefore also equivalent to finding half-periods of $\bw_{\tau}$ of type $(+-)$.
From  \ref{C:sym:isom:p:eq:1} any element in $\Sym(X_{\tau})$ can be written in the form 
$ \tbar^{i} \circ \TTT_{2k\pt} \circ \MMM$,
for some $i\in \Z_{2}$, $k\in \Z$ and $\MMM \in \orth{n-1}$.
From \ref{E:sym:p:eq:1} we obtain
$$ X_{\tau} \circ  \tbar^{i} \circ \TTT_{2k\pt} \circ \MMM = \tbartilde\phantom{}^{i} \circ \ttilde_{2k\pthat} \circ \MMM \circ X_{\tau},$$
and hence $\tbar^{i} \circ \TTT_{2k\pt} \circ \MMM \in \Per(X_{\tau})$ if and only if 
$\tbartilde\phantom{}^{i} \circ \ttilde_{2k\pthat} \circ \MMM = \Id \in \orth{2n}$.
Hence $\tbar^{i}\circ\TTT_{2k\pt} \circ \MMM \in \Per(X_{\tau})$ if and only if $\tbartilde\phantom{}^{i} \circ \ttilde_{2k\pthat} = \MMM^{-1}$ for some $k\in \Z$ and 
$\MMM \in \orth{n-1} \subset \orth{n} \subset \orth{2n}$. But this is equivalent to finding all points in the intersection $\widetilde{\dihedral{}} \cap \orth{n-1}$.
Hence the result follows from \ref{L:symtilde:p:eq:1}.
\end{proof}

\subsubsection*{The structure of $\Symtilde(X_{\tau})$ for $p>1$ and $p\neq q$}\phantom{ab}
The analogous results for $p>1$ and $p \neq q$ follow. First for any $(j,k)\in \Z_{2}\times \Z_{2}$ define 
the involution $\widetilde{\rho_{jk}} \in \orth{p} \times \orth{q} \subset \orth{n} \subset \unit{n}$ by 
$$\widetilde{\rho}_{jk} := 
\left(\begin{array}{cc}(-1)^{j}\Id_{p} & 0 \\0 & (-1)^{k}\Id_{q}\end{array}\right).$$
\begin{lemma}[Structure of $\Symtilde(X_{\tau})$ for $p>1,p\neq q$]
\addtocounter{equation}{1}
\label{L:symtilde:p:neq:q}
Let $\widetilde{\dihedral{}}$ be the group generated by $\tbartildeplus$ and $\tbartildeminus$ (defined in \ref{E:tbartildeplus} and \ref{E:tbartildeminus}).
For $0<\abs{\tau}<\taumax$ , we have 
$$ \widetilde{\mathbf{D}} \cong 
\begin{cases}
\dihedral{\infty} \quad \text{if\ } k_{0}=\infty;\\
\dihedral{k_{0}} \quad \textit{if\ } k_{0} \text{\ is finite;}
\end{cases}
$$
and 
$$\widetilde{\dihedral{}} \cap \IsomslpmJ = \widetilde{\dihedral{}} \cap \unit{n} = \langle \ttilde_{2\pthat}\rangle \subset \sunit{n}.$$

The structure of $\Symtilde(X_{\tau})$ is given as follows:
\begin{itemize}
\item[(i)] If $k_{0}$, the order of the rotational period, is infinite or odd then
$$ \Symtilde(X_{\tau}) \cong \widetilde{\mathbf{D}} \times \orth{p}\times \orth{q}.$$
\item[(ii)] If $k_{0}$ is even then 
$$\widetilde{\dihedral{}} \cap \orth{p} \times \orth{q} = \langle \ttilde_{k_{0}\pthat} \rangle = 
\langle \widetilde{\rho}_{jk}\rangle \cong \Z_{2},$$
where $\widetilde{\rho}_{jk}$ is the involution defined above 
and $i=q/\hcf(p,q)$ and $j=p/\hcf(p,q)$.
Furthermore, $\langle \widetilde{\rho}_{jk}\rangle$ belongs to the centre of $\Symtilde(X_{\tau})$.
Hence $ \Symtilde(X_{\tau}) $ is an internal central product of $\widetilde{\dihedral{}}$ and 
$\orth{p}\times \orth{q}$ identifying the central subgroup $\widetilde{\dihedral{}} \cap \orth{p}\times \orth{q}  =\langle \widetilde{\rho}_{jk}\rangle \cong \Z_{2}$.
\end{itemize}
\end{lemma}
\begin{proof}
This is very similar to the proof of Lemma \ref{L:symtilde:p:eq:1}.

Recall the presentation $\langle x, y \, |\, x^{2}=y^{2}=(xy)^{k}=1\rangle$ for the finite dihedral group 
$\dihedral{k}$ and the presentation $\langle x, y \, |\, x^{2}=y^{2}=1\rangle$ for the infinite dihedral group 
$\dihedral{\infty}$. Since by \ref{P:commute:p:neq:q}.ii 
$\tbartildeplus \circ \tbartildeminus = \ttilde_{2\pthat}$ we have $\widetilde{\dihedral{}} \cong \dihedral{k_{0}}$
if the rotational period has finite order $k_{0 }$ 
and  $\widetilde{\dihedral{}} \cong \dihedral{\infty}$ otherwise.
Any nontrivial element in $\widetilde{\dihedral{}}$ can be written as an alternating word in 
its two generators $\tbartildeplus$ and $\tbartildeminus$. By \ref{P:commute:p:neq:q}.iv
both generators act antiholomorphically, and hence so does any word with an odd number of letters. 
By \ref{P:commute:p:neq:q}.ii any word in $\widetilde{\dihedral{}}$ with an even number of letters lies in the cyclic subgroup 
$\widetilde{\cyclic{}} = \langle \ttilde_{2\pthat} \rangle \subset \sunit{n}$ and hence 
$\widetilde{\dihedral{}} \cap \IsomslpmJ = \widetilde{\cyclic{}} = \langle \ttilde_{2\pthat} \rangle$ as claimed.

By \ref{P:commute:p:neq:q}.iii $\widetilde{\dihedral{}} = \langle \tbartilde^{+},\tbartilde^{-} \rangle$ 
centralises $\orth{p}\times \orth{q}$ and therefore 
$\widetilde{\dihedral{}} \cdot \orth{p} \times \orth{q}$ forms a group in which 
both factors are normal subgroups. As in the $p=1$ case
in general  $\widetilde{\dihedral{}} \cap \orth{p} \times \orth{q} \neq  (\Id)$.
Clearly, $\widetilde{\dihedral{}} \cap \orth{p} \times \orth{q} \subset \widetilde{\dihedral{}} \cap \sunit{n}^{\pm} = 
\langle \ttilde_{2\pthat} \rangle$.
Using the definitions of $\ttilde_{x}$, $\tcheck_{x}$ and $\orth{p} \times \orth{q} \subset \orth{n} \subset \sunit{n}^{\pm}$ we find 
$$\ttilde_{2k\pthat} \in \orth{p} \times \orth{q} \iff \tcheck_{2k\pthat} = 
\left(\begin{array}{cc}  \pm 1& 0 \\0 & \pm 1\end{array}\right) = \rho_{jk}, \quad \text{for some\ }
(j,k)\in \Z_{2}\times \Z_{2}.$$
Hence by \ref{L:w:half:period:a}, 
$\widetilde{\dihedral{}} \cap \orth{p} \times \orth{q} \neq (\Id)$ if and only if 
$\bw_{\tau}$ admits strict half-periods.
By \ref{L:w:half:period} $\bw_{\tau}$ admits strict half-periods if and only if $k_{0}$ is even. 
Moreover, for fixed $p$ and $q$ only strict half-periods of type $(jk)$ occur where $j=q/\hcf(p,q)$ and $k=p/\hcf(p,q)$. Finally, it easy to check that $\langle \widetilde{\rho}_{jk}\rangle$ belongs to the centre 
of $\Symtilde(X_{\tau}) = \widetilde{\dihedral{}}\cdot \orth{p}\times \orth{q}$. 
\end{proof}

\begin{remark}
\addtocounter{equation}{1}
\label{R:pm:holo:p:gt:1}
Lemma \ref{L:symtilde:p:neq:q} implies that
every element $\mtilde \in \Symtilde(X_\tau)$ can be written (non-uniquely) in the form
$$ \mtilde = \DDD \circ \MMM,$$ 
where $\DDD \in \widetilde{\dihedral{}}$ is an alternating word in $\tbartildeplus$ and $\tbartildeminus$ 
and $\MMM \in \orth{p} \times \orth{q} \subset \sunit{n}^{\pm}$.
In particular, every element of $\Symtilde(X_{\tau})$ belongs to $\Isomslpm$ and hence acts on $\C^{n}$ 
either holomorphically (if and only if $\DDD$ is a word with an even number of letters) or anti-holomorphically. 
Moreover, the subgroup of holomorphic symmetries of $X_{\tau}$, 
$\Symtilde(X_{\tau}) \cap \IsomslpmJ = \langle \ttilde_{2\pthat} \rangle \cdot \orth{p} \times \orth{q}$. 
Finally $\mtilde=(\mtilde_{1},\mtilde_{2}) \in \orth{p} \times \orth{q}$ belongs to $\sunit{n}$ if and only if 
$\det(\mtilde_{1}) \,\det(\mtilde_{2}) = +1$.
\end{remark}

\begin{corollary}[Structure of $\Per(X_{\tau})$ for $p>1$, $p\neq  q$]
\addtocounter{equation}{1}
\label{C:xtau:period:p:neq:1}
For $0<\abs{\tau}<\taumax$, we have  
\begin{equation*}
\Per(X_{\tau}) = 
\begin{cases}
(\Id) & \text{if\ }k_{0}=\infty;\\
\langle \TTT_{2k_{0}\pt} \rangle & \text{if\ } k_{0} \text{\ is odd;}\\
\langle \TTT_{k_{0}\pt} \circ (-1)^{j}\Id_{\Sph^{p-1}} \circ (-1)^{k} \Id_{\Sph^{q-1}}\rangle & \text{if\ } k_{0} \text{\ is even,}
\end{cases}
\end{equation*}
where $j=q/\hcf(p,q)$ and $k=p/\hcf(p,q)$.
\end{corollary}
\noindent
The proof follows from  Lemma \ref{L:symtilde:p:neq:q} in the same way that Corollary \ref{C:xtau:period:p:eq:1} followed from \ref{L:symtilde:p:eq:1}.

\smallskip
\subsubsection*{The structure of $\Symtilde(X_{\tau})$ for $p>1$ and $p=q$}\phantom{ab}
Finally we determine the structure of $\Symtilde(X_{\tau})$ in the case $p=q$.
Recall from \ref{P:commute:p:eq:q}.i  that 
$$\tbartildeminus = \tbartilde \circ \tbartildeplus \circ \tbartilde.$$
Hence the involutions $\tbartildeplus$ and $\tbartilde$ generate a subgroup
$\widetilde{\dihedral{}} = \langle \tbartilde, \tbartildeplus \rangle \subset \Symtilde(X_\tau)$.

\begin{lemma}[Structure of $\Symtilde(X_{\tau})$ for $p=q$]
\addtocounter{equation}{1}
\label{L:symtilde:p:eq:q}
Let $\widetilde{\dihedral{}} = \langle \tbartilde, \tbartildeplus \rangle \subset \Symtilde(X_{\tau})$. 
For $0<\abs{\tau}<\taumax,$  we have 
$$ \widetilde{\mathbf{D}} \cong 
\begin{cases}
\dihedral{\infty} \quad & \text{if\ } k_{0}=\infty;\\
\dihedral{2k_{0}} \quad & \textit{if\ } k_{0} \text{\ is finite;}
\end{cases}
$$
and 
$$ 
\widetilde{\dihedral{}} \cap \IsomslpmJ = \widetilde{\dihedral{}} \cap \sunit{n}^{\pm} = \langle \tbartilde, \ttilde_{2\pthat} \rangle.$$
The structure of $\Symtilde(X_{\tau})$ is given as follows:
\begin{itemize}
\item[(i)] If $k_{0}$, the order of the rotational period, is infinite or odd then 
$$\widetilde{\dihedral{}} \cap \orth{p} \times \orth{p} = (\Id)$$
and 
\begin{equation}
\addtocounter{theorem}{1}
\label{E:symtilde:xtau:p:eq:q}
\Symtilde(X_\tau) \cong \orth{p} \times \orth{p} \rtimes_{\widetilde{\rho}} \widetilde{\dihedral{}},
\end{equation}
where the twisting homomorphism $\widetilde{\rho}: \widetilde{\dihedral{}} \ra \Aut{\orth{p}\times\orth{p}}$ is given by
$$\widetilde{\rho}(\gamma) =
\begin{cases}
\Id & \text{if $\gamma \in \widetilde{\dihedral{}}$ is a word containing an even number of copies of $\tbartilde$},\\
\EEE'  & \text{if $\gamma \in \widetilde{\dihedral{}}$ is a word containing an odd number of copies of $\tbartilde$},
\end{cases}
$$
where $\EEE'$ is the involution defined in \ref{E:exchange'}.
\item[(ii)] If $k_{0}$ is even then 
$$\widetilde{\dihedral{}} \cap \orth{p} \times \orth{p} = \langle \ttilde_{k_{0}\pthat} \rangle = \langle -\Id\rangle \cong \Z_{2}.$$
Furthermore, $\langle \ttilde_{k_{0}\pthat} \rangle = \langle -\Id\rangle \cong \Z_{2}$ belongs to the centre of $\Symtilde(X_{\tau})$.
\end{itemize}
\end{lemma}
\noindent
N.B.  for $k_{0}$ finite, $\widetilde{\dihedral{}} \cong \dihedral{2k_{0}}$ not $\dihedral{k_{0}}$ as in the cases $p=1$ and $p \neq q$.
\begin{proof}
Recalling again the presentations $\langle x, y \, |\, x^{2}=y^{2}=(xy)^{k}=1\rangle$ and
$\langle x, y \, |\, x^{2}=y^{2}=1\rangle$ for the finite dihedral group $\dihedral{k}$ and infinite dihedral group $\dihedral{\infty}$ respectively, 
we see that $\widetilde{\dihedral{}}$ is isomorphic to a finite or infinite dihedral group depending on whether or not $(\tbartildeplus \circ \tbartilde)^{k}=\Id$ 
for some $k\in \Z$. From \ref{P:commute:p:eq:q}.i and \ref{P:commute:p:neq:q}.ii we have 
\begin{equation}
\addtocounter{theorem}{1}
\label{E:tbplus:tb}
 (\tbartildeplus \circ \tbartilde)^{2} = \tbartildeplus \circ \tbartildeminus = \ttilde_{2\pthat}.
\end{equation}
Hence if $k_{0}$ is finite then $ (\tbartildeplus \circ \tbartilde)^{2k_{0}}=\Id$ by the definition of $k_{0}$. If $k$ is any even natural number less than $k_{0}$ then 
$ (\tbartildeplus \circ \tbartilde)^{k}\neq \Id$ (again by the definition of $k_{0}$). 
By \ref{P:commute:p:neq:q}.iv and \ref{P:commute:p:eq:q}.iii $\tbartilde$ and $\tbartildeplus$ are 
holomorphic and antiholomorphic respectively. Hence 
$ (\tbartildeplus \circ \tbartilde)^{k}$  is antiholomorphic for any odd integer $k$ and so cannot be the identity. 
Therefore $(\tbartildeplus \circ \tbartilde)$ has order exactly $2k_{0}$ as claimed. 
If $k_{0}$ is infinite then $\tbartildeplus \circ \tbartilde$ cannot have finite order (the order cannot be odd by the antiholomorphic argument above and 
by \ref{E:tbplus:tb} an even order would imply $k_{0}$ is finite) and hence $\widetilde{\dihedral{}} \cong \dihedral{\infty}$ as claimed.

Any nontrivial element in $\widetilde{\dihedral{}}$ 
can be written as an alternating word in its two generators $\tbartilde$ and $\tbartildeplus$.
An element of $\widetilde{\dihedral{}}$ acts holomorphically if and only if it contains the antiholomorphic isometry 
$\tbartildeplus$ an even number of times. 
Hence a holomorphic isometry in $\widetilde{\dihedral{}}$ has 
either (a) an even number of both generators or (b) an even number of $\tbartildeplus$s and an odd number 
of $\tbartilde$s. In case (b) any such element is equal to $\ttilde_{2k\pthat}\circ \tbartilde$ for some $k\in \Z$.   
In case (a), by \ref{E:tbplus:tb} any such word is of the form $\ttilde_{2k\pthat}$ for some $k\in \Z$. 
Hence $\widetilde{\dihedral{}} \cap \IsomslpmJ  = \langle \tbartilde, \ttilde_{2\pthat} \rangle$ as claimed. 
  
By considering the conjugation action of $\widetilde{\dihedral{}}$ on 
$\orth{p} \times \orth{p} \subset \orth{2p}$ we see that 
the set $\widetilde{\dihedral{}} \cdot \orth{p} \times \orth{p}$ coincides with the set $\orth{p} \times \orth{p} \cdot \widetilde{\dihedral{}}$, and
hence forms a group, which coincides with $\Symtilde(X_\tau)$.
It is easy to see that $\orth{p} \times \orth{p}$ is a normal subgroup of $\Symtilde(X_\tau)=\widetilde{\dihedral{}} \cdot \orth{p} \times \orth{p}$.
It remains only to analyse the intersection $\widetilde{\dihedral{}}\cap \orth{p} \times \orth{p}$. 
As in the previous cases $\widetilde{\dihedral{}}\cap \orth{p} \times \orth{p} \subset \widetilde{\dihedral{}}\cap \IsomslpmJ = 
\langle \tbartilde, \ttilde_{2\pthat} \rangle$.  
Written in block diagonal form (using \ref{E:ttilde} and \ref{E:tbartilde:p:eq:q}) $\ttilde_{2k\pthat}\circ \tbartilde$ is purely off-diagonal 
and hence not contained in $\orth{p}\times \orth{p}$ for any $k\in \Z$. 
Arguing as in the $p=1$ and $p\neq q$ cases 
we find that $\ttilde_{2k\pthat} \in \orth{p} \times \orth{p}$ if and only if $\tcheck_{2k\pthat} = \pm \Id$ and
hence by \ref{L:w:half:period:a} $\widetilde{\dihedral{}} \cap \orth{p}\times \orth{p} \neq (\Id)$ if and only if 
$\bw_{\tau}$ admits strict half-periods of type $(--)$. Thus by \ref{L:w:half:period} 
$\widetilde{\dihedral{}} \cap \orth{p}\times \orth{p} =(\Id)$ if $k_{0}$ is infinite or odd. 
If $k_{0}$ is even then any odd multiple of $k_{0}\pt$ is a strict half-period of type $(--)$.
Therefore in this case $\widetilde{\dihedral{}} \cap \orth{p}\times \orth{p} = \langle \ttilde_{k_{0}\pthat} \rangle = \langle -\Id\,
\rangle \cong \Z_{2}$. It is easy to see that $\langle \ttilde_{k_{0}\pthat} \rangle = \langle -\Id\, \rangle$ belongs to the centre of 
both $\widetilde{\dihedral{}}$ and $\orth{p}\times \orth{p}$ and hence also to the centre of $\Symtilde(X_{\tau}) = \widetilde{\dihedral{}}\cdot \orth{p}\times \orth{p}$.
\end{proof}

\begin{remark}
\addtocounter{equation}{1}
\label{R:pm:holo:p:eq:q}
Lemma \ref{L:symtilde:p:eq:q} implies that
every element $\mtilde \in \Symtilde(X_\tau)$ can be written (non-uniquely) in the form
$$ \mtilde = \DDD \circ \MMM,$$ 
where $\DDD \in \widetilde{\dihedral{}}$ is an alternating word in $\tbartilde$ and $\tbartildeplus$ 
and $\MMM \in \orth{p} \times \orth{p} \subset \sunit{n}^{\pm}$.
In particular, every element of $\Symtilde(X_{\tau})$ belongs to $\Isomslpm$ and hence acts on $\C^{n}$ 
$\pm$-holomorphically. The subgroup of holomorphic symmetries of $X_{\tau}$ is  $\Symtilde(X_{\tau}) \cap 
\IsomslpmJ = \langle \tbartilde, \ttilde_{2\pthat} \rangle \cdot \orth{p} \times \orth{p}$.   
\end{remark}

\begin{corollary}[Structure of $\Per(X_{\tau})$ for $p=q$]
\addtocounter{equation}{1}
\label{C:xtau:period:p:eq:q}
For $0<\abs{\tau}<\taumax$, we have
\begin{equation*}
\Per(X_{\tau}) = 
\begin{cases}
(\Id) & \text{if\ }k_{0}=\infty;\\
\langle \TTT_{2k_{0}\pt} \rangle & \text{if\ } k_{0} \text{\ is odd;}\\
\langle \TTT_{k_{0}\pt} \circ (-\Id_{\Sph^{p-1}}, -\Id_{\Sph^{p-1}})\rangle & \text{if\ } k_{0} \text{\ is even,}
\end{cases}
\end{equation*}
\end{corollary}
\noindent
The proof follows from  Lemma \ref{L:symtilde:p:eq:q} as in the cases $p=1$ and $p>1$, $p \neq q$.

\section{Geometric features of $X_{\tau}$}
\label{S:xtau:limit}
This section describes various geometric features of $X_{\tau}$ with particular emphasis on 
its geometry as $\tau \ra 0$, the action of $\Sym(X_{\tau})$ on various subdomains of $\cylpq$ 
and the action of $\Symtilde(X_{\tau})$ on various equatorial spheres associated with $X_{\tau}$.

\subsection*{Waists, bulges and approximating spheres} \phantom{ab}
In this section we describe distinguished subsets of $\cylpq$ called the \emph{waists} and \emph{bulges} of $X_{\tau}$ 
and describe the action of $\Sym(X_{\tau})$ on these subsets. We also attach to each bulge a $p+q-1$ dimensional 
 equatorial subsphere 
of $\Sph^{2(p+q)-1}$, called the \emph{approximating sphere} of the bulge and describe symmetries associated with 
these approximating spheres.
The terminology approximating sphere is justified by \ref{E:bw:Xtau} where we show  
that for $\tau$ sufficiently close to $0$ the image of each bulge under $X_{\tau}$ is close to its 
approximating sphere.

Fix admissible integers $p$ and $q$ and let $X_{\tau}: \cylpq \ra \Sph^{2(p+q)-1}$ be the $1$-parameter family of $\sorth{p}\times \sorth{q}$-equivariant special Legendrian immersions defined in \ref{D:X:tau} and $g_{\tau}$ denote the pullback 
metric on $\cylpq$ induced by $X_{\tau}$. 
Throughout this section we assume that $\abs{\tau}<\taumax$.
\begin{definition}
\addtocounter{equation}{1}
\label{D:waist}
A \emph{waist} of  $(\cylpq, g_{\tau})$ is a meridian $\{t\} \times \merpq$ of $\cylpq$ on which the radius of one spherical factor of the meridian is minimal. 
\end{definition}
\subsubsection*{Waists for $p=1$}\phantom{ab}
If $p=1$ then a waist is any meridian $\{t\} \times \Sph^{n-2}$ such that 
$y_{\tau}(t)= y_{min}$. Recall from \ref{E:y0:p:eq:1} that our choice of initial conditions for $\bw_\tau$ in the case $p=1$ forces $y_\tau$ to have 
a maximum at $t=0$ and a minimum at $t=\pt$. Hence using the symmetries of $y_\tau$ described in \ref{E:y:sym:p:eq:1}
 $y_\tau$ has maxima at precisely $2k\pt$ and minima at precisely $(2k+1)\pt$ for each $k\in \Z$.
See Figure \ref{fig:w2:pequals1} for an illustration.
Therefore the meridian $\{t\} \times \Sph^{n-2}$ is a waist of $X_{\tau}$ if and only if $t \in (2\Z+1) \pt$. 
For any $k\in \Z$ we define the $k$th waist $W[k]$ of $\cylone$ to be 
\begin{equation}
\addtocounter{theorem}{1}
\label{E:waist:k:p:eq:1}
W[k] = \{ (2k-1)\pt \} \times \Sph^{n-2}.
\end{equation}

\subsubsection*{Waists for $p>1$}\phantom{ab}
If $p>1$ then a waist is any meridian $\{t\} \times \merpq$ such that 
either $y_{\tau}(t)= y_{min}$ or $y_{\tau}(t)=y_{max}$; we call a waist on which $y(t)=\ymax$ a \emph{waist of type $1$}, 
since it is the radius of the first spherical factor $\Sph^{p-1}$ which is minimal on such a waist. 
Similarly, a waist on which $y_{\tau}(t)=\ymin$ is called a \emph{waist of type $2$}, since the radius of the second spherical 
factor $\Sph^{q-1}$ is minimal on such a waist. 
Recall from \ref{E:y:max:min} 
that $y_\tau$ attains a maximum at $t=-\pt^-$, a minimum at $t=\pt^+$,  
is decreasing on $(-\pt^-,\pt^+)$ and increasing on $(\pt^+, \pt + \pt^+)$---see Figure \ref{fig:w2:pneq1}.
Hence $\{t\} \times \merpq$ is a waist of type $1$ if and only if 
$t +\pt^{-}  \in 2\pt \Z$ and a waist of type $2$ if and only if $t - \pt^{+} \in 2\pt \Z$.
For any $k\in \Z$ we define the $k$th waist $W[k]$ of $\cylpq$ by
\begin{subequations}
\addtocounter{theorem}{1}
\label{E:waist:k:p:neq:1}
\begin{align}
W[2l+1] &:= \{2l\pt + \pt^{+}\} \times \merpq \quad \text{if $k=2l+1$ some $l\in \Z$};\\
W[2l] &:= \{2l\pt -\pt^{-}\} \times \merpq \quad \text{if $k=2l$ \quad \ \ some $l\in \Z$}.
\end{align}
\end{subequations}

\subsubsection*{The action of $\Isom(\cylpq,g_{\tau})$ on waists}\phantom{ab}
Any element of $\Isom(\cylpq,g_{\tau})$  permutes the waists of $\cylpq$. If $p=1$ then all waists are isometric and 
$\Isom(\cylone,g_{\tau})$ acts transitively on the set of waists. 
If $p>1$ and $p \neq q$ then waists of type $1$ and type $2$ are not isometric and therefore $\Isom(\cylpq, g_{\tau})$ 
cannot act transitively on the set of all waists. However, all waists of fixed type  are isometric 
and $\Isom(\cylpq, g_{\tau})$ acts transitively on the set of waists of fixed type.

If $p>1$ and $p=q$, waists of type $1$ and type $2$ are isometric (recall \ref{E:ymin:ymax}). 
Recall from \ref{P:isom:pb}.iii that $\Isom(\cylpp,g_{\tau}) = \dihedral{} \cdot \orth{p} \times \orth{p}$ 
where $\dihedral{} = \langle \tbar \circ \EEE, \, \tbar_{\pt/2} \rangle$.
Isometries containing an even number of copies of $\tbar \circ \EEE$ preserve the type of any waist, while 
isometries containing an odd number of copies of $\tbar \circ \EEE$ exchange the two types of waist. 
The full group $\Isom(\cylpp,g_{\tau})$ acts transitively on the set of all waists. 

\subsubsection*{Bulges}\phantom{ab}
The set of all waists $W$ of $(\cylpq,g_{\tau})$ is a hypersurface with countably many components $W[k]$ ($k \in \Z$) 
and the complement of $W$ in $\cylpq$ has countably many components. 
\begin{definition}
\addtocounter{equation}{1}
\label{D:bulges}
A \emph{bulge} of $(\cylpq,g_{\tau})$ is a connected component of $(\cylpq \setminus W)$. For any $k\in \Z$ 
the $k$th bulge $\hat{S}[k]$ of $(\cylpq,g_{\tau})$ is the unique connected component of $(\cylpq \setminus W)$ whose 
closure contains the two consecutive waists $W[k]$ and $W[k+1]$. We call $W[k]$ and $W[k+1]$ the \emph{boundary waists} 
of the bulge $\hat{S}[k]$.
\end{definition}
\noindent
More concretely, for $p=1$ the $k$th bulge of $\cylone$ is
\begin{equation}
\addtocounter{theorem}{1}
\label{E:bulge:k:p:eq:1}
\hat{S}[k] = (\,(2k-1)\pt, (2k+1) \pt\, ) \times \Sph^{n-2} = \TTT_{2k\pt} \hat{S}[0],
\end{equation}
while for $p>1$ the $k$th bulge of $\cylpq$ is 
\begin{subequations}
\addtocounter{theorem}{1}
\label{E:bulge:k:p:neq:1}
\begin{align}
\hat{S}[2l] &= (-\pt^{-}+2l \pt, -\pt^{-}+(2l+1)\pt) \times \merpq = \TTT_{2l \pt} \hat{S}[0] \quad 
\text{if $k=2l$};\\ 
\hat{S}[2l+1] &= (\pt^{+}+2l \pt, \pt^{+}+(2l+1)\pt) \times \merpq = \TTT_{2l \pt} \hat{S}[1] \quad 
\text{if $k=2l+1$}.
\end{align}
\end{subequations}
Since any isometry in $\Isom(\cylpq, g_{\tau}) = \Sym(X_{\tau})$ permutes the waists of $(\cylpq,g_{\tau})$ 
it also permutes the bulges of $\cylpq$.  Moreover $\Isom(\cylpq, g_{\tau})$ acts transitively on the set of all bulges.

\begin{definition}
\addtocounter{equation}{1}
\label{D:sym:k}
For any $k\in \Z$ we define $\Sym_{k}(X_{\tau})$ to be the subgroup of $\Sym(X_{\tau})=\Isom(\cylpq, g_{\tau})$ 
leaving the $k$th bulge $\hat{S}[k]$ invariant.
\end{definition}
Since $\Isom(\cylpq,g_{\tau})=\Sym(X_{\tau})$ acts transitively on the set of all bulges the subgroups $\Sym_{k}(X_{\tau})$ 
are all conjugate in $\Sym(X_{\tau})$ . In particular they are all isomorphic as groups.
\begin{lemma}[Structure of $\Sym_{k}(X_{\tau})$; cf. Lemma \ref{L:symtilde:app:sph}]
\addtocounter{equation}{1}
\label{L:sym:k}
For any fixed $k\in \Z$ we have
$$
\Sym_{k}(X_{\tau}) = 
\begin{cases}
\langle \tbar_{2k\pt} \rangle \cdot \orth{n-1} \cong \orth{1} \times \orth{n-1} \quad & \text{if $p=1$;}\\
\orth{p} \times \orth{q} \quad & \text{if $p>1$ and $p \neq q$;}\\
\langle \tbar_{k\pt} \circ \EEE \rangle \cdot \orth{p} \times \orth{p} \cong \orth{p} \times \orth{p} \rtimes \Z_{2} 
\quad & \text{if $p>1$ and $p=q$.}
\end{cases}
$$
\end{lemma}
\begin{proof}
An element of $\Sym(X_{\tau})$ belongs to $\Sym_{k}(X_{\tau})$ if and only if leaves invariant 
the union of the two boundary waists $W[k]$ and $W[k+1]$.
The lemma now follows using the structure of $\Isom(\cylpq, g_{\tau})= \Sym(X_{\tau})$ established in \ref{P:isom:pb} 
to determine its action on the set of waists $W$. See also the proof of Lemma \ref{L:symtilde:app:sph} for a more detailed 
proof of the analogous result in the case of the action of $\Symtilde(X_{\tau})$ on the approximating spheres.
\end{proof}


\subsubsection*{$(p,q)$-marked special Legendrian spheres and approximating spheres}\phantom{ab}
Fix admissible integers $p$ and $q$. We now define the important concept of a \emph{$(p,q)$-marked special Legendrian sphere}.
We will see shortly that we can associate a $(p,q)$-marked SL sphere $\Sph[k]$ 
to every bulge of $X_{\tau}: \cylpq \ra \Sph^{2(p+q)-1}$.
Moreover, for $\tau$ sufficiently small the image of the $k$th bulge $\hat{S}[k]$ under $X_{\tau}$ 
is close to the marked SL sphere $\Sph[k]$.

\begin{definition}
\addtocounter{equation}{1}
\label{D:marked:sphere}
If $(p,q)=(1,n-1)$ then  a \emph{$(p,q)$-marked SL sphere} is a pair $\{\pm e, \Sph\}$ 
consisting of an equatorial $n-1$ sphere $\Sph$ of $\Sph^{2n-1}$ which is special Legendrian (for the correct orientation)
and a pair of antipodal points $\pm e \in \Sph$.
We call $\pm e \subset \Sph$ the \emph{attachment set} (or alternatively the \emph{marked set}) 
of the $(1,n-1)$-marked SL sphere $(\pm e, \Sph)$.

If $p>1$ then a \emph{$(p,q)$-marked SL sphere} is a triple consisting of an equatorial special Legendrian 
(for the correct orientation) 
 $p+q-1$ sphere $\Sph$ 
of $\Sph^{2(p+q)-1}$, an equatorial subsphere $\Sph_{p-1} \subset \Sph$ of dimension $p-1$ 
and the orthogonal equatorial subsphere $\Sph_{q-1} \subset \Sph$. 
We call $\Sph_{p-1}\cup \Sph_{q-1} \subset \Sph$ the \emph{attachment set} or \emph{marked set} 
of the $(p,q)$-marked SL sphere $(\Sph_{p-1},\Sph_{q-1},\Sph)$. 
\end{definition}

A $(1,n-1)$-marked SL sphere is equivalent to a pair $\{l ,\Pi_{n}\}$ where $\Pi_{n} \subset \C^{n}$ is a special Lagrangian $n$-plane 
and $l = \langle \pm e \rangle \subset \Pi_{n}$ is an unoriented real line in the $n$-plane $\Pi_{n}$.
We call $\Sph^{n-1} \subset \R^{n} \subset \C^{n}$ the \emph{standard special Legendrian sphere} 
and $\{\pm e_{1}, \Sph^{n-1}\}$ \emph{the standard $(1,n-1)$-marked special Legendrian sphere}, 
where $e_{1}, \ldots , e_{n}$ is the standard oriented orthonormal basis of $\R^{n}$.

For $p>1$ a $(p,q)$-marked SL sphere is equivalent to a triple $\{\Pi_{p},\Pi_{q}=\Pi_{p}^{\perp},\Pi_{p+q}\}$ 
where $\Pi_{p+q} \subset \C^{p+q}$ 
is a special Lagrangian $p+q$-plane and $\Pi_{p} \subset \Pi_{p+q}$ is a real $p$-plane in $\Pi_{p+q}$ 
and $\Pi_{q}= \Pi_{p}^{\perp} \subset \Pi_{p+q}$. 
The $(p,q)$-marked SL sphere with $\Pi_{p+q} = \R^{p+q} \subset \C^{p+q}$,  
$\Pi_{p} = \R^{p} \times \{0\} \subset \Pi_{p+q}$ 
and $\Pi_{q} = \{0\} \times \R^{q} = \Pi_{p}^{\perp} \subset \Pi_{p+q}$ 
we call  the \emph{standard $(p,q)$-marked SL sphere}. 
The choice of a real $p$-plane $\Pi_{p} \subset \Pi_{p+q}$ determines the $q$-plane $\Pi_{q}$ 
as the orthogonal complement of $\Pi_{p}$ inside $\Pi_{p+q}$.

\medskip

Fix admissible integers $p$ and $q$ and $\tau$ satisfying $\abs{\tau} < \taumax$. 
To each bulge $\hat{S}[k]$ of $X_{\tau}$ we now associate a $(p,q)$-marked SL sphere 
called its approximating (marked) sphere $\Sph[k]$. 
\begin{definition}
\addtocounter{equation}{1}
\label{D:approx:sphere}
For each $k\in \Z$ we define a $(p,q)$-marked sphere $\Sph[k]$ as follows.
For $k=0$ we define $\Sph[0]$ to be the standard $(p,q)$-marked SL sphere defined following \ref{D:marked:sphere}.\\
For $p=1$ we define
\begin{equation}
\addtocounter{theorem}{1}
\label{E:app:sph:k:p1}
\Sph[k]: =\ttilde_{2k\pthat} \Sph[0] \qquad \quad \text{if $k \neq 0$}.
\end{equation}
For $p>1$ we define
\begin{equation}
\addtocounter{theorem}{1}
\label{E:app:sph:k:pgt1}
\Sph[k]: = 
\begin{cases}
\ttilde_{2l\pthat} \,\Sph[0] \quad & \text{if $k=2l$};\\
\ttilde_{2l \pthat} \circ \tbartildeplus \,\Sph[0]  \quad & \text{if $k=2l+1$}.
\end{cases}
\end{equation}
$\Sph[k]$ is called the \emph{approximating $(p,q)$-marked sphere} (or more simply the approximating sphere) 
associated with the $k$th bulge $\hat{S}[k]$ of $X_{\tau}$. 
\end{definition}

Note that since $\ttilde_{x} \in \sunit{n} \subset \Isomsl$ for all $x$, for $p=1$ if we orient 
the central marked sphere $\Sph[0]$ so that it is special Legendrian then the orientation 
$\Sph[k] = \ttilde_{2k\pthat} \Sph[0]$ 
inherits from $\Sph[0]$  via $\ttilde_{2k\pthat}$ also makes $\Sph[0]$ special Legendrian.
However, for $p>1$ recall from \ref{P:commute:p:neq:q} that $\tbartilde_{+} \in \Isomslpm \setminus \Isomsl$;  
this occurs because the corresponding symmetry $\tbar_{\pt^{+}} \in \Sym(X_{\tau}) \subset \Diff(\cylpq)$ 
reverses orientation on $\cylpq$. Hence if we orient the central marked sphere $\Sph[0]$ so that it is 
special Legendrian then the orientation inherited on any odd approximating sphere 
$\Sph[2l+1]= \ttilde_{2l\pthat} \circ \tbartilde_{+}\Sph[0]$ from $\Sph[0]$ makes it anti-special Legendrian.

\begin{lemma}[Action of $\Symtilde(X_{\tau})$ on the approximating spheres; cf. Lemma \ref{L:sym:k}]
\addtocounter{equation}{1}
\label{L:symtilde:app:sph}
\phantom{ab} \hfill
\begin{enumerate}
\item[(i)] $\Symtilde(X_{\tau})$ acts transitively on the approximating marked spheres of $X_{\tau}$
\item[(ii)] $\Symtilde_{k}(X_{\tau})$, the subgroup of $\Symtilde(X_{\tau})$ leaving the $k$th approximating sphere 
$\Sph[k]$ invariant is
\begin{equation}
\addtocounter{theorem}{1}
\label{E:symtilde:k}
\Symtilde_{k}(X_{\tau}) = 
\begin{cases}
\langle \tbartilde_{2k\pt} \rangle \cdot \orth{n-1} \quad &\text{if $p=1$};\\
\orth{p} \times \orth{q} \quad & \text{if $p>1$ and $p\neq q$};\\
\langle \tbartilde_{k\pt} \rangle \cdot \orth{p} \times \orth{p} \quad & \text{if $p>1$ and $p=q$}.
\end{cases}
\end{equation}
\end{enumerate}
\end{lemma}

\begin{proof}
The main point of (i) is to verify that $\Symtilde(X_{\tau})$ maps every approximating sphere to another approximating sphere. 
Transitivity of the action then follows immediately from the definitions \ref{E:app:sph:k:p1} and \ref{E:app:sph:k:pgt1} 
and the structure of $\Symtilde(X_{\tau})$ already established. The proof that $\Symtilde(X_{\tau})$ permutes the
approximating spheres is straightforward using the definitions \ref{E:app:sph:k:p1} and \ref{E:app:sph:k:pgt1}, 
the structure results for $\Symtilde(X_{\tau})$ and the properties of the generators of $\widetilde{\dihedral{}}$ 
described in \ref{P:commute:p:eq:1}, \ref{P:commute:p:neq:q} and \ref{P:commute:p:eq:q} (for the three cases 
$p=1$, $p>1$ and $p \neq q$ and $p>1$ and $p=q$ respectively). For completeness, we give the details below 
since we use these results in the proof of part (ii).

(i) Case $p=1$: Recall from \ref{L:symtilde:p:eq:1} that $\Symtilde(X_{\tau}) = \widetilde{\dihedral{}} \cdot \orth{n-1}$ 
where $\widetilde{\dihedral{}} = \langle \tbartilde, \ttilde_{2\pthat} \rangle$.
First we show that $\MMM \,\Sph[k] = \Sph[k]$ for any $k \in \Z$ and any $\MMM \in \orth{n-1}$.
Clearly, $\MMM \in \orth{n-1} \subset \orth{n}$ preserves the standard special Legendrian sphere $\Sph^{n-1}\subset \R^{n}\subset \C^{n}$ and the vector $e_{1}\in \Sph^{n-1}$ and hence the approximating  marked sphere
$\Sph[0]$. Hence  using definition \ref{E:app:sph:k:p1} and \ref{P:commute:p:eq:1}.ii
we have
\begin{equation}
\addtocounter{theorem}{1}
\label{E:app:sph:orth:p1}
\MMM\, \Sph[k]  = \MMM \circ \ttilde_{2k\pthat}\, \Sph[0] = \ttilde_{2k\pthat} \circ \MMM \,\Sph[0] = \ttilde_{2k\pthat}\, \Sph[0] = \Sph[k], \quad \forall k \in \Z, \ \MMM \in \orth{n-1}.
\end{equation}
Using the structure of  $\Symtilde(X_{\tau})$ recalled above it suffices to show that $\tbartilde$ and 
$\ttilde_{2\pthat}$, the generators of $\widetilde{\dihedral{}}$, permute the approximating spheres.
From the definition of $\Sph[k]$  we have 
\begin{equation}
\addtocounter{theorem}{1}
\label{E:app:sph:act:per:p1}
\ttilde_{2\pthat}\, \Sph[k] = \ttilde_{2\pthat} \circ \ttilde_{2k\pthat}\, \Sph[0] = \Sph[k+1] 
\quad \text{for any $k \in \Z$.}
\end{equation}
From the definition of $\tbartilde$ (in \ref{E:tbartilde:p:eq:1}) we see immediately that it leaves the
approximating marked sphere $\Sph[0]$ invariant (it exchanges the two points in the marked set $\pm e_{1}$). 
Hence we have 
\begin{equation}
\addtocounter{theorem}{1}
\label{E:app:sph:tb:p1}
\tbartilde \, \Sph[k] = \tbartilde \circ \ttilde_{2k\pthat} \, \Sph[0] = \ttilde_{-2k\pthat} \circ \tbartilde \, \Sph[0] = 
\ttilde_{-2k\pthat} \, \Sph[0] = \Sph[-k], \quad \text{ for any $k\in \Z$},
\end{equation}
where we have used \ref{P:commute:p:eq:1}(i) to obtain the second equality.

Case $p>1$ and $p \neq q$: Recall from \ref{L:symtilde:p:neq:q} that $\Symtilde(X_{\tau}) = \widetilde{\dihedral{}} \cdot \orth{p} \times \orth{q}$ where $\widetilde{\dihedral{}} = \langle \tbartildeplus, \tbartildeminus \rangle$ (defined in 
\ref{E:tbartildeplus} and \ref{E:tbartildeminus} respectively).
Clearly, $\orth{p} \times \orth{q} \subset \orth{p+q}$ preserves the standard $(p,q)$-marked sphere and hence 
the approximating  marked sphere $\Sph[0]$. 
Using definition \ref{E:app:sph:act:per:pgt1} and \ref{P:commute:p:neq:q}.iii 
it follows that any $\MMM \in \orth{p} \times \orth{q}$ leaves every approximating sphere  invariant. 
As in the case $p=1$ it suffices now to exhibit the action of the generators 
$\tbartildeplus$ and $\tbartildeminus$ of $\widetilde{\dihedral{}}$ on the approximating spheres.

We claim that for any $k\in \Z$ we have
\begin{subequations}
\addtocounter{theorem}{1}
\label{E:app:sph:act:pneq}
\begin{align}
\label{E:app:sph:tbp}
\tbartildeplus \, \Sph[k] &= \Sph[1-k], \\
\label{E:app:sph:tbm}
\tbartildeminus \, \Sph[k] &= \Sph[-1-k].
\end{align}
\end{subequations}

To prove \ref{E:app:sph:tbp} we consider the cases $k$ even and odd separately. 
For even $k=2l$  using the definition of $\Sph[k]$ given in \ref{E:app:sph:k:pgt1} we have
$$\tbartildeplus \, \Sph[2l]= \tbartildeplus \circ \ttilde_{2l\pthat}\, \Sph[0] = \ttilde_{-2l\pthat} \circ \tbartildeplus\, \Sph[0] = \ttilde_{-2l\pthat} \, \Sph[1] = \Sph[1-2l],$$
where we used \ref{P:commute:p:neq:q}(i) in the second equality.
In a similar way for odd $k=2l+1$ we obtain
$$ \tbartildeplus \, \Sph[2l+1] = \tbartildeplus \circ \ttilde_{2l\pthat} \, \Sph[1] = \ttilde_{-2l\pthat} \circ \tbartildeplus
\, \Sph[1] = \ttilde_{-2l\pthat}\, \Sph[0] = \Sph[-2l],$$
where we used \ref{P:commute:p:neq:q}(i) in the second equality and the fact that $\tbartildeplus \, \Sph[1] = \Sph[0]$ 
(because by definition $\Sph[1] = \tbartildeplus \Sph[0]$ and $\tbartildeplus$ is an involution).
Using the fact that by \ref{P:commute:p:neq:q}(i) $\tbartildeminus \circ \ttilde_{x} = \ttilde_{-x} \circ \tbartildeminus$ 
the proof of \ref{E:app:sph:tbm} is almost identical to the proof of \ref{E:app:sph:tbp} given above.
\begin{remark}
\addtocounter{equation}{1}
\label{R:app:sph:per:pgt1}
It follows immediately from \ref{E:app:sph:act:pneq} and \ref{P:commute:p:neq:q}(ii) that 
\begin{equation}
\addtocounter{theorem}{1}
\label{E:app:sph:act:per:pgt1}
\ttilde_{2\pthat}\, \Sph[k] = \tbartildeplus \circ \tbartildeminus \, \Sph[k] = \Sph[k+2] \quad \text{for any $k\in \Z$},
\end{equation}
i.e. the rotational period $\ttilde_{2\pthat}$ maps $\Sph[k] \mapsto \Sph[k+2]$
unlike the case $p=1$ (recall \ref{E:app:sph:act:per:p1}) where it maps $\Sph[k] \mapsto \Sph[k+1]$. This reflects a fundamental 
difference in the geometry of $X_{\tau}$ in the cases $p=1$ and $p>1$.
\end{remark}

Case $p>1$ and $p=q$: recall from \ref{L:symtilde:p:eq:q} that $\Symtilde(X_{\tau}) = \widetilde{\dihedral{}} \cdot \orth{p} \times \orth{p}$ where $\widetilde{\dihedral{}} = \langle \tbartildeplus, \tbartilde \rangle$ with 
$\tbartilde$ as defined in \ref{E:tbartilde:p:eq:q}. Given the results of the previous part it suffices to exhibit the action of 
$\tbartilde$ on the approximating spheres. 
We want to prove that
\begin{equation}
\addtocounter{theorem}{1}
\label{E:app:sph:tb:peq}
\tbartilde \, \Sph[k] = \Sph[-k] \quad \text{ for any $k\in \Z$}.
\end{equation}
It follows immediately from \ref{E:tbartilde:p:eq:q} that $\tbartilde$ 
preserves the standard $(p,p)$-marked sphere (but exchanges the two components of the marked set) 
and therefore the approximating marked sphere $\Sph[0]$. 
Using the invariance of $\Sph[0]$ under  $\tbartilde$ we see that for even  $k=2l$  then by \ref{E:app:sph:k:pgt1} we have 
$$ \tbartilde \, \Sph[2l] = \tbartilde \circ \ttilde_{2l\pthat} \, \Sph[0] = \ttilde_{-2l\pthat} \circ \tbartilde \, \Sph[0] = 
\ttilde_{-2l\pthat}\, \Sph[0] = \Sph[-2l],$$
where we used \ref{P:commute:p:eq:q}(i) to obtain the second equality.
For $k=1$ we have
$$ \tbartilde \, \Sph[1] = \tbartilde \circ \tbartildeplus \, \Sph[0] = \tbartildeminus \circ \tbartilde \, \Sph[0] = \tbartildeminus \Sph[0] = \ttilde_{-2\pthat} \circ \tbartildeplus \, \Sph[0] = \Sph[-1],$$
where we have used \ref{P:commute:p:eq:q}(i) and \ref{P:commute:p:neq:q}(ii) to obtain the second and fourth inequalities
respectively. From this we see more generally that for odd $k=2l+1$ 
$$ \tbartilde\, \Sph[2l+1] = \tbartilde \circ \ttilde_{2l\pthat} \,\Sph[1] = \ttilde_{-2l\pthat} \circ \tbartilde \,\Sph[1] = 
\ttilde_{-2l\pthat} \circ \Sph[-1] = \Sph[-1-2l].$$

(ii) Case $p=1$: any element in $\Symtilde(X_{\tau})$ can be written in the form 
$ \ttilde_{2\pt}^{i} \circ \tbartilde^{j} \circ \MMM$ for some $i\in \Z$, $j\in \Z_{2}$ and $\MMM \in \orth{n-1}$.
Using \ref{E:app:sph:orth:p1}, \ref{E:app:sph:act:per:p1} and \ref{E:app:sph:tb:p1} we have 
$$
\ttilde_{2\pthat}^{i} \circ \tbartilde\phantom{}^{j} \circ \MMM \, \Sph[k] = \Sph[i+(-1)^{j}k] \quad \text{for any $k\in \Z$}.
$$
It follows that for any fixed $k\in \Z$, $\Symtilde_{k}(X_{\tau}) =\langle \tbartilde_{2k\pthat} \rangle  \cdot\orth{n-1}$ as claimed.

Case $p>1$ and $p \neq q$: any element in $\Symtilde(X_{\tau})$ can be written in the form
$ \ttilde_{2\pt}^{i} \circ \tbartildeplus^{j} \circ \MMM$ for some $i\in \Z$,  $j \in \Z_{2}$ and 
$\MMM \in \orth{p} \times \orth{q}$.
Using \ref{E:app:sph:tbp} and \ref{E:app:sph:act:per:pgt1} we have
$$
\ttilde_{2\pthat}^{i} \circ \tbartildeplus\phantom{}^{j} \circ \MMM \, \Sph[k] = 
\begin{cases}
\Sph[k+2i] \quad & \text{if $j=0$};\\
\Sph[1-k+2i] \quad & \text{if $j=1$}.
\end{cases}
$$
It follows that for any $k\in \Z$, $\Symtilde_{k}(X_{\tau}) = \orth{p} \times \orth{q}$ as claimed.

Case $p>1$ and $p=q$: 
Since any element $\mtilde \in \Symtilde(X_{\tau})$ acts on the approximating spheres it therefore defines a map 
$\widetilde{m}: \Z \ra \Z$ (in fact $\widetilde{m} \in \Isom{\Z}$) by 
$$ \mtilde\, \Sph[k] = \Sph[\widetilde{m}(k)].$$
We say that $\mtilde \in \Symtilde(X_{\tau})$ is a \emph{parity preserving} symmetry if 
$\widetilde{m}(k) -k \equiv 0 \mod{2}$  for all $k\in \Z$ and a \emph{parity reversing} symmetry 
if $\widetilde{m}(k)-k \equiv 1 \mod{2}$ for all $k\in \Z$.
Since any element  $\MMM \in \orth{p} \times \orth{p}$ leaves each $\Sph[k]$ invariant it is 
parity preserving, and from \ref{E:app:sph:tb:peq} and \ref{E:app:sph:tbp} respectively we see 
that $\tbartilde$ is parity preserving while 
$\tbartildeplus$ is parity reversing. The composition of parity preserving and reversing symmetries satisfies the obvious 
properties: the composition of two symmetries of the same parity is parity preserving, while 
composition of two symmetries of the opposite parity is parity reversing.
Since $\Symtilde(X_{\tau}) = \widetilde{\dihedral{}} \cdot \orth{p} \times \orth{p}$ where 
$\widetilde{\dihedral{}} = \langle \tbartildeplus, \tbartilde \rangle$ it follows that any $\mtilde \in \Symtilde(X_{\tau})$
has a definite parity. 
The set of all parity preserving symmetries in $\Symtilde(X_{\tau})$ forms a subgroup $\Symtilde_{+}(X_{\tau})$.
Clearly, if $\mtilde \in \Symtilde_{k}(X_{\tau})$ for some $k\in \Z$ then $\mtilde \in \Symtilde_{+}(X_{\tau})$.
One can verify that the subgroup $\Symtilde_{+}(X_{\tau})$ 
coincides with the subgroup of holomorphic isometries of $\Symtilde(X_{\tau})$ and 
hence by \ref{L:symtilde:p:eq:q} 
any element in $\Symtilde_{+}(X_{\tau})$ can be written in the form 
$ \ttilde_{2\pthat}^{i} \circ \tbartilde\phantom{}^{j} \circ \MMM$ for some $i\in \Z$, $j \in \Z_{2}$ and $\MMM \in \orth{p} \times 
\orth{p}$. Using \ref{E:app:sph:act:per:pgt1} and \ref{E:app:sph:tb:peq} we have 
$$\ttilde_{2\pthat}^{i} \circ \tbartilde\phantom{}^{j} \circ \MMM \, \Sph[k] = \Sph[2i + (-1)^{j}k] \quad \text{for all $k \in \Z$}.$$
It follows that for any fixed $k \in \Z$, $\Symtilde_{k}(X_{\tau}) = \langle \tbartilde_{k\pthat} \rangle \cdot \orth{p} 
\times \orth{p}$ as claimed.
\end{proof}

\noindent
It follows from \ref{L:sym:k} and \ref{L:symtilde:app:sph} together with \ref{P:xtau:sym:p:eq:1}, \ref{P:xtau:sym:p:neq:q}  and \ref{P:xtau:sym:p:eq:q} that $\Symtilde_{k}(X_{\tau}) = \rho(\Sym_{k}(X_{\tau}))$ 
where $\rho:  \Sym(X_{\tau}) \ra \Symtilde(X_{\tau})$ is the homomorphism defined in \ref{L:sym:subgroups} 
and $\Sym_{k}(X_{\tau})$ is defined in \ref{D:sym:k}.

\subsection*{Repositioning $X_{\tau}$ and marked spheres}
Recall from Appendix \ref{A:isom:sl} the groups $\Isomsl$ and $\Isomslpm$ defined in \ref{D:isomsl}. 
Fix $0<\abs{\tau }<\taumax$ and let $X_{\tau}$ be the corresponding $\sorth{p} \times \sorth{q}$-invariant 
special Legendrian immersion. We can use elements of $\Isomslpm$ to reposition the image 
$X_{\tau} (\cylpq) \subset \Sph^{2(p+q)-1}$ to another subset of $\Sph^{2(p+q)-1}$ that is also 
special Legendrian (with the correct choice of orientation). 
As we reposition $X_{\tau}$ in this way, the approximating spheres of $X_{\tau}$ are also repositioned. 
In our gluing constructions we will ``fuse'' a finite number of repositioned copies of $X_{\tau}$ 
at one shared approximating sphere. To achieve this 
we need to be able to reposition $X_{\tau}$ keeping some particular approximating 
sphere fixed, but changing its marking. 

 To this end we study the action of $\Isomslpm$ on  $(p,q)$-marked approximating spheres. 
\begin{definition}
\addtocounter{equation}{1}
\label{D:stab}
Let $\Sph$ be a $(p,q)$-marked special Legendrian sphere. We define 
$$\Stab(\Sph):= \{ A \in \Isomslpm \,|\, A\, \Sph = \Sph\},$$
where $A\,\Sph = \Sph$ means  $A\,\Sph$ and $\Sph$ are equal as $(p,q)$-marked spheres.
\end{definition}

\begin{lemma}[Action of $\Isomslpm$ on $(p,q)$-marked spheres]
\addtocounter{equation}{1}
\label{L:repos}
\hfill
\begin{itemize}
\item[(i)]
For the standard $(p,q)$-marked special Legendrian sphere $\Sph$ (defined following \ref{D:marked:sphere}) 
we have  
\begin{equation}
\addtocounter{theorem}{1}
\label{E:stab:marked}
\Stab(\Sph)
= 
\begin{cases}
\orth{1} \times \orth{n-1} \cdot \langle \ccong \rangle \cong \orth{1} \times \orth{n-1} \times \Z_{2}
\quad & \quad \text{if $(p,q)=(1,n-1)$;}\\
\orth{p} \times \orth{q} \cdot \langle \ccong \rangle \cong \orth{p} \times \orth{q} \times \Z_{2}
\quad & \quad \text{if $p>1$ and $p\neq q$};\\
(\orth{p} \times \orth{p} \cdot \langle \tbartilde \rangle) \cdot \langle \ccong \rangle 
\cong (\orth{p} \times \orth{p} \rtimes \Z_{2}) \times \Z_{2}
\quad & \quad \text{if $p>1$ and $p=q$}
\end{cases}
\end{equation}
where $\ccong \in \orth{2n}$ is defined in \ref{E:ctilde} and $\tbartilde \in \orth{2p}$ is the restriction 
to  $\R^{2p}\subset \C^{2p}$ of the $\orth{4p}$ transformation defined in \ref{E:tbartilde:p:eq:q}. 
\item[(ii)]
For any $(p,q)$-marked special Legendrian sphere $\Sph$, $\Stab(\Sph)$ is conjugate to the group given in 
\ref{E:stab:marked}.
\item[(iii)]
The set of all $(p,q)$-markings of the  standard (unmarked) special Legendrian sphere $\Sph$
is parametrised by the homogeneous space 
$$\orth{n} \cdot \langle \ccong \rangle /\Stab(\Sph).$$
\end{itemize}
\end{lemma}

\begin{proof}
(i) Since $n \ge 3$, from \ref{L:isom:sl}.ii we have $\Isomslpm = \sunit{n}^{\pm} \cdot \langle \ccong \rangle$ 
where 
$$\sunit{n}^{\pm} = \{ U \in \unit{n}\, | \, \det{}_{\C}U = \pm 1\}.$$
Clearly, the subgroup of $\sunit{n}^{\pm}$ which leaves $\R^{n}\subset \C^{n}$ invariant is $\orth{n}$ 
and from its definition $\ccong$ also leaves $\R^{n}$ invariant. 
Hence the subgroup of $\Isomslpm$ that sends the (unmarked) standard special Legendrian sphere $\Sph^{n-1}$ 
to itself is $\orth{n} \cdot \langle \ccong \rangle \cong \orth{n} \times \Z_{2}$ (since $\ccong$ centralises 
$\orth{n}$ we have a direct rather than semidirect product structure).
Hence for any $(p,q)$, the group $\Stab(\Sph)$ for the standard marked $(p,q)$-sphere $\Sph$ is a subgroup of 
$\orth{n} \cdot \langle \ccong \rangle \subset \Isomslpm$. 

For any $(p,q)$ $\ccong$ fixes the standard $(p,q)$-marked SL sphere $\Sph$ 
(since it fixes $\R^{n}$ pointwise). 
For $p=1$ the subgroup of $\orth{n}$ fixing $e_{1}\in \Sph^{n-1}$ is clearly $\orth{1} \times \orth{n-1}$. 
Similarly, for $p>1$ and $p \neq q$ the subgroup of $\orth{n}$ fixing 
$(\R^{p}\times \{0\} \cup \{0\} \times \R^{q}) \subset \R^{p} \times \R^{q}$
is clearly $\orth{p} \times \orth{q}$. 
If $p>1$ and $p=q$ then the subgroup of $\orth{n}$ fixing 
$(\R^{p}\times \{0\} \cup \{0\} \times \R^{p}) \subset \R^{p} \times \R^{p}$
is the semidirect product group $\orth{p} \times \orth{p} \cdot \langle \tbartilde \rangle$. 
This proves \ref{E:stab:marked}. 
(ii) follows from the fact that $\Isomslpm$ acts transitively on marked $(p,q)$-spheres. 
(iii) follows immediately.
\end{proof}

One should compare the results of Lemmas \ref{L:symtilde:app:sph} and \ref{L:repos}. 
Specialising to $k=0$ in \ref{E:symtilde:k} and comparing with \ref{E:stab:marked} 
we see that for any admissible $(p,q)$  we have 
$$\Stab(\Sph[0]) = \Symtilde_{0}(X_{\tau}) \cdot \langle \ccong \rangle,$$ 
and moreover, we have
$$
\frac{\Stab(\Sph[0])}{\Symtilde_{0}(X_{\tau})} 
\cong 
\begin{cases}
\langle \ccong \rangle \cong \Z_{2}\quad & \text{if $p>1$;}\\
\langle \RRR_{1} \rangle \cdot \langle\ccong \rangle / \langle \RRR_{1}\cdot \ccong \rangle \cong 
(\Z_{2} \times \Z_{2})/\Z_{2}
\quad & \text{if $p=1$;}
\end{cases}
$$
where $\RRR_{1} \in \orth{n} \subset \unit{n}$ denotes reflection in the complex hyperplane $z_{1}=0$.

\subsection*{The limiting geometry of $X_{\tau}$ as $\tau \ra 0$}
\phantom{ab}
This section describes the geometry of $X_{\tau}$ as $\tau \ra 0$ concentrating on the almost spherical 
regions of $X_{\tau}$ that asymptotically resemble equatorial spheres and on the necks which 
asymptotically resemble small Lagrangian catenoids or the product of a unit sphere with a small Lagrangian catenoid. 
The fact that $X_{\tau}$ degenerates to a union of very simple geometric objects is fundamental 
to our gluing constructions in \cite{haskins:kapouleas:invent,haskins:kapouleas:survey,haskins:kapouleas:hd2,haskins:kapouleas:hd3}.

\subsection*{Almost spherical regions of $X_{\tau}$ and approximating spheres}
Recall that by \ref{P:w:tau}  $\bw_\tau$ depends analytically on $\tau \in (-\taumax,\taumax)$
and the image of $\bw_{0}$ is contained in $\Sph^{1} \subset \R^2 \subset \C^2$.
This implies that $X_\tau$ depends analytically on $\tau$ and that  
$X_0$ gives a parametrisation of $\Sph[0] \setminus \M[0]$ 
where $\Sph[0]$ denotes the standard $(p,q)$-marked special Legendrian sphere
(recall \ref{D:approx:sphere}) and $\M[0]$ is its marked set  
(two orthogonal equators of dimension $p-1$ and $q-1$ if $p>1$ 
or two antipodal points if $p=1$).

Because of the analytic dependence on $\tau$,  $X_{\tau}$ smoothly converges to $X_{0}$ as $\tau \ra 0$ 
on any compact subset $K \subset \subset \cylpq$. 
Define 
$$S[0]:=[-b,b] \times \merpq,$$ then by choosing 
$b\in \R^{+}$ sufficiently large we can ensure the image $X_{0}(S[0])$ contains any given compact 
subset of $\Sph[0]\setminus \M[0]$.
If $|\tau|$ is small enough in terms of $b$,
then $\bw_\tau$ and therefore $X_\tau$ satisfy
\addtocounter{theorem}{1}
\begin{equation}
\label{E:bw:Xtau}
\|\bw_\tau-\bw_0:C^k([-b,b])\|\le C(b,k)|\tau|,
\qquad
\|X_\tau-X_0:C^k(S[0])\|\le C(b,k)|\tau|,
\end{equation}
where we can use the standard metric of the cylinder to define the $C^k$ norm and the constant
$C(b,k)$ depends only on $b$ and $k$.
This motivates us to call $S[0]$ an \emph{almost spherical region} of $X_{\tau}$. 
Note that the definition of $S[0]$ depends on a choice of $b$ which we do not make precise here, 
but which is supposed to be chosen large enough as needed.  The freedom to choose an appropriate $b$ 
to define the almost spherical regions is needed in our gluing constructions \cite{haskins:kapouleas:invent,haskins:kapouleas:survey,haskins:kapouleas:hd2,haskins:kapouleas:hd3}. 
\ref{E:bw:Xtau} implies that for $\tau$ sufficiently small the image under $X_{\tau}$ of 
the almost spherical region $S[0] \subset \subset \hat{S}[0]$ is close to its approximating sphere $\Sph[0]$ 
and converges as $\tau \ra 0$ to a fixed compact subset of  $\Sph[0]\setminus \M[0]$ (depending on the choice of $b$).
This explains the origin of the terminology approximating sphere and also one of the roles played by the marked set.


If we have fixed the almost spherical region $S[0]$ as above, then we can mimic the definition 
of $\hat{S}[k]$ in terms of $\hat{S}[0]$ (recall \ref{E:bulge:k:p:eq:1}, \ref{E:bulge:k:p:neq:1}) to define the \emph{$k$th almost spherical region} $S[k] \subset \subset \hat{S}[k] \subset \cylpq$ of $X_{\tau}$ in terms of $S[0]$ by 
$$
S[k]:=
\begin{cases}
\TTT_{2k\pt}S[0] \quad & \text{if $p=1$;}\\
\TTT_{2l\pt}S[0] \quad & \text{if $p>1$ and $k=2l$;}\\
\TTT_{2l\pt} \circ \tbar_{\pt^{+}}S[0] \quad & \text{if $p>1$ and $k=2l+1$.}
\end{cases}
$$
Because $\pt \ra \infty$ when $p=1$ and $\pt^{+},  \pt^{-} \ra \infty$ as $\tau \ra 0$ 
(see \ref{P:asymptotics-ptau} for a precise statement), every almost spherical region $S[k]$ with $k\neq 0$ 
``slides off the end'' of $\cylpq$ as $\tau \ra 0$. 
However, by using an element of $\Isom(\cylpq,g_{\tau})$ to bring $S[k]$ back to $S[0]$ we can also infer 
the small $\tau$ behaviour of $X_{\tau}$ on the other almost spherical regions $S[k]$. 
Using the relevant symmetries of $X_{\tau}$ (from \ref{P:xtau:sym:p:eq:1} and \ref{P:xtau:sym:p:neq:q})
we see that on the $k$th almost spherical region $S[k] \subset \subset \hat{S}[k]$,   
$X_{\tau}$ satisfies the analogue of \ref{E:bw:Xtau} with $X_{0}$ replaced by the embedding 
$X[k]: \cylpq \ra \Sph^{2(p+q)-1}$ defined as
$$ 
X[k] :=  
\begin{cases}
\ttilde_{2k\pthat} \circ X_{0} \circ \TTT_{-2k\pt} \quad & \text{if $p=1$;}\\
\ttilde_{2l\pthat} \circ X_{0} \circ \TTT_{-2l\pt} \quad & \text{if $p>1$ and $k=2l$;}\\
\ttilde_{2l\pthat}\circ \tbartildeplus \circ X_{0} \circ \tbar_{\pt^{+}} \circ \TTT_{-2l\pt} \quad & \text{if $p>1$ and $k=2l+1$.}
\end{cases}
$$
For $k\neq 0$ $X[k]$ itself depends on $\tau$. 
The image of $\cylpq$ under $X[k]$ is $\Sph[k]\setminus \M[k]$, 
where $\Sph[k]$ denotes the $k$th $(p,q)$-marked approximating sphere and $\M[k]$ is its marked set, 
both of which also depend on $\tau$ for $k \neq 0$. 
Nevertheless, we have that for $\tau$ sufficiently small the image of the almost spherical region $S[k]\subset \subset \hat{S}[k]$ under $X_{\tau}$ is close 
to its approximating sphere $\Sph[k]$. 

Each almost spherical region $S[k]$ of $\cylpq$ connects to its neighbouring almost spherical regions $S[k-1]$ and $S[k+1]$ in the 
two adjacent bulges $\hat{S}[k-1]$ and $\hat{S}[k+1]$ 
via a pair of transition regions whose images under $X_{\tau}$ for $\tau$ sufficiently small 
localise near the two components of the marked set $\M[k]$ of the $k$th approximating marked sphere $\Sph[k]$. 
This pair of transition regions is centred on the two boundary waists $W[k]$ and $W[k+1]$ of $\hat{S}[k]$ (recall \ref{D:bulges}). 
In the next section we study the geometry of $X_{\tau}$ in the vicinity of the waists as $\tau \ra 0$.

\subsection*{Waists, necks and Lagrangian catenoids}
Recall (\ref{D:waist}) 
that a waist $W[k]$ is a meridian of $\cylpq$ on which the radius of one spherical factor of the meridian is minimal. 
The vicinity of any meridian we call a \emph{neck}. The necks are the regions of $(\cylpq,g_{\tau})$ where the magnitude of the curvature is largest and where as $\tau \ra 0$ the curvature becomes unbounded. We will show below that 
for $\tau$ sufficiently small any neck of $X_{\tau}$---appropriately scaled and repositioned---is a small perturbation
of a (truncated) Lagrangian catenoid (recall \ref{D:lag:cat:std}) in the case $p=1$,
or  the product of a large round sphere with a Lagrangian catenoid in the case $p>1$.
This is reminiscent of the Delaunay surfaces \cite[Lemma A.2.1]{kapouleas:annals}
whose highly curved regions approximate a scaled repositioned catenoid.

We first study in detail the case $p=1$.
Recall that for $p=1$ all waists are isometric and $\Isom(\cylpq,g_{\tau})$ acts transitively on the waists. 
Hence without loss of generality we can concentrate on the neck containing the first waist $W[1]:=\{\pt\}\times \Sph^{n-2}$.
We first magnify the immersion $X_{\tau}$ while simultaneously repositioning and scaling the $t$ variable,
by taking
\addtocounter{theorem}{1}
\begin{equation}
\label{E:Xtilde:p1}
\widetilde{X}_\tau(\tilde{t},\sigma)=
\frac1\beta\, W\, X_\tau(\beta^{3-n} \, \tilde{t}+\pt,\sigma),
\end{equation}
where $W\in \unit{n}$ and $\beta>0$ are defined by
\addtocounter{theorem}{1}
\begin{equation}
\label{E:bW:p1}
\beta:= |w_2(\pt)|=\sqrt{\ymin},
\qquad
W = \left(
\begin{matrix}
\frac{|w_1(\pt)|}{w_1(\pt)}& 0  \\
0 & e^{i\pi/2(n-1)}\, \frac{|w_2(\pt)|}{w_2(\pt)}\Id_{n-1} \\
\end{matrix}
\right).
\end{equation}
Geometrically, $\beta$ is the radius of the sphere $\Sph^{n-2}$ on the waist.
Note that using \ref{E:ydot:cx} and that $\dot{y}(\pt)=0$ we have
$\det W=-1$.
We can write (recall \ref{D:X:tau})
\addtocounter{theorem}{1}
\begin{equation}
\label{E:Xtilde:z1z2:p1}
\widetilde{X}_\tau(\tilde{t},\sigma)  =
\left(z_1(\,\tilde{t}\,) + \frac{\sqrt{1-\beta^2}}{\beta},
\  z_2(\,\tilde{t}\,) \cdot \sigma\right),
\end{equation}
where
\addtocounter{theorem}{1}
\begin{equation}
\label{E:z1z2:p1}
z_1(\,\tilde{t}\,)=
\frac{\sqrt{1-\beta^2}}{\beta}
\,\left(\frac{w_1(\beta^{3-n} \, \tilde{t}+\pt)} {w_1(\pt)} -1\right),
\qquad
z_2(\,\tilde{t}\,)= 
e^{i\pi/2(n-1)}\,
\frac{w_2(\beta^{3-n} \, \tilde{t}+\pt)} {w_2(\pt)}.
\end{equation}
In terms of the new coordinates $z_{1}$, $z_{2}$ and the rescaled time parameter $\tilde{t}$, 
\ref{E:odes:p:n} is equivalent to 
\addtocounter{theorem}{1}
\begin{equation}
\label{E:odes:z1z2:p1}
\begin{aligned}
 \frac{dz_1}{d\tilde{t}} & = 
\beta\,\overline{z}_2^{n-1},
\\
 \frac{dz_2}{d\tilde{t}} & =
(\sqrt{1-\beta^2} + \beta\overline{z}_1)\,\overline{z}_2^{n-2},
\end{aligned}
\quad \text{with initial data $z_1(0)=0, \ \  z_2(0)= e^{i\pi/2(n-1)}$.}
\end{equation}
By standard ODE theory \ref{E:odes:z1z2:p1} has a unique (real analytic) maximal solution for each $\beta\in\R$
which we denote by
$\bz_\beta=(z_{1,\beta},z_{2,\beta})$, and
which depends analytically on $\beta\in (-1,1)$.
When $\beta=0$ the system simplifies: we obtain $z_{1,0}\equiv0$
and $z_{2,0}$ satisfies equation \ref{E:n:twist:sl:ode} with $n$ replaced by $n-1$ and 
with initial condition $z_{2,0}(0)=e^{i\pi/2(n-1)}$. 
Hence (recall \ref{D:lag:cat:std}) $X(\tilde{t},\sigma):=z_{2,0}(\,\tilde{t}\,) \sigma$ 
is the standard embedding of the unit Lagrangian catenoid in $\C^{n-1}$. 
By modifying \ref{E:Xtilde:z1z2:p1} we define a new $\sorth{n-1}$-invariant embedding 
$\widehat{X}_\tau: (-T,T) \times \Sph^{n-2} \ra \C^{n}$ by 
\addtocounter{theorem}{1}
\begin{equation}
\label{E:Xhat:p1}
\widehat{X}_\tau(\tilde{t},\sigma)  =
\left( \frac{\sqrt{1-\beta^2}}{\beta},
  z_{2,0}(\,\tilde{t}\,) \cdot \sigma\right),
\end{equation}
where $2T$ is the lifetime of the standard embedding of the unit Lagrangian catenoid in $\C^{n-1}$; 
as discussed in \ref{D:lag:cat:std} the lifetime $T$ is finite when $n-1>2$ and infinite when $n-1=2$. 
$\widehat{X}_{\tau}$ is independent of $\tau$ except for the translation by $\sqrt{1-\beta^{2}}/\beta$ 
in the first factor and its image is the standard unit $n-1$ dimensional Lagrangian catenoid in 
$\{0\} \times \C^{n-1} \subset \C \times \C^{n-1}$ 
translated in the $x$ direction of the extra $\C$ factor.  
As $\tau \ra 0$ the translation makes the Lagrangian catenoid drift to infinity.

If we take  $b$ large if $n=3$ or $b<T$ but close to $T$ if $n>3$,
and we restrict $|\tilde{t}|\le b$, the image under $\widehat{X}_{\tau}$ is
a truncated Lagrangian catenoid which exhausts the whole Lagrangian catenoid
as $b\to T-$ ($b\to\infty$ when $n=3$).
By the smooth dependence of $\bz_\beta$ on $\beta$,
we conclude that if $\tau$ is small enough depending on $b$,
and $\beta$ is defined as in \ref{E:bW:p1},
we have
\addtocounter{theorem}{1}
\begin{equation}
\label{E:Xtilde:Xhat:p1}
\begin{gathered}
\|\bz_\beta-\bz_0:C^k([-b,b])\|\le C(b,k)\beta,
\\
\|\widetilde{X}_\tau-\widehat{X}_\tau:C^k([-b,b]\times\merone)\|\le C(b,k)\beta,
\end{gathered}
\end{equation}
where we can use the standard metric of the cylinder or alternatively the pullback of the Euclidean metric
by $\widehat{X}_\tau$ to define the $C^k$ norm, and the constant $C(b,k)$ depends only on $b$ and $k$.

Now we study the limiting geometry of the necks when $p>1$. 
Recall that for $p>1$ waists (and hence necks) come in two types: 
type $1$ waists where the radius of the first spherical factor $\Sph^{p-1}$
is minimal and type $2$ waists where the radius of the second spherical factor $\Sph^{q-1}$ is minimal.
We concentrate now on the case of a type 2 neck.
 
Since all type $2$ waists are isometric and $\Isom(\cylpq,g_{\tau})$ acts transitively on them 
we can without loss of generality deal with the waist $W[1] = \{\pt^{+}\}\times \merpq$ (recall \ref{E:waist:k:p:neq:1}).
As in the case $p=1$ we magnify the immersion $X_{\tau}$ while repositioning and scaling the $t$ variable,
by taking
\addtocounter{theorem}{1}
\begin{equation}
\label{E:Xtilde}
\widetilde{X}_\tau(\tilde{t},\sigma_1,\sigma_2)=
\frac1\beta\, W\, X_\tau(\beta^{2-q} \, \tilde{t}+\pt^+,\sigma_1,\sigma_2),
\end{equation}
where $W\in \unit{n}$ and $\beta>0$ are defined by
\addtocounter{theorem}{1}
\begin{equation}
\label{E:bW}
\beta:= |w_2(\pt^+)|=\sqrt{\ymin},
\qquad
W = \left(
\begin{matrix}
\frac{|w_1(\pt^+)|}{w_1(\pt^+)}\Id_{p}& 0  \\
0 & e^{i\pi/2q}\, \frac{|w_2(\pt^+)|}{w_2(\pt^+)}\Id_{q} \\
\end{matrix}
\right).
\end{equation}
Geometrically, $\beta$ is the radius of the second spherical factor $\Sph^{q-1}$ on the waist.
As in the previous case using \ref{E:ydot:cx} and that $\dot{y}(\pt^+)=0$ we have
$\det W=-1$.
We can write (recall \ref{D:X:tau})
\addtocounter{theorem}{1}
\begin{equation}
\label{E:Xtilde:z1z2}
\widetilde{X}_\tau(\tilde{t},\sigma_1, \sigma_2)  =
\left(\left(z_1(\,\tilde{t}\,) + \frac{\sqrt{1-\beta^2}}{\beta}\right)\cdot \sigma_1,
\  z_2(\,\tilde{t}\,) \cdot \sigma_2\right),
\end{equation}
where
\addtocounter{theorem}{1}
\begin{equation}
\label{E:z1z2}
z_1(\,\tilde{t}\,)=
\frac{\sqrt{1-\beta^2}}{\beta}
\,\left(\frac{w_1(\beta^{2-q} \, \tilde{t}+\pt^+)} {w_1(\pt^+)} -1\right),
\qquad
z_2(\,\tilde{t}\,)= 
e^{i\pi/2q}\,
\frac{w_2(\beta^{2-q} \, \tilde{t}+\pt^+)} {w_2(\pt^+)}.
\end{equation}
In terms of the new coordinates $z_{1}$, $z_{2}$ and the rescaled time parameter $\tilde{t}$, 
\ref{E:odes:p:n} is equivalent to 
\addtocounter{theorem}{1}
\begin{equation}
\label{E:odes:z1z2}
\begin{aligned}
 \frac{dz_1}{d\tilde{t}} & = 
\beta\,(\sqrt{1-\beta^2} + \beta\overline{z}_1)^{p-1}\,\overline{z}_2^q,
\\
 \frac{dz_2}{d\tilde{t}} & =
(\sqrt{1-\beta^2} + \beta\overline{z}_1)^{p}\,\overline{z}_2^{q-1},
\end{aligned}
\quad \ \ 
\text{with initial data $z_1(0)=0,\ \ z_2(0)= e^{i\pi/2q}$.}
\end{equation}
As in the case $p=1$ \ref{E:odes:z1z2} has a unique maximal solution $\bz_\beta=(z_{1,\beta},z_{2,\beta})$,
for any $\beta\in\R$ which depends analytically on $\beta \in (-1,1)$. 
When $\beta=0$ \ref{E:odes:z1z2} again simplifies: $z_{1,0}\equiv 0$ and $z_{2,0}$ 
satisfies the equation for the standard embedding of the unit Lagrangian catenoid in $\C^{q}$.
Therefore following \ref{E:Xhat:p1} we define a new $\sorth{p} \times \sorth{q}$-invariant 
embedding $\widehat{X}_{\tau}: (-T,T) \times \merpq \ra \C^{p+q}$
\addtocounter{theorem}{1}
\begin{equation}
\label{E:Xhat}
\widehat{X}_\tau(\tilde{t},\sigma_1,\sigma_2)  =
\left( \frac{\sqrt{1-\beta^2}}{\beta}\cdot\sigma_1,\,
  z_{2,0}(\,\tilde{t}\,) \cdot \sigma_2\right),
\end{equation}
where $2T$ denotes the lifetime of the standard embedding of the unit Lagrangian catenoid in $\C^{q}$
(which as we already discussed is finite if $q>2$ and infinite if $q=2$). 
The image of $\widehat{X}_{\tau}$ is the product of a $p-1$ sphere $r\cdot \Sph^{p-1}$ 
of large radius $r=\sqrt{1-\beta^{2}}/\beta \simeq \beta^{-1}$ with a unit $q$-dimensional standard Lagrangian catenoid.
Arguing as before we conclude
that if $\tau$ is small enough depending on $b$,
and $\beta$ is defined as in \ref{E:bW},
we have
\addtocounter{theorem}{1}
\begin{equation}
\label{E:Xtilde:Xhat}
\begin{gathered}
\|\bz_\beta-\bz_0:C^k([-b,b])\|\le C(b,k)\beta,
\\
\|\widetilde{X}_\tau-\widehat{X}_\tau:C^k([-b,b]\times\merpq)\|\le C(b,k)\beta,
\end{gathered}
\end{equation}
where we use the pullback of the Euclidean metric
by $\widehat{X}_\tau$ to define the $C^k$ norm, the constant $C(b,k)$ depends only on $b$ and $k$,
and $b<T$ where $T=T_{1}$ is the lifetime for the standard unit Lagrangian catenoid in $\C^{q}$ defined in 
 \ref{D:lag:cat:std}. 
Finally the case of any type 1 neck is similar except that $\widehat{X}_\tau$ now becomes an $\sorth{p} \times \sorth{q}$-
invariant embedding
of the product of the  unit standard Lagrangian catenoid in $\C^{p}$ with a large round $q-1$ sphere in $\C^{q}$. 
We omit the details.

\subsection*{Synthesis}
We combine the results from the previous two sections to describe qualitatively the geometry of $X_{\tau}$ for small $\tau$.

We begin our discussion 
with the case $p=1$. In this case each domain of periodicity of $g_{\tau}$ contains a single bulge $\hat{S}[k] \subset \cylpq$. 
Inside each bulge we fixed a compact subset $S[k] \subset \hat{S}[k]$ (depending on the choice of a sufficiently large 
real number $b$) which we called an almost spherical region of $X_{\tau}$. 
As $\tau \ra 0$ the image of the $k$th almost spherical region $S[k]$ under $X_{\tau}$ tends to the complement 
of a small tubular neighbourhood of the marked set $\M[k] \subset \Sph[k]$ inside the $k$th approximating sphere 
$\Sph[k] = \ttilde_{2k\pthat} \Sph[0]$. $S[k]$ connects to its neighbouring almost spherical regions $S[k-1]$ and 
$S[k+1]$ via two small transition regions which are localised near the 
marked set $\M[k] = \pm \ttilde_{2k\pthat}(e_{1})$ of $\Sph[k]$. 
The core of each transition region---the necks---are the vicinity of the two boundary waists $W[k]$ and $W[k+1]$ of $\hat{S}[k]$. 
As $\tau \ra 0$ on each neck $X_{\tau}$ approaches an embedding of the $n-1$ dimensional Lagrangian catenoid
of size $\sqrt{\ymin} \simeq (2\tau)^{1/n-1}$ located close to one of the two points $\pm \ttilde_{2k\pthat}(e_{1})$.

In the limit as $\tau \ra 0$ almost spherical regions tend to (subsets) of the approximating spheres, 
while a transition region connecting neighbouring almost spherical regions tends to a point of intersection of the 
corresponding approximating spheres. It follows from \ref{E:pt:hat:asymp} and \ref{P:asymptotics-pthat} that as $\tau \ra 0$
the rotational period of $X_{\tau}$ satisfies
\begin{equation}
\addtocounter{theorem}{1}
\ttilde_{2\pthat} \ra \left(
\begin{matrix}
-1 & 0  \\
0 & e^{-i\pi/(n-1)}\Id_{n-1} \\
\end{matrix}
\right).
\end{equation}
Hence in the $\tau \ra 0$ limit, the real $n$-planes in $\C^n$ associated to the almost spherical region
$S[0]$ and the almost spherical region $S[1]$ are
$\R \oplus \R^{n-1}$ and $\R \oplus e^{-i\pi/(n-1)} \R^{n-1}$ respectively.
This is consistent with the fact that the Lagrangian catenoid in $\C^{n-1}$ is asymptotic
to the union of two $n-1$ planes (which up to rotation we can take to be) $\R^{n-1}$ and $e^{-i\pi/(n-1)} \R^{n-1}$.
See Figure \ref{fig:as:cylinder} for a schematic illustration of the intrinsic geometry of $X_\tau$ 
in the case $p=1$.
\iffigureson
\begin{figure}
\centering
\input{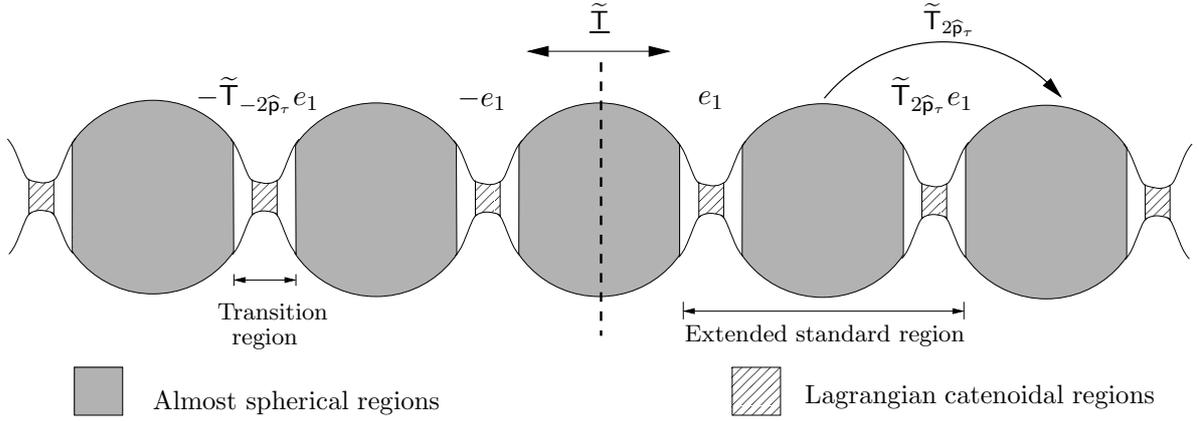}
\caption{Schematic presentation of the intrinsic geometry
of a special Legendrian cylinder $X_\tau$ with small $\tau$ and $p=1$}
\label{fig:as:cylinder}
\end{figure}
\else
\fi

We now discuss the limiting geometry of $X_{\tau}$ when $p>1$. 
In this case each domain of periodicity of $g_{\tau}$ contains not one but two bulges $\hat{S}[k]$ and $\hat{S}[k+1]$. 
For each bulge one of its two boundary waists is a waist of type $1$ (where the radius 
of the first spherical factor $\Sph^{p-1}$ is minimal) and the other is a waist of type $2$
(where the radius of the second spherical factor $\Sph^{q-1}$ is minimal). 
Moreover, since waists of type $1$ and $2$ alternate along $\cylpq$ one of the two bulges in a domain of periodicity 
will have a type $1$ waist at its left-hand boundary and a type $2$ waist at its right-hand boundary, 
while the other bulge will have a type $2$ waist at its LH boundary and a type $1$ waist at its RH boundary.
Hence while the reflectional symmetry $\TTT_{2k\pt} \circ \tbar_{\pt^{+}}$ exchanges the adjacent almost spherical regions 
$S[k]$ and $S[k+1]$ there is no purely translational symmetry that achieves this (unlike the case $p=1$). 
Instead the basic rotational period $\ttilde_{2\pthat}$ of $X_{\tau}$ sends $S[k]$ to $S[k+2]$. 
This reflects a fundamental difference in the geometry of $X_{\tau}$ between the cases $p=1$ and $p>1$.

Inside each bulge we fixed a compact subset $S[k] \subset \hat{S}[k]$ (depending on the choice of a sufficiently large real number
$b$) which we called the $k$th almost spherical region of $X_{\tau}$. 
As $\tau \ra 0$ the image of the $k$th almost spherical region $S[k]$ under $X_{\tau}$ tends to the complement 
of a small tubular neighbourhood of the marked set $\M[k] \subset \Sph[k]$ inside the $k$th approximating sphere. 
The marked set $\M[k]$ is a generalised $(p,q)$-Hopf link, i.e. two orthogonal equatorial subspheres in $\Sph^{p+q-1}$
of dimensions $p-1$ and $q-1$. $S[k]$ connects to its neighbouring almost spherical regions $S[k-1]$ and $S[k+1]$ via two 
transition regions which are localised near the two components of the marked set $\M[k]$. The core of each transition 
region---the necks---is the vicinity of one of the waists $W[k]$ and $W[k+1]$: one of type $1$ and one of type $2$. 
On the neck containing the type $1$ waist $X_{\tau}$ approaches an embedding of the product of a
Lagrangian catenoid in $\C^{p}$ of size $\sqrt{1-\ymax} \simeq (2\tau)^{1/p}$ with a round sphere $\Sph^{q-1}$ of radius $1$
as $\tau \ra 0$. This type $1$ neck localises to the equatorial $q-1$ sphere of the marked set $\M[k]$.
On the type $2$ neck $X_{\tau}$ approaches an embedding of the product of a round sphere $\Sph^{p-1}$ of radius $1$ 
with a Lagrangian catenoid in $\C^{q}$ of size $\sqrt{\ymin} \simeq (2\tau)^{1/q}$ which localises to the 
equatorial $p-1$ sphere of the marked set $\M[k]$. 
In particular, when $p \neq q$ necks of type $1$ and necks of type $2$ are not isometric and hence no 
symmetry can take a type $1$ neck to a type $2$ neck. 

However, when $p=q$ type $1$ and type $2$ necks are isometric and extra symmetries exist that exchange the two neck types;
the symmetry $\tbar_{k\pt} \circ \EEE$ (recall \ref{L:sym:k}) sends the bulge $\hat{S}[k]$ to itself 
but exchanges its two boundary waists $W[k]$ and $W[k+1]$. 

In the limit as $\tau \ra 0$ almost spherical regions tend to (subsets) of the approximating spheres, while a transition region 
connecting neighbouring almost spherical regions tends to the equatorial subsphere formed by the intersection 
of the corresponding approximate spheres. It follows from \ref{E:pt:hat:asymp} and \ref{P:asymptotics-pthat}
that as $\tau \ra 0$ the reflection $\tbartildeplus$ that sends $\Sph[0]$ to $\Sph[1]$ converges to the reflection
$$ (z,w) \mapsto (\overline{z}, e^{-i\pi/q}\overline{w}), \qquad \qquad \text{for\ } (z,w)\in \C^{p}\times \C^{q}.$$
Hence in the $\tau \ra 0$ limit, the real $n$-planes in $\C^{n}$ associated to the almost spherical regions $S[0]$ and $S[1]$ 
are $\R^{p}\oplus \R^{q}$ and $\R^{p} \oplus e^{-i\pi/q}\R^{q}$. This is consistent both with the asymptotic geometry of the Lagrangian catenoid in $\C^{q}$ and the fact that the neck concentrates on a round $\Sph^{p-1}$. 
Similarly, the reflection $\tbartildeminus$ that sends $S[0]$ to $S[-1]$ converges as $\tau \ra 0$ to the reflection
$$ (z,w) \mapsto (e^{-i\pi/p}\overline{z}, \overline{w}),  \qquad \qquad \text{for\ } (z,w)\in \C^{p}\times \C^{q}.$$
Hence in the $\tau \ra 0$ limit the real $n$-planes in $\C^{n}$ associated to the almost spherical regions $S[0]$ and $S[-1]$ 
are $\R^{p}\oplus \R^{q}$ and $e^{-i\pi/p}\R^{p} \oplus \R^{q}$ respectively, which is again consistent 
with the asymptotic geometry of the $p$ dimensional catenoid and concentration of the neck on a round $\Sph^{q-1}$.

\section{Torques}
\label{S:torques}
Suppose $M$ is an oriented $m$-dimensional submanifold of the ambient manifold $(\overline{M},\overline{g})$ 
and $\ki\in \iso(\overline{M},\gbar) $ is a Killing field on $(\overline{M},\overline{g})$. Given any oriented hypersurface 
$\Sigma \subset M$ we define the $\ki$-flux through $\Sigma$ by 
\begin{equation}
\addtocounter{theorem}{1}
\label{E:kflux}
\mathcal{F}_{\ki}(\Sigma):= \int_{\Sigma}{\gbar(\ki,\eta)} \dvol_{\Sigma}\, ,
\end{equation}
where $\eta$ is the unit conormal to $\Sigma$, chosen so that the orientation defined by $\Sigma$ and $\eta$ agrees with that 
of $M$. An immediate consequence of the First Variation of Volume formula \cite[7.6]{simon:book} is
\begin{lemma}
\addtocounter{equation}{1}
\label{L:kflux}
If $M$ is an oriented $m$-dimensional minimal submanifold of $(\overline{M},\gbar)$, $\Sigma$ is an oriented hypersurface 
of $M$ and $\ki \in \iso(\overline{M},\gbar)$ then the $\ki$-flux through $\Sigma$, $\mathcal{F}_{\ki}(\Sigma)$, 
depends only on the homology class $[\Sigma] \in H_{m-1}(M,\R)$.
\end{lemma}
In other words, when $M$ is a minimal submanifold of  $(\overline{M},\gbar)$ the $\ki$-flux map defined in \ref{E:kflux} 
induces a linear map $\mathcal{F}: H_{m-1}(M,\R) \ra \iso(\overline{M},\gbar)^{*}$.
In particular, if $(\overline{M},\gbar) = (\Sph^{2n-1},g_{\text{std}})$ then $\iso(\overline{M},\gbar) = \lorth{2n}$ 
and we call the map $\mathcal{F}: H_{m-1}(M,\R) \ra \lorth{2n}^{*}$ the \emph{torque of $M$}.
For special Legendrian submanifolds of $\Sph^{2n-1}$ it is also convenient to define the \emph{restricted torque of $M$}, 
which is the restriction of the torque to the subalgebra $\lsunit{n} \subset \lorth{2n}$.

\begin{prop}
\addtocounter{equation}{1}
\label{P:torques}
For $p>1$ the $\lsunit{n}$ restricted torque of the $\sorth{p}\times\sorth{q}$-invariant special Legendrian immersion
$X_{\tau}: \cylpq \ra \Sph^{2(p+q)-1}$ 
is given by 
\begin{equation}
\addtocounter{theorem}{1}
\label{E:torques}
\mathcal{F}_{\ki}(X_{\tau}) = 
\begin{cases}
2\tau \left(\frac{1}{p} \sum_{i=1}^{p}{\lambda_{i}} - \frac{1}{q} \sum_{j=1}^{q} {\mu_{j}}\right) \vol(\Sph^{p-1}) \vol(\Sph^{q-1}) \quad & \ki = i \diag(\lambda_{1}, \ldots ,\lambda_{p},\mu_{1},\ldots ,\mu_{q});\\
0 \quad & \text{if\ } \ki \in \lsunit{n}\ \text{is off-diagonal},
\end{cases}
\end{equation}
where we implicitly use the homology class of any meridian in $\cylpq$.

For $p=1$ the $\lsunit{n}$ restricted torque of the $\sorth{n-1}$-invariant special Legendrian immersion 
$X_{\tau}: \cylone \ra \Sph^{2n-1}$ is given by 
\begin{equation}
\addtocounter{theorem}{1}
\label{E:torques:p:eq:1}
\mathcal{F}_{\ki}(X_{\tau}) = 
\begin{cases}
2\tau \left(\lambda - \frac{1}{n-1} \sum_{j=1}^{n-1}{\mu_{j}}\right)  \vol(\Sph^{n-2}) \quad & \ki = i \diag(\lambda,\mu_{1},\ldots ,\mu_{n-1});\\
0 \quad & \ki \in \lsunit{n}\ \text{is off-diagonal}.
\end{cases}
\end{equation}
\end{prop}
In particular, if we take $\ki=\mathsf{t}$ to be the generator of the $1$-parameter subgroup $\{\ttilde_{x}\}_{x\in \R}$ 
(defined in \ref{E:ttilde}) 
associated to the rotational period $\ttilde_{2\pthat}$ of $X_{\tau}$ then we obtain
$$ \mathcal{F}_{\mathsf{t}}(X_{\tau}) =
\begin{cases}
2\tau \frac{n}{pq} \vol(\Sph^{p-1}) \vol(\Sph^{q-1}), & \quad \text{if\ } p>1;\\
2\tau \frac{n}{n-1} \vol(\Sph^{n-2}), & \quad \text{if\ } p=1.
\end{cases}
$$
\begin{proof}
We give the proof in the case $p>1$. The result in the case $p=1$ follows by making the obvious adjustments to the argument below.

\emph{Case $p>1$:} 
By the homological invariance of $\mathcal{F}_{\ki}(\Sigma)$ we may evaluate the $\ki$-flux on any meridian 
$\{t_{0}\} \times \merpq$ of $\cylpq$.
From \ref{P:X:tau}.iii the vector field $\partial_{t}$ is orthogonal to any meridian $\{t_{0}\} \times \merpq$. 
Hence the unit conormal is given by $\eta = \partial_{t}X_{\tau}/\abs{\partial_{t}X_{\tau}}$. By the definition of $X_{\tau}$ in terms of $\bw_{\tau}$ we have $\abs{\partial_{t}X_{\tau}} = \abs{\dot{\bw}}$.
Using \ref{E:w:reparam} and \ref{P:X:tau}.ii the volume form induced on the meridian  $\{t_{0}\} \times \merpq$  is 
$$ \abs{w_{1}}^{p-1}\abs{w_{2}}^{q-1} \dvol_{\Sph^{p-1}} \wedge \dvol_{\Sph^{q-1}} =
\abs{\partial_{t}X_{\tau}} \dvol_{\Sph^{p-1}} \wedge \dvol_{\Sph^{q-1}}.$$
Therefore 
\begin{equation}
\addtocounter{theorem}{1}
\label{E:kflux:simple}
\mathcal{F}_{\ki} = \int_{t=t_{0}}{\ki \cdot \frac{\partial_{t}X_{\tau}}{\abs{\partial_{t}X_{\tau}}} \abs{\partial_{t}X_{\tau}} \dvol_{\Sph^{p-1}} \wedge \dvol_{\Sph^{q-1}}} = \int_{t=t_{0}}{\ki \cdot \partial_{t}X_{\tau} \dvol_{\Sph^{p-1}} \wedge \dvol_{\Sph^{q-1}}},
\end{equation}
where $t=t_{0}$ is a shorthand for the meridian $\{t_{0}\}\times \merpq$ on which the $\R$ coordinate $t$ equals $t_{0}$.

\emph{$\ki \in \lsunit{n}$ is diagonal:}
If $\ki =  i \diag(\lambda_{1}, \ldots ,\lambda_{p},\mu_{1},\ldots ,\mu_{q}) \in \lsunit{n}$ a short computation shows that
$$ \ki \cdot \partial_{t}X_{\tau} = \Imag(\overline{w}_{1}\dot{w}_{1}) \sum_{i=1}^{p}{\lambda_{i}(\sigma_{1}^{i})^{2}} + \Imag(\overline{w}_{2}\dot{w}_{2}) \sum_{i=1}^{q}{\mu_{j}(\sigma_{2}^{j})^{2}},
$$
where $\sigma_{1}=(\sigma_{1}^{1},\ldots ,\sigma_{1}^{p})\in \Sph^{p-1}\subset \R^{p}$ and $\sigma_{2} = (\sigma_{2}^{1},\ldots ,\sigma_{2}^{q}) \in \Sph^{q-1}\subset \R^{q}$.
Hence using \ref{E:odes:p:n} and the definition of $\tau$ we have 
\begin{equation}
\addtocounter{theorem}{1}
\label{E:k:dot:xt}
\ki \cdot \partial_{t}X_{\tau} = 2\tau \left(\sum_{i=1}^{p}{\lambda_{i}(\sigma_{1}^{i})^{2}} -  \sum_{j=1}^{q}{\mu_{j}(\sigma_{2}^{j})^{2}}\right).
\end{equation}
By symmetry we have 
\begin{equation}
\addtocounter{theorem}{1}
\label{E:sphere:ints}
\int_{\Sph^{p-1}}{(\sigma_{1}^{i})^{2}}\dvol_{\Sph^{p-1}} = \frac{1}{p} \vol{\Sph^{p-1}} \quad \text{and} \quad 
\int_{\Sph^{q-1}}{(\sigma_{2}^{j})^{2}}\dvol_{\Sph^{q-1}} = \frac{1}{q} \vol{\Sph^{q-1}}, 
\end{equation}
for $1\le i \le p$ and $1\le j \le q$.
Combining \ref{E:kflux:simple}, \ref{E:k:dot:xt} and \ref{E:sphere:ints} we obtain 
$$ \mathcal{F}_{\ki} = 
2\tau \left(\frac{1}{p} \sum_{i=1}^{p}{\lambda_{i}} - \frac{1}{q} \sum_{j=1}^{q} {\mu_{j}}\right) \vol(\Sph^{p-1}) \vol(\Sph^{q-1}). $$
\emph{$\ki \in \lsunit{n}$ is  off-diagonal:}  Off-diagonal elements $\ki \in \lsunit{n}$ can be decomposed as $\lsorth{n} \oplus \,i \,\symoff{n}$ where $\symoff{n}$ denotes the off-diagonal real symmetric  $n \times n$ matrices. By linearity it suffices to prove 
$\mathcal{F}_{\ki}=0$ for any $\ki \in \lsorth{n}$ and $\ki\in \,i \,\symoff{n}$.

First we show that $\mathcal{F}_{\ki}$ vanishes for any $\ki \in \lsorth{n} \subset \lsunit{n}$.
Let $e_{1}, \ldots , e_{n}$ denote the standard unitary basis of $\C^{n}$. 
For $i\neq j \in \{1,\ldots ,n\}$ define $\mathsf{R}_{ij}\in \lsorth{n}$ by 
$$ \mathsf{R}_{ij}(v) = (e_{i}\cdot v) \, e_{j} -( e_{j}\cdot v) \, e_{i}, \quad \text{for any\ } v\in \R^{n}.$$
$\{\mathsf{R}_{ij}\}$ for $i<j \in \{1,\ldots ,n\}$ forms a basis for $\lsorth{n} \subset \lsunit{n}$.
Using the definition of $X_{\tau}$ and $\mathsf{R}_{ij}$ we find
$$
\mathsf{R}_{ij} X_{\tau}= 
\begin{cases}
w_{1} (\sigma_{1}^{i}e_{j} - \sigma_{1}^{j}e_{i}) & \text{for \ } i,j \in \{1, \ldots ,p\};\\
w_{2}  (\sigma_{2}^{i'}e_{j} - \sigma_{2}^{j'}e_{i}) & \text{for \ } i',j' \in \{1, \ldots ,q\};\\
w_{1} \sigma_{1}^{i}e_{j} - w_{2} \sigma_{2}^{j'}e_{i}  & \text{for \ } i\in \{1,\ldots , p\}, \  j' \in \{1, \ldots ,q\},
\end{cases}
$$
where $i'=i-p$ and $j'=j-p$.
Taking the inner product with $\partial_{t}X_{\tau}$ we obtain
\begin{equation}
\addtocounter{theorem}{1}
\label{E:flux:sorth}
\mathsf{R}_{ij} X_{\tau} \cdot \partial_{t}X_{\tau} = 
\begin{cases}
0 &  \text{for \ } i,j \in \{1, \ldots ,p\};\\
0 & \text{for \ } i',j' \ \in \{1, \ldots ,q\};\\
\Real(\overline{w}_{1}\dot{w}_{2} - \overline{w}_{2}\dot{w}_{1}) \sigma_{1}^{i}\sigma_{2}^{j'} & \text{for \ } i\in \{1,\ldots , p\}, \  j' \in \{1, \ldots ,q\}.
\end{cases}
\end{equation}
Clearly we have
\begin{equation}
\addtocounter{theorem}{1}
\label{E:integrand:sorth}
\int_{\Sph^{p-1}\times \Sph^{q-1}}{\sigma_{1}^{i} \, \sigma_{2}^{j'} \dvol_{\Sph^{p-1}} \wedge \dvol_{\Sph^{q-1}}} = \int_{\Sph^{p-1}}{\sigma_{1}^{i} \dvol_{\Sph^{p-1}}} \, \int_{\Sph^{q-1}}{\sigma_{2}^{j'} \dvol_{\Sph^{q-1}}}=0.
\end{equation}
Combining \ref{E:kflux:simple}, \ref{E:flux:sorth} and \ref{E:integrand:sorth} we conclude $\mathcal{F}_{\ki}=0$ for $\ki = \mathsf{R}_{ij}$ 
and hence by linearity $\mathcal{F}_{\ki}=0$ for all $\ki\in \lsorth{n} \subset \lsunit{n}$.

Now we show that $\mathcal{F}_{\ki}=0$ for any $\ki \in \,i \,\symoff{n}$.
For $i<j \in \{1,\ldots ,n\}$ define $\mathsf{S}_{ij}\in \symoff{n}$ by 
$$ \mathsf{S}_{ij}(v) = (e_{i}\cdot v) \, e_{j} + (e_{j}\cdot v) \, e_{i}, \quad \text{for any\ } v\in \R^{n}.$$
$\{\sqrt{-1}\,\mathsf{S}_{ij}\}$ for $i<j \in \{1,\ldots ,n\}$ forms a basis for $i\, \symoff{n} \subset \lsunit{n}$.
Using the definition of $X_{\tau}$ and $\mathsf{S}_{ij}$ we find
$$
\sqrt{-1}\,\mathsf{S}_{ij} X_{\tau}= \sqrt{-1}
\begin{cases}
w_{1} (\sigma_{1}^{i}e_{j} + \sigma_{1}^{j}e_{i}) & \text{for \ } i,j \in \{1, \ldots ,p\};\\
w_{2}  (\sigma_{2}^{i'}e_{j} + \sigma_{2}^{j'}e_{i}) & \text{for \ } i',j' \in \{1, \ldots ,q\};\\
w_{1} \sigma_{1}^{i}e_{j} + w_{2} \sigma_{2}^{j'}e_{i}  & \text{for \ } i\in \{1,\ldots , p\}, \  j' \in \{1, \ldots ,q\},
\end{cases}
$$
where as above $i'=i-p$ and $j'=j-p$.
Taking the inner product with $\partial_{t}X_{\tau}$ we obtain
\begin{equation}
\addtocounter{theorem}{1}
\label{E:flux:symoff}
\mathsf{S}_{ij} X_{\tau} \cdot \partial_{t}X_{\tau} = 
\begin{cases}
2 \Imag(\overline{w}_{1} \dot{w}_{1}) \sigma_{1}^{i}\sigma_{1}^{j} &  \text{for \ } i,j \in \{1, \ldots ,p\};\\
2 \Imag(\overline{w}_{2} \dot{w}_{2}) \sigma_{2}^{i'}\sigma_{2}^{j'} & \text{for \ } i',j' \in \{1, \ldots ,q\};\\
\Imag(\overline{w}_{1}\dot{w}_{2} + \overline{w}_{2}\dot{w}_{1}) \sigma_{1}^{i}\sigma_{2}^{j'} & \text{for \ } i\in \{1,\ldots , p\}, \  j' \in \{1, \ldots ,q\}.
\end{cases}
\end{equation}
For any $i\neq j$ we have 
\begin{equation}
\addtocounter{theorem}{1}
\label{E:integrand:symoff1}
\int_{\Sph^{p-1}\times \Sph^{q-1}}{\sigma_{1}^{i}\sigma_{1}^{j} \dvol_{\Sph^{p-1}}\wedge \dvol_{\Sph^{q-1}}} 
= \vol(\Sph^{q-1}) \int_{\Sph^{p-1}}{\sigma_{1}^{i}\sigma_{1}^{j} \dvol_{\Sph^{p-1}}}=0,
\end{equation}
since for any $i \neq j$, $\sigma_{1}^{i}\sigma_{1}^{j}$ is an eigenvalue of the Laplacian on $\Sph^{p-1}$ 
with eigenvalue $\lambda=2p$, and hence is $L^{2}$-orthogonal to the constant functions.
(Alternatively, one can consider the involution mapping 
$\sigma_{1}^{i} \mapsto -\sigma_{1}^{i}$ and
$\sigma_{1}^{k} \mapsto -\sigma_{1}^{k}$ for any $k \notin \{i,j\}$ and fixing all other components of $\sigma_{1}$.
Clearly this symmetry preserves $\dvol_{\Sph^{p-1}}$ but sends $\sigma_{1}^{i}\sigma_{1}^{j} \mapsto -\sigma_{1}^{i}\sigma_{1}^{j}$. Hence the integral in \ref{E:integrand:symoff1} is odd under this symmetry and therefore vanishes.) 
Similarly, we have 
\begin{equation}
\addtocounter{theorem}{1}
\label{E:integrand:symoff2}
 \int_{\Sph^{p-1}\times \Sph^{q-1}}{\sigma_{2}^{i'}\sigma_{2}^{j'} \dvol_{\Sph^{p-1}}\wedge \dvol_{\Sph^{q-1}}} 
= \vol(\Sph^{p-1}) \int_{\Sph^{q-1}}{\sigma_{2}^{i'}\sigma_{2}^{j'} \dvol_{\Sph^{q-1}}}=0,
\end{equation}
for any $i\neq j$.
For any $i\neq j$, combining \ref{E:kflux:simple}, \ref{E:integrand:sorth}--\ref{E:integrand:symoff2} implies that $\mathcal{F}_{\ki}=0$ for $\ki = \sqrt{-1}\, \mathsf{S}_{ij}$ 
and hence by linearity $\mathcal{F}_{\ki}=0$ for all $\ki\in i \, \symoff{n} \subset \lsunit{n}$.
\end{proof}

\begin{remark}
\addtocounter{equation}{1}
\label{R:flux:reposn}
If $\mtilde \in \sunit{n}$ then $\mtilde \circ X_{\tau}: \cylpq \ra \Sph^{2(p+q)-1}$ is also a $1$-parameter family 
of special Legendrian immersions and hence we may consider its torque or restricted torque. 
A simple computation shows that 
the torque or restricted torque of $X_{\tau}$ determines the torque of $\mtilde \circ X_{\tau}$ by
\begin{equation}
\addtocounter{theorem}{1}
\label{E:flux:reposn}
\mathcal{F}_{k}(\mtilde \circ X_{\tau}) = \mathcal{F}_{\widetilde{\ki}} (X_{\tau}), \quad \text{where} \ 
\widetilde{\ki} = \mtilde^{t}\, \ki \,\mtilde \quad \text{and} \ \ki \in \lsunit{n}.
\end{equation}
\end{remark}

\begin{remark}
\addtocounter{equation}{1}
\label{R:flux:sym}
If $G$ is a Lie subgroup of $\orth{m}$ then one can obtain restrictions on the possible torques $\mathcal{F}$ 
of any $G$-invariant minimal submanifold of $\Sph^{m-1}$, in terms of the coadjoint action of $G \subset \orth{m}$ 
on $\lorth{m}^{*}$. We will describe this in detail elsewhere.
\end{remark}

\section{Precise asymptotics as $\tau \ra 0$}
\label{S:asymptotics}

In order to describe the asymptotics it helps to introduce the following notation:
We define functions of $\tau$ by
\addtocounter{theorem}{1}
\begin{equation}
\label{E:Tk}
T_k(\tau):=
\begin{cases}
\tau^{-1+2/k}  , \qquad &\text{for $k>2$;}\\
\log{\tau^{-1}} , \qquad &\text{for $k=2$},
\end{cases}
\end{equation}
and introduce the notation $f_1 \sim f_2$ for functions $f_1$ and $f_2$ of $\tau$ to mean that
\addtocounter{theorem}{1}
\begin{equation}
\label{E:sim}
\frac{f_2(\tau)}{f_1(\tau)}\to\,1
\quad\text{as}\quad
\tau\to0.
\end{equation}
Using this notation we have the following :

\begin{prop}[Small $\tau$ asymptotics of the period and partial-periods]\hfill
\addtocounter{equation}{1}
\label{P:asymptotics-ptau}
\begin{enumerate}
\item[(i)]
For $p>1$, $\pt^{+}$ and $\pt^{-}$ are analytic functions of $\tau$ for $0<\abs{\tau}<\taumax$. 
For $p=1$, $\pt$ is an analytic function of $\tau$ for $0<\abs{\tau}<\taumax$. 
\item[(ii)]
In the case $p>1$
we have
\begin{equation}
\addtocounter{theorem}{1}
\label{E:pt:pm}
\pt^+\sim 
b_q\,T_q(\tau),
\qquad
\quad
\pt^- \sim
b_p\,T_p(\tau),
\end{equation}
where 
\begin{equation}
\addtocounter{theorem}{1}
\label{E:b:q}
b_2:=1,\qquad
b_k:=4^{-1+\frac1k}\int_1^\infty\frac{dz}{\sqrt{z^k-1}\,}=
4^{-1+\frac1k}\frac{\sqrt{\pi}\, \Gamma(\frac12-\frac1k)} {\Gamma(-\frac1k)} \quad \text{for $k\ge 2$},
\end{equation}
where $\Gamma$ is the gamma function. 
We also have
\begin{equation}
\addtocounter{theorem}{1}
\label{E:pt:tau:0:p:eq:1}
\pt \sim
b_q\,T_{q}(\tau)
\quad\text{when}\quad
1\le p< q,
\qquad
\pt \sim
2b_q\,T_q(\tau)
\quad\text{when}\quad
2\le p= q .
\end{equation}
\end{enumerate}
\end{prop}
\begin{remark}
\addtocounter{equation}{1}
\label{R:period:cat:vs:pt}
For $k>2$ we note that by \ref{L:lagn:cat:ode}.vii, the expression $b_{k}T_{k}(\tau)$ appearing in \ref{E:pt:pm} 
is exactly half the lifetime of a Lagrangian catenoid in $\C^{k}$ of size $2\tau$ parametrised by \ref{E:n:twist:sl:ode}. 
In light of the geometry of the high curvature regions of $X_{\tau}$ described in Section \ref{S:xtau:limit} 
this does not come as a surprise. 
\end{remark}

\begin{proof}
For $p=1$, $\dot{y}_{\tau}(\pt)=0$ and for $p>1$, $\dot{y}_{\tau}(\pt^{+})=\dot{y}_{\tau}(-\pt^{-})=0$ 
and locally the vanishing of $\dot{y}$ determines $\pt$ and $\pt^{+}$ and $\pt^{-}$. 
Moreover,  $\ddot{y}=2f'(y(t))$ is nonzero at $t=\pt$ if $p=1$ or at either $\pt^{+}$ and $-\pt^{-}$ if $p>1$
for all $\tau \in (0,\taumax)$.
Analyticity of $\pt^{+}$, $\pt^{-}$ in the case $p>1$ (and hence $\pt=\pt^{+}+\pt^{-}$) and $\pt$ in the case $p=1$ 
now follows from the real analytic Implicit Function Theorem. 

Assume now that $p>1$.
By using \ref{E:y:dot},      \ref{E:y:ic:p:neq:1}, and \ref{E:y:max:min}, we have that 
$$
\pt^+ =
\int_{\ymin}^{q/n} \frac{dy}{\,2\sqrt{y^q(1-y)^p-4\tau^2\,}\,},
\qquad
\pt^- =
\int^{\ymax}_{q/n} \frac{dy}{\,2\sqrt{y^q(1-y)^p-4\tau^2\,}\,}.
$$
Clearly if we substitute the limits
$\ymin$ and $\ymax$
in the above integrals by
$\ymin+\delta$ and $\ymax-\delta$
where $\delta$ is a small positive number,
the integrals we get converge as $\tau\to0$ to constants which depend only on $\delta$.
Moreover since for $y\in[\ymin,\ymin+\delta]$
we have 
$$
(1-\ymin-\delta)^{p/2}
\sqrt{\max(0,y^q-4(\tau')^2)\,}
\le
\sqrt{y^q(1-y)^p-4\tau^2\,}
\le 
\sqrt{y^q-4\tau^2\,},
$$
where $\tau':=\tau(1-\ymin-\delta)^{-p/2}$,
and for $y\in[\ymax-\delta,\ymax]$ we have 
$$
(\ymax-\delta)^{q/2}
\sqrt{\max(0,(1-y)^p-4(\tau'')^2)\,}
\le
\sqrt{y^q(1-y)^p-4\tau^2\,}
\le 
\sqrt{(1-y)^p-4\tau^2\,},
$$
where $\tau''=\tau(\ymax-\delta)^{-q/2}$,
it is enough to prove 
$$
\int^{\ymin+\delta}_{\ymin} \frac{dy}{\,\sqrt{y^q-4\tau^2\,}\,}\sim 2b_q \, T_q(\tau),
\qquad
\int_{\ymax-\delta}^{\ymax} \frac{dy}{\,\sqrt{(1-y)^p-4\tau^2\,}\,}\sim 2b_p \, T_p(\tau).
$$
This follows easily by using \ref{E:y:maxmin:asym}
and integration by substitution
(substituting $z=y(4\tau^2)^{-1/q}$ or $z=(1-y)(4\tau^2)^{-1/p}$ respectively),
and concludes the proof when $p>1$ (recall also \ref{E:part:periods}).

When $p=1$ by using 
\ref{E:y:dot},      \ref{E:y:ic:p:eq:1}, and \ref{E:y0:p:eq:1},  we have
$$
\pt =
\int^{\ymax}_{\ymin} \frac{dy}{\,2\sqrt{y^{n-1}(1-y)-4\tau^2\,}\,}
$$
and as before the proof reduces to 
$$
\int^{\ymin+\delta}_{\ymin} \frac{dy}{\,\sqrt{y^{n-1}-4\tau^2\,}\,}\sim 2b_{n-1}\,T_{n-1}(\tau).
$$
\end{proof}

We introduce now some convenient notation.
Note that the definition of $\Phi_v$ is motivated by the fact that if $\Phi$ is Legendrian
then $\Phi_v$ is also Legendrian 
(see e.g. \cite[Lemma 2.4]{le:min:legn}).

\begin{definition}
\addtocounter{equation}{1}
\label{D:Phi:v}
If $\Phi:\Sigma\to\Sph^{2n-1}\subset\C^n$ is an immersion and $V$ is a normal
(to $\Phi$ in $\Sph^{2n-1}$)
small vector field, we define 
$\Phi_V:\Sigma\to\Sph^{2n-1}$ 
by
$$
\Phi_V=\frac{\Phi+V}{|\Phi+V|},
$$
where we consider $\Phi$ and $V$ as $\C^n$-valued and $|.|$ is the standard length.
If $\Phi$ is Legendrian and $v:\Sigma\to\R$ a function
with locally small enough $C^1$ norm, we also write $\Phi_v$
for $\Phi_
{2v\,J\frac{\partial}{\partial r}+J\nabla v}
$.
\end{definition}

In order to understand the asymptotics of $\pthat$ 
we prove first the following lemma:

\begin{lemma}
\addtocounter{equation}{1}
\label{L:infinitesimal-force}
Let $\mathsf{t} = \left.\frac{d\ttilde_{x}}{dx}\right|_{x=0}\in \lsunit{n}$ be the generator of the $1$-parameter subgroup $\{\ttilde_{x}\}$.
Suppose that $(X_\tau)_\phi$ is special Legendrian where $\phi:\cylpq\to\R$
is a smooth function which depends only on $t$.
The $\mathsf{t}$-flux through the meridian $(\{t\}\times\merpq)$ of $(X_\tau)_\phi$ is given by
$$ \mathcal{F}_{\mathsf{t}} =
\frac{n}{pq} \vol(\merpq)\,
\left(2\tau + (q-ny)\dot{\phi}+n\dot{y}\phi\right)+\hot,
$$
where $\vol(\merpq)=\vol(\Sph^{p-1}) \vol(\Sph^{q-1})$ if $p>1$
or $\vol(\merpq)=\vol(\Sph^{n-2})$ if $p=1$,
and ``$\hot$'' stands for terms which are quadratic or higher order in $\phi$
and its derivatives.
\end{lemma}

\begin{remark}
\addtocounter{equation}{1}
\label{R:linop:integral}
Since $\phi$ depends only on $t$ the linearised equation
$$\Delta_{X^*g_{\Sph^{2n-1}}}\phi+2n\phi=0$$
reduces to  
\addtocounter{theorem}{1}
\begin{equation}
\label{E:linear:phi}
\ddot{\phi}=-2n|\dot{\bw}|^{2}\phi.
\end{equation}
By \ref{E:w:reparam} and \ref{E:y:ddot} we have
$$
(\,(q-ny) \dot{\phi}+n\dot{y}\phi\,) \dot{\phantom{1} }
=
(q-ny) \, |\dot{\bw}|^{2} \, (\Delta_{X^*g_{\Sph^{2n-1}}}\phi+2n\phi),
$$
which shows that the first order linear ODE 
\begin{equation}
\addtocounter{theorem}{1}
\label{E:linop:integral}
(q-ny)\dot{\phi}+n\dot{y}\phi=A
\end{equation}
for any constant $A$, is a first integral of the second order linearised equation
\ref{E:linear:phi}.
Clearly, $\phi=q-ny$ satisfies \ref{E:linop:integral} with $A=0$ and hence is a solution 
of the linearised equation \ref{E:linear:phi}. One can also establish this directly by taking its 
second derivative and using \ref{E:y:ddot} and \ref{E:w:reparam}.

There is a simple geometric explanation for the solution $q-ny$ to the rotationally invariant 
linearised equation and for the fact that this solution 
satisfies \ref{E:linop:integral} with constant $A=0$. 
For any special Legendrian $X$ the variation vector field $V$ 
associated with the $1$-parameter variation $\ttilde_{x} \circ X$ arises from a function $\varphi$
which solves the linearised equation. The function $\varphi$ is  $\varphit \circ X$, where 
\begin{equation}
\addtocounter{theorem}{1}
\label{E:ft}
\varphit(z_1, ... , z_n)  =  
\frac1{2p}\sum_{i=1}^p|z_i|^2
-
\frac1{2q}\sum_{i=p+1}^n|z_i|^2, 
\end{equation}
which satisfies 
\begin{equation}
\addtocounter{theorem}{1}
\label{E:J:ft}
 J\nabla \varphit = \mathsf{t} = \left.\frac{d\ttilde_{x}}{dx}\right|_{x=0},
\end{equation}
 i.e. it is the function whose associated Hamiltonian vector field is $\mathsf{t}$, the infinitesimal generator of $\{\ttilde_{x}\}$.
Recall that the $1$-parameter subgroup $\{\ttilde_{x}\}$ commutes with every $\MMM \in \sorth{p} \times \sorth{q}$. 
Hence for any $x$, $\ttilde_{x} \circ X_{\tau}$ is also an $\sorth{p} \times \sorth{q}$-invariant special Legendrian 
congruent to $X_{\tau}$. Thus $\varphit \circ X_{\tau}$ should be a rotationally invariant function which is just
$$
\varphit \circ X_{\tau}=\frac{q-ny}{2pq}.
$$
$q-ny$ satisfies \ref{E:linop:integral} with $A=0$ 
because $\ttilde_{x} \circ X_{\tau}$ is just a repositioning of $X_{\tau}$ and does not correspond 
to changing the value of $\tau$ to some nearby $\tau'$. The variation vector field corresponding to varying $\tau$ in $X_{\tau}$
also arises from a rotationally invariant solution $\varphi$ of the linearised equation, but by \ref{L:infinitesimal-force}
it must satisfy \ref{E:linear:phi} with a nonzero constant $A$ because it corresponds to changing the value of $\tau$.

Finally, recall (see e.g. \cite[\S 27]{arnold:ode}) that if $\psi$ and $\phi$ are solutions of the second order linear equation \ref{E:linear:phi} then the Wronskian 
$$W_{\psi,\phi}(t):= \psi \dot{\phi} - \dot{\psi} \phi,$$
is constant and $\psi$, $\phi$ span the space of solutions of \ref{E:linear:phi} if and only if this constant is nonzero.
The expression on the LHS of \ref{E:linop:integral} is precisely the Wronskian where $\psi=q - ny$.
Hence to find a basis of solutions for \ref{E:linear:phi} it suffices to find a function $\phi$ satisfying 
\ref{E:linop:integral} with say $A=1$. 
\end{remark}

\begin{proof}
The proof is a rather long calculation making full use of the expression for $X_{\tau}$ in terms of $\bw_{\tau}$, the definition 
of the perturbation of a Legendrian submanifold by a function and repeated use of the 
equations satisfied by $\bw_{\tau}$, in particular \ref{E:w:reparam}, \ref{E:odes:p:n}, \ref{E:ydot:cx} and \ref{E:y:dot}.
Because of the importance of the lemma for calculating the asymptotics of $\pthat$ as $\tau \ra 0$ we give the main steps 
in the calculation. 

To simplify the notation we write $X=X_\tau$ and $Y=(X_\tau)_\phi$.
To compute the $\mathsf{t}$-flux through the meridian $(\{t\}\times\merpq)$ of $Y$ using \ref{E:kflux}
we need to compute the following:
the pullback metric $Y^*g_{\tiny \,\Sph^{2n-1}}$, 
the unit conormal $\eta$ to the meridian, the volume form induced on the meridian by $Y^*g_{\tiny \,\Sph^{2n-1}}$ 
and the inner product $\eta \cdot \mathsf{t}$.

We proceed to calculate $Y^*g_{\tiny \,\Sph^{2n-1}}$.
We first assume that $p>1$.
The definition of $(X)_\phi$ yields
\begin{equation}
\addtocounter{theorem}{1}
\label{E:Y}
Y=X+i|\dot{\bw}|^{-2}\dot{\phi}\,\dot{X}+2i\phi \,X + \hot,
\end{equation}
and therefore
\begin{equation}
\addtocounter{theorem}{1}
\label{E:Y:dot}
\dot{Y}=\dot{X}+i\ddot{X}\,(|\dot{\bw}|^{-2}\dot{\phi})
+i\dot{X}\, (\,\dot{\phi}\, (|\dot{\bw}|^{-2})\!\!\dot{\phantom{X}} \hspace{-0.4em}
+ |\dot{\bw}|^{-2}\ddot{\phi}+2\phi\,)
+2iX\,\dot{\phi}+\hot
\end{equation}
We compute that 
$$ \abs{\dot{Y}}^{2} = \abs{\dot{X}}^{2} + 2\abs{\dot{\bw}}^{-2}\dot{\phi} \,\dot{X} \cdot i\ddot{X} + \hot$$
(many terms vanish using the fact that $X$ is Legendrian in $\Sph^{2n-1}$ and because we only keep terms linear in 
$\phi$ and its derivatives). From the definition of $X_{\tau}$ in terms of $\bw_{\tau}$ we find
$$ \dot{X} \cdot i\ddot{X} = \Imag( \dot{w}_{1}\ddot{\overline{w}}_{1} + \dot{w}_{2}\ddot{\overline{w}}_{2} ). $$ 
Differentiation of \ref{E:odes:p:n} to compute the second derivatives of $\bw$ 
and subsequent persistent use of \ref{E:odes:p:n} to replace all first derivatives of $\bw$ eventually yields 
$$ \dot{w}_{1}\ddot{\overline{w}}_{1} + \dot{w}_{2}\ddot{\overline{w}}_{2} = 
\overline{w}_{1}^{p}\overline{w}_{2}^{q}\, (1-y)^{p-1}y^{q-1}\, \left( (p-1) \frac{\abs{w_{2}}^{2}}{\abs{w_{1}}^{2}} + (p-q) - (q-1) \frac{\abs{w_{1}}^{2}}{\abs{w_{2}}^{2}}\right).$$
Using \ref{E:w:reparam}, \ref{E:ydot:cx} and some algebraic manipulation we obtain 
$$ \dot{X} \cdot i\ddot{X} = 2\tau\abs{\dot{\bw}}^{2} \left(\frac{p-1}{|w_1|^2}-\frac{q-1}{|w_2|^2}\right),$$
and hence 
\begin{equation}
\addtocounter{theorem}{1}
\label{E:norm:Ydot}
\abs{\dot{Y}}^{2} = \abs{\dot{\bw}}^{2}+ 4\tau \left(\frac{p-1}{|w_1|^2}-\frac{q-1}{|w_2|^2}\right) \dot{\phi} + \hot
\end{equation}

We denote by $x$ and $y$ parametrisations of $\Sph^{p-1}$ and $\Sph^{q-1}$
with coordinates $\{x^j\}$ and $\{y^k\}$ respectively and calculate 
$$
\frac{\partial Y}{\partial x^j}
=
\frac{\partial X}{\partial x^j}
+
i \frac{\partial^2 X}{\partial t \partial x^j}
|\,\dot{\bw}|^{-2}\dot{\phi}
+
i\frac{\partial X}{\partial x^j}\, 2\phi + \hot
$$
and similarly for $\partial Y/\partial y^k$.
Using $X=w_{1} x + w_{2}y$ one can verify that 
$$ \dot{X} \cdot \tfrac{\partial{X}}{\partial{x^{j}}} = \dot{X} \cdot i\tfrac{\partial{X}}{\partial{x^{j}}} =
\dot{X} \circ i \tfrac{\partial^{2}X}{\partial{t}\partial{x^{j}}} = \tfrac{\partial{X}}{\partial{x^{j}}} \cdot i\ddot{X} = 
\tfrac{\partial{X}}{\partial{x^{j}}} \cdot iX =0,$$
for all $j$. The same vanishing is also true if we use the coordinates $y^{k}$ in place of $x^{j}$.
Since the only terms in $\dot{Y} \cdot \tfrac{\partial{Y}}{\partial{x}^{j}}$ or 
$\dot{Y} \cdot \tfrac{\partial{Y}}{\partial{y}^{k}}$
that are linear in $\phi$ and its derivatives are linear combinations of these vanishing terms above we conclude that 
\begin{equation}
\addtocounter{theorem}{1}
\label{E:dotY:Yx}
\dot{Y} \cdot \frac{\partial{Y}}{\partial{x}^{j}} = \dot{Y} \cdot \frac{\partial{Y}}{\partial{y}^{k}} = 0 + \hot
\quad \text{for all $j$ and $k$.}
\end{equation}
One also has 
\begin{equation}
\addtocounter{theorem}{1}
\label{E:Yx:Yy}
\frac{\partial{Y}}{\partial{x^{j}}} \cdot \frac{\partial{Y}}{\partial{y^{k}}} =0 + \hot \quad 
\text{for any $j$ and $k$}.
\end{equation}

We compute 
$$ \frac{\partial{X}}{\partial{x^{j}}} \cdot i\frac{\partial^{2}X}{\partial{t}\partial{x^{j'}}} = 
\Imag(w_{1} \dot{\overline{w}}_{1}) \frac{\partial{x}}{\partial{x^{j}}} \cdot \frac{\partial{x}}{\partial{x^{j'}}} 
=-2\tau \frac{\partial{x}}{\partial{x^{j}}} \cdot \frac{\partial{x}}{\partial{x^{j'}}},$$
and 
$$ \frac{\partial{X}}{\partial{y^{k}}} \cdot i\frac{\partial^{2}X}{\partial{t}\partial{y^{k'}}} = 
\Imag(w_{2} \dot{\overline{w}}_{2}) \frac{\partial{y}}{\partial{y^{k}}} \cdot \frac{\partial{y}}{\partial{y^{k'}}} 
=2\tau \frac{\partial{y}}{\partial{y^{k}}} \cdot \frac{\partial{y}}{\partial{y^{k'}}}.$$
Hence using the fact that 
$$\frac{\partial{X}}{\partial{x^{j}}} \cdot i  \frac{\partial{X}}{\partial{x^{j'}}} = 
\frac{\partial{X}}{\partial{y^{k}}} \cdot    i\frac{\partial{X}}{\partial{y^{k'}}} =0,$$
for all $j,j'$ and $k,k'$ we obtain 
\begin{equation}
\addtocounter{theorem}{1}
\label{E:Yx:Yx}
 \frac{\partial{Y}}{\partial{x^{j}}} \cdot  \frac{\partial{Y}}{\partial{x^{j'}}} = 
(\abs{w_{1}}^{2} - 4\tau \abs{\dot{\bw}}^{-2} \dot{\phi}) \frac{\partial{x}}{\partial{x^{j}}} \cdot   \frac{\partial{x}}{\partial{x^{j'}}}+ \hot,
\end{equation}
and 
\begin{equation}
\addtocounter{theorem}{1}
\label{E:Yy:Yy}
 \frac{\partial{Y}}{\partial{y^{k}}} \cdot   \frac{\partial{Y}}{\partial{y^{k'}}} = 
(\abs{w_{2}}^{2} + 4\tau \abs{\dot{\bw}}^{-2} \dot{\phi}) \frac{\partial{y}}{\partial{y^{k}}} \cdot   \frac{\partial{y}}{\partial{y^{k'}}}+\hot
\end{equation}
Combining \ref{E:norm:Ydot}, \ref{E:dotY:Yx}, \ref{E:Yx:Yy}, \ref{E:Yx:Yx} and \ref{E:Yy:Yy} we have 
\begin{multline}
\addtocounter{theorem}{1}
\label{E:pb:y:pgt1}
Y^*g_{\tiny \,\Sph^{2n-1}}=
\left(|\dot{\bw}|^{2}+4\tau\left(\frac{p-1}{|w_1|^2}-\frac{q-1}{|w_2|^2}\right)\dot{\phi}\right)\,dt^2
+
\\
(|w_1|^2-4\tau|\dot{\bw}|^{-2}\dot{\phi})\,g_{\tiny \,\Sph^{p-1}}
+
(|w_2|^2+4\tau|\dot{\bw}|^{-2}\dot{\phi})\,g_{\tiny \,\Sph^{q-1}} + \hot 
\end{multline}
for $p>1$.
By the same methods for $p=1$ we obtain 
\begin{equation}
\addtocounter{theorem}{1}
\label{E:pb:y:peq1}
Y^*g_{\tiny \,\Sph^{2n-1}}=
\left(|\dot{\bw}|^{2}-4\tau(n-2)|w_2|^{-2}\dot{\phi}\right)\,dt^2
+
(|w_2|^2+4\tau|\dot{\bw}|^{-2}\dot{\phi})\,g_{\tiny \,\Sph^{n-2}} + \hot
\end{equation}
Hence in both cases the unit conormal $\eta$ to the meridian $\{t\} \times \merpq$ is 
$$ \eta = \frac{\dot{Y}}{\abs{\dot{Y}}} + \hot$$
Combining this with \ref{E:pb:y:pgt1} and \ref{E:pb:y:peq1} we have  
\begin{equation*}
\eta \cdot \mathsf{t} \dvol = 
\begin{cases}
 \left( {\abs{\dot{Y}}^{-2}}\,(|w_1|^2-4\tau|\dot{\bw}|^{-2}\dot{\phi})^{p-1}\, 
(|w_2|^2+4\tau|\dot{\bw}|^{-2}\dot{\phi})^{q-1} \right)^{1/2} \dot{Y} \cdot {\mathsf{t}} \dvol_{\Sph^{p-1}} \dvol_{\Sph^{q-1}} \quad &\text{if $p>1$;}\\
\left( {\abs{\dot{Y}}^{-2}}\,(|w_2|^2+4\tau|\dot{\bw}|^{-2}\dot{\phi})^{n-2} \right)^{1/2} \dot{Y} \cdot {\mathsf{t}} \dvol_{\Sph^{n-2}} \quad &\text{if $p=1;$}
\end{cases}
\end{equation*}
up to higher order terms.
Expanding and keeping only the lowest order terms we find
$$ \left( {\abs{\dot{Y}}^{-2}}\,(|w_1|^2-4\tau|\dot{\bw}|^{-2}\dot{\phi})^{p-1}\, 
(|w_2|^2+4\tau|\dot{\bw}|^{-2}\dot{\phi})^{q-1} \right)^{1/2} = 1 - 4\tau \abs{\dot{\bw}}^{-2} 
\left( \frac{p-1}{|w_1|^2}-\frac{q-1}{|w_2|^2}\right) \dot{\phi}$$
and 
$$\left( {\abs{\dot{Y}}^{-2}}\,(|w_2|^2+4\tau|\dot{\bw}|^{-2}\dot{\phi})^{n-2} \right)^{1/2} = 
1 + 4\tau (n-2) \abs{\dot{\bw}}^{-2}{|w_2|^{-2}}\dot{\phi},$$
up to higher order terms. 
Hence we have 
\begin{equation}
\addtocounter{theorem}{1}
\label{E:eta:dot:t}
\eta \cdot \mathsf{t} \dvol = 
\begin{cases}
\left(1 - 4\tau \abs{\dot{\bw}}^{-2}  \left( \frac{p-1}{|w_1|^2}-\frac{q-1}{|w_2|^2}\right) \dot{\phi}\right)
 \dot{Y} \cdot {\mathsf{t}} \dvol_{\Sph^{p-1}} \dvol_{\Sph^{q-1}} + \ \hot \quad &\text{if $p>1$;}\\
\left(1 + 4\tau (n-2)\abs{\dot{\bw}}^{-2}{|w_2|^{-2}}\dot{\phi}\right) \dot{Y} \cdot {\mathsf{t}} \dvol_{\Sph^{n-2}} +\ \hot \quad &\text{if $p=1$}.
\end{cases}
\end{equation}

It remains to calculate $\dot{Y} \cdot \mathsf{t}$.
Recall that $\mathsf{t}$ at the point $(z_1, ... , z_n)\in\C^n$ is given by
$$
\mathsf{t}(z_1, ... , z_n)  =  i\,(z_1/p, ... ,z_p/p,-z_{p+1}/q, ... ,-z_n/q).
$$
Since $Y$ is special Legendrian, and $\phi$ depends only on $t$ 
it satisfies the linearized equation \ref{E:linear:phi}.
Substituting \ref{E:linear:phi} into  \ref{E:Y:dot}
and also using the expression for $Y$ from \ref{E:Y} we calculate that
\begin{multline*}
\dot{Y} \cdot {\mathsf{t}}=\\
\frac1p\Real\left(
-i\dot{w}_1\overline{w}_1+
\ddot{w}_1\overline{w}_1  |\dot{\bw}|^{-2}\dot{\phi}+
\dot{w}_1\overline{w}_1
(|\dot{\bw}|^{-2})\!\!\dot{\phantom{X}} \dot{\phi}
-
\dot{w}_1\overline{w}_1
2n\phi
+|w_1|^2 2\dot{\phi}
-|\dot{w}_1|^2
|\dot{\bw}|^{-2} \dot{\phi}
\right)
\\
-\frac1q\Real\left(
-i\dot{w}_2\overline{w}_2+
\ddot{w}_2\overline{w}_2  |\dot{\bw}|^{-2}\dot{\phi}+
\dot{w}_2\overline{w}_2
(|\dot{\bw}|^{-2})\!\!\dot{\phantom{X}} \dot{\phi}
-
\dot{w}_2\overline{w}_2
2n\phi
+|w_2|^2 2\dot{\phi}
-|\dot{w}_2|^2
|\dot{\bw}|^{-2} \dot{\phi}
\right).
\end{multline*}
We claim that this expression for $\dot{Y} \cdot {\mathsf{t}}$ can be simplified to 
\begin{equation}
\addtocounter{theorem}{1}
\label{E:t:ydot}
\dot{Y} \cdot {\mathsf{t}} =
\frac n {pq}
\left( 2\tau
+\left(q|w_1|^2-p|w_2|^2+8\tau^2|\dot{\bw}|^{-2}\left( \frac{p-1}{|w_1|^2}-\frac{q-1}{|w_2|^2}\right)\right)\dot{\phi}
+n\dot{y}\phi\right).
\end{equation}
Granted this claim the Lemma follows by using \ref{E:eta:dot:t} and \ref{E:t:ydot} to evaluate the $\mathsf{t}$-flux 
integral \ref{E:kflux} up to higher order terms.

For completeness we indicate how to obtain \ref{E:t:ydot}. The zero order terms and the terms involving only $\phi$ are easily 
computed using \ref{E:odes:p:n}, \ref{E:ydot:cx} and \ref{E:w:reparam}. Combining the eight terms involving $\dot{\phi}$ 
in the expression above \ref{E:t:ydot} to yield the coefficient of $\dot{\phi}$ in \ref{E:t:ydot} is more involved.
First we rewrite the eight terms appearing as the coefficient of $\dot{\phi}$ in the form 
\begin{multline}
\addtocounter{theorem}{1}
\label{E:phi:dot:coeff}
\frac{1}{p}\left(\phantom{|} \partial_{t}(\Real(\overline{w}_{1}\dot{w}_{1})\abs{\dot{\bw}}^{-2}) + 2\abs{w_{1}}^{2}-2\abs{w_{2}^{2}}\right) - \frac{1}{q}\left( \partial_{t}(\Real(\overline{w}_{2}\dot{w}_{2})\abs{\dot{\bw}}^{-2}) + 2\abs{w_{2}}^{2}-2\abs{w_{1}^{2}}\right)  =\\
\frac{n}{pq} \left( \partial_{t}(\Real(\overline{w}_{1}\dot{w}_{1})\abs{\dot{\bw}}^{-2}) +2\abs{w_{1}}^{2}-2\abs{w_{2}}^{2}\right).
\end{multline}
Rewrite $\Real(\overline{w}_{1}\dot{w}_{1})\abs{\dot{\bw}}^{-2} $ as $\Real(\overline{w}_{1}\overline{w}_{2}w_{1}^{1-p} w_{2}^{1-q})$ using \ref{E:odes:p:n} and \ref{E:w:reparam}. Repeated use of \ref{E:odes:p:n} yields
$$ \partial_{t}(\overline{w}_{1}\overline{w}_{2}w_{1}^{1-p} w_{2}^{1-q}) = \abs{w_{2}}^{2}-\abs{w_{1}}^{2} + 
\frac{\overline{w}_{1}^{p}\overline{w}_{2}^{q}}{w_{1}^{p}w_{2}^{q}}
\left( (1-p) \abs{w_{2}}^{2} - (1-q)\abs{w_{1}}^{2}\right),
$$
while \ref{E:w:reparam}, \ref{E:ydot:cx} and \ref{E:y:dot} imply that 
$$\Real\left(\frac{\overline{w}_{1}^{p}\overline{w}_{2}^{q}}{w_{1}^{p}w_{2}^{q}}\right) = 1 - \frac{8\tau^{2}}{f(y)} = 1 - \frac{8\tau^{2}}{\abs{\dot{\bw}}^{2}\abs{w_{1}}^{2}\abs{w_{2}}^{2}}.$$
Hence 
\begin{equation}
\addtocounter{theorem}{1}
\label{E:phi:dot:coeff2}
\partial_{t}(\Real(\overline{w}_{1}\dot{w}_{1})\abs{\dot{\bw}}^{-2})  = (2-p) \abs{w_{2}}^{2} - (2-q) \abs{w_{1}}^{2} + 8 \tau^{2} \abs{\dot{\bw}}^{-2} \left( \frac{p-1}{\abs{w_{1}}^{2}} - \frac{q-1}{w_{2}^{2}}\right).
\end{equation}
Combining \ref{E:phi:dot:coeff} with \ref{E:phi:dot:coeff2} gives us the coefficient of $\dot{\phi}$ as it appears in  
\ref{E:t:ydot}.
\end{proof}

The asymptotics of the angular period $\pthat$ as $\tau \ra 0$ will follow from the following result which expresses 
the derivative of the angular period for all values of $\tau$ in terms of the behaviour of a particular (rotationally invariant) 
solution $Q$ (depending on $\tau$) of the linearised equation \ref{E:linear:phi}.
\begin{lemma}
\addtocounter{equation}{1}
\label{L:phat:dt}
The angular period $\pthat$ is an analytic function of $\tau$ for $\tau \in (0,\taumax)$.
For any $0<\tau<\taumax$ the derivative of the angular period $\pthat$ satisfies 
\begin{equation}
\addtocounter{theorem}{1}
\label{E:pthat:dt:peq1}
\frac {d\pthat}{d\tau}  =
4(n-1)\left(\frac{Q(\pt)}{\,q-ny(\pt)}-\frac{Q(0)}{\,q-ny(0)}\right), \qquad \text{when $p=1$;}\\
\end{equation}
\begin{equation}
\addtocounter{theorem}{1}
\label{E:pthat:dt:pgt1}
\frac {d\pthat}{d\tau}  =
4pq\left( \frac{Q(\pt^+)}{\,q-ny(\pt^+)} - \frac{Q(-\pt^-)}{\,q-ny(-\pt^-)}\right),
\qquad \text{when $p>1$;}
\end{equation}
where $Q(t)$ is the unique solution to the rotationally-invariant linearised equation \ref{E:linear:phi}
with initial data
\begin{eqnarray}
\addtocounter{theorem}{1}
\label{E:Q:peq1}
n\,\dot{y}(\pt^{*})\,Q(\pt^*) &= 1,
\quad
\dot{Q}(\pt^*) & =0,
\qquad\text{when $p=1$};\\
\addtocounter{theorem}{1}
\label{E:Q:pgt1}
n\,\dot{y}(0)\,Q(0) &= 1,
\quad
\dot{Q}(0) & =0,
\qquad\text{when $p>1$;}
\end{eqnarray}
where for $p=1$  $\pt^*$  is the unique $t\in (0,\pt)$ such that
$y(\pt^*) = \tfrac{n-1}{n}=\tfrac{q}{n}$.
\end{lemma}

\begin{remark}
\addtocounter{equation}{1}
\label{R:Q:smooth}
For $p=1$ $y(\pt^{*})=q/n$ and $\pt^{*}$ is locally characterised by this property. 
Also $\dot{y}(\pt^{*}) = -4 \sqrt{\taumax^{2}-\tau^{2}} \neq 0$ for $\tau \in (-\taumax,\taumax)$. 
Hence by the real analytic Implicit Function Theorem $\pt^{*}$ is an analytic function of $\tau$ in $(-\taumax,\taumax)$. 
In particular it is bounded independent of $\tau$ as $\tau \ra 0$. 
For $p>1$ $y(0)=q/n$  and $\dot{y}(0) = - 4 \sqrt{\taumax^{2}-\tau^{2}}$. 

Hence in both cases the initial conditions for $Q$  vary analytically with $\abs{\tau}<\taumax$. 
Also by \ref{P:X:tau}.i the coefficients of the linearised equation \ref{E:linear:phi} depend analytically on $\tau \in (-\taumax,\taumax)$.
Combining all these facts we see that the solution $Q$ to \ref{E:linear:phi} defined above depends analytically on $\tau \in (-\taumax,\taumax)$. Therefore, if $t$ stays in a bounded interval $I\subset \R$ then $\sup_{t\in I}\abs{Q(t)}$ is bounded independent of $\tau$ as $\tau \ra 0$. In particular, the term $Q(0)$ appearing in \ref{E:pthat:dt:peq1} is bounded as $\tau \ra 0$.
\end{remark}

\begin{proof}
Real analyticity of $\pthat$ for $\tau \in (0,\taumax)$ follows from real analyticity of $\bw_{\tau}$, $\pt^{+}$, $\pt^{-}$ 
and $\pt$ and the definition of $\pthat$ (\ref{E:pthat}).
We fix any $\tau\in (0,\taumax)$ and consider $\sigma$ sufficiently close to $\tau$ which we will allow to vary.

Consider first the case $p=1$. By \ref{P:xtau:sym:p:eq:1} $X:=X_{\tau}$ has the symmetries
$$ \tbartilde \circ X = X \circ \tbar, \qquad \tbartilde_{\pthat} \circ X = X \circ \tbar_{\pt}.$$
$Y:= X_{\sigma}$ shares the $\tbartilde$ symmetry 
\begin{equation}
\addtocounter{theorem}{1}
\label{E:Y:tbar}
\tbartilde \circ Y = Y \circ \tbar,
\end{equation}
but not the symmetry with respect to $\tbartilde_{\pthat}$ (because we have changed from $\tau$ to $\sigma$).
However,  the following repositioned and reparametrised version of $X_{\sigma}$
$$Z:=\ttilde_{\pthat-\psighat} \circ X_\sigma \circ \TTT_{\psig-\pt},$$
does share the other ($\sigma$-independent) symmetry of $X$, i.e.
\begin{equation}
\addtocounter{theorem}{1}
\label{E:Z:tbarpt}
\tbartilde_{\pthat} \circ Z  = Z \circ \tbar_{\pt}.
\end{equation}
Since $\{\ttilde_{x}\}$  commutes with $\orth{n-1}$ by \ref{P:commute:p:eq:1}ii the 
immersion $Z$ is $\orth{n-1}$-invariant like $X$ and $Y$.

When $p>1$ we write 
$$X:=X_\tau, \quad 
Y:=\ttilde_{x^-} \circ X_\sigma \circ \TTT_{\pt^--\psig^-}, \quad 
Z:=\ttilde_{\pthat-\psighat} \circ Y \circ \TTT_{\psig-\pt},$$
where $x^-$ is defined to be the small number
which ensures that the symmetries of $X$ in
\ref{E:tbar:ptminus} and \ref{E:tbar:ptplus} 
(or \ref{E:tbar:ptminus:p:eq:q} and \ref{E:tbar:ptplus:p:eq:q})
apply to $Y$ and $Z$ respectively as
\begin{equation}
\addtocounter{theorem}{1}
\label{E:Y:tbarminus}
\tbartildeminus \circ Y  = Y \circ \tbar_{\,-\pt^-}
\end{equation}¥
and
\begin{equation}
\addtocounter{theorem}{1}
\label{E:Y:tbarplus}
\tbartildeplus \circ Z  = Z \circ \tbar_{\pt^+},
\end{equation}
where 
$\tbartildeminus$ and $\tbartildeplus$ are defined in \ref{E:tbartildeminus}
and \ref{E:tbartildeplus} respectively (and are independent of $\sigma$).


Provided $\sigma$ is sufficiently close to $\tau$
we clearly have 
unique small
vector fields $V$ and $W$
normal to $\left.X\right|_{(-2\pt,2\pt)\times\merpq}$,
and diffeomorphisms close to the identity $D_\sigma, E_\sigma:(-2\pt,2\pt)\times\merpq\to\cylpq$,
such that on $(-2\pt,2\pt)\times\merpq$
$$Y=X_V\circ D_\sigma, \quad \text{and} \quad Z=X_W\circ E_\sigma.$$
Clearly $V, W, D_\sigma, E_\sigma$ are smooth and depend smoothly on $\sigma$.
Moreover by the appropriate version of the
Legendrian neighbourhood theorem
(see e.g. \cite[Lemma 2.4]{le:min:legn})
there are unique small smooth functions $\phitilde_{\sigma},\varphitilde_{\sigma}:(-2\pt,2\pt)\times\merpq\to\R$
depending smoothly on $\sigma$ 
such that 
$$
V=2\phitilde_{\sigma}\,J\frac{\partial}{\partial r}+J\nabla\phitilde_{\sigma},
\qquad
W=2\varphitilde_{\sigma}\,J\frac{\partial}{\partial r}+J\nabla\varphitilde_{\sigma},
$$
and therefore by \ref{D:Phi:v} 
$X_V=X_{\phitilde_{\sigma}}$ and $X_W=X_{\varphitilde_{\sigma}}$.


We want to show that $\phitilde_{\sigma}$ and $\varphitilde_{\sigma}$ inherit certain symmetries from the 
symmetries of $X$, $Y$ and $Z$ given above. We claim that $\phitilde_{\sigma}$ and $\varphitilde_{\sigma}$ depend only on $t$
and that \begin{eqnarray}
\addtocounter{theorem}{1}
\label{E:phitilde:sym:peq1}
-\phitilde_{\sigma}\circ \tbar=\phitilde_{\sigma}, \quad -\varphitilde_{\sigma} \circ \tbar_{\pt} = \varphitilde_{\sigma} 
\qquad & \text{if $p=1$;}\\
\addtocounter{theorem}{1}
\label{E:phitilde:sym:pgt1}
-\phitilde_{\sigma}\circ \tbar_{-\pt^{-}}=\phitilde_{\sigma}, \quad -\varphitilde_{\sigma} \circ \tbar_{\pt^{+}} = \varphitilde_{\sigma} \qquad & \text{if $p>1$.}
\end{eqnarray}

To see the first symmetry of \ref{E:phitilde:sym:peq1} we combine \ref{P:commute:p:eq:1}.iii with the $\tbartilde$ symmetry of $X$ to obtain
$$\tbartilde\circ X_{\phitilde_{\sigma}}=(\tbartilde\circ X)_{-\phitilde_{\sigma}}=(X\circ\tbar)_{-\phitilde_{\sigma}}=X_{-\phitilde_{\sigma}\circ\tbar}\circ\tbar.$$
Combining this with the symmetry \ref{E:Y:tbar} of $Y$  we conclude
$$
X_{-\phitilde_{\sigma}\circ\tbar}\circ\tbar \circ D_\sigma =X_{\phitilde_{\sigma}}\circ D_\sigma \circ \tbar.
$$
The uniqueness statement in the Legendrian neighbourhood theorem now implies that
$-\phitilde_{\sigma}\circ \tbar=\phitilde_{\sigma}$ as required. 
Arguing in the same way using the symmetry of $X$ and  $Z$ (\ref{E:Z:tbarpt}) with respect to $\tbartilde_{\pthat}$ 
 we conclude that $\varphitilde$ satisfies the second symmetry of \ref{E:phitilde:sym:peq1}.
The analogous argument using the symmetry of $X$, $Y$ and $Z$ under any 
$\MMM \in \orth{n-1}$ (recall \ref{E:orth:p:eq:1})
implies that 
$$\phitilde_{\sigma}\circ \MMM=\phitilde_{\sigma}, \quad \varphitilde_{\sigma}\circ \MMM=\varphitilde_{\sigma},
\quad \text{for any\ }  \MMM \in \orth{n-1},$$
and therefore $\phitilde_{\sigma}$ and $\varphitilde_{\sigma}$ depend only on $t$.
For $p>1$ the same sort of arguments establish the symmetries in \ref{E:phitilde:sym:pgt1} and the rotational symmetry
of $\phitilde_{\sigma}$ and $\varphitilde_{\sigma}$.

Linearising $Z:=\ttilde_{\pthat-\psighat} \circ Y \circ \TTT_{\psig-\pt}$ around $\sigma=\tau$, using 
\ref{E:J:ft} and comparing normal components we obtain the following important equality
\addtocounter{theorem}{1}
\begin{equation}
\label{E:varphi-phi}
\varphi=\phi\,-\left(\left. \frac { d\pthat } {d\tau}\right|_\tau \right)\varphit\circ X, 
\end{equation}
with $\varphit$ as defined in \ref{E:ft} and
$$
\varphi=\left.\frac{d\varphitilde_{\sigma}}{d\sigma}\right|_{\sigma=\tau}
\quad \text{and} \quad 
\phi=\left.\frac{d\phitilde_{\sigma}}{d\sigma}\right|_{\sigma=\tau}.
$$
Recall that 
$$ \varphit \circ X = \frac{q-ny}{2pq}.$$
Differentiating the expression for the linearised $\mathsf{t}$-flux from lemma 
\ref{L:infinitesimal-force} we find that 
$\phi$ satisfies
\begin{equation}
\addtocounter{theorem}{1}
\label{E:phi:integral}
(q-ny)\dot{\phi}+n\dot{y}\phi =2.
\end{equation}
Using the initial conditions for $Q$ given in \ref{E:Q:peq1} and \ref{E:Q:pgt1} we see that the Wronskian $W(t)$
of $q-ny$ with $Q$ satisfies
\begin{equation}
\addtocounter{theorem}{1}
\label{E:q:integral}
 W(t) := (q-ny)\dot{Q} + n\dot{y} \,Q \equiv 1,
\end{equation}
and hence by Remark \ref{R:linop:integral} $Q$ and $q-ny$ span the solution space of the linearised equation 
\ref{E:linear:phi}.
Hence from \ref{E:phi:integral} there is a unique constant $b$ such that 
\begin{equation}
\addtocounter{theorem}{1}
\label{E:phi}
\phi=2\,\left(\,b(q-ny) + Q \,\right).
\end{equation}
$\phi$ and $\varphi$ inherit the symmetries of $\phitilde_{\sigma}$ and $\varphitilde_{\sigma}$ 
\ref{E:phitilde:sym:peq1} and \ref{E:phitilde:sym:pgt1} (for the case $p=1$ and $p>1$ respectively).
In particular we have
$$ \phi(0)=0, \quad \text{and} \quad \varphi(\pt)=0 \quad \text{when $p=1$;}$$
and 
$$ \phi(-\pt^{-}) = 0, \quad \text{and} \quad \varphi(\pt^{+}) = 0 \quad \text{when $p>1$.}$$
We determine $b$ by using the values of $\phi$ given above
\begin{equation}
\addtocounter{theorem}{1}
\label{E:b}
 b = -\frac{Q(0)}{q-ny(0)} \quad \text{if $p=1$,} \quad \text{or}  \quad 
b= - \frac{Q(-\pt^{-})}{q-ny(-\pt^{-})} \quad \text{if $p>1$.}
\end{equation}
Similarly \ref{E:varphi-phi} together with the above values of $\varphi$  implies
$$ \frac{d \pthat}{d\tau} = \frac{2(n-1)}{q-ny(\pt)} \phi(\pt) \quad \text{if $p=1$},$$
and 
$$ \frac{d \pthat}{d\tau} = \frac{2pq}{q-ny(\pt^{+})} \phi(\pt^{+}) \quad \text{if $p>1$}.$$
Combining these expressions with \ref{E:phi} and \ref{E:b} yields \ref{E:pthat:dt:peq1} and \ref{E:pthat:dt:pgt1}.
\end{proof}

\begin{prop}
\addtocounter{equation}{1}
\label{P:asymptotics-pthat}
For $\tau>0$ the angular period satisfies
\addtocounter{theorem}{1}
\begin{equation}
\label{E:phat}
\frac{d \pthat}{d \tau} \sim 
\frac{4p}q\pt
\qquad\quad
\pthat- \frac{\pi}{2} \sim
\frac{4p}q\tau\pt.
\end{equation}
\end{prop}

\begin{proof}
This will follow easily from Lemma \ref{L:phat:dt} by estimating the appropriate values of $Q$ when $\tau$ is sufficiently small, 
once we have proved that $\pthat\to \frac{\pi}{2}$ as $\tau\to 0$.

We deal first with the case $p=1$.
Recall that $\pt^*$ is the unique $t\in (0,\pt)$ such that $y(\pt^*) = \tfrac{n-1}{n}$ and that it 
approaches a finite limit as $\tau \ra 0$. 
Recall the function $\Psi$ defined in \ref{E:psi:def}. 
The initial condition for $y$ together with
\ref{E:psi:real} and \ref{E:psi:imag} implies that $\Psi \in (0,\tfrac{\pi}{2})$
for $t\in (0,\pt)$. 
From \ref{E:psi:real} we have $\cos{\Psi(\pt^{*})} = \frac{\tau}{\taumax}$ and therefore 
\addtocounter{theorem}{1}
\begin{equation}
\label{E:Psi:max}
\Psi(\pt^*) = \tfrac{\pi}{2}  - \arcsin(\tfrac{\tau}{\taumax}) = \tfrac{\pi}{2} + \alpha_\tau,
\end{equation}
where $\alpha_\tau = -\arcsin(\tau/\taumax)$ as in \ref{E:alpha:tau}.
\ref{E:Psi:max} implies that 
$$ \lim_{\tau \ra 0} \Psi(\pt^{*}) := \psi_{1}(\pt^{*})+(n-1)\psi_{2}(\pt^{*}) = \frac{\pi}{2}.$$

%
Using \ref{E:y:psi} we calculate
$$
\frac{d\psi_2}{dy} =
\frac{\dot{\psi}_2}{\dot{y}} =
-\frac{2\tau}{y\dot{y}}
\qquad
\text{ and }
\qquad
\frac{d\psi_1}{dy} =
 \frac{\dot{\psi}_1}{\dot{y}} = \frac{2\tau}{(1-y)\dot{y}} ,
$$
and therefore
$$
 \psi_2(\pt^*) = 
\int^{\ymax}_{(n-1)/n}
\frac{2\tau}{y\dot{y}} dy
\qquad
\text{and}
\qquad
\psi_1(\pt) - \psi_1(\pt^*) =
\int_{(n-1)/n}^{\ymin}\frac{2\tau}{(1-y)\dot{y}} dy.
$$
We claim that both these integrals converge to zero and hence $\pthat=  \psi_1(\pt)$  (recall \ref{E:psi:period:p:eq:1})
converges to $\pi/2$ as desired. 
We can see that these integrals converge to zero as follows. Since in the integral for $\psi_{2}(\pt^{*})$,
$y$ belongs to the interval $(\tfrac{n-1}{n},y_{\text{max}}) \subset (\tfrac{n-1}{n},1)$ we have
\begin{equation*}
2\tau \pt^* < -\psi_2(\pt^*) < \frac{2n\tau}{n-1} \pt^*.
\end{equation*}
Since $\pt^{*}$ is bounded as $\tau \ra 0$, $\psi_{2}(\pt^{*})$ converges to zero as $\tau \ra 0$.
Similarly, using the obvious upper and lower bounds for $1-y$ in the integral for $\psi_{1}(\pt) - \psi_{1}(\pt^{*})$ we obtain
\begin{equation*}
2\tau (\pt- \pt^*) < \psi_1(\pt) - \psi_1(\pt^*) < 2\tau n (\pt - \pt^*).
\end{equation*}
Hence by the asymptotics for $\pt$ established in \ref{P:asymptotics-ptau} we see
$\psi_1(\pt) - \psi_1(\pt^*) \ra 0$.

The argument in the case $p>1$ is very similar. 
At $t= \pt^+$ (or $t=-\pt^-$) we have $\dot{y}=0$ and $y=\ymin$ (or $y=\ymax$).
Hence \ref{E:psi:real:pneq1} and \ref{E:psi:imag:pneq1} imply that $e^{i(\Psi + \alpha_\tau)} = e^{-i\pi/2}$ and therefore 
$\Psi = -\tfrac{\pi}{2} - \alpha_\tau$ at $t=\pt^+$ (or $t=-\pt^-$).
It follows using \ref{E:psi:period:p:neq:1} that
\begin{equation}
\addtocounter{theorem}{1}
\label{E:pthat:pgt1}
\pthat = \tfrac{\pi}{2} + \alpha_\tau + p\,\psi_1(\pt^+) + q\, \psi_2(-\pt^-).
\end{equation}
By analysing the functions $\psi_1$ on $(0,\pt^+)$ and
$\psi_2$ on $(-\pt^-,0)$ as above we find 
\begin{equation*}
2p\tau \pt^+ < p\psi_1(\pt^+) < 2n\tau \pt^+,
\quad \text{and} \quad 
2q\tau \pt^- < q\psi_2(-\pt^-) < 2n\tau \pt^-.
\end{equation*}
Hence by \ref{P:asymptotics-ptau} and the definition of $\alpha_{\tau}$ 
all three nonconstant terms on the RHS of \ref{E:pthat:pgt1} converge to zero as $\tau \ra 0$.


It remains only to understand the small $\tau$ asymptotics of the values of $Q$
which appear in \ref{E:Q:peq1} and \ref{E:Q:pgt1}. 
To achieve this we subdivide the interval $(0,\pt)$ when $p=1$
or $(-\pt^-,\pt^+)$ when $p>1$ as in the proof of
\ref{P:asymptotics-ptau}.
By Remark \ref{R:Q:smooth} we obtain bounds on $Q(t)$ independent of $\tau$ except when $y(t)$ is close to $\ymin$
and in the case $p>1$ also when $y(t)$ is close to $\ymax$.
To deal with these regions we notice that 
away from zeros of $q-ny$ the first order ODE for $Q$ \ref{E:q:integral} can be rewritten as 
\begin{equation}
\addtocounter{theorem}{1}
\label{E:Q:integral2}
\left(
\frac Q {q-ny}
\right)
\dot{\!\!\!\!\!\phantom{\frac11}  }
=\frac1{\,(q-ny)^2\,}.
\end{equation}
Using \ref{E:y:dot} and \ref{E:Q:integral2} 
we see that in the vicinity of $\ymax$ or $\ymin$, 
$\dot{Q}$ is close to $(q-ny)^{-1}$,
which is close to either $-1/p$ or $1/q$ respectively.
Using the asymptotics from \ref{P:asymptotics-ptau} we conclude
$$Q(\pt)\sim\frac{1}{n-1}\pt \quad \text{when $p=1$},$$
or 
$$Q(-\pt^-)\sim \frac{1}{p}\pt^- \quad \text{and} \quad Q(\pt^+)\sim\frac{1}{q}\pt^+ \quad \text{when $p>1$},$$
which together with Lemma \ref{L:phat:dt} implies the result for $\frac{d\pthat}{d\tau}$ claimed.
\end{proof}

We also need the limiting behaviour of $\pthat$ as $\tau \ra \taumax$
\begin{lemma}
\addtocounter{equation}{1}
\label{L:ptau:taumax}
In the limit as $\tau \ra \taumax$ we have
\addtocounter{theorem}{1}
\begin{equation}
\label{E:pthat:taumax:p:eq:1}
\lim_{\tau \ra \taumax}{\pthat} = \pi \sqrt{\frac{2pq}{n}}.
\end{equation}
\end{lemma}

\begin{proof}
When $\tau = \taumax$, we have $y \equiv q/n$ and $p\,\dot{\psi}_1 \equiv 2n \taumax$.
Hence we have
$$ \lim_{\tau \ra \taumax}{2\pthat} = \lim_{\tau \ra \taumax}{p\, \psi_1(2\pt)} = 4n\taumax \lim_{\tau \ra \taumax}{\pt}.$$
The asymptotics for $\pthat$ now follow from the asymptotics for $\pt$ established in \ref{E:pt:taumax}.
\end{proof}

\section{Closed twisted {SL} curves and closed embedded special Legendrians.}
\label{S:embedded}

In this section we combine our results about the behaviour of $\pthat$  with our 
earlier results about periods and half-periods of 
$\bw_{\tau}$ to prove the existence of infinitely many closed $(p,q)$-twisted SL curves. 
Whether the curve $\bw_{\tau}$ closes is determined by the period lattice $\Per(\bw_{\tau})$, 
since by Corollary \ref{C:wtau:closed} $\bw_{\tau}$ is a closed curve if and only if $\Per(\bw_{\tau}) \neq (0)$. 
We are also interested in $\sorth{p} \times \sorth{q}$-invariant special Legendrian embeddings of closed manifolds arising from $X_{\tau}$; this is equivalent to looking for curves $\bw_{\tau}$ for which the associated 
curve of isotropic $\sorth{p} \times \sorth{q}$-orbits $\mathcal{O}_{\bw}$ 
is closed and this is determined by the half-period lattice 
$\Perh(\bw_{\tau})$ (recall \ref{D:w:period} and \ref{P:xtau:embed}). 


\begin{theorem}
\addtocounter{equation}{1}
\label{T:w:closed}
Fix admissible integers $p$ and $q$. There exists a countably infinite dense subset $N \subset (0,\taumax)$ such that 
$\tau \in N$ if and only if the $(p,q)$-twisted SL curve $\bw_{\tau}$ is closed. 
\end{theorem}
\begin{proof}
Define 
$$N:= \{ \tau \in (0,\taumax) \, | \, \pthat \in \pi \Q\}.$$
By Corollary \ref{C:wtau:closed} $\bw_{\tau}$ is a closed curve if and only if $\Per(\bw_{\tau}) \neq (0)$ and 
by Lemma \ref{L:w:half:period} $\Per(\bw_{\tau}) \neq 0$ if and only if $\pthat \in \pi \Q$.
Hence $\bw_{\tau}$ is closed if and only if $\tau \in N$. 
\ref{P:asymptotics-pthat} and \ref{L:ptau:taumax} imply that $\pthat$ is a nonconstant analytic function of $\tau$ on the interval $(0,\taumax)$. 
For any $\delta>0$ sufficiently small $\pthat|_{[\delta,\taumax-\delta]}$ is also nonconstant analytic
and hence the closed interval $[\delta,\taumax-\delta]$ 
contains only finitely many points at which $\tfrac{d\pthat}{d\tau}$ vanishes. 
In fact, it follows from the small $\tau$ asymptotics of $\tfrac{d\pthat}{d\tau}$  given in \ref{P:asymptotics-pthat} 
that $\frac{d\pthat}{d\tau}>0$ 
on $(0,\delta]$ for any $\delta$ sufficiently small. 
Hence for any $\delta>0$ sufficiently small 
there exists a countable dense subset of $(0,\taumax-\delta]$ for which $\pthat \in \pi \Q$. 
\end{proof}

From \ref{L:w:half:period} the condition $\tau \in N$ is equivalent to the condition that the rotational period $\tcheck_{2\pthat}$ (recall \ref{E:M:defn}) of $\bw_{\tau}$ 
is of finite order $k_{0}$  (recall Definition \ref{D:rp:order}).

\begin{theorem}
\addtocounter{equation}{1}
\label{T:Xtau:embed}
Choose $\tau \in N$ and let $k_{0}$ be the order of the rotational period $\tcheck_{2\pthat}$. 
The $\sorth{p}\times \sorth{q}$-invariant special Legendrian immersion 
$X_{\tau}: \cylpq \ra \Sph^{2(p+q)-1}$ factors through a special Legendrian embedding of the 
closed manifold $\cylpq/\Per(X_{\tau})$ where 
$\Per(X_{\tau}) \cong \Z \subset \Sym(X_{\tau}) \subset \Diff(\cylpq)$ is the 
following infinite cyclic subgroup
$$
\Per(X_{\tau}) = 
\begin{cases}
\langle (\TTT_{k_{0}\pt},-\Id_{\Sph^{n-1}}) \rangle \quad & \text{if $p=1$ and $k_{0}$ is even and $n$ is odd;}\\
\langle (\TTT_{k_{0}\pt},(-1)^{j}\Id_{\Sph^{p-1}},(-1)^{k}\Id_{\Sph^{q-1}}) \rangle \quad & \text{if $p>1$ and $k_{0}$ is even and $n$ is odd;}\\
\langle \TTT_{2k_{0}\pt} \rangle \quad & \text{otherwise;}
\end{cases}
$$
where $j=q/\hcf(p,q)$ and $k=p/\hcf(p,q)$.
\end{theorem}
\noindent
In the third case above the closed manifold is diffeomorphic to  $S^{1}\times S^{p-1}\times S^{q-1}$ if $p>1$ and to
$S^{1}\times S^{n-2}$ if $p=1$. In the first case the manifold is diffeomorphic to a $\Z_{2}$ quotient of 
$S^{1}\times S^{n-1}$ and in the second case to a $\Z_{2}$ quotient of $S^{1} \times S^{p-1} \times S^{q-1}$.
\begin{proof}
\ref{P:xtau:sym:p:eq:1}, \ref{P:xtau:sym:p:neq:q} and \ref{P:xtau:sym:p:eq:q} give us the structure of $\Per(X_{\tau})$ 
in the cases $p=1$, $p>1$, $p \neq q$ and $p>1$, $p=q$ respectively. 
The result follows by combining the structure of $\Per(X_{\tau})$ with \ref{P:xtau:embed}.
\end{proof}

For our gluing applications it will be convenient to use special Legendrian ``necklaces'' with 
topology $S^{1}\times S^{n-2}$ if $p=1$ or $S^{1}\times S^{p-1}\times S^{p-1}$ if $p>1$ and $p=q$ 
as in the third case of Theorem \ref{T:Xtau:embed}.  
By \ref{T:Xtau:embed} it suffices to find $\tau \in N$ so that the rotational period has odd order $k_{0}$. 
By using the asymptotics of $\pthat$ as $\tau \ra 0$ we can prove that there are infinitely many such special Legendrian necklaces
by using a well-chosen sequence of values of $\tau$ going to $0$.

\begin{lemma}[Existence of special Legendrian necklaces for $p=1$ and $p=q$, $p\ge2$]\hfill
\addtocounter{equation}{1}
\label{L:no:half:periods}
\begin{enumerate}
\item[(i)]
Suppose $p=1$, $q=n-1$. 
For any $\mb \in \N$ sufficiently large, there exists a unique small positive number $\tb$
satisfying
\addtocounter{theorem}{1}
\begin{equation}
\label{E:tb:p1}
\ptbhat = \left( \frac{(n-1)\mb}{2(n-1)\mb -1}\right) \pi > \frac{\pi}{2}.
\end{equation}
By choosing $\mb$ sufficiently large, $\tb$ can be chosen
as close to $0$ as desired (see \ref{E:tb:mb}).
Let $k_0$ be the (large) positive odd integer defined in terms of $\mb$ by
\addtocounter{theorem}{1}
\begin{equation}
\label{E:k0:p1}
k_0 = 2(n-1)\mb -1.
\end{equation}
Then $\Per(X_\tb) = \langle\TTT_{2k_0 \ptb}\rangle $, $\ttilde_{k_0 \ptbhat}= \ttilde_{-k_0\ptbhat}$,
$\tbartilde_{k_0\ptbhat} = \tbartilde_{-k_0\ptbhat} = \tbartilde$, and $X_\tb$ factors through an embedding of
$(\R/2k_0 \ptb \Z ) \times \Sph^{n-2}$.
\item[(ii)]
Suppose $p=q \ge 2$. 
For any $\mb \in \N$ sufficiently large, there exists a unique small positive number $\tb$
satisfying
\addtocounter{theorem}{1}
\begin{equation}
\label{E:tb:pq}
\ptbhat = \left( \frac{p\mb}{2p\mb -1}\right) \pi > \frac{\pi}{2}.
\end{equation}
By choosing $\mb$ sufficiently large, $\tb$ can be chosen
as close to $0$ as desired (see \ref{E:tb:mb}).
Let $k_0$ be the (large) positive odd integer defined in terms of $\mb$ by
\addtocounter{theorem}{1}
\begin{equation}
\label{E:k0:pp}
k_0 = 2p\mb -1.
\end{equation}
Then $\Per(X_\tb) = \langle \TTT_{2k_0 \ptb}^l  \rangle$, $\ttilde_{k_0 \ptbhat}= \ttilde_{-k_0\ptbhat}$,
$\tbartilde_{k_0\ptbhat} = \tbartilde_{-k_0\ptbhat} = \tbartilde$, and $X_\tb$ factors through an embedding of
$(\R/2k_0 \ptb \Z ) \times \Sph^{p-1} \times \Sph^{p-1}$.
\end{enumerate}
\end{lemma}

When $n=3$, \ref{E:tb:p1} specialises to
$$ \ptbhat = \frac{\pi}{2}\left(\frac{4\mb}{4\mb-1}\right) = \frac{\pi}{2} \left(1 + \frac{1}{m}\right),
\quad \text{where\ } m=4\mb -1.$$
This agrees with equation 4.3 in \cite{haskins:kapouleas:invent} and
hence Lemma \ref{L:no:half:periods} generalises \cite[Lemma 4.5]{haskins:kapouleas:invent}.

\begin{proof}
Existence and uniqueness of small $\tb>0$ satisfying \ref{E:tb:p1} or \ref{E:tb:pq} follow from \ref{E:phat}. 
Moreover, using \ref{P:asymptotics-ptau} we conclude also that 
\addtocounter{theorem}{1}
\begin{equation}
\label{E:tb:mb}
\mb \sim 
\begin{cases}
c_{n-1}/(\tb T_{n-1}(\tb)\,) \quad \text{where} \quad c_{n-1} = \pi /16(n-1)\,b_{n-1} & \text{if $p=1$;}\\
d_{p}/(\tb T_{p}(\tb)\,) \quad \quad \quad\text{where} \quad  d_{p} = \pi/32p \,b_{p} & \text{if $p>1$ and $p=q$};
\end{cases}
\end{equation}
and $b_{k}$ are the constants defined in \ref{E:b:q}. The rest of the lemma follows from 
Theorem \ref{T:Xtau:embed} once we show \ref{E:k0:p1} and \ref{E:k0:pp}, since in both cases $k_{0}$ is odd. 
From \ref{E:k0:alt} the order of the rotational period, $k_{0}$, can be found as 
$$k_{0}= \min\{k\in \Z^{+}|\,k\pthat \in \lcm(p,q)\pi \Z\}.$$
(i) Using the form of $\ptbhat$ assumed in \ref{E:tb:p1} we see 
$$ k\ptbhat \in (n-1)\pi\Z \iff k\mb \in (2(n-1)\mb -1)\Z \iff  k\mb \in (2(n-1)\mb -1)\Z \cap \mb \Z.$$
It is easily check that $\hcf(\mb,2(n-1)\mb-1)=1$ and hence $(2(n-1)\mb -1)\Z \cap \mb \Z = \mb (2(n-1)\mb -1)\Z$.
Therefore $k_{0 }=2(n-1)\mb-1$ as claimed.\\
(ii) Similarly from \ref{E:tb:pq} we have 
$$ k\ptbhat \in p\pi\Z \iff k\mb \in (2\mb p-1)\Z \iff  k\mb \in (2\mb p-1)\Z \cap \mb \Z
\iff k\mb \in \mb (2\mb p-1)\Z$$
since  $\hcf(\mb,2\mb p-1)=1$. Hence $k_{0}=2\mb p-1$ as claimed. 
\end{proof}

\appendix

\section{Direct, central and semidirect products of groups}
\label{A:groups}
We recall some elementary group theory needed in several parts of the paper. 
Let $K, N \subset G$ be any two subsets of a group $G$. We write
$KN = \{ kn\, |\, k\in K, n\in N\} \subset G$. Sometimes we will also 
use the notation $K\cdot N$.  If both $K$ and $N$
are subgroups of $G$, then $KN$ is a subgroup if and only if
$KN=NK$ \cite[p.22]{isaacs}. Moreover, $N$ is a normal subgroup of
$KN$ if and only if $kNk^{-1} =N$ for all $k\in K$ and similarly
for $K$. 

In particular, if $N$ centralises $K$, i.e. every element
of $N$ commutes with every element of $K$, then clearly $KN=NK$
and both $K$ and $N$ are normal subgroups of the group $H=KN$. In this case we have
$K \cap N \subseteq Z(H)$, where $Z(H)$ denotes the centre of $H$,
and we say that $KN$ is an  \textit{(internal) central product} of $K$ and
$N$ identifying $K\cap N$ \cite[p. 29]{gorenstein}. If in fact $K\cap N=1$, then
$KN$ is the \textit{(internal) direct product of $K$ and $N$}, and
$KN$ is isomorphic to the (external) direct product $K \times N$.

If $K$ is a normal subgroup of $KN$ then conjugation by any element of $N$ gives a homomorphism
$\rho: N \ra \Aut{K}$ and the kernel of $\rho$ is the centraliser of $K$ in $KN$.
If $K$ is a normal subgroup of $KN$ and $K\cap N=1$, then $KN$ is the
\textit{semidirect product of $K$ by $N$}, $K \rtimes N$.
To make explicit the conjugation action of $N$ on $K$ we often write, $K\rtimes_{\rho}N$ 
and write down the twisting homomorphism $\rho: N \ra \Aut{K}$.

\section{Lagrangian and special Lagrangian isometries of $\C^{n}$}
\label{A:isom:sl}
Let $\Lag$ denote the Grassmannian of unoriented Lagrangian $n$-planes in $\C^{n}$, 
$\Slag$ denote the Grassmannian of (necessarily oriented) special Lagrangian $n$-planes in $\C^{n}$
and define $\pm \Slag:=\{ \Pi\, | \, \pm \Pi \in \Slag\}$.

\addtocounter{equation}{1}
\begin{definition}[Lagrangian and special Lagrangian isometries]\hfill
\label{D:isomsl}	
\begin{enumerate}
\item[(i)]
Define $\Isoml:= \{ A \in \orth{2n} \, |\, A (\Pi) \in \Lag \,\text{for all\ } \Pi \in \Lag\}$.
Elements of $\Isoml$ we call \emph{Lagrangian isometries}.
\item[(ii)]
Define $\Isomsl:= \{ A \in \orth{2n} \, |\, A (\Pi) \in \Slag \, \text{for all\ } \Pi \in \Slag\}$.
Elements of $\Isomsl$ we call \emph{special Lagrangian isometries}.
\item[(iii)]
Define $\Isomslpm:= \{ A \in \orth{2n} \, |\, A (\Pi) \in \pm \Slag \,\text{for all\ } \Pi \in \Slag\}$.
Elements of $\Isomslpm$ we call \emph{$\pm$-special Lagrangian isometries} and elements of 
$\Isomslpm\setminus \Isomsl$ we call \emph{anti-special Lagrangian isometries}.
\end{enumerate}
\end{definition}
\noindent
Note we do not assume a priori that $\Isomsl \subset \Isoml$ (see also Lemma \ref{L:isom:sl}), but from our definitions we do have 
$\Isomsl \subset \Isomslpm$. 

Define $\ccong \in \orth{2n}$ by 
\begin{equation}
\addtocounter{theorem}{1}
\label{E:ctilde}
\ccong(\bz) = \overline{\bz}, \quad  \quad \quad \text{where $\bz \in \C^{n}$}.
\end{equation}
Since $\ccong$ satisfies
$$ \ccong^{*}J = -J, \quad \ccong^{*}\omega = - \omega, \quad \ccong^{*}\Omega = \overline{\Omega},$$
we see in particular that $\ccong$ belongs to both $\Isoml$ and $\Isomsl$.

The following result on the structure of $\Isoml$ and $\Isomsl$ 
is presumably well-known but since it is important for our study of the discrete symmetries 
of $\sorth{p} \times \sorth{q}$-invariant special Legendrians and we are not aware of a suitable reference we give its proof.
For completeness, we state the result also for dimension $2$ although it is not used in this paper. 
\begin{lemma}[Structure of $\Isoml$,  $\Isomsl$ and $\Isomslpm$]\hfill
\addtocounter{equation}{1}
\label{L:isom:sl}
\begin{enumerate}
\item[(i)]
$\Isoml = \unit{n} \cdot \langle \ccong \rangle \simeq \unit{n} \rtimes_{\rho} \Z_{2}$ where the twisting 
homomorphism $\rho: \Z_{2} \ra \Aut{\unit{n}}$ is determined by 
$\rho(1)\, U = \overline{U}$ for any $U\in \unit{n}$.
\item[(ii)]
For $n>2$, we have 
\begin{eqnarray*} 
\Isomsl &=& \sunit{n} \cdot \langle \ccong \rangle \simeq \sunit{n} \rtimes_{\rho} \Z_{2},\\
\Isomslpm &=& \sunit{n}^{\pm} \cdot \langle \ccong \rangle \simeq \sunit{n}^{\pm} \rtimes_{\rho} \Z_{2},
\end{eqnarray*}
where  
$$\sunit{n}^{\pm}:= \{ U \in \unit{n} \, | \, \det{}_{\C}{U}= \pm 1 \} \simeq \sunit{n} \rtimes \Z_{2},$$ 
and $\rho$ is the restriction of the twisting homomorphism defined in (i) to $\sunit{n}^{\pm}$.
In particular, $\Isomsl$ and $\Isomslpm$ are subgroups of $\Isoml$. 
\item[(iii)]
$\Isomsl(2) =\text{U}_{I}(2)$  and $\Isomslpm = \text{U}_{I}(2) \cdot \langle \RRR_{1} \rangle \simeq 
\text{U}_{I}(2) \rtimes_{\rho} \Z_{2}$ 
where 
$\text{U}_{I}(2)$ denotes the unitary group of $\C^{2}$ with respect to the complex structure $I$ on $\C^{2}$ 
defined by right multiplication by the imaginary quaternion $I \in \Imag{\HH}$, 
$\RRR_{1}(z_{1},z_{2}) := (-z_{1},z_{2})$ and $\rho: \Z_{2} \ra \text{U}_{I}(2)$ 
is the homomorphism defined by $\rho(1)\, U = \RRR_{1} U \RRR_{1}$. 
$\text{U}_{I}(2)$ satisfies 
$$\text{U}_{I}(2) \cap \unit{2} = \sunit{2},$$ 
where $\unit{2}$ and $\sunit{2}$ 
are the unitary and special unitary groups with respect to the standard complex structure $J$ on $\C^{2}$ (defined by 
right multiplication by $J \in \Imag{\HH}$).
In particular, there exist special Lagrangian isometries of $\C^{2}$ which are not Lagrangian isometries. 
\end{enumerate}
\end{lemma}
In fact, the following stronger version of Lemma \ref{L:isom:sl} holds: 
\emph{if $n>2$ any diffeomorphism of $\C^{n}$ which preserves the special Lagrangian 
differential ideal $\mathcal{I}$ generated by $\omega$ and $\Imag{\Omega}$ is a product of 
a dilation with some element of $\Isomslpm$} \cite{bryant:personal}.
\begin{proof}[Proof of \ref{L:isom:sl}]
(i) We first prove that $\unit{n} \cdot \langle \ccong \rangle$ forms a subgroup of $\orth{2n}$ 
isomorphic to the semidirect product claimed above. 
To prove that $\unit{n} \cdot \langle \ccong \rangle$ forms a subgroup of $\orth{2n}$
it suffices, by the group theory discussion in Appendix \ref{A:groups},  to prove that 
$\unit{n} \cdot \langle \ccong \rangle =  \langle \ccong \rangle \cdot \unit{n}$.
Since conjugation by $\ccong$ acts on $\glc{n}$ by $ \ccong \MMM \ccong = \overline{\MMM}$,
conjugation by $\ccong$ leaves $\unit{n}$ invariant 
and hence $\unit{n} \cdot \langle \ccong \rangle$ forms a group in which $\unit{n}$ is a normal subgroup. 
Since $\ccong \notin \unit{n}$ we have $\langle \ccong \rangle \cap \unit{n}=1$ and therefore 
 $\unit{n} \cdot \langle \ccong \rangle$ has the semidirect product structure claimed.
 
Since both $\unit{n}$ and $\ccong$ belong to $\Isoml$, the subgroup generated by them 
$\unit{n} \cdot \langle \ccong \rangle$ is clearly a subgroup of $\Isoml$. We want to show that 
any element in $\Isoml$ belongs to $\unit{n} \cdot \langle \ccong \rangle$.
Recall that $\unit{n}$ acts transitively on the set of all (unoriented) Lagrangian planes in $\C^{n}$ with stabiliser conjugate to 
$\orth{n} \subset \unit{n}$. Hence given any $\LLL \in \Isoml$ there exists $U \in \unit{n}$ so that 
$U^{-1}\LLL$ maps the standard Lagrangian plane $\R^{n} \subset \C^{n}$ 
to itself and moreover fixes $\R^{n}$ pointwise.
Given any $v \in \Sph^{n-1} \in \R^{n}$, consider the $n-1$ dimensional isotropic plane 
$I_{v}:= \langle v \rangle^{\perp}\subset \R^{n} \subset \C^{n}$.
The pencil $LP_{v}$ of all Lagrangian $n$-planes containing the isotropic plane $I_{v}$ is given by 
$$LP_{v}=\bigcup_{\theta \in \Sph^{1}}L_{v,\theta} = \bigcup_{\theta \in \Sph^{1}}\, I_{v} \oplus \langle \cos{\theta} \,v + \sin{\theta}\, Jv\rangle.$$
Since $U^{-1}\LLL$ fixes $I_{v}$ pointwise it maps the pencil of Lagrangian planes $LP_{v}$ to itself
and therefore leaves the complex line $l_{v}$ spanned by $v$ and $Jv$ invariant.
Moreover, since $U^{-1}\LLL|_{{l_{v}}}$  is an isometry of the complex line $l_{v}$ which fixes $v \in l_{v}$ we must have   
$$
U^{-1} \LLL(Jv) = \pm Jv.
$$
By continuity the same sign must occur for any 
$v\in \Sph^{n-1} \subset \R^{n}$ and hence we have either $U^{-1}\LLL = \Id$ or $U^{-1} \LLL = \ccong$ 
as required.

(ii) \emph{Structure of $\Isomsl$:} 
Since $\sunit{n}$ and $\ccong$ are contained in $\Isomsl$, 
$\sunit{n} \cdot \langle \ccong \rangle$ is a subgroup 
of $\Isomsl$. We want to show that every element in $\Isomsl$ belongs to $\sunit{n} \cdot \langle \ccong \rangle$. 
Recall that $\sunit{n}$ acts transitively on the set of all special Lagrangian $n$-planes in $\C^{n}$ with stabiliser conjugate 
to $\sorth{n}$. Hence given any $\LLL \in \Isomsl$ there exists $S \in \sunit{n}$ so that $S^{-1}\LLL$ 
maps the standard special Lagrangian $n$-plane $\R^{n} \subset \C^{n}$ to itself preserving its orientation.
Therefore $S^{-1}\LLL$ restricted to $\R^{n}$ is some element of $\sorth{n}$. Hence by using the freedom to change $S$ by 
an element of the stabiliser $\sorth{n}$ we can arrange that $S^{-1}\LLL= \Id$ on $\R^{n}$.
Therefore given any $\LLL \in \Isomsl$ there exists $S \in \sunit{n}$ such that $\LLL':=S^{-1}\LLL \in \Isomsl$ 
restricted to $\R^{n} \subset \C^{n}$ is the identity. The result follows if we can prove that this implies that 
$\LLL' = \Id$ or $\LLL'= \ccong$. To prove this we need to assume that $n>2$.

Given any isotropic $n-2$ plane $\Pi \subset \C^{n}$ we consider the space of all Lagrangian $n$-planes containing 
$\Pi$. Given any $n$-plane extension $\Pi''$ of $\Pi$ we can choose orthonormal vectors $v_{n-1}$ and $v_{n}$ 
which are orthogonal to $\Pi$ and so that $\Pi'' = \Pi \oplus \langle v_{n-1},v_{n}\rangle$. 
Since $\Pi$ is isotropic,  
$\Pi'' = \Pi \oplus \langle v_{n-1},v_{n}\rangle $ is Lagrangian if and only if $v_{n-1}$ and $v_{n}$ are also 
orthogonal to $J\Pi$ and $\omega(v_{n-1},v_{n})=0$. In other words, Lagrangian $n$-planes containing 
$\Pi$ are in one-to-one correspondence with $2$-planes in the complex $2$-plane $(\Pi\oplus J\Pi)^{\perp}$
which are Lagrangian with respect to the symplectic structure given by restriction of the symplectic form $\omega$ on $\C^{n}$ 
to $(\Pi \oplus J\Pi)^{\perp}$
(the restriction of $\omega$ to any complex subspace is nondegenerate). 
A refinement of this correspondence shows that special Lagrangian extensions of the oriented isotropic $n-2$ plane $\Pi$ 
are in one-to-one correspondence with $2$-planes in the complex $2$-plane $(\Pi \oplus J\Pi)^{\perp}$
that are special Lagrangian with respect to the holomorphic $(2,0)$-form $\Omega_{2} = \iota_{v_{1}\wedge \cdots \wedge v_{n-2}} \Omega$ where $v_{1}, \ldots ,v_{n-2}$ is an oriented orthonormal basis of $\Pi$.

Let $e_{1}, \ldots ,e_{n}$ denote the standard oriented orthonormal basis of $\R^{n}$.
Since $\LLL'$ leaves invariant $\R^{n}$ it also leaves invariant the perpendicular $n$-plane 
$i\R^{n} \subset \C^{n}$. Hence $\LLL'|_{i\R^{n}} \in \orth{n}$. Let $(l_{jk}')$ denote the matrix of 
$\LLL '|_{i\R^{n}}$ with respect to the basis $ie_{1}, \ldots , ie_{n}$.
Choose any pair of integers $1 \le j<k \le n$ and consider the oriented codimension two isotropic subspace of 
$\R^{n} \subset \C^{n}$ defined by $\Pi_{j,k}:= \langle e_{j},e_{k} \rangle^{\perp} \subset \R^{n}$. 
Since $\LLL'$ fixes $\R^{n}$ pointwise it fixes $\Pi_{j,k}\subset \R^{n}$ and hence takes any 
special Lagrangian $n$-plane containing $\Pi_{j,k}$ into another special Lagrangian $n$-plane containing $\Pi_{j,k}$.
Hence by the previous paragraph $\LLL'$  sends the complex $2$-plane $\C^{2}_{e_{j},e_{k}}$ spanned by $e_{j}$ and $e_{k}$ 
to itself and also it 
fixes $\R^{2}_{{e_{j},e_{k}}} \subset \C^{2}_{{e_{j},e_{k}}}$ pointwise. Hence 
$\LLL'|_{\C^{2}_{e_{j},e_{k}}} \in (\Id )\times \orth{2} \subset \orth{2} \times \orth{2} \subset \orth{4}$ 
(with this splitting thought of with respect to the basis $e_{j}, e_{k}, ie_{j}, ie_{k}$) 
and therefore $l_{jm}'=l_{km}'=0$ for any $m\notin \{j,k\}$
(consider the norm of the vectors formed by the $j$th and $k$th rows or columns of $\LLL'|_{i\R^{n}}\in \orth{n}$).
Since we are free to choose any $1\le j<k \le n$  this forces all off-diagonal terms of $(l_{jk}')$ to vanish and hence 
$l_{jj}' = \pm 1$ for all $j=1, \ldots ,n$. 

Finally we have to show that either $l_{jj}'=1$ for all $j$ or $l_{jj}'=-1$ for all $j$.
Suppose $\LLL'|_{i\R^{n}} \neq \pm \Id$, then without loss of generality we may suppose that $l_{11}'=1$ and $l_{22}'=-1$ 
and $l_{jj} = \pm 1$ for $j>2$. The oriented $n$-plane $\xi= -Je_{1} \wedge Je_{2} \wedge e_{3} \wedge \ldots \wedge e_{n}$ 
is a special Lagrangian $n$-plane which as an unoriented $n$-plane is invariant under $\LLL'$; but by our choice 
of $l_{11}'$ and $l_{22}'$ $\LLL'$ reverses the orientation of $\xi$ and hence $\LLL' \notin \Isomsl$. 
Therefore $\LLL = \Id$ or $\LLL= \ccong$ as claimed. 

\emph{Structure of $\Isomslpm$:}
It follows directly from the definition that $\sunit{n}^{\pm}$ forms a subgroup of $\unit{n}$ which is 
invariant under complex conjugation. Hence by the first part of (i)  $\sunit{n}^{\pm} \cdot \langle \ccong \rangle$ 
has the semidirect product structure claimed. The semidirect product structure of $\sunit{n}^{\pm}$ itself 
can be seen as follows. Consider the subgroup of $\unit{n}$ generated by $\sunit{n}$ and reflection in the complex hyperplane 
$z_{1}=0$,  
$\RRR_{1} \in \orth{n} \subset \unit{n}$ defined by 
$$ \RRR _{1}(e_{i}) = 
\begin{cases}
-e_{i} \quad & \text{if $i=1$;}\\
\phantom{-}e_{i} \quad & \text{otherwise.}
\end{cases}
$$
Since $\det_{\C}(\RRR_{1})=-1$ 
any element in $\sunit{n}^{\pm}$ can be written in the form $\langle \RRR_{1} \rangle \cdot \sunit{n}$, 
conjugation by $\RRR_{i}$ defines an involution of $\sunit{n}$ and $\RRR_{1} \notin \sunit{n}$. 
Hence $\langle \RRR_{1} \rangle \cdot \sunit{n} 
\simeq  \sunit{n} \rtimes \Z_{2}$ where the twisting homomorphism $\rho: \Z_{2} \ra \Aut{\sunit{n}}$ is determined 
by $\rho(1)\, U = \RRR_{1}\, U \,\RRR_{1}$. It is straightforward to check that $\RRR_{1}$ 
satisfies $\RRR^{*}_{1}\Real{\Omega} = - \Real{\Omega}$ and hence belongs to 
$\Isomslpm \setminus \Isomsl$.
Moreover, we see that $\LLL \in \Isomslpm \setminus \Isomsl$ if and only if 
$\RRR_{1} \cdot \LLL \in \Isomsl$. 
Hence the structure of $\Isomslpm$ now follows from the structure of $\Isomsl$ already established.

(iii) Using the standard basis $e_{1}=1$, $e_{2}=I$, $e_{3}=J$ and $e_{4}=K$ for the quaternions $\HH \cong \C^{2}$, 
the standard complex structure $J$ on $\C^{2}$ 
can be represented by the action of right multiplication by the unit imaginary quaternion $J \in \Imag(\HH)$. 
With respect to the complex structure $I$ defined by right multiplication by the unit imaginary quaternion $I$ 
we have 
$$ \omega_{I}:= g( \cdot, I \cdot ) = \Real{\Omega_{J}}.$$
Hence the special Lagrangian $2$-planes of $(\C^{2}, J,\omega_{J},\Omega_{J})$ are exactly the $I$-complex lines in $\C^{2}$.
The rest of (iii) follows from this well-known fact. We omit the details since we do not use the result in this paper.  
However, for concreteness we exhibit special Lagrangian isometries which are not Lagrangian isometeries.  
The $1$-parameter subgroup $\{M_{\theta}\} \cong \sorth{2} \subset \text{U}_{I}(2)\subset  \orth{4}$ defined by 
\begin{equation}
\addtocounter{theorem}{1}
\label{E:sl:isom:ex}
M_{\theta} = 
\left(
\begin{array}{rrrr}
1 & 0 & 0 & 0\\
0 & 1 & 0 & 0\\
0 & 0 & \cos{\theta} & \sin{\theta}\\
0 & 0 & -\sin{\theta} & \cos{\theta}
\end{array}
\right)
\end{equation}
acts on the standard symplectic form $\omega=\omega_{J}$ and standard holomorphic $(2,0)$-form $\Omega=\Omega_{J}$ 
on $\C^{2}$ by 
\begin{equation}
\addtocounter{theorem}{1}
\label{E:mtheta:act}
M_{\theta}^{*}\omega = \cos{\theta} \,\omega - \sin{\theta} \Imag{\Omega}, \qquad 
M_{\theta}^{*}\Omega = \Real{\Omega} + i (\cos{\theta} \Imag{\Omega} + \sin{\theta}\,\omega).
\end{equation}
In particular $\{M_{\theta}\}$ is not a subgroup of $\Isoml$ but is a subgroup of $\Isomsl$.
When $\theta =\pi$ , \ref{E:mtheta:act} reduces to $\omega \mapsto -\omega$ and 
$\Omega \ra \overline{\Omega}$, which is consistent with the fact that $M_{\pi} = \ccong$ 
where $\ccong$ denotes the complex conjugation on $\C^{2}$ defined with respect to the standard
complex structure $J$.
\end{proof}
\begin{corollary}
\addtocounter{equation}{1}
\label{C:sl:isom}
If $n>2$ then any special Lagrangian isometry $\LLL \in \Isomsl$ satisfies 
$$ \LLL^{*}\omega = \omega, \quad \LLL^{*}\Omega = \Omega, \qquad \text{or} \qquad 
\LLL^{*}\omega = -\omega, \quad \LLL^{*}\Omega = \overline{\Omega},$$
while any anti-special Lagrangian isometry $\LLL \in \Isomslpm\setminus \Isomsl$ 
satisfies 
$$ \LLL^{*}\omega = \omega, \quad \LLL^{*}\Omega = -\Omega, \qquad \text{or} \qquad 
\LLL^{*}\omega = -\omega, \quad \LLL^{*}\Omega = -\overline{\Omega}.$$
\end{corollary}

Corollary \ref{C:sl:isom} implies that every $\pm$-special Lagrangian 
isometry of $\C^{n}$ sends the complex structure 
$J$ to $\pm J$. In other words, every $\pm$-special Lagrangian isometry of $\C^{n}$ is either 
a holomorphic or anti-holomorphic isometry of $\C^{n}$. Hence we may also define another subgroup 
of $\Isomslpm$ by  
$ \IsomslpmJ:= \{ A\in \Isomslpm\, | \, AJ=JA \}$ where $J$ denotes the standard complex structure on $\C^{n}$.
\ref{L:isom:sl} implies that 
\begin{equation}
\addtocounter{theorem}{1}
\label{E:isomslpm:j}
\IsomslpmJ = \Isomslpm \cap \,\unit{n} = \sunit{n}^{\pm}.
\end{equation}

Corollary \ref{C:sl:isom} also implies that every special Lagrangian isometry of $\C^{n}$ preserves the calibration $\Real{\Omega}$ 
and that every anti-special Lagrangian isometry of $\C^{n}$ sends $\Real{\Omega}$ to $-\Real{\Omega}$. 
Clearly, an isometry of $\C^{n}$ that preserves the calibration $\Real{\Omega}$ defines a special 
Lagrangian isometry. However, in general it is not always the case that an isometry sending calibrated planes to 
calibrated planes must preserve the calibration. For example, for any $c\in (0,1)$ the calibration on $\R^{4}$ defined by 
$$ \phi = dx_{1}\wedge dy_{1} + c \,dx_{2}\wedge dy_{2},$$
calibrates only the $2$-plane $\xi=e_{1}\wedge J e_{1}$. The isometry
$$ (x_{1},x_{2},y_{1},y_{2}) \mapsto (x_{1},x_{2},y_{1},-y_{2}),$$
leaves $\xi$ invariant but does not preserve $\phi$.

\bibliographystyle{amsplain}
\bibliography{paper}

\def\cprime{$'$} \def\cprime{$'$}
\providecommand{\bysame}{\leavevmode\hbox to3em{\hrulefill}\thinspace}
\providecommand{\MR}{\relax\ifhmode\unskip\space\fi MR }
\providecommand{\MRhref}[2]{%
  \href{http://www.ams.org/mathscinet-getitem?mr=#1}{#2}
}
\providecommand{\href}[2]{#2}
\begin{thebibliography}{10}

\bibitem{anciaux}
Henri Anciaux, \emph{Legendrian submanifolds foliated by {$(n-1)$}-spheres in
  {$\Bbb S\sp {2n+1}$}}, Mat. Contemp. \textbf{30} (2006), 41--61, XIV School
  on Differential Geometry (Portuguese). \MR{MR2373502 (2009f:53084)}

\bibitem{arnold}
V.~I. Arnol{\cprime}d and A.~B. Givental{\cprime}, \emph{Symplectic geometry},
  Dynamical systems, IV, Encyclopaedia Math. Sci., vol.~4, Springer, Berlin,
  2001, pp.~1--138. \MR{1 866 631}

\bibitem{arnold:ode}
Vladimir~I. Arnold, \emph{Ordinary differential equations}, Universitext,
  Springer-Verlag, Berlin, 2006, Translated from the Russian by Roger Cooke,
  Second printing of the 1992 edition. \MR{MR2242407 (2007b:34001)}

\bibitem{becker}
Katrin Becker, Melanie Becker, and Andrew Strominger, \emph{Fivebranes,
  membranes and non-perturbative string theory}, Nuclear Phys. B \textbf{456}
  (1995), no.~1-2, 130--152. \MR{97k:81112}

\bibitem{bryant:personal}
Robert Bryant, \emph{Personal communication}, 2010.

\bibitem{butscher}
Adrian Butscher, \emph{Regularizing a singular special {L}agrangian variety},
  Comm. Anal. Geom. \textbf{12} (2004), no.~4, 733--791. \MR{2104075}

\bibitem{carberry:mcintosh}
Emma Carberry and Ian McIntosh, \emph{Minimal {L}agrangian 2-tori in
  {$\mathbb{CP}\sp 2$} come in real families of every dimension}, J. London
  Math. Soc. (2) \textbf{69} (2004), no.~2, 531--544. \MR{2040620}

\bibitem{castro:li:urbano}
Ildefonso Castro, Haizhong Li, and Francisco Urbano, \emph{Hamiltonian-minimal
  {L}agrangian submanifolds in complex space forms}, Pacific J. Math.
  \textbf{227} (2006), no.~1, 43--63. \MR{MR2247872 (2007k:53092)}

\bibitem{castro:urbano:foliated}
Ildefonso Castro and Francisco Urbano, \emph{On a minimal {L}agrangian
  submanifold of {$\bold C\sp n$} foliated by spheres}, Michigan Math. J.
  \textbf{46} (1999), no.~1, 71--82. \MR{MR1682888 (2000b:53079)}

\bibitem{castro:urbano:construct}
\bysame, \emph{On a new construction of special {L}agrangian immersions in
  complex {E}uclidean space}, Q. J. Math. \textbf{55} (2004), no.~3, 253--265.
  \MR{MR2082092 (2005g:53087)}

\bibitem{geiges}
Hansj{\"o}rg Geiges, \emph{Contact geometry}, Handbook of differential
  geometry. {V}ol. {II}, Elsevier/North-Holland, Amsterdam, 2006, pp.~315--382.
  \MR{MR2194671 (2007c:53123)}

\bibitem{gorenstein}
Daniel Gorenstein, \emph{Finite groups}, Harper \& Row Publishers, New York,
  1968. \MR{MR0231903 (38 \#229)}

\bibitem{gross:slgfib1}
Mark Gross, \emph{Special {L}agrangian fibrations. {I}. {T}opology}, Integrable
  systems and algebraic geometry (Kobe/Kyoto, 1997), World Sci. Publishing,
  River Edge, NJ, 1998, pp.~156--193. \MR{2000e:14066}

\bibitem{gross:slgfib2}
\bysame, \emph{Special {L}agrangian fibrations. {II}. {G}eometry. {A} survey of
  techniques in the study of special {L}agrangian fibrations}, Surveys in
  differential geometry: differential geometry inspired by string theory, Surv.
  Differ. Geom., vol.~5, Int. Press, Boston, MA, 1999, pp.~341--403.
  \MR{2001j:53065}

\bibitem{gross:slg:examples}
\bysame, \emph{Examples of special {L}agrangian fibrations}, Symplectic
  geometry and mirror symmetry (Seoul, 2000), World Sci. Publishing, River
  Edge, NJ, 2001, pp.~81--109. \MR{2003f:53085}

\bibitem{harvey:lawson}
Reese Harvey and H.~Blaine Lawson, Jr., \emph{Calibrated geometries}, Acta
  Math. \textbf{148} (1982), 47--157. \MR{85i:53058}

\bibitem{haskins:thesis}
Mark Haskins, \emph{Constructing {S}pecial {L}agrangian {C}ones}, Ph.D. thesis,
  University of Texas at Austin, 2000.

\bibitem{haskins:slgcones}
\bysame, \emph{Special {L}agrangian cones}, Amer. J. Math. \textbf{126} (2004),
  no.~4, 845--871. \MR{2075484}

\bibitem{haskins:complexity}
\bysame, \emph{{The geometric complexity of special Lagrangian $T^2$-cones}},
  Invent. Math. \textbf{157} (2004), 11--70.

\bibitem{haskins:kapouleas:hd2}
Mark Haskins and Nikolaos Kapouleas, \emph{{H}igher dimensional special
  {L}agrangian cones by gluing $\text{SO}(p)\times \text{SO}(p)$-invariant
  cones}, In preparation.

\bibitem{haskins:kapouleas:hd3}
\bysame, \emph{{S}pecial {L}agrangian cones by gluing ${SO}(n-1)$-invariant
  cones}, In preparation.

\bibitem{haskins:kapouleas:invent}
\bysame, \emph{Special {L}agrangian cones with higher genus links}, Invent.
  Math. \textbf{167} (2007), no.~2, 223--294. \MR{MR2270454}

\bibitem{haskins:kapouleas:survey}
\bysame, \emph{Gluing constructions of special {L}agrangian cones}, Handbook of
  geometric analysis. {N}o. 1, Adv. Lect. Math. (ALM), vol.~7, Int. Press,
  Somerville, MA, 2008, pp.~77--145. \MR{MR2483363}

\bibitem{hitchin:moduli}
Nigel Hitchin, \emph{The moduli space of special {L}agrangian submanifolds},
  Ann. Scuola Norm. Sup. Pisa Cl. Sci. (4) \textbf{25} (1997), no.~3-4,
  503--515 (1998), Dedicated to Ennio De Giorgi. \MR{2000c:32075}

\bibitem{hitchin:slg:lectures}
\bysame, \emph{Lectures on special {L}agrangian submanifolds}, Winter School on
  Mirror Symmetry, Vector Bundles and Lagrangian Submanifolds (Cambridge, MA,
  1999), AMS/IP Stud. Adv. Math., vol.~23, Amer. Math. Soc., Providence, RI,
  2001, pp.~151--182. \MR{2003f:53086}

\bibitem{isaacs}
I.~Martin Isaacs, \emph{Algebra}, Brooks/Cole Publishing Co., Pacific Grove,
  CA, 1994, A graduate course. \MR{MR1276273 (95k:00003)}

\bibitem{joyce:symmetries}
Dominic Joyce, \emph{Special {L}agrangian {$m$}-folds in {$\mathbb{C}\sp m$}
  with symmetries}, Duke Math. J. \textbf{115} (2002), no.~1, 1--51. \MR{1 932
  324}

\bibitem{joyce:syz}
\bysame, \emph{Singularities of special {L}agrangian fibrations and the {SYZ}
  conjecture}, Comm. Anal. Geom. \textbf{11} (2003), no.~5, 859--907.
  \MR{2004m:53094}

\bibitem{joyce:sing:survey}
\bysame, \emph{Special {L}agrangian submanifolds with isolated conical
  singularities. {V}. {S}urvey and applications}, J. Differential Geom.
  \textbf{63} (2003), no.~2, 279--347. \MR{2 015 549}

\bibitem{joyce:u1:i}
\bysame, \emph{{$\rm U(1)$}-invariant special {L}agrangian 3-folds. {I}.
  {N}onsingular solutions}, Adv. Math. \textbf{192} (2005), no.~1, 35--71.
  \MR{MR2122280 (2006e:53094)}

\bibitem{joyce:u1:ii}
\bysame, \emph{{$\rm U(1)$}-invariant special {L}agrangian 3-folds. {II}.
  {E}xistence of singular solutions}, Adv. Math. \textbf{192} (2005), no.~1,
  72--134. \MR{MR2122281 (2006h:53047)}

\bibitem{joyce:u1:iii}
\bysame, \emph{{$\rm U(1)$}-invariant special {L}agrangian 3-folds. {III}.
  {P}roperties of singular solutions}, Adv. Math. \textbf{192} (2005), no.~1,
  135--182. \MR{MR2122282 (2006j:53078)}

\bibitem{kapouleas:bulletin}
Nikolaos Kapouleas, \emph{Constant mean curvature surfaces in {E}uclidean
  three-space}, Bull. Amer. Math. Soc. (N.S.) \textbf{17} (1987), no.~2,
  318--320. \MR{88g:53013}

\bibitem{kapouleas:annals}
\bysame, \emph{Complete constant mean curvature surfaces in {E}uclidean
  three-space}, Ann. of Math. (2) \textbf{131} (1990), no.~2, 239--330.
  \MR{93a:53007a}

\bibitem{kapouleas:cmp}
\bysame, \emph{Slowly rotating drops}, Comm. Math. Phys. \textbf{129} (1990),
  no.~1, 139--159. \MR{91c:76024}

\bibitem{kapouleas:jdg}
\bysame, \emph{Compact constant mean curvature surfaces in {E}uclidean
  three-space}, J. Differential Geom. \textbf{33} (1991), no.~3, 683--715.
  \MR{93a:53007b}

\bibitem{krantz}
Steven~G. Krantz and Harold~R. Parks, \emph{A primer of real analytic
  functions}, second ed., Birkh\"auser Advanced Texts: Basler Lehrb\"ucher.
  [Birkh\"auser Advanced Texts: Basel Textbooks], Birkh\"auser Boston Inc.,
  Boston, MA, 2002. \MR{MR1916029 (2003f:26045)}

\bibitem{le:min:legn}
H{\^o}ng-V{\^a}n L{\^e}, \emph{A minimizing deformation of {L}egendrian
  submanifolds in the standard sphere}, Differential Geom. Appl. \textbf{21}
  (2004), no.~3, 297--316. \MR{MR2091366 (2005g:53171)}

\bibitem{y:lee}
Yng-Ing Lee, \emph{Embedded special {L}agrangian submanifolds in {C}alabi-{Y}au
  manifolds}, Comm. Anal. Geom. \textbf{11} (2003), no.~3, 391--423. \MR{2 015
  752}

\bibitem{marle}
Paulette Libermann and Charles-Michel Marle, \emph{Symplectic geometry and
  analytical mechanics}, Mathematics and its Applications, vol.~35, D. Reidel
  Publishing Co., Dordrecht, 1987. \MR{88c:58016}

\bibitem{mcduff:salamon}
Dusa McDuff and Dietmar Salamon, \emph{Introduction to symplectic topology},
  second ed., Oxford Mathematical Monographs, The Clarendon Press Oxford
  University Press, New York, 1998. \MR{2000g:53098}

\bibitem{mcintosh:slg}
Ian McIntosh, \emph{Special {L}agrangian cones in {$\mathbb{C}^3$} and
  primitive harmonic maps}, J. London Math. Soc. (2) \textbf{67} (2003), no.~3,
  769--789. \MR{1 967 705}

\bibitem{schoen}
R.~Schoen, \emph{The existence of weak solutions with prescribed singular
  behavior for a conformally invariant scalar equation}, Comm. Pure Appl. Math.
  \textbf{41} (1988), no.~3, 317--392. \MR{89e:58119}

\bibitem{schoen:wolfson:1996}
R.~Schoen and J.~Wolfson, \emph{Minimizing volume among {L}agrangian
  submanifolds}, Differential equations: La Pietra 1996 (Florence), Proc.
  Sympos. Pure Math., vol.~65, Amer. Math. Soc., Providence, RI, 1999,
  pp.~181--199. \MR{99k:53130}

\bibitem{schoen:wolfson}
\bysame, \emph{Minimizing area among {L}agrangian surfaces: the mapping
  problem}, J. Differential Geom. \textbf{58} (2001), no.~1, 1--86.
  \MR{2003c:53119}

\bibitem{simon:book}
Leon Simon, \emph{Lectures on geometric measure theory}, Proceedings of the
  Centre for Mathematical Analysis, Australian National University, vol.~3,
  Australian National University Centre for Mathematical Analysis, Canberra,
  1983. \MR{MR756417 (87a:49001)}

\bibitem{syz}
Andrew Strominger, Shing-Tung Yau, and Eric Zaslow, \emph{Mirror symmetry is
  {$T$}-duality}, Nuclear Phys. B \textbf{479} (1996), no.~1-2, 243--259.
  \MR{97j:32022}

\bibitem{thomas:yau}
R.~P. Thomas and S.-T. Yau, \emph{Special {L}agrangians, stable bundles and
  mean curvature flow}, Comm. Anal. Geom. \textbf{10} (2002), no.~5,
  1075--1113. \MR{1 957 663}

\end{thebibliography}
\end{document}